\documentclass[10pt]{article}
\usepackage{geometry}
\usepackage{mythmsubsub}
\usepackage{amscd}
\usepackage{tikz-cd}
\usepackage[shortlabels]{enumitem}
\usepackage[bookmarks]{hyperref}
\usepackage{amsthm}
\usepackage{titlesec}
\numberwithin{equation}{thm}

\setlist{leftmargin=1.2cm, topsep=0.5em, itemsep=0em}
\makeatletter
\@addtoreset{thm}{section}

\makeatother
\begin{document}
\newtheorem{nonec}{}[section]
\newtheorem{thmc}[nonec]{Theorem}
\title{Triangulated categories of motives over fs log schemes}
\author{Doosung Park}
\date{}
\maketitle
\begin{abstract}
  We introduce the notion of log motivic triangulated categories, which is the theoretical framework for understanding the motivic aspect of cohomology theories for fs log schemes. Then we study the Grothendieck six operations formalism for log motivic triangulated categories.
\end{abstract}
\tableofcontents
\titlelabel{\thetitle.\;\;}
\titleformat*{\section}{\center \bf}
\section{Introduction}
\begin{nonec}\label{0.1}
  This paper is devoted to study triangulated categories of motives over fs log schemes with rational coefficients and their six operations formalism. Throughout the introduction, let $\Lambda$ be a fixed ring. For simplicity, assume also that every log scheme we deal with in the introduction is a noetherian fs log schemes over the spectrum of a fixed field or Dedekind domain.
\end{nonec}
\begin{nonec}\label{0.2}
  As illustrated in \cite[16.2.18]{CD12}, $\mathbb{A}^1$-weak equivalences and the \'etale topology ``generate'' the right homotopy equivalences needed to produce the motivic cohomology. However, in the category of fs log schemes, we may need more homotopy equivalences. For example, consider morphisms
  \[Y\stackrel{g}\rightarrow X\stackrel{f}\rightarrow S\]
  of fs log schemes satisfying one of the following conditions:
  \begin{enumerate}[(a)]
    \item $f$ is exact log smooth, and $g$ is the verticalization (\ref{0.9}) $X^{\rm ver}\rightarrow X$ of $f$.
    \item $f$ is the identity, and $g$ is a pullback of $\mathbb{A}_u:\mathbb{A}_M\rightarrow \mathbb{A}_P$ where $u:M\rightarrow {\rm spec}\,P$ is a proper birational morphism of monoschemes.
    \item $f$ is the identity, the morphism $\underline{g}:\underline{Y}\rightarrow \underline{X}$ of underlying schemes is an isomorphism, and the homomorphism
    \[\overline{\mathcal{M}}_{Y,\overline{y}}^{\rm gp}\rightarrow \overline{\mathcal{M}}_{X,\overline{g(y)}}^{\rm gp}\]
    of groups is an isomorphism for any point $y$ of $Y$.
    \item $f$ is the projection $S\times \mathbb{A}_\mathbb{N}\rightarrow S$, and $g$ is the $0$-section $S\times {\rm pt}_\mathbb{N}\rightarrow S\times \mathbb{A}_\mathbb{N}$ where ${\rm pt}_\mathbb{N}$ denotes the reduced strict closed subscheme of $\mathbb{A}_\mathbb{N}$ whose image is the origin.
  \end{enumerate}
  For each type (a)--(d), we should expect that $g:Y\rightarrow X$ is {\it homotopy equivalent} over $S$ in some sense because the Betti realization of $g$ seems to be homotopy equivalent over the Betti realization of $S$. It is not clear that $\mathbb{A}^1$-weak equivalences and the \'etale topology can make such morphisms of the types (a)--(d) as homotopy equivalences.
\end{nonec}
\begin{nonec}\label{0.4}
  We also need the Grothendieck six operations formalism. Let $\mathscr{T}$ be a triangulated fibered category over the category of fs log schemes. The formalism should contain the following information.
  \begin{enumerate}[(1)]
    \item There exist 3 pairs of adjoint functors as follows:
    \[\begin{split}
      f^*:\mathscr{T}(S)\rightleftarrows \mathscr{T}(X):f_*&,f:X\rightarrow S\text{ any morphism},\\
      f_!:\mathscr{T}(X)\rightleftarrows \mathscr{T}(S):f^!&,f:X\rightarrow S\text{ any separated morphism of finite type},\\
      (\otimes,Hom)&,\text{ symmetric closed monoidal structure on }\mathscr{T}(X).
    \end{split}\]
    \item There exists a structure of a covariant (resp.\ contravariant) 2-functors on $f\mapsto f_*$, $f\mapsto f_!$  (resp.\ $f\mapsto f^*$, $f\mapsto f^!$).
    \item There exists a natural transformation
    \[\alpha_f:f_!\rightarrow f_*\]
    which is an isomorphism when $f$ is proper. Moreover, $\alpha$ is a morphism of $2$-functors.
    \item For any separated morphism of finite type $f:X\rightarrow S$, there exist natural transformations
    \[\begin{split}
      f_!K\otimes_S L&\stackrel{\sim}\longrightarrow f_!(K\otimes_X f^*L),\\
      Hom_S(f_!K,L)&\stackrel{\sim}\longrightarrow f_*Hom_X(K,f^!L),\\
      f^!Hom_S(L,M)&\stackrel{\sim}\longrightarrow Hom_X(f^*L,f^!M).
    \end{split}\]
    \item {\it Localization property.} For any strict closed immersion $i:Z\rightarrow S$ with complementary open immersion $j$, there exists a distinguished triangle of natural transformations as follows:
        \[j_!j^!\stackrel{ad'}\longrightarrow {\rm id}\stackrel{ad}\longrightarrow i_*i^*\stackrel{\partial_i}\longrightarrow j_!j^![1]\]
        where $ad'$ (resp.\ $ad$) denotes the counit (resp.\ unit) of the relevant adjunction.
    \item {\it Base change.} Consider a Cartesian diagram
    \[\begin{tikzcd}
      X'\arrow[r,"g'"]\arrow[d,"f'"]&X\arrow[d,"f"]\\
      S'\arrow[r,"g"]&S
    \end{tikzcd}\]
    of fs log schemes. Assume that one of the following conditions is satisfied: $f$ is strict, $f$ is exact log smooth, $g$ is strict, or $g$ is exact log smooth. Then there exists a natural isomorphism
    \[g^*f_!\longrightarrow f_!'g'^*.\]
    \item {\it Lefschetz duality.} Let $f:X\rightarrow S$ be an exact log smooth morphism of fs log schemes of relative dimension $d$, and let $j:X^{\rm ver/f}\rightarrow X$ denote its verticalization of $X$ via $f$. Then there exist natural isomorphisms
        \[j_*j^*f^!(-d)[-2d]\stackrel{\sim}\longrightarrow f^*,\]
        \[f^!\stackrel{\sim}\longrightarrow j_\sharp j^*f^*(d)[2d].\]
  \end{enumerate}
  Here, the formulations (1)--(5) are extracted from \cite[Introduction A.5.1]{CD12}.
\end{nonec}
\begin{nonec}
  In (\ref{2.9.1}), borrowing a terminology from \cite[2.4.45]{CD12}, we introduce the notion of {\it log motivic triangulated category.} This is the theoretical framework for understanding the motivic aspect of cohomology theories for log schemes. The following is our first main theorem.
\end{nonec}
\begin{thmc}[\ref{2.9.3}]\label{0.5}
  A log motivic triangulated category satisfies the properties (1)--(6) in (\ref{0.4}) and the homotopy properties (Htp--5), (Htp--6), and (Htp--7) in (\ref{2.4.1}).
\end{thmc}
\begin{nonec}\label{0.11}
  In (\ref{0.2}), the morphisms $g:Y\rightarrow X$ of types (c) and (d) are homotopy equivalent over $S$ in the sense of (Htp--5) and (Htp--6). See (\ref{2.4.1}) for the precise statements. The morphisms of types (a) and (b) are inverted in the definition of log motivic triangulated categories.
\end{nonec}
\begin{nonec}\label{0.6}
  In \cite[5.1]{Nak97}, the proper base change theorem is proved in the context of the derived category of Kummer log \'etale sheaves with a more general condition than that of our formalism (6), but we do not know that such a generalization is possible to our situation.
\end{nonec}
\begin{nonec}
    We do not prove (7) in (\ref{0.4}) for log motivic triangulated categories. The following is a weaker version for log motivic derivators (\ref{10.1.8}), which is called the {\it Poincar\'e duality.} Log motivic derivators are extensions of log motivic trianglated categories to diagrams of fs log schemes.
  \begin{thmc}[\ref{9.7.2}] Let $\mathscr{T}$ be a log motivic derivator satisfying strict \'etale descent, and let $f:X\rightarrow S$ be a separated vertical exact log smooth morphism of fs log schemes of relative dimension $d$. Then there exist a natural isomorphism
      \[f^!(-d)[-2d]\stackrel{\sim}\longrightarrow f^*.\]
  \end{thmc}
\end{nonec}
\begin{nonec}
  {\it Organization.} In Section 2, we first review the notion of premotivic triangulated categories. Then we review properties of morphisms in \cite{Ayo07} and \cite{CD12}. Many properties of morphisms in them are trivially generalized to properties for {\it strict} morphisms. We end this section by introducing the notion of log motivic triangulated categories.

  In Section 3, we discuss results on log schemes and motives that will be needed in the later sections.

  In Section 4, we construct purity transformations. Let $f:X\rightarrow S$ be a vertical exact log smooth morphism of fs log schemes of relative dimension $d$. Unlike the case of usual schemes, the diagonal morphism
  \[\underline{X}\rightarrow \underline{X}\times_{\underline{S}}\underline{X}\]
  of underlying schemes is {\it not} a regular embedding in general. Hence we cannot apply the theorem of Morel and Voevodsky \cite[2.4.35]{CD12}. To resolve this obstacle, we assume that the diagonal morphism $X\rightarrow X\times_S X$ has a compactified version of an exactification
  \[X\stackrel{c}\rightarrow E\rightarrow X\times_S X\]
  in some sense. Then $c$ becomes a strict regular embedding, so when $f$ is a proper exact log smooth morphism, we can apply [loc.\ cit] to construct the purity transformation
  \[f_\sharp \longrightarrow f_*(d)[2d].\]

  In Section 5, we introduce the notions of the semi-universal and universal support properties, which are generalizations of the support property for {\it non proper} morphisms. Then we prove that Kummer log smooth morphisms satisfy the semi-universal support property. We next prove that morphisms satisfying the semi-universal support property enjoy some good properties. Then we prove the semi-universal support property for $\mathbb{A}_\mathbb{\theta}:\mathbb{A}_{\mathbb{N}^2}\rightarrow \mathbb{A}_\mathbb{N}$ where $\theta:\mathbb{N}\rightarrow \mathbb{N}\oplus \mathbb{N}$ denotes the diagonal morphism and the projection $\mathbb{A}_\mathbb{N}\times {\rm pt}_\mathbb{N}\rightarrow {\rm pt}_\mathbb{N}$. We end this section by proving the support property under the additional axiom (ii) of (\ref{2.9.1}).

  In Section 6, we prove the base change and various homotopy properties.

  In Section 7, we select axioms of algebraic derivators to define the notion of premotivic triangulated prederivators, and we also define log motivic derivators. We prove several consequences of the axioms.

  In Section 8, we introduce the notion of compactified exactifications. Applying these to various natural transformations defined in Section 4, we construct the Poincar\'e duality for vertical exact log smooth separated morphism $f:X\rightarrow S$ with an fs chart having some conditions, and we show the purity. Then we collect the local constructions of purity transformations using the notion of premotivic triangulated prederivators, and we discuss its canonical version.
\end{nonec}
\begin{nonec}\label{0.7}
  {\it General terminology and conventions.}
  \begin{enumerate}[(1)]
  \item
  We fix a base fs log scheme $\mathcal{S}$. Then we fix a full subcategory $\mathscr{S}$ of the category of noetherian fs log schemes over $\mathcal{S}$ such that
  \begin{enumerate}[(i)]
    \item $\mathscr{S}$ is closed under finite sums and pullbacks via morphisms of finite type,
    \item if $S$ belongs to $\mathscr{S}$ and $X\rightarrow S$ is strict quasi-projective, then $X$ belongs to $\mathscr{S}$,
    \item if $S$ belongs to $\mathscr{S}$, then $S\times \mathbb{A}_M$ belongs to $\mathscr{S}$ for every fs monoscheme $M$,
    \item If $S$ belongs to $\mathscr{S}$, then $\underline{S}$ is belongs to $\mathscr{S}$,
    \item for any separated morphism $f:X\rightarrow S$ of $\mathscr{S}$-schemes, the morphism $\underline{f}:\underline{X}\rightarrow \underline{S}$ of underlying schemes admits a compactification in the sense of \cite[3.2.5]{SGA4}, i.e., we have a factorization
        \[\underline{X}\rightarrow Y\rightarrow \underline{S}\]
        in $\mathscr{S}$ such that the first arrow is an open immersion and the second arrow is a strict proper morphism.
  \end{enumerate}
   For example, as in \cite[2.0]{CD12}, $\mathcal{S}$ can be the spectrum of a field or Dedekind domain, and then $\mathscr{S}$ can be the category of noetherian fs log schemes over $\mathcal{S}$.
  \item When $S$ is an object of $\mathscr{S}$, we say that $S$ is an $\mathscr{S}$-scheme.
  \item When $f$ is a morphism in a class $\mathscr{P}$ of morphisms in $\mathscr{S}$, we say that $f$ is a $\mathscr{P}$-morphism.
  \item If we have an adjunction $\alpha:\mathcal{C}\leftrightarrows \mathcal{D}:\beta$ of categories, then the unit ${\rm id}\rightarrow \beta\alpha$ is denoted by $ad$, and the counit $\alpha\beta\rightarrow {\rm id}$ is denoted by $ad'$.
  \item We mainly deal with fs log schemes. The fiber products of fs log schemes and fiber coproducts of fs monoids are computed in the category of fs log schemes and fs monoids respectively unless otherwise stated.
  \item An abbreviation of the strict \'etale topology is s\'et.
\end{enumerate}
\end{nonec}
\begin{nonec}\label{0.8}
  {\it Terminology and conventions for monoids.}
  \begin{enumerate}[(1)]
    \item For a monoid $P$, we denote by ${\rm Spec}\,P$ the set of prime ideals of $P$. Note that $K\mapsto (P-K)$ for ideals $K$ of $P$ gives one-to-one correspondence between ${\rm Spec}\,P$ and the set of faces of $P$.
    \item A homomorphism $\theta:P\rightarrow Q$ of monoids is said to be {\it strict} if $\overline{\theta}:\overline{P}\rightarrow \overline{Q}$ is an isomorphism.
    \item A homomorphism $\theta:P\rightarrow Q$ of monoids is said to be {\it locally exact} if for any face $G$ of $Q$, the induced homomorphism $P_{\theta^{-1}(G)}\rightarrow Q_G$ is exact.
    \item Let $\theta:P\rightarrow Q$ be a homomorphism of monoids. A face $G$ of $Q$ is said to be {\it $\theta$-critical} if $\theta^{-1}(G)=\theta^{-1}(Q^*)$. Such a face $G$ is said to be {\it maximal $\theta$-critical} if $G$ is maximal among $\theta$-critical faces.
    \item A homomorphism $\theta:P\rightarrow Q$ of monoids is said to be {\it vertical} if the cokernel of $\theta$ computed in the category of integral monoids is a group. Equivalently, $\theta$ is vertical if $\theta(P)$ is not contained in any proper face of $Q$.
  \end{enumerate}
\end{nonec}
\begin{nonec}\label{0.9}
  {\it Terminology and conventions for log schemes.}
  \begin{enumerate}[(1)]
      \item For a monoid $P$, we denote by ${\bf A}_P$ the log scheme whose underlying scheme is $\mathbb{Z}[P]$ and the log structure is induced by the homomorphism $P\rightarrow \mathbb{Z}[P]$.
      \item For a monoid $P$ with an ideal $K$, we denote by $\mathbb{A}_{(P,K)}$ the strict closed subscheme of $\mathbb{A}_P$ whose underlying scheme is ${\rm Spec}\,\mathbb{Z}[P]/\mathbb{Z}[K]$.
      \item For a sharp monoid $P$, we denote by ${\rm pt}_P$ the log scheme $\mathbb{A}_{(P,P^+)}$.
      \item For a log scheme $S$, we denote by $\underline{S}$ the underlying scheme of $S$, and we denote by $\mathcal{M}_S$ the \'etale sheaf of monoids on $\underline{S}$ given by $S$.
      \item For a morphism $f:X\rightarrow S$ of log schemes, $\underline{f}$ denotes the morphism $\underline{X}\rightarrow \underline{S}$ of underlying schemes.
      \item For a morphism $f:X\rightarrow S$ of fs log schemes, we say that $f$ is a monomorphism if it is a monomorphism in the category of fs log schemes. Equivalently, $f$ is a monomorphism if and only if the diagonal morphism $X\rightarrow X\times_S X$ is an isomorphism.
      \item For a morphism $f:X\rightarrow S$ of fine log schemes and a point $x\in X$, we say that $f$ is {\it vertical at $x$} if the induced homomorphism
      \[\overline{\mathcal{M}}_{S,f(x)}\rightarrow \overline{\mathcal{M}}_{X,x}\]
      is vertical. Then the set
      \[X^{{\rm ver}/f}:=\{x\in X:f\text{ is vertical at }x\}\]
      is an open subset of $X$, and we regard it as an open subscheme of $X$. The induced morphism $X^{{\rm ver}/f}\rightarrow S$ is said to be the {\it verticalization} of $f$, and the induced morphism $X^{{\rm ver}/f}\rightarrow X$ is said to be the {\it verticalization} of $X$ via $f$.
  \end{enumerate}
\end{nonec}
\begin{nonec}\label{0.10}
  {\it Terminology and conventions for monoschemes.}
  \begin{enumerate}[(1)]
    \item For a monoid $P$, we denote by ${\rm spec}\,P$ the monoscheme associated to $P$ defined in \cite[II.1.2.1]{Ogu17}.
    \item For a monoscheme $M$ (see \cite[II.1.2.3]{Ogu17} for the definition of monoschemes), we denote by $\mathbb{A}_M$ the log scheme associated to $M$ defined in \cite[\S III.1.2]{Ogu17}.
  \end{enumerate}
\end{nonec}
\begin{none}
  {\it Acknowledgements.} Most part of this paper comes from the author's dissertation. The author is grateful to Brad Drew and Arthur Ogus, and Martin Olsson for helpful conversations or communications.
\end{none}
\section{Properties of premotivic triangulated categories}
\begin{none}
  Throughout this section, we fix a class $\mathscr{P}$ of morphisms of $\mathscr{S}$ containing all strict smooth morphisms and stable by compositions and pullbacks.
\end{none}
\subsection{Premotivic triangulated categories}
\begin{none}
  We will review $\mathscr{P}$-premotivic triangulated categories and exchange structures formulated in \cite[\S 1]{CD12}. First recall from \cite[A.1]{CD13} the definition of $\mathscr{P}$-primotivic triangulated categories as follows.
\end{none}
\begin{df}\label{1.1.3}
  A {\it $\mathscr{P}$-premotivic triangulated category} $\mathscr{T}$ over $\mathscr{S}$ is a fibered category over $\mathscr{S}$ satisfying the following properties:
  \begin{enumerate}
    \item[(PM--1)] For any object $S$ in $\mathscr{S}$, $\mathscr{T}(S)$ is a symmetric closed monoidal triangulated category.
    \item[(PM--2)] For any morphism $f:X\rightarrow S$ in $\mathscr{S}$, the functor $f^*$ is monoidal and triangulated, and admits a right adjoint denoted by $f_*$.
    \item[(PM--3)] For any $\mathscr{P}$-morphism $f:X\rightarrow S$, the functor $f^*$ admits a left adjoint denoted by $f_\sharp$.
    \item[($\mathscr{P}$-BC)] $\mathscr{P}$-base change: For any Cartesian square
    \[\begin{tikzcd}
      X'\arrow[r,"f'"]\arrow[d,"g'"]&X\arrow[d,"f"]\\
      S'\arrow[r,"g"]&S
    \end{tikzcd}\]
    in $\mathscr{S}$ with $f\in \mathscr{P}$, the exchange transformation defined by
    \[Ex:f'_\sharp g'^*\stackrel{ad}\longrightarrow f'_\sharp g'^*f^*f_\sharp \stackrel{\sim}\longrightarrow f'_\sharp f'^*g^*f_\sharp \stackrel{ad'}\longrightarrow g^*f_\sharp\]
    is an isomorphism.
    \item[($\mathscr{P}$-PF)] $\mathscr{P}$-projection formula: For any $\mathscr{P}$-morphism $f:X\rightarrow S$, and any objects $K$ in $\mathscr{T}(X)$ and $L$ in
        $\mathscr{T}(S)$, the exchange transformation defined by
    \[Ex:f_\sharp (K\otimes_X f^*L)\stackrel{ad}\longrightarrow f_\sharp (f^*f_\sharp K\otimes_X f^*L)\stackrel{\sim}\longrightarrow f_\sharp f^*(f_\sharp K\otimes_S
    L)\stackrel{ad'}\longrightarrow f_\sharp K\otimes_S L\]
    is an isomorphism.
  \end{enumerate}
  We denote by $Hom_S$ the internal Hom in $\mathscr{T}(S)$.
\end{df}
\begin{rmk}\label{1.1.4}
  Note that the axiom (PM--2) implies the following.
  \begin{enumerate}[(1)]
    \item For any morphism $f:X\rightarrow S$ in $\mathscr{S}$ and objects $K$ and $L$ of $\mathscr{T}(S)$, we have the natural transformation
    \begin{equation}\label{1.1.4.1}
      f^*(K)\otimes_X f^*(L)\stackrel{\sim}\longrightarrow f^*(K\otimes_S L)
    \end{equation}
    with the coherence conditions given in \cite[2.1.85, 2.1.86]{Ayo07}.
    \item For any morphism $f:X\rightarrow S$ in $\mathscr{S}$, we have the natural transformation
    \[f^*(1_S)\stackrel{\sim}\longrightarrow 1_X\]
    with the coherence conditions given in \cite[2.1.85]{Ayo07}.
  \end{enumerate}
\end{rmk}
\begin{df}\label{1.1.5}
  Let $\mathscr{T}$ be a $\mathscr{P}$-premotivic triangulated category.
  \begin{enumerate}[(1)]
    \item Let $f:X\rightarrow S$ be a $\mathscr{P}$-morphism in $\mathscr{S}$. We put $M_S(X)=f_\sharp 1_X$ in $\mathscr{T}(S)$. It is called the {\it motive} over $S$ represented by $X$.
    \item A {\it cartesian section} of $\mathscr{T}$ is the data of an object $A_S$ of $\mathscr{T}(S)$ for each object $S$ of $\mathscr{S}$ and of isomorphisms
    \[f^*A_S\stackrel{\sim}\longrightarrow A_X\]
     for each morphism $f:X\rightarrow S$ in $\mathscr{S}$, subject to following coherence conditions:
    \begin{enumerate}[(i)]
      \item the morphism ${\rm id}^*A_S\stackrel{\sim}\longrightarrow A_S$ is the identity morphism,
      \item if $g:Y\rightarrow X$ is another morphisms in $\mathscr{S}$, then the diagram
      \[\begin{tikzcd}
        g^*f^*A_S\arrow[r,"\sim"]\arrow[d,"\sim"]&g^*A_X\arrow[r,"\sim"]&A_Y\arrow[d,"{\rm id}"]\\
        (gf)^*A_S\arrow[rr,"\sim"]&&A_Y
      \end{tikzcd}\]
      in $\mathscr{T}(Y)$ commutes.
    \end{enumerate}
      The tensor product of two cartesian sections is defined termwise.
    \item A set of {\it twists} $\tau$ for $\mathscr{T}$ is a set of Cartesian sections of $\mathscr{T}$ stable by tensor product. For short, we say also that $\mathscr{T}$ is $\tau$-twisted.
  \end{enumerate}
\end{df}
\begin{none}\label{1.1.6}
  Let $i$ be an object of $\tau$. It defines a section $i_S$ for each object $S$ of $\mathscr{S}$, and for an object $K$ of $\mathscr{T}(S)$, we simply put
  \[K\{i\}=K\otimes_S i_S.\]
  Then when $i,j\in \tau$, we have
  \[K\{i+j\}=(K\{i\})\{j\}.\]
  Note also that by (\ref{1.1.4}(1)), for a morphism $f:X\rightarrow S$ in $\mathscr{S}$, we have the natural isomorphism
  \[f^*(K\{i\})\stackrel{\sim}\longrightarrow (f^*K)\{i\}.\]
\end{none}
\begin{none}\label{1.1.7}
  Let $\mathscr{T}$ be a $\mathscr{P}$-premotivic triangulated category. Consider a commutative diagram
  \[\begin{tikzcd}
    X'\arrow[r,"f'"]\arrow[d,"g'"]&X\arrow[d,"f"]\\
    S'\arrow[r,"g"]&S
  \end{tikzcd}\]
  in $\mathscr{S}$. We associates several exchange transformations as follows.
  \begin{enumerate}
    \item[(1)] We obtain the exchange transformation
    \[f^*g_*\stackrel{Ex}\longrightarrow g'_*f'^*\]
    by the adjunction of the exchange transformation
    \[f'_\sharp g'^*\stackrel{Ex}\longrightarrow g^*f_\sharp\]
    in (\ref{1.1.3}). Note that it is an isomorphism when $f$ is a $\mathscr{P}$-morphism by ($\mathscr{P}$-BC).
    \item[(2)] Assume that $f$ is a $\mathscr{P}$-morphism. Then we obtain the exchange transformation
    \[Ex:f_\sharp g_*\stackrel{ad}\longrightarrow f_\sharp g_*f'^*f_\sharp'\stackrel{Ex^{-1}}\longrightarrow f_\sharp f^*g_*'f_\sharp'\stackrel{ad'}\longrightarrow g_*' f_\sharp'.\]
    \item[(3)] Assume that $g_*$ and $g_*'$ have right adjoints, denoted by $g^!$ and $g'^!$ respectively. If the exchange transformation
    \[f^*g_*\stackrel{Ex}\longrightarrow g'_*f'^*\]
    is an isomorphism, then we obtain the exchange transformation
    \[Ex:f'^*g'^!\stackrel{ad}\longrightarrow g^!g_*f'^*g'^!\stackrel{Ex^{-1}}\longrightarrow g^!f^*g_*'g'^!\stackrel{ad'}\longrightarrow g^!f^*.\]
    \item[(4)] For objects $K$ of $\mathscr{T}(X)$ and $L$ of $\mathscr{T}(S)$, we obtain the exchange transformation
    \[Ex:f_*K\otimes_S L\stackrel{ad}\longrightarrow f_*f^*(f_*K\otimes_S L)\stackrel{\sim}\longrightarrow f_*(f^*f_*K\otimes_X f^*L)\stackrel{ad'}\longrightarrow f_*(K\otimes_X
    f^*L).\]
    \item[(5)] For objects $K$ of $\mathscr{T}(S)$ and $L$ of $\mathscr{T}(X)$, we obtain the natural isomorphism
    \[Ex:Hom_S(K,f_*L)\longrightarrow f_*Hom_T(f^*K,L)\]
    by the adjunction of (\ref{1.1.4.1}).
    \item[(6)] Assume that $f$ is a $\mathscr{P}$-morphism. For objects $K$ and $L$ of $\mathscr{T}(S)$ and $K'$ of $\mathscr{T}(X)$, we obtain the exchange transformations
    \[f^*Hom_S(K,L)\stackrel{Ex}\longrightarrow Hom_X(f^*K,f^*L),\]
    \[Hom_S(f_\sharp K',L)\stackrel{Ex}\longrightarrow f_*Hom_X(K',f^*L)\]
    by the adjunction of the $\mathscr{P}$-projection formula.
    \item[(7)] Assume that $f$ is a $\mathscr{P}$-morphism and that the diagram is Cartesian. Then we obtain the exchange transformation
    \[Ex:M_{S'}(X')=f_\sharp'1_{X'}\stackrel{\sim}\longrightarrow f_\sharp'g'^*1_{X}\stackrel{Ex}\longrightarrow g^*f_\sharp 1_{X}=g^*M_S(X).\]
    Note that it is an isomorphism.
    \item[(8)] Assume that $f$ and $g$ are $\mathscr{P}$-morphisms and the the diagram is Cartesian. Then we obtain the composition
    \[\begin{split}
    M_S(X\times_S S')&=f_\sharp g_\sharp'f'^*1_S'\stackrel{Ex}\longrightarrow f_\sharp f^*g_\sharp 1_{S'}\\
    &\stackrel{\sim}\longrightarrow f_\sharp (1_X\otimes_X f^*g_\sharp 1_{S'} ) \stackrel{Ex}\longrightarrow f_\sharp 1_X\otimes_S g_\sharp 1_{S'}=M_S(X)\otimes M_S(S')
    \end{split}.\]
    Note that it is an isomorphism.
    \item[(9)] Assume that $\mathscr{T}$ is $\tau$-twisted and that $f$ is a $\mathscr{P}$-morphism. For $i\in \tau$ and an object $K$ of $\mathscr{T}(X)$, we obtain the exchange
        transformation
    \[Ex:f_\sharp ((-)\{i\})\stackrel{\sim}\longrightarrow f_\sharp((-)\otimes_X f^*1_S\{i\})\stackrel{Ex}\longrightarrow (f_\sharp (-))\{i\}.\]
    Note that it is an isomorphism.
    \item[(10)] Assume that $\mathscr{T}$ is $\tau$-twisted. For $i\in \tau$, we obtain the exchange transformation
    \[Ex:(f_*(-))\{i\}\stackrel{\sim}\longrightarrow f_*(-)\otimes_S 1_S\{i\}\stackrel{Ex}\longrightarrow f_*((-)\otimes_X f^*1_S\{i\})=f_*((-)\{i\}).\]
    If twists are $\otimes$-invertible, then it is an isomorphism since its right adjoint is the natural isomorphism
    \[f^*((-)\{-i\})\stackrel{\sim}\longrightarrow (f^*(-))\{-i\}.\]
  \end{enumerate}
\end{none}
\subsection{Elementary properties}\label{Sec2.0}
\begin{none}
  From \S \ref{Sec2.0} to \S \ref{Sec2.10}, let $\mathscr{T}$ be a $\mathscr{P}$-premotivic triangulated category.
\end{none}
\begin{none}
  In \cite{Ayo07} and \cite{CD12}, the adjoint property, base change property, $\mathbb{A}^1$-homotopy property, localization property, projection formula, purity, $t$-separated property, stability, and support property are discussed. Many of them can be trivially generalized to properties for {\it strict} morphisms. We also introduce base change properties for non strict morphisms and other homotopy properties. In the last section, we introduce the notion of log motivic triangulated categories, which will be the central topic in later sections.
\end{none}
\begin{none}\label{2.1.1}
  Recall from \cite[\S 2.1, 2.2.13]{CD12} the following definitions.
  \begin{enumerate}[(1)]
    \item We say that $\mathscr{T}$ is {\it additive} if for any $\mathscr{S}$-schemes $S$ and $S'$, the obvious functor
    \[\mathscr{T}(S\amalg S')\rightarrow \mathscr{T}(S)\times\mathscr{T}(S')\]
    is an equivalence of categories.
    \item Let $f:X\rightarrow P$ be a proper morphism of $\mathscr{S}$-schemes. We say that $f$ satisfies the {\it adjoint property}, denoted by (${\rm Adj}_f$), if the functor
    \[f_*:\mathscr{T}(X)\rightarrow \mathscr{T}(S)\]
    has a right adjoint. When $({\rm Adj}_f)$ is satisfied for any proper morphism $f$, we say that $\mathscr{T}$ satisfies the {\it adjoint property}, denoted by (Adj).
    \item Let $t$ be a topology on $\mathscr{S}$ generated by a pretopology $t_0$ on $\mathscr{S}$. We say that $T$ is {\it $t$-seperated}, denoted by ($t$-sep), if for any $t_0$-cover $\{u_i:X_i\rightarrow S\}_{i\in I}$ of $S$, the family of functors $(f_i^*)_{i\in I}$ is conservative.
  \end{enumerate}
\end{none}
\begin{df}
  A class $\mathcal{G}$ of objects of a triangulated category $\mathcal{T}$ is called {\it generating} if the family of functors
    \[{\rm Hom}_\mathcal{T}(X[n],-)\]
    for $X\in \mathcal{G}$ and $n\in \mathbb{Z}$ is conservative.
\end{df}
\begin{df}
  Let $\mathscr{T}$ be a $\mathscr{P}$-premotivic triangulated category over $\mathscr{S}$. We say that $\mathscr{T}$ is {\it generated by $\mathscr{P}$ and $\tau$} if for any object $S$ of $\mathscr{S}$, the family of objects of the form
    \[M_S(X)\{i\}\]
    for a $\mathscr{P}$-morphism $X\rightarrow S$ and $i\in \tau$ generates $\mathscr{T}(S)$.
\end{df}
\begin{none}\label{2.1.2}
  Let $t$ be a topology on $\mathscr{S}$ generated by a pretopology $t_0$ on $\mathscr{S}$ such that any $t_0$-cover is consisted with $\mathscr{P}$-morphisms. Assume that $\mathscr{T}$ satisfies ($t$-sep) and that $\mathscr{T}$ is generated by $\mathscr{P}$ and $\tau$. Let $S$ be an $\mathscr{S}$-scheme, and let $\mathscr{P}'/S$ be a class of $\mathscr{P}$-morphisms $X\rightarrow S$ such that for any $\mathscr{P}$-morphism $g:Y\rightarrow S$, there is a $t$-cover $\{u_i:Y_i\rightarrow Y\}_{i\in I}$ such that each composition $gu_i:Y_i\rightarrow S$ is in $\mathscr{P}'/S$. In this setting, we will show that the family of objects of the form
  \[M_S(X)\{i\}\]
  for morphism $X\rightarrow S$ in $\mathscr{P}'/S$ and $i\in \tau$ generates $\mathscr{T}(S)$.

  Since $\mathscr{T}$ is generated by $\mathscr{P}$ and $\tau$, the family of functors
  \[{\rm Hom}_{\mathscr{T}(S)}(M_S(X)\{i\},-)={\rm Hom}_{\mathscr{T}(X)}(1_X\{i\},f^*(-))\]
  for $\mathscr{P}$-morphism $f:X\rightarrow S$ and $i\in \tau$ is conservative. By assumption, there is a $t_0$-cover $\{u_j:X_j\rightarrow X\}_{j\in I}$ such that each composition $fu_j:X_j\rightarrow S$ is in $\mathscr{P}'/S$. Applying ($t$-sep), we see that the family of functors
  \[{\rm Hom}_{\mathscr{T}(X_j)}(u_j^*1_X\{i\},u_j^*f^*(-))={\rm Hom}_{\mathscr{T}(S)}(M_S(X_j)\{i\},-)\]
  for $\mathscr{P}$-morphism $f:X\rightarrow S$, $j\in I$, and $i\in \tau$ is conservative. This implies the assertion.
\end{none}
\subsection{Localization property}
\begin{none}\label{2.2.1}
  Let $i:Z\rightarrow S$ be a strict closed immersion of $\mathscr{S}$-schemes, and let $j:U\rightarrow S$ be its complement. Recall from (\ref{1.1.3}) that $\mathscr{T}$ satisfies ($\mathscr{P}$-BC). According to \cite[2.3.1]{CD12}, we have the following
  consequences of ($\mathscr{P}$--BC):
  \begin{enumerate}[(1)]
    \item the unit ${\rm id}\stackrel{ad}\longrightarrow j^*j_\sharp$ is an isomorphism,
    \item the counit $j^*j_*\stackrel{ad'}\longrightarrow {\rm id}$ is an isomorphism,
    \item $i^*j_\sharp=0$,
    \item $j^*i_*=0$,
    \item the composition $j_\sharp j^*\stackrel{ad'}\longrightarrow {\rm id}\stackrel{ad}\longrightarrow i_*i^*$ is zero.
  \end{enumerate}
\end{none}
\begin{df}\label{2.2.2}
  We say that $\mathscr{T}$ satisfies the {\it localization property}, denoted by (Loc), if
  \begin{enumerate}[(1)]
    \item $\mathscr{T}(\emptyset)=0$,
    \item For any {\it strict} closed immersion $i$ of $\mathscr{S}$-schemes and its complement $j$, the pair of functors $(j^*,i^*)$ is conservative, and the conunit
        $i^*i_*\stackrel{ad'}\longrightarrow {\rm id}$ is an isomorphism.
  \end{enumerate}
\end{df}
\begin{none}\label{2.2.3}
  Assume that $\mathscr{T}$ satisfies (Loc). Consequences formulated in \cite[\S 2.3]{CD12} and in the proof of \cite[3.3.4]{CD12} are as follows.
  \begin{enumerate}[(1)]
    \item For any closed immersion $i$ of $\mathscr{S}$-schemes, the functor $i_*$ admits a right adjoint $i^!$.
    \item For any closed immersion $i$ of $\mathscr{S}$-schemes and its complement $j$, there exists a unique natural transformation $\partial_i:i_*i^*\rightarrow j_\sharp j^*[1]$
        such that the triangle
        \[j_\sharp j^*\stackrel{ad'}\longrightarrow {\rm id}\stackrel{ad}\longrightarrow i_*i^*\stackrel{\partial_i}\longrightarrow j_\sharp j^*[1]\]
        is distinguished.
    \item For any closed immersion $i$ of $\mathscr{S}$-schemes and its complement $j$, there exists a unique natural transformation $\partial_i:j_*j^*\rightarrow i_* i^![1]$ such
        that the triangle
        \[i_* i^!\stackrel{ad'}\longrightarrow {\rm id}\stackrel{ad}\longrightarrow j_*j^*\stackrel{\partial_i}\longrightarrow i_* i^![1]\]
        is distinguished.
    \item Let $S$ be an $\mathscr{S}$-scheme, and let $S_{red}$ denote the reduced scheme associated with $S$. The closed immersion $\nu:S_{red}\rightarrow S$ induces an equivalence
        of categories
        \[\nu^*:\mathscr{T}(S)\rightarrow \mathscr{T}(S_{red}).\]
    \item For any partition $(S_i\stackrel{\nu_i}\longrightarrow S)_{i\in I}$ of $S$ by locally closed subsets, the family of functors $(\nu_i^*)_{i\in I}$ is conservative.
    \item The category $\mathscr{T}$ is additive.
    \item The category $\mathscr{T}$ satisfies the strict Nisnevich separation property (in the case of usual schemes, note that (Loc) implies the cdh separation property).
    \item For any $\mathscr{S}$-scheme $S$ and any strict Nisnevich distinguished square
    \[\begin{tikzcd}
      X'\arrow[r,"g'"]\arrow[d,"f'"]&X\arrow[d,"f"]\\
      T'\arrow[r,"g"]&T
    \end{tikzcd}\]
    of $\mathscr{P}/S$-schemes, the associated Mayer-Vietoris sequence
    \[p_\sharp h_\sharp h^*p^*K\longrightarrow p_\sharp f_\sharp f^*p^*K\oplus p_\sharp g_\sharp g^*p^*K\longrightarrow p_\sharp p^*K\longrightarrow p_\sharp h_\sharp h^*p^*K[1]\]
    is a distinguished triangle for any object $K$ of $\mathscr{T}(S)$ where $h=fg'$ and $p$ denotes the structural morphism $T\rightarrow S$.
  \end{enumerate}
\end{none}
\subsection{Support property}
\begin{df}\label{2.3.1}
  Following \cite[2.2.5]{CD12}, we say that a proper morphism $f$ of $\mathscr{S}$-schemes satisfies the support property, denoted by (${\rm Supp}_f$), if for any Cartesian diagram
  \[\begin{tikzcd}
    X'\arrow[r,"g'"]\arrow[d,"f'"]&X\arrow[d,"f"]\\
    S'\arrow[r,"g"]&S
  \end{tikzcd}\]
  of $\mathscr{S}$-schemes such that $g$ is an open immersion, the exchange transformation
  \[Ex:g_\sharp f'_*\rightarrow f_*g'_\sharp\]
  is an isomorphism. We say that $\mathscr{T}$ satisfies the support property (resp.\ the strict support property), denoted by (Supp) (resp.\ (sSupp)), if the support property is satisfied for any proper morphism (resp.\ for any strict proper morphism).
\end{df}
\begin{none}\label{2.3.3}
  Let $f:X\rightarrow S$ be a separated morphism of $\mathscr{S}$-schemes. Choose a
  compactification
  \[\underline{X}\rightarrow \underline{S'}\rightarrow \underline{S}\]
  of $\underline{f}$. Following \cite[5.4]{Nak97}, $f$ can be factored as
  \[X\stackrel{f_1}\longrightarrow \underline{X}\times_{\underline{S}}S\stackrel{f_2}\longrightarrow \underline{S'}\times_{\underline{S}}S\stackrel{f_3}\longrightarrow S\]
  where
  \begin{enumerate}[(i)]
    \item $f_1$ denotes the morphism induced by $X\rightarrow \underline{X}$ and $X\rightarrow S$,
    \item $f_2$ denotes the morphism induced by $\underline{X}\rightarrow \underline{S'}$,
    \item $f_3$ denotes the projection.
  \end{enumerate}
  The morphisms $f_1$ and $f_3$ are proper, and the morphism $f_2$ is an open immersion. Hence we can use the argument of \cite[\S 2.2]{CD12}. A summary of [loc.\ cit] is as follows.

  Assume that $\mathscr{T}$ satisfies (Supp). For any separated morphism of finite type $f:X\rightarrow S$ of $\mathscr{S}$-schemes, we can associate a functor
  \[f_!:\mathscr{T}(X)\rightarrow \mathscr{T}(S)\]
  with the following properties:
  \begin{enumerate}[(1)]
    \item For any separated morphism of finite types $f:X\rightarrow Y$ and $g:Y\rightarrow Z$ of $\mathscr{S}$-schemes, there is a natural isomorphism
    \[(gf)_!\rightarrow g_!f_!\]
    with the usual cocycle condition with respect to the composition.
    \item For any separated morphism of finite type $f:X\rightarrow S$ of $\mathscr{S}$-schemes, there is a natural transformation
    \[f_!\rightarrow f_*,\]
    which is an isomorphism when $f$ is proper. Moreover, it is compatible with compositions.
    \item For any open immersion $j:U\rightarrow S$, there is a natural isomorphism
    \[f_!\rightarrow f_\sharp.\]
    It is compatible with compositions.
    \item For any Cartesian diagram
    \[\begin{tikzcd}
      X'\arrow[r,"g"]\arrow[d,"f'"]&X\arrow[d,"f"]\\
      S'\arrow[r,"g"]&S
    \end{tikzcd}\]
    of $\mathscr{S}$-schemes such that $f$ is separated of finite type, there is an exchange transformation
    \[Ex:g^*f_!\rightarrow f'_!g'^*\]
    compatible with horizontal and vertical compositions of squares such that the diagrams
    \[\begin{tikzcd}
      g^*f_!\arrow[r,"Ex"]\arrow[d,"\sim"]&f'_!g'^*\arrow[d,"\sim"]\\
      g^*f_*\arrow[r,"Ex"]&f'_*g'^*
    \end{tikzcd}\quad
    \begin{tikzcd}
      g^*f_!\arrow[r,"Ex"]\arrow[d,"\sim"]&f'_!g'^*\arrow[d,"\sim"]\\
      g^*f_\sharp\arrow[r,"Ex^{-1}"]&f'_\sharp g^*
    \end{tikzcd}\]
    of functors commutes when $f$ is proper in the first diagram and is open immersion in the second diagram.
    \item For any Cartesian diagram
    \[\begin{tikzcd}
      X'\arrow[r,"g"]\arrow[d,"f'"]&X\arrow[d,"f"]\\
      S'\arrow[r,"g"]&S
    \end{tikzcd}\]
    of $\mathscr{S}$-schemes such that $f$ is separated of finite type and $g$ is a $\mathscr{P}$-morphism, there is an exchange transformation
    \[Ex:g_\sharp f'_!\rightarrow f_!g'_\sharp\]
    compatible with horizontal and vertical compositions of squares such that the diagrams
    \[\begin{tikzcd}
      g_\sharp f'_!\arrow[r,"Ex"]\arrow[d,"\sim"]&f_!g'_\sharp\arrow[d,"\sim"]\\
      g_\sharp f'_*\arrow[r,"Ex"]&f_*g'_\sharp
    \end{tikzcd}\quad
    \begin{tikzcd}
      g_\sharp f'_!\arrow[r,"Ex"]\arrow[d,"\sim"]&f_!g'_\sharp\arrow[d,"\sim"]\\
      g_\sharp f'_\sharp\arrow[r,"\sim"]&f_\sharp g'_\sharp
    \end{tikzcd}\]
    of functors commutes when $f$ is proper in the first diagram and is open immersion in the second diagram.
    \item For any separated morphism of finite type $f:X\rightarrow S$ of $\mathscr{S}$-schemes, and for any $K\in \mathscr{T}(X)$ and $L\in \mathscr{T}(S)$, there is an exchange
        transformation
    \begin{equation}\label{2.3.3.1}
      Ex:f_!K\otimes_S L\rightarrow f_!(K\otimes_X f^*L)
    \end{equation}
    compatible with compositions such that the diagrams
    \[\begin{tikzcd}
      f_!K\otimes_S L\arrow[r,"Ex"]\arrow[d,"\sim"]&f_!(K\otimes_X f^*L)\arrow[d,"\sim"]\\
      f_!K\otimes_S L\arrow[r,"Ex"]&f_*(K\otimes_X f^*L)
    \end{tikzcd}\quad
    \begin{tikzcd}
      f_!K\otimes_S L\arrow[r,"Ex"]\arrow[d,"\sim"]&f_!(K\otimes_X f^*L)\arrow[d,"\sim"]\\
      f_\sharp K\otimes_S L\arrow[r,"Ex^{-1}"]&f_\sharp (K\otimes_X f^*L)
    \end{tikzcd}\]
    of functors commutes when $f$ is proper in the first diagram and is open immersion in the second diagram.
    \item Assume that $\mathscr{T}$ satisfies (Adj). For any separated morphism of finite type $f:X\rightarrow S$ of $\mathscr{S}$-schemes, the functor $f_!$ admits a right adjoint
    \[f^!:\mathscr{T}(S)\rightarrow\mathscr{T}(X).\]
  \end{enumerate}
\end{none}
\subsection{Homotopy properties}
\begin{df}\label{2.4.1}
  Let $S$ be an $\mathscr{S}$-scheme. Let us introduce the following homotopy properties.
  \begin{enumerate}
    \item[(Htp--1)] Let $f$ denote the projection $\mathbb{A}_S^1\rightarrow S$. Then the counit
    \[f_\sharp f^*\stackrel{ad'}\longrightarrow {\rm id}\]
    is an isomorphism.
    \item[(Htp--2)] Let $f:X\rightarrow S$ be an exact log smooth morphism of $\mathscr{S}$-schemes, and let $j:X^{\rm ver}\rightarrow X$ denote the verticalization of $X$ via $f$. Then the natural
        transformation
    \[f_\sharp j_\sharp j^* \stackrel{ad'}\longrightarrow f_\sharp\]
    is an isomorphism.
    \item[(Htp--3)] Let $S$ be an $\mathscr{S}$-scheme with an fs chart $P$, let $\theta:P\rightarrow Q$ be a vertical homomorphism of exact log smooth over $S$ type (see (\ref{0.1.2}) for the definition), and let $G$ be a $\theta$-critical face of $Q$. Consider the induced morphisms
        \[S\times_{\mathbb{A}_P}\mathbb{A}_{Q_G}\stackrel{j}\rightarrow S\times_{\mathbb{A}_P}\mathbb{A}_Q\stackrel{f}\rightarrow S.\]
        Then the natural transformation
        \[f_\sharp j_\sharp j^*f^*\stackrel{ad'}\longrightarrow f_\sharp f^*\]
        is an isomorphism.
    \item[(Htp--4)] Let $f:X\rightarrow S$ be a surjective proper log \'etale monomorphism of $\mathscr{S}$-schemes. Then the unit
        \[{\rm id}\stackrel{ad}\longrightarrow f_*f^*\]
        is an isomorphism.
    \item[(Htp--5)] Let $f:X\rightarrow S$ be a morphism of $\mathscr{S}$-schemes with the same underlying schemes such that the induced homomorphism
        $\overline{\mathcal{M}}_{S,f(x)}^{\rm gp}\rightarrow \overline{\mathcal{M}}_{X,x}^{\rm gp}$ is an isomorphism for all $x\in X$. Then $f^*$ is an equivalence of categories.
    \item[(Htp--6)] Let $S$ be an $\mathscr{S}$-scheme, let $f:S\times\mathbb{A}_\mathbb{N}\rightarrow S$ denote the projection, and let $i:S\times {\rm pt}_\mathbb{N}\rightarrow S\times \mathbb{A}_\mathbb{N}$ denote the 0-section. Then the natural transformation
        \[f_*f^*\stackrel{ad}\longrightarrow f_*i_*i^*f^*\]
        is an isomorphism.
    \item[(Htp--7)] Assume that $\mathscr{T}$ satisfies (Supp). Under the notations and hypotheses of (Htp--3), the natural transformation
        \[f_!j_\sharp j^*f^*\stackrel{ad'}\longrightarrow f_!f^*\]
        is an isomorphism.
  \end{enumerate}
\end{df}
\begin{none}
  (1) Note that the right adjoint versions of (Htp--1), (Htp--2), and (Htp--3) are as follows.
  \begin{enumerate}[(i)]
    \item Under the notations and hypotheses of (Htp--1), the unit
    \[{\rm id}\stackrel{ad}\longrightarrow f_*f^*\]
    is an isomorphism.
    \item Under the notations and hypotheses of (Htp--2), the natural transformation
    \[f^*\stackrel{ad}\longrightarrow j_*f^*f^*\]
    is an isomorphism.
    \item Under the notations and hypotheses of (Htp--3), the unit
    \[f_*f^*\stackrel{ad}\longrightarrow f_*j_*j^*f^*\]
    is an isomorphism.
  \end{enumerate}
  (2) Let $P$ be an fs monoid. For any proper birational morphism $u:M\rightarrow {\rm spec}\,P$ of monoschemes, the associated morphism $\mathbb{A}_u:\mathbb{A}_M\rightarrow \mathbb{A}_P$ is a surjective proper log \'etale monomorphism of $\mathscr{S}$-schemes.
\end{none}
\subsection{Purity}
\begin{df}\label{2.5.1}
  Let $S$ be an $\mathscr{S}$-scheme, let $p$ denote projection $\mathbb{A}_S^1\rightarrow S$, and let $a$ denote the zero section $S\rightarrow \mathbb{A}_S^1$. Then we denote by
  $1_S(1)$ the element $p_\sharp a_*1_S[2]$, and we say that $\mathscr{T}$ satisfies the {\it stability property}, denoted by (Stab), if $1_S(1)$ is $\otimes$-invertible.
\end{df}
\begin{rmk}\label{2.5.2}
  Note that our definition is different from the definition in \cite[2.4.4]{CD12}, but if we assume (Loc) and (Zar-Sep), then they are equivalent by the following result.
\end{rmk}
\begin{prop}\label{2.5.10}
  Assume that $\mathscr{T}$ satisfies (Loc), (Zar-Sep), and (Stab). Let $f:X\rightarrow S$ be a strict smooth separated morphism of $\mathscr{S}$-schemes, and let $i:S\rightarrow X$ be its section. Then the functor
  \[f_\sharp i_*\]
  is an equivalence of categories.
\end{prop}
\begin{proof}
  It follows from the implication (i)$\Leftrightarrow$(iv) of \cite[2.4.14]{CD12}.
\end{proof}
\begin{df}\label{2.5.3}
  Let $f:X\rightarrow S$ be a separated $\mathscr{P}$-morphism of $\mathscr{S}$-schemes. We denote by $a$ the diagonal morphism $X\rightarrow X\times_S X$ and by $p_2$ the second
  projection $X\times_S X\rightarrow X$. Then we put
  \[\Sigma_f=p_{2\sharp} a_*.\]
  If we assume (Adj), then we put
  \[\Omega_f=a^!p_2^*.\]
  Note that $\Sigma_f$ is left adjoint to $\Omega_f$.
\end{df}
\begin{df}\label{2.5.4}
  Let $f:X\rightarrow S$ be a $\mathscr{P}$-morphism of $\mathscr{S}$-schemess. Assume that ($f$ is proper) or ($f$ is separated and $\mathscr{T}$ satisfies (Supp)). Consider the Cartesian diagram
  \[\begin{tikzcd}
    X\times_S X\arrow[d,"p_1"]\arrow[r,"p_2"]&X\arrow[d,"f"]\\
    X\arrow[r,"f"]&S
  \end{tikzcd}\]
  of $\mathscr{S}$-schemes, and let $a:X\rightarrow X\times_S X$ denote the diagonal morphism. Following \cite[2.4.24]{CD12}, we define the natural transformation
  \[\mathfrak{p}_f:f_\sharp \stackrel{\sim}\longrightarrow f_\sharp p_{1!}a_*\stackrel{Ex}\longrightarrow f_!p_{2\sharp}a_*= f_!\Sigma_f.\]
  The right adjoint of $\mathfrak{p}_f$ is denoted by
  \[\mathfrak{q}_f:\Omega_ff^!\longrightarrow f^*.\]
\end{df}
\begin{df}\label{2.5.5}
  Let $f$ be a $\mathscr{P}$-morphism of $\mathscr{S}$-schemes, and assume that ($f$ is proper) or ($f$ is separated and $\mathscr{T}$ satisfies (Supp)). We say that $f$ is {\it pure}, denoted by (${\rm Pur}_f$), if the natural transformation $\mathfrak{p}_f$ is an isomorphism. Note that if we assume (Adj), then $f$ is pure if and only if $\mathfrak{q}_f$ is an isomorphism. We say that $\mathscr{T}$ satisfies the purity, denoted by (Pur), if $\mathscr{T}$ satisfies $({\rm Pur}_f)$ for any exact log smooth separated morphism $f$.
\end{df}
\begin{rmk}\label{2.5.6}
  Note that our definition is different from the definition in \cite[2.4.25]{CD12} in which the additional condition that $\Sigma_f$ is an isomorphism is assumed. However, if we assume (Loc), (Zar-Sep), and (Stab), then the definitions are equivalent by (\ref{2.5.10}).
\end{rmk}
\begin{thm}\label{2.5.8}
  Assume that $\mathscr{T}$ satisfies (Htp--1), (Loc), and (Stab). Then $\mathscr{T}$ satisfies (sSupp).
\end{thm}
\begin{proof}
  The conditions of the theorem of Ayoub \cite[2.4.28]{CD12} are satisfied, and the same proof of [loc.\ cit] can be applied even if $S$ is not a usual scheme. The consequence is that the
  projection $\mathbb{P}_S^1\rightarrow S$ is pure for any $\mathscr{S}$-scheme $S$. Then the conclusion follow from the proof of \cite[2.4.26(2)]{CD12}.
\end{proof}
\begin{thm}\label{2.5.7}
  Assume that $\mathscr{T}$ satisfies (Htp--1), (Loc), (Stab), and (Supp). Then any strict smooth separated morphism is pure.
\end{thm}
\begin{proof}
  As in the proof of (\ref{2.5.8}), $\mathbb{P}_S^1\rightarrow S$ is pure for any $\mathscr{S}$-scheme $S$. Then the conclusion follows from the proof of \cite[2.4.26(3)]{CD12}.
\end{proof}
\begin{thm}\label{2.5.9}
  Assume that $\mathscr{T}$ satisfies (Htp--1), (Loc), (Stab) and (Supp). Consider a Cartesian diagram
  \[\begin{tikzcd}
    X'\arrow[d,"f'"]\arrow[r,"g'"]&X\arrow[d,"f"]\\
    S'\arrow[r,"g"]&S
  \end{tikzcd}\]
  of $\mathscr{S}$-schemes such that $f$ is strict smooth separated and $g$ is separated. Then the exchange transformation
  \[f_\sharp g_!'\stackrel{Ex}\longrightarrow f_\sharp'g_!\]
  is an isomorphism.
\end{thm}
\begin{proof}
  It follows from (\ref{2.5.7}) and the proof of \cite[2.4.26(3)]{CD12}.
\end{proof}
\subsection{Base change property}
\begin{none}\label{2.6.1}
  Consider a Cartesian diagram
  \[\begin{tikzcd}
    X'\arrow[r,"g'"]\arrow[d,"f'"]&X\arrow[d,"f"]\\
    S\arrow[r,"g"]&S
  \end{tikzcd}\]
  of $\mathscr{S}$-schemes. When $f$ is proper, let us introduce the following base change properties.
  \begin{enumerate}
    \item[(B${\rm C}_{f,g}$)] The exchange transformation $g^*f_*\rightarrow f'_*g'^*$ is an isomorphism.
    \item[(BC--1')] For all $f$ and $g$ such that $f$ is strict and proper, (B${\rm C}_{f,g}$) is satisfied.
    \item[(BC--2')] For all $f$ and $g$ such that $f$ is an exact log smooth morphism and proper, (B${\rm C}_{f,g}$) is satisfied.
    \item[(BC--3')] For all $f$ and $g$ such that $g$ is strict and $f$ is proper, (B${\rm C}_{f,g}$) is satisfied.
    \item[(BC--4')] For all $f$ and $g$ such that $g$ is a $\mathscr{P}$-morphism and $f$ is proper, (B${\rm C}_{f,g}$) is satisfied.
  \end{enumerate}
  On the other hand, when $f$ is just assumed to be separated but $\mathscr{T}$ satisfies (Supp), we have the following base change properties.
  \begin{enumerate}
    \item[(B${\rm C}_{f,g}$)] The exchange transformation $g^*f_!\rightarrow f'_!g'^*$ is an isomorphism.
    \item[(BC--1)] For all $f$ and $g$ such that $f$ is strict, (B${\rm C}_{f,g}$) is satisfied.
    \item[(BC--2)] For all $f$ and $g$ such that $f$ is an exact log smooth morphism, (B${\rm C}_{f,g}$) is satisfied.
    \item[(BC--3)] For all $f$ and $g$ such that $g$ is strict, (B${\rm C}_{f,g}$) is satisfied.
    \item[(BC--4)] For all $f$ and $g$ such that $g$ is a $\mathscr{P}$-morphism, (B${\rm C}_{f,g}$) is satisfied.
  \end{enumerate}
\end{none}
\begin{prop}\label{2.6.7}
  If $\mathscr{T}$ satisfies (Loc), then (B${\rm C}_{f,g}$) is satisfied for all $f$ and $g$ such that $f$ is a strict closed immersion.
\end{prop}
\begin{proof}
  It follow from the proof of \cite[2.3.13(1)]{CD12}, but we repeat the proof for the convenience of reader. Let $h'$ denote the complement of $f'$. Then by (Loc), the pair $(f'^*,h'^*)$ of functors is conservative, so it suffices to show that the natural transformations
  \[f'^*g^*f_*\stackrel{ad}\longrightarrow f'^*f'_*g'^*,\]
  \[h'^*g^*f_*\stackrel{ad}\longrightarrow h'^*f_*'g'^*\]
  are isomorphisms. The first one is an isomorphism since the counits $f^*f_*\stackrel{ad}\longrightarrow {\rm id}$ and $f'^*f_*'\stackrel{ad'}\longrightarrow {\rm id}$ are isomorphisms by (Loc), and the second one is an isomorphism by (\ref{2.2.1}(4)).
\end{proof}
\begin{prop}\label{2.6.3}
  If $\mathscr{T}$ satisfies (Supp), then the property (BC--n') implies (BC--n) for $n=1,2,3,4$.
\end{prop}
\begin{proof}
  It follows from (\ref{2.3.3}(4)).
\end{proof}
\begin{prop}\label{2.6.4}
  If $\mathscr{T}$ satisfies (Supp), then the category $\mathscr{T}$ satisfies (BC--4).
\end{prop}
\begin{proof}
  It follows from (\ref{2.6.3}) and ($\mathscr{P}$-BC).
\end{proof}
\begin{prop}\label{2.6.5}
  Assume that
  \begin{enumerate}[(i)]
    \item $\mathscr{T}$ satisfies (Loc) and (Supp),
    \item for any $\mathscr{S}$-scheme $S$, the projection $\mathbb{P}_S^1\rightarrow S$ is pure.
  \end{enumerate}
  Then the category $\mathscr{T}$ satisfies (BC--1).
\end{prop}
\begin{proof}
  The conditions of \cite[2.4.26(2)]{CD12} are satisfied, and the same proof of [loc.\ cit] can be applied even if $S$ is an fs log scheme.
\end{proof}
\begin{prop}\label{2.6.6}
  Assume that $\mathscr{T}$ satisfies (Loc), (Supp), and (Zar-Sep). Then the category $\mathscr{T}$ satisfies (BC--3).
\end{prop}
\begin{proof}
  By (\ref{2.6.3}), it suffices to show (BC--3'). Consider a Cartesian diagram
  \[\begin{tikzcd}
    X'\arrow[r,"g'"]\arrow[d,"f'"]&X\arrow[d,"f"]\\
    S'\arrow[r,"g"]&S
  \end{tikzcd}\]
  of $\mathscr{S}$-schemes such that $g$ is strict and that $f$ is proper. By (Zar-Sep), the question is Zariski local on $S'$, so we reduce to the case when $S'$ is affine. Then the morphism $g$ is quasi-projective, so we reduce to the cases when
  \begin{enumerate}[(1)]
    \item $g$ is an open immersion,
    \item $g$ is a strict closed immersion,
    \item $g$ is the projection $\mathbb{P}_S^1\rightarrow S$.
  \end{enumerate}
  In the cases (1) and (3), we are done by (BC--4), so the remaining is the case (2). Hence assume that $g$ is a strict closed immersion.

  Let $h:S''\rightarrow S$ denote the complement of $g$, and consider the commutative diagram
  \[\begin{tikzcd}
    X'\arrow[r,"g'"]\arrow[d,"f'"]&X\arrow[d,"f"]\arrow[r,"h'",leftarrow]&X''\arrow[d,"f"]\\
    S'\arrow[r,"g"]&S\arrow[r,"h",leftarrow]&S''
  \end{tikzcd}\]
  of $\mathscr{S}$-schemes where each square is Cartesian. Since the pair $(g_*',h_\sharp')$ generates $\mathscr{T}(X)$ by (Loc), it suffices to show that the natural transformations
  \begin{equation}\label{2.6.6.1}
    g^*f_*g_*'\stackrel{Ex}\rightarrow f'_*g'^*g_*',\quad g^*f_*h_\sharp'\stackrel{Ex}\rightarrow f'_*g'^*h_\sharp'
  \end{equation}
  are isomorphisms.

  Consider the commutative diagram
  \[\begin{tikzcd}
    g^*f_*g_*'\arrow[r,"Ex"]\arrow[rd,"ad'"']&f_*'g'^*g_*'\arrow[d,"ad'"]\\
    &f_*'
  \end{tikzcd}\]
  of functors. The diagonal and vertical arrows are isomorphisms, so the horizontal arrow is an isomorphism. By (Supp), $f_*h_\sharp'\cong h_\sharp f_*$. Since $g^*h_\sharp=0$ and $g'^*h_\sharp'=0$, $g^*f_*h_\sharp'\cong 0\cong f_*'g'^*h_\sharp'$. Thus the natural transformations in (\ref{2.6.6.1}) are isomorphisms.
\end{proof}
\begin{prop}\label{2.6.8}
  Assume that $\mathscr{T}$ satisfies (Loc). Consider a Cartesian diagram
  \[\begin{tikzcd}
    X'\arrow[r,"g'"]\arrow[d,"f'"]&X\arrow[d,"f"]\\
    S'\arrow[r,"g"]&S
  \end{tikzcd}\]
  of $\mathscr{S}$-schemes where $f$ and $g$ are strict closed immersions such that $f(X)\cup g(X)=S'$. We put $h=fg'$. Then for any object $K$ of $\mathscr{T}(S)$, the commutative diagram
  \[\begin{tikzcd}
    K\arrow[r,"ad"]\arrow[d,"ad"]&f_*f^*K\arrow[d,"ad"]\\
    g_*g^*K\arrow[r,"ad"]&h_*h^*K
  \end{tikzcd}\]
  in $\mathscr{T}(S)$ is homotopy Cartesian.
\end{prop}
\begin{proof}
  Let $u:S''\rightarrow S$ denote the complement of $g$. Then $u$ factors through $X$ by assumption, and let $u':S''\rightarrow X$ denote the morphism. and consider the commutative diagram
  \[\begin{tikzcd}
    X'\arrow[r,"g'"]\arrow[d,"f'"]&X\arrow[d,"f"]\arrow[r,"u'",leftarrow]&X''\arrow[d,"g''"]\\
    S'\arrow[r,"g"]&S\arrow[r,"u",leftarrow]&S''
  \end{tikzcd}\]
  where each square is Since the pair $(g^*,u^*)$ of functors is conservative, it suffices to prove that the diagrams
  \[\begin{tikzcd}
    g^*K\arrow[r,"ad"]\arrow[d,"ad"]&g^*f_*f^*K\arrow[d,"ad"]\\
    g^*g_*g^*K\arrow[r,"ad"]&g^*h_*h^*K
  \end{tikzcd}\quad \begin{tikzcd}
    u^*K\arrow[r,"ad"]\arrow[d,"ad"]&u^*f_*f^*K\arrow[d,"ad"]\\
    u^*g_*g^*K\arrow[r,"ad"]&u^*h_*h^*K
  \end{tikzcd}\]
  are homotopy Cartesian. The first diagram is isomorphic to
  \[\begin{tikzcd}
    g^*K\arrow[d,"{\rm id}"]\arrow[r,"ad"]&f_*'h^*K\arrow[d,"{\rm id}"]\\
    g^*K\arrow[r,"ad"]&f_*'h^*K
  \end{tikzcd}\]
  by (Loc) and (\ref{2.6.7}), and it is homotopy Cartesian. For the second diagram, since its lower horizontal arrow is an isomorphism by ($\mathscr{P}$-BC), it suffices to show that the natural transformation
  \[u^*\stackrel{ad}\longrightarrow u^*f_*f^*K\]
  is an isomorphism. Since $u=fu'$, we have the natural transformations
  \[u'^*f^*\stackrel{ad}\longrightarrow u'^*f^*f_*f^*K\stackrel{ad'}\longrightarrow u'^*f^*K,\]
  whose composition is an isomorphism. The second arrow is an isomorphism by (Loc), so the first arrow is also an isomorphism.
\end{proof}
\subsection{Projection formula}
\begin{df}\label{2.10.1}
  For a proper morphism $f:X\rightarrow S$ of $\mathscr{S}$-schemes, we say that $f$ satisfies the {\it projection formula}, denoted by (${\rm PF}_f$), if the exchange transformation
  \[f_*K\otimes_X L\stackrel{Ex}\longrightarrow f_*(K\otimes_Y f^*L)\]
  is an isomorphism for any objects $K$ of $\mathscr{T}(X)$ and $L$ of $\mathscr{T}(S)$. We say that $\mathscr{T}$ satisfies the {\it projection formula}, denoted by (PF), if (${\rm PF}_f$) is satisfied for any proper morphism $f$.
\end{df}
\begin{none}\label{2.10.2}
  Assume that $\mathscr{T}$ satisfies (PF) and (Supp). Let $f:X\rightarrow Y$ be a separated morphism of $\mathscr{S}$-schemes. Then by the proof of \cite[2.2.14(5)]{CD12}, the exchange transformation
  \[f_!K\otimes_X L\stackrel{\sim}\longrightarrow f_!(K\otimes_Y f^*L)\]
  is an isomorphism for any objects $K$ of $\mathscr{T}(X)$ and $L$ of $\mathscr{T}(S)$. If we assume further (Adj), then by taking adjunctions of the above exchange transformation, we obtain the natural transformations
    \[\begin{split}
    Hom_S(f_!K,L)&\stackrel{\sim}\longrightarrow f_*Hom_X(K,f^!L),\\
    f^!Hom_X(L,M)&\stackrel{\sim}\longrightarrow Hom_X(f^*L,f^!M).
    \end{split}\]
  for any objects $K$ of $\mathscr{T}(X)$ and $L$ and $M$ of $\mathscr{T}(S)$.
\end{none}
\subsection{Orientation}
\begin{df}\label{2.8.1}
  Let $p:E\rightarrow S$ be a vector bundle of rank $n$ of $\mathscr{S}$-schemes, and let $i_0:S\rightarrow E$ denote its $0$-section. Then an isomorphism
  \[\mathfrak{t}_E:p_\sharp i_{0*}\longrightarrow 1_S(n)[2n]\]
  is said to be an {\it orientation} of $E$. When $\mathscr{T}$ satisfies (Adj) and (Stab), we denote by
  \[\mathfrak{t}_E':1_S(-n)[-2n]\longrightarrow i_0^!p^*\]
  its right adjoint.

  Recall from \cite[2.4.38]{CD12} that a collection $\mathfrak{t}$ of orientations for all vector bundles $E\rightarrow S$ of $\mathscr{S}$-schemes with the compatibility conditions (a)--(c) in [loc.\ cit] is said to be an {\it orientation} of $\mathscr{T}$.
\end{df}
\begin{none}\label{2.8.2}
  Note that by (\ref{2.5.10}), if $\mathscr{T}$ satisfies (Loc), (Zar-Sep), and (Stab), then any vector bundle has an orientation.
\end{none}
\subsection{Log motivic categories}\label{Sec2.10}
\begin{df}\label{2.9.1}
  Let $\mathscr{T}$ be a $eSm$-premotivic triangulated category. Borrowing a terminology from \cite[2.4.45]{CD12}, we say that $\mathscr{T}$ is a {\it log motivic triangulated category} if
  \begin{enumerate}[(i)]
    \item $\mathscr{T}$ satisfies (Adj), (Htp--1), (Htp--2), (Htp--3), (Htp--4), (Loc), (k\'et-Sep), and (Stab),
    \item $\mathscr{T}$ is generated by $eSm$ and $\tau$,
    \item for any $\mathscr{S}$-scheme $S$ with the trivial log structure, the morphism $S\times\mathbb{A}_\mathbb{N}\rightarrow S\times\mathbb{A}^1$ removing the log structure satisfies the support property.
  \end{enumerate}
\end{df}
\begin{none}
  In \cite[2.4.45]{CD12}, motivic triangulated category is defined, and in \cite[2.4.50]{CD12}, the six operations formalism is given for motivic triangulated cateogires. Following this spirit, we introduced our notion of log motivic triangulated category, which will satisfy the log version of the six operations formalism (1)--(5) in (\ref{0.4}).

  Now, we state our main theorem.
\end{none}
\begin{thm}\label{2.9.3}
  A log motivic triangulated category satisfies the properties (1)--(6) in (\ref{0.4}) and the homotopy properties (Htp--5), (Htp--6), and (Htp--7).
\end{thm}
\begin{none}\label{2.9.2}
  Here is the outline of the proof of the above theorem. Let $\mathscr{T}$ be a log motivic triangulated category over $\mathscr{S}$.
  \begin{enumerate}[(1)]
    \item In (\ref{2.6.3}), (\ref{2.6.4}), (\ref{2.6.5}), and (\ref{2.6.6}), we have proven that $\mathscr{T}$ satisfies (BC--1), (BC--3), and (BC--4).
    \item In (\ref{5.4.4}), we show that $\mathscr{T}$ satisfies (PF).
    \item In (\ref{5.8.5}), we show that $\mathscr{T}$ satisfies (Supp).
    \item In (\ref{6.1.9}), we show that $\mathscr{T}$ satisfies (Htp--5).
    \item In (\ref{6.2.1}), we show that $\mathscr{T}$ satisfies (Htp--6).
    \item In (\ref{6.3.1}), we show that $\mathscr{T}$ satisfies (Htp--7).
    \item In (\ref{6.4.4}), we show that $\mathscr{T}$ satisfies (BC--2).
  \end{enumerate}
  Then (Supp) implies the statements (1)--(3) in (\ref{0.4}) by (\ref{2.3.3}), and (Adj), (Supp), and (PF) implies the statement (4) in (\ref{0.4}) by (\ref{2.10.2}). The statement (5) is in the axioms of log motivic triangulated categories. The statement (6) is the combination of (BC--1), (BC--2), (BC--3), and (BC--4).
\end{none}
\section{Some results on log geometry and motives}
\subsection{Charts of log smooth morphisms}
\begin{df}\label{0.1.1}
  Let $f:X\rightarrow S$ be a morphism of fine log schemes with a fine chart $\theta:P\rightarrow Q$. Consider the following conditions:
  \begin{enumerate}[(i)]
    \item $\theta$ is injective, the order of the torsion part of the cokernel of $\theta^{\rm gp}$ is invertible in $\mathcal{O}_X$, and the induced morphism $X\rightarrow S\times_{\mathbb{A}_P}\mathbb{A}_Q$ is strict \'etale,
    \item $\theta$ is locally exact,
    \item $\overline{\theta}$ is Kummer.
  \end{enumerate}
  Then we say that
  \begin{enumerate}[(1)]
    \item $\theta$ is of log smooth type if (i) is satisfied,
    \item $\theta$ is of exact log smooth type if (i) and (ii) are satisfied,
    \item $\theta$ is of Kummer log smooth type if (i) and (iii) are satisfied,
  \end{enumerate}
\end{df}
\begin{df}\label{0.1.2}
  Let $S$ be a fine log schemes with a fine chart $P$. Let $\theta:P\rightarrow Q$ be a homomorphism of fine monoids. Consider the following conditions:
  \begin{enumerate}[(i)]
    \item $\theta$ is injective and the order of the torsion part of the cokernel of $\theta^{\rm gp}$ is invertible in $\mathcal{O}_S$,
    \item $\theta$ is locally exact,
    \item $\overline{\theta}$ is Kummer.
  \end{enumerate}
  Then we say that
  \begin{enumerate}[(1)]
    \item $\theta$ is of log smooth over $S$ type if (i) is satisfied,
    \item $\theta$ is of exact log smooth over $S$ type if (i) and (ii) are satisfied,
    \item $\theta$ is of Kummer log smooth over $S$ type if (i) and (iii) are satisfied,
  \end{enumerate}
\end{df}
\begin{prop}\label{0.1.3}
  Let $f:X\rightarrow S$ be a morphism of fs log schemes, and let $P$ be an fs chart of $S$. If $f$ is log smooth, then strict \'etale locally on $X$, there is an fs chart $\theta:P\rightarrow Q$ of $f$ of log smooth type.
\end{prop}
\begin{proof}
  By \cite[IV.3.3.1]{Ogu17}, there is a {\it fine} chart $\theta:P\rightarrow Q$ of $f$ of log smooth type. We may further assume that $Q$ is exact at some point $x$ of $X$ by \cite[II.2.3.2]{Ogu17}. Then $Q$ is saturated since $\overline{Q}\cong \overline{\mathcal{M}}_{X,x}$ is saturated, so $\theta$ is an fs chart.
\end{proof}
\begin{prop}\label{0.1.4}
  Let $f:X\rightarrow S$ be a morphism of fs log schemes, let $x$ be a point of $X$, and let $P$ be an fs chart of $S$ exact at $s:=f(x)$. If $f$ is exact log smooth (resp.\ Kummer log smooth), then strict \'etale locally on $X$, there is an fs chart $\theta:P\rightarrow Q$ of $f$ of exact log smooth type (resp.\ Kummer log smooth type).
\end{prop}
\begin{proof}
  Let us use the notations and hypotheses of the proof of (\ref{0.1.3}). When $f$ is exact, by the proof of \cite[3.5]{NO10}, the homomorphism $\theta$ is critically exact. Then by \cite[ I.4.7.7]{Ogu17}, $\theta$ is locally exact. Thus we are done for the case when $f$ is exact log smooth.

  When $f$ is Kummer, the homomorphism $\overline{\mathcal{M}}_{X,x}\rightarrow \overline{\mathcal{M}}_{S,s}$ is Kummer. Thus the homomorphism $\overline{\theta}$ is Kummer since $P$ is exact at $s$ and $Q$ is exact at $x$. This proves the remaining case.
\end{proof}
\begin{prop}\label{0.1.5}
  Let $g:S'\rightarrow S$ be a strict closed immersion of fs log schemes, and let $f':X'\rightarrow S'$ be a log smooth (resp.\ exact log smooth, resp.\ Kummer log smooth) morphism of fs log schemes. Then strict \'etale locally on $X'$, there is a Cartesian diagram
  \[\begin{tikzcd}
    X'\arrow[d,"f'"]\arrow[r,"g'"]&X\arrow[d,"f"]\\
    S'\arrow[r,"g"]&S
  \end{tikzcd}\]
  of fs log schemes such that $f$ is log smooth (resp.\ exact log smooth, resp.\ Kummer log smooth).
\end{prop}
\begin{proof}
  Let $x'$ be a point of $X'$, and we put $s'=f'(x')$ and $s=g(s')$. Strict \'etale locally on $S$, we can choose an fs chart $P$ of $S$ exact at $s$ by \cite[II.2.3.2]{Ogu17}. Then $P$ is also an fs chart of $S'$ exact at $s'$. By (\ref{0.1.3}) and (\ref{0.1.4}), strict \'etale locally on $X'$, $f'$ has an fs chart $\theta:P\rightarrow Q$ of log smooth type (resp.\ exact log smooth type, resp.\ Kummer log smooth type) such that the induced morphism
  \[h':X'\rightarrow S'\times_{\mathbb{A}_P}\mathbb{A}_Q\]
  is strict \'etale. Then by \cite[IV.18.1.1]{EGA}, Zariski locally on $X'$, there is a Cartesian diagram
  \[\begin{tikzcd}
    X'\arrow[d,"h'"]\arrow[r,"g'"]&X\arrow[d,"h"]\\
    S'\times_{\mathbb{A}_P}\mathbb{A}_Q\arrow[r,"g''"]&S\times_{\mathbb{A}_P}\mathbb{A}_Q
  \end{tikzcd}\]
  such that $h$ is strict \'etale and $g''$ denotes the morphism induced by $g$. The remaining is to put $f=ph$ where $p:S\times_{\mathbb{A}_P}\mathbb{A}_Q\rightarrow S$ denotes the projection.
\end{proof}
\subsection{Change of charts}
\begin{none}\label{0.2.1}
  Let $S$ be a fine log scheme, let $\alpha:P\rightarrow \Gamma(S,\mathcal{M}_S)$ and $\alpha':P'\rightarrow \Gamma(S,\mathcal{M}_S)$ be fine charts of $S$, and let $s$ be a geometric point of $S$. Assume that one of the following conditions is satisfied:
  \begin{enumerate}[(a)]
    \item $\alpha$ is neat at $s$,
    \item $\alpha$ is exact at $s$, and $P'^{\rm gp}$ is torsion free.
  \end{enumerate}
  In this setting, strict \'etale locally on $S$ near $s$, we will explicitly construct a chart $\alpha'':P''\rightarrow\Gamma(S,\mathcal{M}_S)$ and homomorphisms $\beta:P\rightarrow P''$ and $\beta':P'\rightarrow P''$ such that $\alpha''\beta=\alpha$ and $\alpha''\beta'=\alpha'$.

  By \cite[II.2.3.9]{Ogu17}, strict \'etale locally on $S$ near $s$, there exist homomorphisms
  \[\kappa:P'\rightarrow P,\quad \gamma:P'\rightarrow \mathcal{M}_S^*\]
  such that $\alpha'=\alpha\circ \kappa+\gamma$. Consider the homomorphisms
  \[\beta:P\rightarrow P\oplus P'^{\rm gp},\quad a\mapsto (a,0),\]
  \[\beta':P'\rightarrow P\oplus P'^{\rm gp},\quad a\mapsto (\kappa(a),a),\]
  \[\alpha'':P\oplus P'^{\rm gp}\rightarrow \Gamma(S,\mathcal{M}_S),\quad (a,b)\mapsto \alpha(a)+\gamma^{\rm gp}(b)\]
  of fs monoids. Then $\alpha''\beta=\alpha$ and $\alpha''\beta'=\alpha'$. The remaining is to show that $\alpha''$ is a chart of $S$. This follows from the fact that the morphism
  \[\mathbb{A}_\beta:\mathbb{A}_{P\oplus P'^{\rm gp}}\rightarrow \mathbb{A}_P\]
  is strict. So far we have discussed the way to compare charts of $S$. In the following two propositions, we will discuss the way to compare charts of birational morphisms.
\end{none}
\begin{prop}\label{0.2.2}
  Let $S$ be an fs log scheme with fs charts $\alpha:P\rightarrow \Gamma(S,\mathcal{M}_S)$ and $\alpha':P'\rightarrow \Gamma(S,\mathcal{M}_S)$, and let $\theta':P'\rightarrow Q'$ be a homomorphism of fs monoid with a homomorphism $\varphi:Q'^{\rm gp}\rightarrow P'^{\rm gp}$ such that $\varphi\circ \theta'^{\rm gp}={\rm id}$. Assume that $P$ is neat at some geometric point $s\in S$. We denote by $\kappa$ the composition
  \[P'\rightarrow \overline{\mathcal{M}}_{S,s}\rightarrow P\]
  where the first (resp.\ second) arrow is the morphism induced by $\alpha$ (resp.\ $\alpha'$). Consider the coCartesian diagram
  \[\begin{tikzcd}
    P'\arrow[r,"\kappa"]\arrow[d,"\theta'"]&P\arrow[d,"\theta"]\\
    Q'\arrow[r,"\kappa'"]&Q
  \end{tikzcd}\]
  of fs monoids. Then strict \'etale locally on $S$, there is an isomorphism
  \[S\times_{\mathbb{A}_{P'}}\mathbb{A}_{Q'}\cong S\times_{\mathbb{A}_P}\mathbb{A}_Q.\]
\end{prop}
\begin{proof}
  Choose homomorphisms $\beta$, $\beta'$, and $\alpha''$ as in (\ref{0.2.1}). Consider the commutative diagram
  \begin{equation}\label{0.2.2.1}\begin{tikzcd}
    P'\arrow[d,"\theta'"]\arrow[r,"\beta'"]&P\oplus P'^{\rm gp}\arrow[d,"\theta''"]\\
    Q'\arrow[r,"\delta'"]&Q\oplus P'^{\rm gp}
  \end{tikzcd}\end{equation}
  of fs monoids where $\theta''$ and $\delta'$ denote the homomorphisms
  \[\theta'':P\oplus P'^{\rm gp}\rightarrow Q\oplus P'^{\rm gp},\quad (a,b)\mapsto (\theta(a),b),\]
  \[\delta':Q'\rightarrow Q\oplus P'^{\rm gp},\quad a\mapsto (\kappa'(a),\varphi(a)).\]
  We will show that the above diagram is coCartesian. The induced commutative diagram
  \[\begin{tikzcd}
    P'^{\rm gp}\arrow[d,"\theta'^{\rm gp}"]\arrow[r,"\beta'^{\rm gp}"]&P^{\rm gp}\oplus P'^{\rm gp}\arrow[d,"\theta''^{\rm gp}"]\\
    Q'^{\rm gp}\arrow[r,"\delta'^{\rm gp}"]&Q^{\rm gp}\oplus P'^{\rm gp}
  \end{tikzcd}\]
  of finitely generated abelian groups is coCartesian. Hence from the description of pushout in the category of fs monoid, to show that (\ref{0.2.2.1}) is coCartesian, it suffices to show that the images of $\eta$ and $\delta$ generate $Q\oplus P'^{\rm gp}$. This follows from the fact that $\kappa':Q'\rightarrow Q$ is surjective.

  We also have the coCartesian diagram
  \[\begin{tikzcd}
    P\arrow[d,"\theta"]\arrow[r,"\beta"]&P\oplus P'^{\rm gp}\arrow[d,"\theta''"]\\
    Q\arrow[r,"\delta"]&Q\oplus P'^{\rm gp}
  \end{tikzcd}\]
  of fs monoids where $\delta$ denotes the first inclusion. Then we have isomorphisms
  $S\times_{\mathbb{A}_{P'}}\mathbb{A}_{Q'}\cong S\times_{\mathbb{A}_{P\oplus P'^{\rm gp}}}\mathbb{A}_{Q\oplus Q'^{\rm gp}}\cong S\times_{\mathbb{A}_P}\mathbb{A}_Q.$
\end{proof}
\begin{prop}\label{0.2.3}
  Let $S$ be an fs log scheme with fs charts $\alpha:P\rightarrow \Gamma(S,\mathcal{M}_S)$ and $\alpha':P'\rightarrow \Gamma(S,\mathcal{M}_S)$, and let $u':M'\rightarrow {\rm spec}\,P'$ be a birational homomorphism of fs monoschemes. Assume that $P$ is neat at some point $s\in S$. Consider the Cartesian diagram
  \[\begin{tikzcd}
    M\arrow[r,"v"]\arrow[d,"u"]&M'\arrow[d,"u'"]\\
    {\rm spec}\,P\arrow[r,"{\rm spec}\,\kappa"]&{\rm spec}\,P'
  \end{tikzcd}\]
  of fs monoschemes where $\kappa$ denotes the homomorphism defined in (\ref{0.2.2}). Then there is an isomorphism
  \[S\times_{\mathbb{A}_{P'}}\mathbb{A}_{M'}\cong S\times_{\mathbb{A}_P}\mathbb{A}_{M}.\]
\end{prop}
\begin{proof}
  Let ${\rm spec}\,Q'\rightarrow M'$ be an open immersion, and let ${\rm spec}\,Q\rightarrow M$ denote the pullback of it via $v:M\rightarrow M'$. Then by (\ref{0.2.2}), there is an isomorphism
  \[S\times_{\mathbb{A}_{P'}}\mathbb{A}_{Q'}\cong S\times_{\mathbb{A}_P}\mathbb{A}_Q.\]
  Its construction is compatible with any further open immersion ${\rm spec}\,Q_1'\rightarrow {\rm spec}\,Q'\rightarrow M'$, so by gluing the isomorphisms, we get an isomorphism
  $S\times_{\mathbb{A}_{P'}}\mathbb{A}_{M'}\cong S\times_{\mathbb{A}_P}\mathbb{A}_{M}.$
\end{proof}
\subsection{Sections of log smooth morphisms}
\begin{lemma}\label{0.3.1}
  Let $f:X\rightarrow S$ be a log \'etale morphism of fs log schemes, and let $i:S\rightarrow X$ be its section. Then $i$ is an open immersion.
\end{lemma}
\begin{proof}
  From the commutative diagram
  \[\begin{tikzcd}
    S\arrow[r,"i"]\arrow[d,"i"]&X\arrow[d,"i'"]\arrow[r,"f"]&S\arrow[d,"i"]\\
    X\arrow[r,"a"]&X\times_S X\arrow[r,"p_2"]&X
  \end{tikzcd}\]
  of fs log schemes where
  \begin{enumerate}[(i)]
    \item $a$ denotes the diagonal morphism,
    \item $p_2$ denotes the second projection,
    \item each square is Cartesian,
  \end{enumerate}
  it suffices to show that $a:X\rightarrow X\times_S X$ is an open immersion. Since the diagonal morphism
  \[\underline{X}\rightarrow\underline{X}\times_{\underline{S}}\underline{X}\]
  is radiciel, it suffices to show that $a$ is strict \'etale by \cite[IV.17.9.1]{EGA}. As in \cite[IV.17.3.5]{EGA}, the morphism $a$ is log \'etale. Thus it suffices to show that $a$ is strict. We will show this in several steps.\\[4pt]
  (I) {\it Locality on $S$.} Let $g:S'\rightarrow S$ be a strict \'etale cover of fs log schemes, and we put $X'=X\times_S S'$. Then the commutative diagram
  \[\begin{tikzcd}
    X'\arrow[r]\arrow[d]&X'\times_{S'}X'\arrow[d]\\
    X\arrow[r]&X\times_S X
  \end{tikzcd}\]
  of fs log schemes is Cartesian, so the question is strict \'etale local on $S$.\\[4pt]
  (II) {\it Locality on $X$.} Let $h:X'\rightarrow X$ be a strict \'etale cover of fs log schemes. Then we have the commutative diagram
  \[\begin{tikzcd}
    &X'\times_S X'\arrow[d,"h''"]\\
    X'\arrow[d,"h"]\arrow[r,"a'"]\arrow[ru,"a''"]&X\times_S X'\arrow[d,"h'"]\\
    X\arrow[r,"a"]&X\times_S X
  \end{tikzcd}\]
  of fs log schemes where
  \begin{enumerate}[(i)]
    \item the small square is Cartesian,
    \item $a''$ denotes the diagonal morphism,
    \item $h'$ and $h''$ denote the morphism induced by $h:X'\rightarrow X$.
  \end{enumerate}
  Assume that $a''$ is strict. Then $a'$ is strict since $h''$ is strict, so $a$ is strict since $h$ is a strict \'etale cover. Conversely, assume that $a$ is strict. Then $a'$ is
  strict, so $a''$ is strict. Thus the question is strict \'etale local on $X$\\[4pt]
  (III) {\it Final step of the proof.} By \cite[IV.3.3.1]{Ogu17}, strict \'etale locally on $X$ and $S$, we have an fs chart $\theta:P\rightarrow Q$ of $f$ such that
  \begin{enumerate}[(i)]
    \item $\theta$ is injective, and the cokernel of $\theta^{\rm gp}$ is finite,
    \item the induced morphism $X\rightarrow S\times_{\mathbb{A}_P}\mathbb{A}_Q$ is strict \'etale.
  \end{enumerate}
  Hence by (I) and (II), we may assume that $(X,S)=(\mathbb{A}_Q,\mathbb{A}_P)$. Then it suffices to show that the diagonal homomorphism
  \[\mathbb{A}_Q\rightarrow \mathbb{A}_{Q}\oplus_{\mathbb{A}_P}\mathbb{A}_Q\]
  is strict. To show this, it suffices to show $a\oplus (-a)\in (Q\oplus_P Q)^*$ for any $a\in Q$. Choose $n\in \mathbb{N}^+$ such that $na\in P^{\rm gp}$. Because the
  summation homomorphism
  \[P^{\rm gp}\oplus_{P^{\rm gp}}P^{\rm gp}\rightarrow P^{\rm gp}\]
  is an isomorphism, the two elements $(na)\oplus 0$ and $0\oplus (na)$ of $Q\oplus_P Q$ are equal. Thus $n(a\oplus (-a))=0$. Since $Q\oplus_P Q$ is an fs monoid, we have $a\oplus (-a)\in Q\oplus_P Q$. This means $a\oplus (-a)\in (Q\oplus_P Q)^*$ since $n(a\oplus (-a))=0$.
\end{proof}
\begin{lemma}\label{0.3.6}
  Let $f:X\rightarrow S$ and $h:Y\rightarrow X$ be morphisms of fine log schemes. Assume that $h$ is surjective. If $h$ and $fh$ are strict, then $f$ is strict.
\end{lemma}
\begin{proof}
  By \cite[III.1.2.11]{Ogu17}, it suffices to show that the induced homomorphism
  \[\overline{\mathcal{M}}_{S,\overline{f(x)}}\rightarrow \overline{\mathcal{M}}_{X,\overline{x}}\]
  of fine monoids is an isomorphism for any point $x$ of $X$. Since $h$ is surjective, we can choose a ray $y\in Y$ whose image in $X$ is $x$. Then we have the induced homomorphisms
  \[\overline{\mathcal{M}}_{S,\overline{f(x)}}\rightarrow\overline{\mathcal{M}}_{X,\overline{x}}\rightarrow \overline{\mathcal{M}}_{Y,\overline{y}} \]
  of fine monoids. The second arrow (resp.\ the composition of the two arrows) is an isomorphism since $h$ (resp.\ $fh$) is strict. Thus the first arrow is also an isomorphism.
\end{proof}
\begin{lemma}\label{0.3.3}
  Let $f:X\rightarrow S$ be a morphism of fine log schemes. Then there is a maximal open subscheme $U$ of $X$ such that the composition $U\rightarrow X\stackrel{f}\rightarrow S$ is strict.
\end{lemma}
\begin{proof}
  Let $U$ denote the set of points $x\in X$ such that the induced homomorphism
  \[\overline{\mathcal{M}}_{S,\overline{f(x)}}\rightarrow \overline{\mathcal{M}}_{X,\overline{x}}\]
  of fine monoids is an isomorphism. If $U$ is open in $X$, then by \cite[III.1.2.11]{Ogu17}, $U$ is the maximal open subscheme $U$ of $X$ such that the composition $U\rightarrow X\stackrel{f}\rightarrow S$ is strict. Thus the remaining is to show that $U$ is open.

  By \cite[1.5]{Kat00}, there is a strict \'etale morphism $h:Y\rightarrow X$ of fine log schemes such that the image of $h$ contains $U$ and that $fh$ is strict. Let $V$ denote the image of $h$, which can be considered as an open subscheme of $X$. Then by (\ref{0.3.6}), the composition $V\rightarrow X\stackrel{f}\rightarrow S$ is strict where the first arrow is the open immersion. Thus $V\subset U$ by the construction of $U$. Then $V=U$, and in particular $U$ is open in $X$.
\end{proof}
\begin{lemma}\label{0.3.4}
  Let $f:X\rightarrow S$ be a Kummer log smooth separated morphism of fs log schemes, and let $i:S\rightarrow X$ be its section. Then $i$ is a strict regular embedding.
\end{lemma}
\begin{proof}
  Since $i$ is a pullback of the diagonal morphism $d:X\rightarrow X\times_S X$, it suffices to show that $d$ is a strict regular embedding. The new question is strict \'etale local on $X$ and $S$, so we may assume that $f$ has an fs chart $\theta:P\rightarrow Q$ of Kummer log smooth type by (\ref{0.1.4}). Then $\overline{\theta}:\overline{P}\rightarrow \overline{Q}$ is Kummer, so by (\ref{D18}), the summation homomorphism $Q\rightarrow Q\oplus_P Q$ is strict. Thus the first inclusion $Q\oplus_P Q\rightarrow Q$ is also strict. In particular, the first projection $p_1:X\times_S X\rightarrow X$ is strict smooth. Then $d$ is the section of a strict smooth separated morphism, so $d$ is a strict regular embedding.
\end{proof}
\begin{lemma}\label{0.3.5}
  Let $\theta:P\rightarrow Q$ be a Kummer homomorphism of fs monoids, and let $\eta:Q\rightarrow P$ be a homomorphism of fs monoids such that $\eta\theta={\rm id}$. If $Q$ is sharp, then $\theta$ is an isomorphism.
\end{lemma}
\begin{proof}
  Let $q\in Q$ be an element not in $\theta(P)$. Since $\theta$ is $\mathbb{Q}$-surjective, we can choose $n\in \mathbb{N}^+$ such that $nq=\theta(p)$ for some $p\in P$. Then
  \[n(q-\theta\eta(q))=\theta(p)-\theta\eta\theta(p)=0,\]
  so $q-\theta\eta(q)\in Q^*$ since $Q$ is saturated. Thus $q-\theta\eta(q)=0$ since $Q$ is sharp, which proves the assertion.
\end{proof}
\subsection{Generating motives}
\begin{none}
  Throughout this subsection, fix a $\tau$-twisted $eSm$-premotivic triangulated category $\mathscr{T}$ over $\mathscr{S}$ generated by $eSm$ and $\tau$ satisfying (Loc) and (s\'et-Sep).
\end{none}
\begin{none}\label{0.5.1}
  Let $S$ be an $\mathscr{S}$-scheme with an fs chart $\alpha:P\rightarrow \mathcal{M}_S$. We denote by $\mathcal{F}_{S,\alpha}$ the family of motives in $\mathscr{T}(S)$ of the form
  \[M_S(S'\times_{\mathbb{A}_{P'},\mathbb{A}_{\theta'}}\mathbb{A}_Q)\{r\}\]
  where
  \begin{enumerate}[(i)]
    \item $S'\rightarrow S$ is a Kummer log smooth morphism with an fs chart $\eta:P\rightarrow P'$ of Kummer log smooth type,
    \item $\theta':P'\rightarrow Q$ is an injective homomorphism of fs monoids such that the cokernel of $\theta'^{\rm gp}$ is torsion free,
    \item $\theta'$ is logarithmic and locally exact,
    \item $r$ is a twist in $\tau$.
  \end{enumerate}
\end{none}
\begin{prop}\label{0.5.2}
  Under the notations and hypotheses of (\ref{0.5.1}), the family $\mathcal{F}_{S,\alpha}$ generates $\mathscr{T}(S)$.
\end{prop}
\begin{proof}
  Let $f:X\rightarrow S$ be an exact log smooth morphism of $\mathscr{S}$-schemes with an fs chart $\theta:P\rightarrow Q$ of exact log smooth type. It suffices to show that the motive $M_S(X)\{r\}$ is in $\langle \mathcal{F}_{S,\alpha}\rangle$ where $r$ is a twist in $\tau$. Here, $\langle \mathcal{F}_{S,\alpha}\rangle$ denotes the localizing subcategory of $\mathcal{T}(S)$ generated by $\mathcal{F}_{S,\alpha}$.\\[4pt]
  (I) {\it Reduction of $S$.} Note first that the question is strict \'etale local on $S$ by (\ref{2.1.2}). Let $i:Z\rightarrow S$ be a strict closed immersion of $\mathscr{S}$-schemes, let $j:U\rightarrow S$ denote its complement, and let $\beta:P\rightarrow \mathcal{M}_Z$ denote the fs chart induced by $\alpha$. Assume that the question is true for $Z$ and $U$. Then by (Loc), to show the question for $S$, it suffices to show that the motive
  \[i_*M_Z(Z'\times_{\mathbb{A}_{P'}}\mathbb{A}_Q)\{r\}\]
  with the similar conditions as in (i)--(iv) of (\ref{0.5.1}) is in $\langle F_{Z,\beta}\rangle$.

  The induced morphism $Z'\rightarrow Z\times_{\mathbb{A}_P}\mathbb{A}_{P'}$ is open since it is smooth, and let $W$ denote its image. Choose an open immersion $Y\rightarrow S\times_{\mathbb{A}_P}\mathbb{A}_{P'}$ such that $W\cong Z\times_S Y$ and that $\eta^{\rm gp}$ is invertible in $\mathcal{O}_{Y}$. By \cite[IV.18.1.1]{EGA}, we may assume $Z'\cong W\times_Y S'$ for some strict smooth morphism $S'\rightarrow Y$ since the question is Zariski local on $Z'$. Then we have $Z'\cong Z\times_S S'$. By (Loc), we have the distinguished triangle
  \[M_S(U\times_S S'\times_{\mathbb{A}_{P'}}\mathbb{A}_Q)\longrightarrow M_S(S'\times_{\mathbb{A}_{P'}}\mathbb{A}_Q)\longrightarrow i_*M_Z(Z'\times_{\mathbb{A}_{P'}}\mathbb{A}_Q)\longrightarrow M_S(U\times_S S'\times_{\mathbb{A}_{P'}}\mathbb{A}_Q)[1]\]
  in $\mathscr{T}(S)$, and this proves the question since $M_S(U\times_S S'\times_{\mathbb{A}_{P'}}\mathbb{A}_Q)$ and $M_S(S'\times_{\mathbb{A}_{P'}}\mathbb{A}_Q)$ are in $\langle \mathcal{F}_{S,\alpha}\rangle$.

  By the proof of \cite[3.5(ii)]{Ols03}, there is a stratification $\{S_i\rightarrow S\}$ of $S$ such that each $S_i$ has a constant log structure. Hence applying the above argument, we reduce to the case when $\alpha:P\rightarrow \mathcal{M}_S$ induces a constant log structure.\\[4pt]
  (II) {\it Construction of $P'$.} We will use induction on
   \[d:=\max_{x\in X}{\rm rk}\,\overline{\mathcal{M}}_{X,x}^{\rm gp}.\]
   If $d=\dim P$, then $f$ is Kummer log smooth, so we are done. Hence let us assume $d>\dim P$.

   We denote by $P'$ the submonoid of $Q$ consisting of elements $p\in Q$ such that $np\in \theta(P)+Q^*$ for some $n\in \mathbb{N}^+$. Then $P'$ is an fs monoid by Gordon's lemma \cite[ I.2.3.19]{Ogu17}. Let $\theta':P'\rightarrow Q$ denote the inclusion. Then the cokernel of $\theta'^{\rm gp}$ is torsion free by construction. We will check the conditions (ii) and (iii) of (\ref{0.5.1}). Since $P'^{\rm gp}=Q^{\rm gp}$, $\theta'$ is logarithmic. For the locally exactness, it suffices to show that $\theta_\mathbb{Q}':\overline{P}_\mathbb{Q}'\rightarrow \overline{Q}_\mathbb{Q}$ is integral. This follows from \cite[I.4.5.3(2), I.4.5.3(1)]{Ogu17}. The remaining is to show that $\theta'^{\rm gp}$ is torsion free.

   Let $G$ be a maximal $\theta'$-critical face of $Q$. Then we have $(\overline{Q})_\mathbb{Q}^{\rm gp}=(\overline{P'})_\mathbb{Q}^{\rm gp}\oplus (\overline{G})_\mathbb{Q}^{\rm gp}$ by \cite[I.4.7.9]{Ogu17}. Thus, for any $q\in Q^{\rm gp}$ such that $nq\in P'^{\rm gp}$ for some $n\in \mathbb{N}^+$, the image of $q$ in $(\overline{P'})_\mathbb{Q}^{\rm gp}\oplus (\overline{G})_\mathbb{Q}^{\rm gp}$ should be in $(\overline{P'})_\mathbb{Q}^{\rm gp}$. This means $nq+p'\in P'+Q^*$ for some $p'\in P'$, so $n(q+p)\in P'+Q^*$. Thus $q+p\in P'$ by the construction of $P'$, so $q\in P'^{\rm gp}$. This proves that $\theta'^{\rm gp}$ is torsion free.\\[4pt]
   (III) {\it Construction of $S'$.} The induced morphism $X\rightarrow S\times_{\mathbb{A}_P}\mathbb{A}_{P'}$ is open by \cite[5.7]{Nak09}, and let $Y$ denote its image. Then $Y$ has the chart $P'$. Note that the induced morphism $h:Y\rightarrow S$ is Kummer log smooth and that the order of the torsion part of the cokernel of $\eta^{\rm gp}$ is invertible in $\mathcal{O}_{Y}$.

   The closed immersion $Y\times_{\mathbb{A}_{P'}}\mathbb{A}_{(P',P'^+)}\rightarrow Y$ is an isomorphism since $S$ has a constant log structure. Thus the projection
   \[\underline{Y\times_{\mathbb{A}_{P'}}\mathbb{A}_{(Q,Q^+)}}\rightarrow \underline{Y}\]
   of underlying schemes is an isomorphism since $\theta':P'\rightarrow Q$ is logarithmic. Consider the pullback
   \[g':X\times_{\mathbb{A}_Q}\mathbb{A}_{(Q,Q^+)}\rightarrow Y\times_{\mathbb{A}_{P'}}\mathbb{A}_{(Q,Q^+)}\]
   of the induced morphism $h:X\rightarrow Y\times_{\mathbb{A}_{P'}}\mathbb{A}_Q$. Since $\theta$ is exact log smooth type, $h$ is strict \'etale, so $g'$ is also strict \'etale. Then there is a unique Cartesian diagram
   \[\begin{tikzcd}
     X\times_{\mathbb{A}_{Q}}\mathbb{A}_{(Q,Q^+)}\arrow[r,"g'"]\arrow[d]&Y\times_{\mathbb{A}_{P'}}\mathbb{A}_{(Q,Q^+)}\arrow[d]\\
     S'\arrow[r,"g"]&Y
   \end{tikzcd}\]
   of $\mathscr{S}$-schemes where the right vertical arrow is the projection. The morphism $g$ is automatically strict \'etale. This verifies the condition (i) of (\ref{0.5.1}), so we have checked the conditions (i)--(iii) of (loc.\ cit) for our constructions of $P'$ and $S'$.\\[4pt]
   (IV) {\it Final step of the proof.} Then we have the commutative diagram
   \[\begin{tikzcd}
     S'\times_{\mathbb{A}_{P'}}\mathbb{A}_Q\arrow[r,leftarrow]\arrow[rddd]\arrow[rrrdd]&S'\times_{\mathbb{A}_{P'}}\mathbb{A}_{(Q,Q^+)}\arrow[r,"\sim"]\arrow[rdd,crossing over] \arrow[ddd,crossing over] &X\times_{\mathbb{A}_Q}\mathbb{A}_{(Q,Q^+)}\arrow[r]\arrow[dd,crossing over]&X\arrow[dd]\\
     \\
     &&Y\times_{\mathbb{A}_{P'}}\mathbb{A}_{(Q,Q^+)}\arrow[r,"v"']\arrow[d]&Y\times_{\mathbb{A}_{P'}}\mathbb{A}_Q\arrow[ld,"p"]\\
     &S'\arrow[r,"g"]&Y
   \end{tikzcd}\]
   of $\mathscr{S}$-schemes. Note that the projection $p$ is log smooth by the conditions (ii) and (iii) of (loc.\ cit). Let $u$ denote the complement of the closed immersion $v:Y\times_{\mathbb{A}_{P'}}\mathbb{A}_{(Q,Q^+)}\rightarrow Y$. Then by (Loc), we have distinguished triangles
   \[p_\sharp u_\sharp u^*M_{Y\times_{\mathbb{A}_{P'}}\mathbb{A}_{Q}}(X)\longrightarrow M_Y(X)\longrightarrow p_\sharp v_*v^*M_{Y\times_{\mathbb{A}_{P'}}\mathbb{A}_{Q}}(X)\longrightarrow p_\sharp u_\sharp u^*M_{Y\times_{\mathbb{A}_{P'}}\mathbb{A}_{Q}}(X)[1],\]
   \[\begin{tikzpicture}[baseline= (a).base]
    \node[scale=.73] (a) at (0,0)
    {\begin{tikzcd}
     p_\sharp u_\sharp u^*M_{Y\times_{\mathbb{A}_{P'}}\mathbb{A}_{Q}}(S'\times_{\mathbb{A}_{P'}}\mathbb{A}_Q)\longrightarrow M_Y(S'\times_{\mathbb{A}_{P'}}\mathbb{A}_Q)
     \longrightarrow p_\sharp v_*v^*M_{Y\times_{\mathbb{A}_{P'}}\mathbb{A}_{Q}}(S'\times_{\mathbb{A}_{P'}}\mathbb{A}_Q)\longrightarrow p_\sharp u_\sharp u^*M_{Y\times_{\mathbb{A}_{P'}}\mathbb{A}_Q}(S'\times_{\mathbb{A}_{P'}}\mathbb{A}_Q)[1].
   \end{tikzcd}};
  \end{tikzpicture}\]
   Let $r$ be a twist in $\tau$. We have isomorphisms
   \[\begin{split}
     v^*M_{Y\times_{\mathbb{A}_{P'}}\mathbb{A}_{Q}}(X)&\cong M_{Y\times_{\mathbb{A}_{P'}}\mathbb{A}_{(Q,Q^+)}}(X\times_{\mathbb{A}_Q}\mathbb{A}_{(Q,Q^+)})\\
     &\cong M_{Y\times_{\mathbb{A}_{P'}}\mathbb{A}_{(Q,Q^+)}}(S'\times_{\mathbb{A}_{P'}}\mathbb{A}_{(Q,Q^+)}) \\
     &\cong v^*M_{Y\times_{\mathbb{A}_{P'}}\mathbb{A}_Q}(S'\times_{\mathbb{A}_{P'}}\mathbb{A}_Q),
   \end{split}\]
   and $h_\sharp M_Y(S'\times_{\mathbb{A}_{P'}}\mathbb{A}_Q)$ is in $\langle \mathcal{F}_{S,\alpha}\rangle$ by definition. Moreover, by induction on $d$,
   \[h_\sharp p_\sharp u_\sharp u^*M_{Y\times_{\mathbb{A}_{P'}}\mathbb{A}_{Q}}(X),\quad h\sharp p_\sharp u_\sharp u^*M_{Y\times_{\mathbb{A}_{P'}}\mathbb{A}_{Q}}(S'\times_{\mathbb{A}_{P'}}\mathbb{A}_Q)\{r\}\]
   are in $\langle \mathcal{F}_{S,\alpha}\rangle$. Thus from the above triangles, we conclude that $h_\sharp M_Y(X)\{r\}\cong M_S(X)\{r\}$ is also in $\langle \mathcal{F}_{S,\alpha}\rangle$.
\end{proof}
\begin{cor}\label{0.5.3}
  Assume that $\mathscr{T}$ satisfies (Htp--3). Let $S$ be an $\mathscr{S}$-scheme with an fs chart $\alpha:P\rightarrow \mathcal{M}_S$. Consider the family of motives in $\mathscr{T}(S)$ of the form
  \[M_S(S')\{r\}\]
  where $r$ is a twist and $S'\rightarrow S$ is a Kummer log smooth morphism with an fs chart $\theta:P\rightarrow P'$ of Kummer log smooth type. Then the family generates $\mathscr{T}(S)$.
\end{cor}
\begin{proof}
  Let $S'\rightarrow S$ be a Kummer log smooth morphism with an fs chart $\theta:P\rightarrow P'$ of Kummer log smooth type, let $\theta':P'\rightarrow Q$ is a logarithmic, locally exact, and injective homomorphism of fs monoids such that the cokernel of $\theta'^{\rm gp}$ is torsion free, and let $G$ be a $\theta'$-critical face of $Q$. Then the induced morphism
  \[M_S(S'\times_{\mathbb{A}_{P'},\mathbb{A}_{\theta'}}\mathbb{A}_{Q_G})\rightarrow M_S(S'\times_{\mathbb{A}_{P'},\mathbb{A}_{\theta'}}\mathbb{A}_Q)\]
  in $\mathscr{T}(S)$ is an isomorphism by (Htp--3). Since the induced morphism
  \[S'\times_{\mathbb{A}_{P'},\mathbb{A}_{\theta'}}\mathbb{A}_{Q_G}\rightarrow S\]
  is Kummer log smooth and has the fs chart $P\rightarrow Q_G$ of Kummer log smooth type, we are done.
\end{proof}
\section{Purity}\label{Sec4}
\begin{none}
  Throughout this section, Fix a $eSm$-premotivic triangulated category satisfying (Adj), (Htp--1), (Loc), and (Stab).
\end{none}
\subsection{Thom transformations}
\begin{df}\label{4.1.3}
  Let $f:X\rightarrow S$ be a morphism of $\mathscr{S}$-schemes, and let $i:S\rightarrow X$ be its section. Assume that $i$ is a strict regular embedding. We have the following definitions.
  \begin{enumerate}[(1)]
    \item $B_SX$ denotes the blow-up of $X$ with center $S$,
    \item $B_S(X\times\mathbb{A}^1)$ denotes the blow-up of $X\times\mathbb{A}^1$ with center $S\times\{0\}$,
    \item $D_SX=B_S(X\times\mathbb{A}^1)-B_SX$,
    \item $N_SX$ denotes the normal bundle of $S$ in $X$.
  \end{enumerate}
  The morphisms $S\stackrel{i}\rightarrow X\stackrel{f}\rightarrow S$ induces the morphisms $D_SS\rightarrow D_S X\rightarrow D_S S$, which is
  \begin{equation}\label{4.1.3.1}
    S\times \mathbb{A}^1\rightarrow D_SX\rightarrow S\times\mathbb{A}^1
  \end{equation}
  since $D_SS=S\times\mathbb{A}^1$.
\end{df}
\begin{df}\label{4.1.1}
  Let $h:X\rightarrow Y$ and $g:Y\rightarrow S$ be morphisms of $\mathscr{S}$-schemes, and we put $f=gh$. Consider a commutative diagram
  \[\begin{tikzcd}
    D_0\arrow[d,"u_0"]\arrow[r,"b"]&D\arrow[d,"u"]\arrow[rd,"q_2"]\\
    X\arrow[r,"a"']&Y\times_S X\arrow[r,"p_2"']&X
  \end{tikzcd}\]
  of $\mathscr{S}$-schemes where $a$ denotes the graph morphism and $p_2$ denotes the second projection. Assume that $b$ is proper. Then we define the following functors:
  \[\Sigma_{g,f}:=p_{2\sharp}a_*,\quad \Omega_{g,f}:=a^!p_2^*,\quad \Omega_{g,f,D}:=u_{0*}b^!q_2^*.\]
  The third notation depends on the morphisms, so we will use it only when no confusion arises.

  When $b$ is a strict regular embedding, consider the diagrams
  \begin{equation}\begin{tikzcd}\label{4.1.1.1}
    D_0\arrow[r,"b"]\arrow[d,"\gamma_1"]&D\arrow[d,"\beta_1"]\arrow[r,"q_2"]&X\arrow[d,"\alpha_1"]\\
    D_0\times\mathbb{A}^1\arrow[r,"d"]\arrow[d,"\phi"]&D_{D_0}D\arrow[r,"s_2"]&X\times \mathbb{A}^1\arrow[d,"\pi"]\\
    D_0&&X
  \end{tikzcd}\quad
  \begin{tikzcd}
    D_0\arrow[r,"e"]\arrow[d,"\gamma_0"]&N_{D_0}D\arrow[d,"\beta_0"]\arrow[r,"t_2"]&X\arrow[d,"\alpha_0"]\\
    D_0\times\mathbb{A}^1\arrow[r,"d"]\arrow[d,"\phi"]&D_{D_0}D\arrow[r,"s_2"]&X\times \mathbb{A}^1\arrow[d,"\pi"]\\
    D_0&&X
  \end{tikzcd}\end{equation}
  of $\mathscr{S}$-schemes where
  \begin{enumerate}[(a)]
    \item each square is Cartesian,
    \item $\alpha_0$ denotes the $0$-section, and $\alpha_1$ denotes the $1$-section,
    \item $d$ and $s_2$ are the morphisms constructed by (\ref{4.1.3.1}),
    \item $\phi$ and $\pi$ denotes the projections.
  \end{enumerate}
  Then we define the following functors:
  \[\Omega_{g,f,D}^d:=u_{0*}\pi_*d^!s_2^*\pi^*,\quad \Omega_{g,f,D}^n:=u_{0*}e^!t_2^*.\]
  Now, assume $u_0={\rm id}$. Then we define the following functor:
  \[\Omega_{g,f,D}^o:=\mathfrak{t}'_{N_XD}.\]
  Here, $\mathfrak{t}'_{N_XD}$ is the right adjoint of an orientation of $N_XD$, and it exists by (\ref{2.8.2}). By (\ref{2.5.10}), the functor $\Omega_{g,f,D}^n$ is an equivalence, and by a theorem of Morel and Voevodsky \cite[2.4.35]{CD12}, $\Omega_{g,f,D}^d$ is also an equivalence. We denote by
  \[\Sigma_{g,f,D}^d,\quad \Sigma_{g,f,D}^n,\quad \Sigma_{g,f,D}^o\]
  the left adjoints (or equivalently right adjoints) of $\Omega_{g,f,D}^d$, $\Omega_{g,f,D}^n$, and $\Omega_{g,f,D}^o$ respectively.

  When $h$ is the identity morphism, we simply put
  \[\Sigma_f:=\Sigma_{f,f},\quad \Omega_f:=\Omega_{f,f},\quad \Omega_{f,D}:=\Omega_{f,f,D},\]
  \[\Omega_{f,D}^d:=\Omega_{f,f,D}^d,\quad \Omega_{f,D}^n:=\Omega_{f,f,D}^n,\quad \Omega_{f,D}^o:=\Omega_{f,f,D}^o,\]
  \[\Sigma_{f,D}^d:=\Sigma_{f,f,D}^d,\quad \Sigma_{f,D}^n:=\Sigma_{f,f,D}^n,\quad \Sigma_{f,D}^o:=\Sigma_{f,f,D}^o,\]
  and when $a$ is a strict regular embedding, we simply put
  \[\Sigma_f^d:=\Sigma_{f,X\times_S X}^d,\quad \Sigma_f^n:=\Sigma_{f,X\times_S X}^n.\]
  These functors are called {\it Thom transformations}.
\end{df}
\begin{none}\label{4.1.2}
  Under the notations and hypotheses of (\ref{4.1.1}), we will frequently assume that $u_0={\rm id}$ and there is a commutative diagram
  \[\begin{tikzcd}
    &I\arrow[d,"w"]\arrow[rd,"r_2"]\\
    X\arrow[ru,"c"]\arrow[r,"b"']&D\arrow[r,"q_2"']&X
  \end{tikzcd}\]
  with the following properties:
  \begin{enumerate}[(i)]
    \item $w$ is an open immersion,
    \item $c$ is a strict closed immersion,
    \item $r_2$ is a {\it strict} smooth morphism.
  \end{enumerate}
\end{none}
\subsection{Transition transformations}
\begin{none}
  In this subsection, we will develop various functorial properties of Thom transformations.
\end{none}
\begin{none}\label{4.2.2}
  Under the notations and hypotheses of (\ref{4.1.1}), consider a commutative diagram
  \[\begin{tikzcd}
   E_0\arrow[d,"v_0"]\arrow[r,"c"]&E\arrow[d,"v"]\arrow[rdd,"r_2"]\\
    D_0\arrow[d,"u_0"]\arrow[r,"b"]&D\arrow[d,"u"]\arrow[rd,"q_2"']\\
    X\arrow[r,"a"']& Y\times_S X\arrow[r,"p_2"']&X
  \end{tikzcd}\]
  of $\mathscr{S}$-schemes, and assume that $b$ and $c$ are proper. Then we have a natural transformation
  \[T_{D,E}:\Omega_{g,f,E}\longrightarrow \Omega_{g,f,D}\]
  in the below three cases. This is called a {\it transition transformation}. Here, when $D=Y\times_S X$, we put $T_{Y\times_S X,E}=T_E$ for simplicity.
  \begin{enumerate}[(i)]
    \item Assume that $v_0$ is the identity and that the exchange transformation
    \[{\rm id}^*b^!\stackrel{Ex}\longrightarrow c^!v^*\]
    is defined and an isomorphism. Then the natural transformation $\Omega_{g,f,E}\stackrel{T_{D,E}}\longrightarrow \Omega_{g,f,D}$ is given by
    \[u_{0*}c^!r_2^*\stackrel{\sim}\longrightarrow u_{0*}c^!v^*q_2^*\stackrel{Ex^{-1}}\longrightarrow u_{0*}b^!r_2^*.\]
    Note that when $v$ is an open immersion and (Supp) is satisfied, then the condition is satisfied.
    \item Assume that the unit ${\rm id}\stackrel{ad}\longrightarrow v_*v^*$ is an isomorphism. Then the natural transformation $\Omega_{g,f,E}\stackrel{T_{D,E}}\longrightarrow \Omega_{g,f,D}$ is given by
    \[u_{0*}v_{0*}c^!r_2^*\stackrel{Ex}\longrightarrow u_{0*}b^!v_*r_2^*\stackrel{\sim}\longrightarrow u_{0*}b^!v_*v^*q_2^*\stackrel{ad^{-1}}\longrightarrow u_{0*}b^!q_2^*.\]
    \item Assume that $v$ is strict \'etale, $v_0$ is the identity, and (Supp) is satisfied. The purity transformation
    \[v^!\stackrel{\mathfrak{q}_v^n}\longrightarrow v^*\]
    whose description is given in (\ref{4.4.5}) is an isomorphism by \cite[2.4.50(3)]{CD12}. Then the natural transformation $\Omega_{g,f,E}\stackrel{T_{D,E}}\longrightarrow \Omega_{g,f,D}$ is given by
    \[u_{0*}c^!r_2^*\stackrel{\sim}\longrightarrow u_{0*}c^!v^*q_2^*\stackrel{(\mathfrak{q}_v^n)^{-1}}\longrightarrow u_{0*}c^!v^!q_2^*\stackrel{\sim}\longrightarrow u_{0*}b^!q_2^*.\]
    Note that $T_{D,E}$ is an isomorphism.
  \end{enumerate}
\end{none}
\begin{none}\label{4.2.5}
  Under the notations and hypotheses of (\ref{4.1.1}), let $\eta:X'\rightarrow X$ be a morphism of $\mathscr{S}$-schemes, and we put $f'=f\eta$. Consider the commutative diagram
  \[\begin{tikzcd}
    D_0'\arrow[rd,"u_0'"']\arrow[dd,"\eta_0"']\arrow[rr,"b'"]&&D'\arrow[rd,"u'"']\arrow[rrrd,"q_2'"]\arrow[dd,"\rho",near start]\\
    &X'\arrow[rr,"a'"',near start,crossing over]&&Y\times_S X'\arrow[rr,"p_2'"']&&X'\arrow[dd,"\eta"]\\
    D_0\arrow[rd,"u_0"']\arrow[rr,"b",near start]&&D\arrow[rd,"u"']\arrow[rrrd,"q_2"]\\
    &X\arrow[uu,leftarrow,crossing over,"\eta"',near end]\arrow[rr,"a"']&&Y\times_S X\arrow[uu,"\eta'",leftarrow,crossing over]\arrow[rr,"p_2"']&&X
  \end{tikzcd}\]
  of $\mathscr{S}$-schemes where the upper layer is a pullback of the lower layer. Then we have the natural transformations
  \[\eta^*\Omega_{g,f,D}\stackrel{Ex}\longrightarrow \Omega_{g,f',D'}\eta^*,\quad \Omega_{g,f,D}\eta_*\stackrel{Ex}\longrightarrow \eta_*\Omega_{g,f',D'},\quad \Omega_{g,f',D'}\eta^!\stackrel{Ex}\longrightarrow \eta^!\Omega_{g,f,D}\]
  given by
  \[\eta^*u_{0*}b^!q_2^*\stackrel{Ex}\longrightarrow u_{0*}'\eta_0^*b^!q_2^*\stackrel{Ex}\longrightarrow u_{0*}'b'^!\rho^*q_2^*\stackrel{\sim}\longrightarrow u_{0*}'b'^!q_2'^*\eta^*,\]
  \[u_{0*}b^!q_2^*\eta_*\stackrel{Ex}\longrightarrow u_{0*}b^!\rho_*q_2'^*\stackrel{Ex^{-1}}\longrightarrow u_{0*}\eta_{0*}b'^!q_2'^*\stackrel{\sim}\longrightarrow \eta_*u_{0*}'b'^!q_2'^*,\]
  \[u_{0*}'b'^!q_2'^*\eta^!\stackrel{Ex}\longrightarrow u_{0*}'b'^!\rho^!q_2^*\stackrel{\sim}\longrightarrow u_{0*}'\eta_0^!b^!q_2^*\stackrel{Ex}\longrightarrow \eta^!u_{0*}b^!q_2^*\]
  respectively. These are called {\it exchange transformations.} Here, to define the first (resp.\ second, resp.\ third) natural transformation, we assume the condition $(CE^*)$ (resp.\ $(CE_*)$, resp.\ $(CE^!)$) whose definition is given below:
  \begin{enumerate}
    \item[$(CE^*)$] The exchange transformation $\eta_0^*b^!\stackrel{Ex}\longrightarrow b'^!\rho^*$ is defined.
    \item[$(CE_*)$] The exchange transformation $\eta_{0*}b'^!\stackrel{Ex}\longrightarrow b^!\rho_*$ is an isomorphism.
    \item[$(CE^!)$] ($\eta$ is proper) or ($\eta$ is separated and (Supp) is satisfied). Moreover, the exchange transformation $q_2'^*\eta^!\stackrel{Ex}\longrightarrow \rho^!q_2^*$ is defined.
  \end{enumerate}
  For example, if $b$ is a strict closed immersion, then by (\ref{2.6.7}), $(CE^*)$ and $(CE_*)$ are satisfied. On the other hand, if $\eta$ is a strict closed immersion, then by (\ref{2.6.7}), $(CE^!)$ is satisfied.

  When $b$ is a strict regular embedding, we similarly have the natural transformations
  \[\eta^*\Omega_{g,f,D}^d\stackrel{Ex}\longrightarrow \Omega_{g,f',D'}^d\eta^*,\quad \Omega_{g,f,D}^d\eta_*\stackrel{Ex}\longrightarrow \eta_*\Omega_{g,f',D'}^d,\]
  \[\eta^*\Omega_{g,f,D}^n\stackrel{Ex}\longrightarrow \Omega_{g,f',D'}^n\eta^*,\quad \Omega_{g,f,D}^n\eta_*\stackrel{Ex}\longrightarrow \eta_*\Omega_{g,f',D'}^n \]
  because the corresponding versions $(CE^*)$ and $(CE_*)$ are always satisfied since the morphisms $d:D_0\times\mathbb{A}^1\rightarrow D_{D_0}D$ and $e:D_0\rightarrow N_{D_0}D$ are strict regular embeddings. Since $t_2$ is exact log smooth, if ($\eta$ is proper) or ($\eta$ is separated and (Supp) is satisfied), we have the natural transformation
  \[\Omega_{g,f',D'}^n\eta^!\stackrel{Ex}\longrightarrow \eta^!\Omega_{g,f,D}^n.\]
  When $s_2$ is exact log smooth, if ($\eta$ is proper) or ($\eta$ is separated and (Supp) is satisfied), we also have the natural transformation
  \[\Omega_{g,f',D'}^d\eta^!\stackrel{Ex}\longrightarrow \eta^!\Omega_{g,f,D}^d\]
  because the corresponding version $(CE^!)$ is satisfied. However, $s_2$ may not be exact log smooth morphism. In this case, assume that (Supp) is satisfied and that the conditions of (\ref{4.1.2}) are satisfied. We put $I'=I\times_D D'$, and consider the diagram
  \[\begin{tikzcd}
    \Omega_{g,f',I'}\eta^!\arrow[d,"T_{D',I'}"]\arrow[r,"Ex"]&\eta^!\Omega_{g,f,I}\arrow[d,"T_{D,I}"]\\
    \Omega_{g,f',D'}\eta^!&\eta^!\Omega_{g,f,D}
  \end{tikzcd}\]
  of functors. The horizontal arrow is defined since the induced morphism $D_XI\rightarrow X\times \mathbb{A}^1$ is strict smooth. The vertical arrows are isomorphism by (\ref{4.2.2}). Now, the definition of
  \[\Omega_{g,f',D'}^d\eta^!\stackrel{Ex}\longrightarrow \eta^!\Omega_{g,f,D}^d\]
  is given by the composition
  \[\Omega_{g,f',D'}\eta^!\stackrel{(T_{D',I'})^{-1}}\longrightarrow \Omega_{g,f',I'}\eta^!\stackrel{Ex}\longrightarrow \eta^!\Omega_{g,f,I}\stackrel{T_{D,I}}\longrightarrow \eta^!\Omega_{g,f,D}.\]
\end{none}
\begin{lemma}\label{4.2.8}
  Under the notations and hypotheses of (\ref{4.2.5}), the exchange transformation
  \[\Omega_{g,f,D}^n\eta_*\stackrel{Ex}\longrightarrow \eta_*\Omega_{g,f',D'}^n\]
  is an isomorphism.
\end{lemma}
\begin{proof}
  It follows from ($eSm$-BC) because the morphism $N_{D_0}D\rightarrow X$ in (\ref{4.1.1.1}) is exact log smooth.
\end{proof}
\begin{lemma}\label{4.2.7}
  Under the notations and hypotheses of (\ref{4.2.5}), assume that $u_0$ is the identity. If $\eta$ is proper, then the exchange transformation
  \[\Omega_{g,f',D'}^n\eta^!\stackrel{Ex}\longrightarrow \eta^!\Omega_{g,f,D}^n\]
  is an isomorphism.
\end{lemma}
\begin{proof}
  Note first that $\Omega_{g,f,D}^n$ and $\Omega_{g,f',D'}^n$ are equivalences of categories by (\ref{2.8.2}). Consider the natural transformation
  \[\eta^!\Sigma_{g,f',D'}^n\stackrel{Ex}\longrightarrow \Sigma_{g,f,D}^n\eta^!\]
  given by the left adjoint of the exchange transformation
  \[\Omega_{g,f,D}^n\eta_*\stackrel{Ex}\longrightarrow \eta_*\Omega_{g,f',D'}^n.\]
  Then consider the commutative diagram
  \[\begin{tikzcd}
    \Omega_{g,f',D'}^n\eta^!\Sigma_{g,f',D'}^n\arrow[r,"Ex"]\arrow[d,"Ex"]&\eta^!\Omega_{g,f,D}^n\Sigma_{g,f',D'}^n\arrow[d,"ad'"]\\
    \Omega_{g,f',D'}^n\Sigma_{g,f,D}^n\eta^!\arrow[r,"ad'"]&\eta^!
  \end{tikzcd}\]
  of functors. By (\ref{4.2.8}), the left vertical arrow is an isomorphism. The right vertical and lower horizontal arrows are also isomorphisms since $\Omega_{g,f,D}^n$ and $\Omega_{g,f',D'}^n$ are equivalences of categories. Thus $\Omega_{g,f,D}^n$ and $\Omega_{g,f',D'}^n$ are equivalences of categories. Then the conclusion follows from the fact that $\Sigma_{g,f',D'}^n$ is an equivalence of categories.
\end{proof}
\begin{lemma}\label{4.2.13}
  Under the notations and hypotheses of (\ref{4.2.5}), assume that $u_0$ is the identity. if $q_2$ is strict smooth separated and $\eta$ is separated, then the exchange transformation
  \[\Omega_{g,f',D'}\eta^!\stackrel{Ex}\longrightarrow \eta^!\Omega_{g,f,D}\]
  is defined and an isomorphism.
\end{lemma}
\begin{proof}
  It is a direct consequence of (\ref{2.5.9}).
\end{proof}
\begin{lemma}\label{4.2.9}
  Under the notations and hypotheses of (\ref{4.2.5}), assume that $u_0$ is the identity. if $\eta$ is an open immersion and (Supp) is satisfied, then the exchange transformation
  \[\Omega_{g,f',D'}^n\eta^!\stackrel{Ex}\longrightarrow \eta^!\Omega_{g,f,D}^n\]
  is an isomorphism.
\end{lemma}
\begin{proof}
  It is a direct consequence of (Supp).
\end{proof}
\begin{lemma}\label{4.2.10}
  Under the notations and hypotheses of (\ref{4.2.5}), assume that $u_0$ is the identity. if $\eta$ is separated and (Supp) is satisfied, then then the exchange transformation
  \[\Omega_{g,f',D'}^n\eta^!\stackrel{Ex}\longrightarrow \eta^!\Omega_{g,f,D}^n\]
  is an isomorphism.
\end{lemma}
\begin{proof}
  It follows from (\ref{4.2.7}) and (\ref{4.2.9}).
\end{proof}
\begin{none}\label{4.2.1}
  Under the notations and hypotheses of (\ref{4.1.1}), we have the natural transformations
  \[\Omega_{g,f,D}^n\stackrel{(T^n)^{-1}}\longleftarrow \Omega_{g,f,D}^d\stackrel{T^d}\longrightarrow \Omega_{g,f,D}\]
  whose descriptions are given below. These are called {\it transition transformations} again.
  \begin{enumerate}[(1)]
    \item The natural transformation
    \[T^d:\Omega_{g,f,D}^d\longrightarrow \Omega_{g,f,D}\]
    is given by
    \[\begin{split}
      \pi_*\Omega_{g\times\mathbb{A}^1,f\times\mathbb{A}^1,D_XD}\pi^*&\stackrel{ad}\longrightarrow \pi_*\alpha_{1*}\alpha_1^*\Omega_{g\times\mathbb{A}^1,f\times\mathbb{A}^1,D_XD} \pi^*\stackrel{\sim}\longrightarrow u_{0*}\alpha_1^*\Omega_{g\times\mathbb{A}^1,f\times\mathbb{A}^1,D_XD}\pi^*\\
      &\stackrel{Ex}\longrightarrow \Omega_{g,f,D}\alpha_1^*\pi^*\stackrel{\sim}\longrightarrow \Omega_{g,f,D}.
    \end{split}\]
    \item The natural transformation
    \[(T^n)^{-1}:\Omega_{g,f,D}^d\longrightarrow \Omega_{g,f,D}^n\]
    is given by
    \[\begin{split}
      \pi_*\Omega_{g\times\mathbb{A}^1,f\times\mathbb{A}^1,D_XD}\pi^*&\stackrel{ad}\longrightarrow \pi_*\alpha_{0*}\alpha_0^*\Omega_{g\times\mathbb{A}^1,f\times\mathbb{A}^1,D_XD} \pi^*\stackrel{\sim}\longrightarrow \alpha_0^*\Omega_{g\times\mathbb{A}^1,f\times\mathbb{A}^1,D_XD}\pi^*\\
      &\stackrel{Ex}\longrightarrow \Omega_{g,f,D}^n\alpha_0^*\pi^*\stackrel{\sim}\longrightarrow \Omega_{g,f,D}^n.
    \end{split}\]
    When $(T^n)^{-1}$ is an isomorphism, its inverse is denoted by $T^n$.
  \end{enumerate}
  When $u_0$ is the identity, we also have the natural transformation
  \[T^o:\Omega_{g,f,D}^o\longrightarrow \Omega_{g,f,D}^n\]
  given by the right adjoint of an orientation of $N_XD$. It depends on the orientation.
\end{none}
\begin{none}\label{4.2.6}
  Under the notations and hypotheses of (\ref{4.2.5}), consider the commutative diagram
  \[\begin{tikzcd}
    E_0'\arrow[dd,"\psi_0"]\arrow[rd,"v_0'"]\arrow[rr,"c'"]&&E'\arrow[rd,"v'"']\arrow[rrrd,"r_2'"]\arrow[dd,"\psi",near end]\\
    &D_0'\arrow[rr,"b'"',near start,crossing over]&&D'\arrow[rr,"q_2'"']&&X'\arrow[dd,"\eta"]\\
    E\arrow[rr,"c",near start]\arrow[rd,"v_0"']&&E\arrow[rd,"v"']\arrow[rrrd,"r_2"]\\
    &D_0\arrow[rr,"b"']\arrow[uu,crossing over,"\rho_0",near end,leftarrow]&&D\arrow[uu,"\rho",leftarrow,crossing over]\arrow[rr,"q_2"']&&X
  \end{tikzcd}\]
  of $\mathscr{S}$-schemes where each small square is Cartesian. Assume that one of the conditions (i) and (ii) of (\ref{4.2.2}) is simultaneously satisfied for both $(D,E)$ and $(D',E')$. Consider the diagrams
  \[\begin{tikzcd}
    \eta^*\Omega_{g,f,E}\arrow[r,"T_{D,E}"]\arrow[d,"Ex"]&\eta^*\Omega_{g,f,D}\arrow[d,"Ex"]\\
    \Omega_{g,f',E'}\eta^*\arrow[r,"T_{D',E'}"]&\Omega_{g,f',D'}\eta^*
  \end{tikzcd}\quad \begin{tikzcd}
    \Omega_{g,f,E}\eta_*\arrow[r,"T_{D,E}"]\arrow[d,"Ex"]&\Omega_{g,f,D}\arrow[d,"Ex"]\eta_*\\
    \eta_*\Omega_{g,f',E'}\arrow[r,"T_{D',E'}"]&\eta_*\Omega_{g,f',D'}
  \end{tikzcd}\quad \begin{tikzcd}
    \Omega_{g,f',E'}\eta^!\arrow[r,"T_{D',E'}"]\arrow[d,"Ex"]&\Omega_{g,f',D'}\eta^!\arrow[d,"Ex"]\\
    \eta^!\Omega_{g,f,E}\arrow[r,"T_{D,E}"]&\Omega_{g,f,D}
  \end{tikzcd}\]
  of functors. Here, in the first (resp.\ second, resp.\ third) case, we assume the condition $(CE^*)$ (resp.\ $(CE_*)$, resp.\ $(CE^!)$) in (\ref{4.2.5}) for $\eta$ and $\eta'$. We will show that the above diagrams commute under suitable conditions. If the condition (i) of (\ref{4.2.2}) is satisfied, then note that $v_0$ and $v_0'$ are the identity, and the assertion can be checked by considering the diagrams
  \[\begin{tikzcd}
    \rho_0^*c^!r_2^*\arrow[r,"\sim"]\arrow[d,"Ex"]&\rho_0^*c^!v^*q_2^*\arrow[r,"Ex^{-1}"]\arrow[d,"Ex"]&\rho_0^*b^!q_2^*\arrow[dd,"Ex"]\\
    c'^!\psi^*r_2^*\arrow[r,"\sim"]\arrow[dd,"\sim"]&c'^!\psi^*v^*q_2^*\arrow[d,"\sim"]\\
    &c'^!v'^*\rho^*q_2^*\arrow[r,"Ex^{-1}"]\arrow[d,"\sim"]&b'^!\rho^*q_2^*\arrow[d,"\sim"]\\
    c'^!r_2'^*\eta^*\arrow[r,"\sim"]&c'^!v'^*q_2'^*\eta^*\arrow[r,"Ex^{-1}"]&b'^!q_2'^*\eta^*
  \end{tikzcd}\quad \begin{tikzcd}
    c^!r_2^*\eta_*\arrow[r,"\sim"]\arrow[dd,"Ex"]&c^!v^*q_2^*\eta_*\arrow[d,"Ex"]\arrow[r,"Ex^{-1}"]&b^!q_2^*\eta_*\arrow[d,"Ex"]\\
    &c^!v^*\rho_*q_2'^*\arrow[r,"Ex^{-1}"]\arrow[d,"Ex"]&b^!\rho_*q_2'^*\arrow[dd,"Ex^{-1}"]\\
    c^!\psi_*r_2'^*\arrow[r,"\sim"]\arrow[d,"Ex^{-1}"]&c^!\psi_*v'^*q_2'^*\arrow[d,"Ex^{-1}"]\\
    \rho_{0*}c'^!r_2'^*\arrow[r,"\sim"]&\rho_{0*}c'^!v'^*q_2'^*\arrow[r,"Ex^{-1}"]&\rho_{0*}b'^!q_2'^*
  \end{tikzcd}\]
  \[\begin{tikzcd}
    c'^!r_2'^*\eta^!\arrow[r,"\sim"]\arrow[dd,"Ex"]&c'^!v'^*q_2'^*\eta^!\arrow[r,"Ex^{-1}"]\arrow[d,"Ex"]&b'^!q_2'^*\eta^!\arrow[d,"Ex"]\\
    &c'^!v'^*\rho^!q_2^*\arrow[r,"Ex^{-1}"]\arrow[d,"Ex"]&b'^!\rho^!q_2^*\arrow[dd,"\sim"]\\
    c'^!\psi^!r_2^*\arrow[r,"\sim"]\arrow[d,"\sim"]&c'^!\psi^!v^*q_2^*\arrow[d,"\sim"]\\
    \rho_0^!c^!r_2^*\arrow[r,"\sim"]&\rho_0^!c^!v^*q_2^*\arrow[r,"Ex^{-1}"]&\rho_0^!b^!q_2^*
  \end{tikzcd}\]
  of functors. If the condition (ii) of (loc.\ cit) is satisfied, then the assertion can be checked by considering the diagrams
  \[\begin{tikzcd}
    \rho_0^*v_{0*}c^!r_2^*\arrow[r,"Ex"]\arrow[d,"Ex"]&\rho_0^*b^!v_*r_2^*\arrow[r,"\sim"]\arrow[d,"Ex"]&\rho_0^*b^!v_*v^*q_2^*\arrow[d,"Ex"]\arrow[r,"ad^{-1}"]&\rho_0^*b^!q_2^* \arrow[dd,"Ex"]\\
    v_{0*}'\psi_0^*c^!r_2^*\arrow[d,"Ex"]&b'^!\rho^*v_*r_2^*\arrow[r,"\sim"]\arrow[d,"Ex"]&b'^!\rho^*v_*v^*q_2^*\arrow[d,"Ex"]\arrow[rd,"ad^{-1}"]\\
    v_{0*}'c'^!\psi^*r_2^*\arrow[r,"Ex"]\arrow[d,"\sim"]&b'^!v_*'\psi^*r_2^*\arrow[d,"\sim"]\arrow[r,"\sim"]&b'^!v_*'v'^*\rho^*q_2^*\arrow[d,"\sim"]\arrow[r,"ad^{-1}"]&b'^!\rho^*q_2^* \arrow[d,"\sim"]\\
    v_{0*}'c'^!r_2'^*\eta^*\arrow[r,"Ex"]&b'^!v_*'r_2'^*\eta^*\arrow[r,"\sim"]&b'^!v_*'v'^*q_2'^*\eta^*\arrow[r,"ad^{-1}"]&b'^!q_2'^*\eta^*
  \end{tikzcd}\]
  \[\begin{tikzcd}
    v_{0*}c^!r_2^*\eta_*\arrow[d,"Ex"]\arrow[r,"Ex"]&b^!v_*r_2^*\eta_*\arrow[d,"Ex"]\arrow[r,"\sim"]&b^!v_*v^*q_2^*\eta_*\arrow[r,"ad^{-1}"]&b^!q_2^*\eta_*\arrow[d,"Ex"]\\
    v_{0*}c^!\psi_*r_2'^*\arrow[d,"Ex^{-1}"]\arrow[r,"Ex"]&b^!v_*\psi_*r_2'^*\arrow[d,"\sim"]\arrow[r,"\sim"]&b^!\rho_*v_*'v'^*q_2'^*\arrow[dd,"Ex^{-1}"]\arrow[r,"ad^{-1}"]& b^!\rho_*q_2'^*\arrow[dd,"Ex^{-1}"]\\
    v_{0*}\psi_{0*}c'^!r_2'^*\arrow[d,"\sim"]&b^!\rho_*v_*'r_2'^*\arrow[d,"Ex^{-1}"]\\
    \rho_{0*}v_{0*}'c'^!r_2'^*\arrow[r,"Ex"]&\rho_{0*}b'^!v_*'r_2'^*\arrow[r,"\sim"]&\rho_{0*}b'^!v_*'v'^*q_2'^*\arrow[r,"ad^{-1}"]&\rho_{0*}b'^!q_2'^*
  \end{tikzcd}\]
  \[\begin{tikzpicture}[baseline= (a).base]
    \node[scale=.94] (a) at (0,0)
    {\begin{tikzcd}
    v_{0*}'c'^!r_2'^*\eta^!\arrow[d,"ad"]\arrow[r,"Ex"]&b'^!v_*'r_2'^*\eta^!\arrow[d,"ad"]\arrow[rr,"\sim"]&&b'^!v_*'v'^*q_2'^*\eta^!\arrow[d,"ad"]\arrow[r,"ad^{-1}"]&b'^!q_2'^*\eta^! \arrow[d,"ad"]\\
    v_{0*}'c'^!\psi^!\psi_*r_2'^*\eta^!\arrow[d,"Ex^{-1}"]\arrow[r,"Ex"]&b'^!v_*'\psi^!\psi_*r_2'^*\eta^!\arrow[d,"Ex^{-1}"]\arrow[r,"Ex"]&b'^!\rho^!v_*\psi_*r_2'^*\eta^!\arrow[d,"Ex^{-1}"] \arrow[r,"\sim"]&b'^!\rho^!\rho_*v_*'v'^*q_2'^*\eta^!\arrow[r,"ad^{-1}"]&b'^!\rho^!\rho_*q_2'^*\eta^!\arrow[d,"Ex^{-1}"]\\
    v_{0*}'c'^!\psi^!r_2^*\eta_*\eta^!\arrow[r,"Ex"]\arrow[d,"ad'"]&b'^!v_*'\psi^!r_2^*\eta_*\eta^!\arrow[d,"ad'"]\arrow[r,"Ex"]&b'^!\rho^!v_*r_2^*\eta_*\eta^!\arrow[d,"ad'"]\arrow[r,"\sim"] &b'^!\rho^!v_*v^*q_2^*\eta_*\eta^!\arrow[d,"ad'"]\arrow[r,"ad^{-1}"]&b'^!\rho^!q_2^*\eta_*\eta^!\arrow[d,"ad^{-1}"]\\
    v_{0*}'c'^!\psi^!r_2^*\arrow[d,"\sim"]\arrow[r,"Ex"]&b'^!v_*'\psi^!r_2^*\arrow[r,"Ex"]&b'^!\rho^!v_*r_2^*\arrow[r,"\sim"]\arrow[dd,"\sim"]&b'^!\rho^!v_*v^*q_2^*\arrow[r,"ad^{-1}"] &b'^!\rho^!q_2^*\arrow[dd,"\sim"]\\
    v_{0*}'\eta^!c^!r_2^*\arrow[d,"Ex"]\\
    \rho_0^!v_{0*}c^!r_2^*\arrow[rr,"Ex"]&&\rho_0^!b^!v_*r_2^*\arrow[r,"\sim"]&\rho_0^!b^!v_*v^*q_2^*\arrow[r,"ad^{-1}"]&\rho^!b^!q_2^*
   \end{tikzcd}};
  \end{tikzpicture}\]
  of functors. When $b$ is a strict regular embedding, we similarly have the commutative diagrams
  \[\begin{tikzcd}
    \eta^*\Omega_{g,f,D}^d\arrow[r,"T^d"]\arrow[d,"Ex"]&\eta^*\Omega_{g,f,D}\arrow[d,"Ex"]\\
    \Omega_{g,f',D'}^d\eta^*\arrow[r,"T^d"]&\Omega_{g,f',D'}\eta^*
  \end{tikzcd}\quad \begin{tikzcd}
    \Omega_{g,f,D}^d\eta_*\arrow[r,"T^d"]\arrow[d,"Ex"]&\Omega_{g,f,D}\arrow[d,"Ex"]\eta_*\\
    \eta_*\Omega_{g,f',D'}^d\arrow[r,"T^d"]&\eta_*\Omega_{g,f',D'}
  \end{tikzcd}\quad \begin{tikzcd}
    \Omega_{g,f',D'}^d\eta^!\arrow[r,"T^d"]\arrow[d,"Ex"]&\Omega_{g,f',D'}\eta^!\arrow[d,"Ex"]\\
    \eta^!\Omega_{g,f,D}^d\arrow[r,"T^d"]&\Omega_{g,f,D}
  \end{tikzcd}\]
  \[\begin{tikzcd}
    \eta^*\Omega_{g,f,D}^n\arrow[r,"(T^n)^{-1}",leftarrow]\arrow[d,"Ex"]&\eta^*\Omega_{g,f,D}^d\arrow[d,"Ex"]\\
    \Omega_{g,f',D'}^n\eta^*\arrow[r,"(T^n)^{-1}",leftarrow]&\Omega_{g,f',D'}^d\eta^*
  \end{tikzcd}\quad \begin{tikzcd}
    \Omega_{g,f,D}^n\eta_*\arrow[r,"(T^n)^{-1}",leftarrow]\arrow[d,"Ex"]&\Omega_{g,f,D}^d\arrow[d,"Ex"]\eta_*\\
    \eta_*\Omega_{g,f',D'}^n\arrow[r,"(T^n)^{-1}",leftarrow]&\eta_*\Omega_{g,f',D'}^d
  \end{tikzcd}\quad \begin{tikzcd}
    \Omega_{g,f',D'}^n\eta^!\arrow[r,"(T^n)^{-1}",leftarrow]\arrow[d,"Ex"]&\Omega_{g,f',D'}^d\eta^!\arrow[d,"Ex"]\\
    \eta^!\Omega_{g,f,D}^n\arrow[r,"(T^n)^{-1}",leftarrow]&\Omega_{g,f,D}^d
  \end{tikzcd}\]
  of functors.  Here, in the third and sixth diagram, we assume that ($\eta$ and $\eta'$ are proper) or ($\eta$ and $\eta'$ are separated and (Supp) is satisfied).
\end{none}
\begin{none}\label{4.2.11}
  Under the notations and hypotheses of (\ref{4.2.2}), if $b$ and $c$ are strict regular embeddings, we have the commutative diagrams
  \[\begin{tikzcd}
    E_0\arrow[rr]\arrow[rd]\arrow[dd]&&E\arrow[rd]\arrow[rrrd]\arrow[dd,near end]\\
    &D_0\arrow[dd]\arrow[rr,near start,crossing over]&&D\arrow[rr]&&X\arrow[dd]\\
    E_0\times \mathbb{A}^1\arrow[rr]\arrow[rd]&&D_{E_0}E\arrow[rd]\arrow[rrrd]\\
    &D_0\times\mathbb{A}^1\arrow[uu,leftarrow,crossing over]\arrow[rr]&&D_{D_0}D\arrow[uu,leftarrow,crossing over]\arrow[rr]&&X\times\mathbb{A}^1
  \end{tikzcd}\]
  \[\begin{tikzcd}
    E_0\arrow[rr]\arrow[dd]\arrow[rd]&&N_{E_0}E\arrow[rd]\arrow[rrrd]\arrow[dd,near end]\\
    &D_0\arrow[dd]\arrow[rr,near start,crossing over]&&N_{D_0}D\arrow[rr]&&X\arrow[dd]\\
    E_0\times \mathbb{A}^1\arrow[rr]\arrow[rd]&&D_{E_0}E\arrow[rd]\arrow[rrrd]\\
    &D_0\times\mathbb{A}^1\arrow[uu,leftarrow,crossing over]\arrow[rr]&&D_{D_0}D\arrow[uu,leftarrow,crossing over]\arrow[rr]&&X\times\mathbb{A}^1
  \end{tikzcd}\]
  of $\mathscr{S}$-schemes as in (\ref{4.1.1}). Thus we similarly obtain the natural transformations
  \[\Omega_{g,f,E}^d\longrightarrow \Omega_{g,f,D}^d,\quad \Omega_{g,f,E}^n\longrightarrow \Omega_{g,f,D}^n\]
  as in (\ref{4.2.2}) when one of the conditions (i)--(iii) of (loc.\ cit) is satisfied. These are again denoted by $T_{D,E}$ and called {\it transition transformations.} We also have the commutative diagram
  \[\begin{tikzcd}
    \Omega_{g,f,E}^n\arrow[d,"T_{D,E}"]\arrow[r,"(T^n)^{-1}",leftarrow]&\Omega_{g,f,E}^d\arrow[d,"T_{D,E}"]\arrow[r,"T^d"]&\Omega_{g,f,E}\arrow[d,"T_{D,E}"]\\
    \Omega_{g,f,D}^n\arrow[r,"(T^n)^{-1}",leftarrow]&\Omega_{g,f,D}^d\arrow[r,"T^d"]&\Omega_{g,f,D}
  \end{tikzcd}\]
  of functors. Note that in the case (iii), the horizontal arrows are isomorphisms as in (loc.\ cit). In the case (i), if (Supp) is satisfied, then the horizontal arrows are isomorphisms.
\end{none}
\begin{none}\label{4.2.14}
  Under the notations and hypotheses of assume that the conditions of (\ref{4.1.2}) are satisfied. Consider the commutative diagram
  \[\begin{tikzcd}
    \Omega_{g,f,I}^n\arrow[d,"T_{D,I}"]\arrow[r,"(T^n)^{-1}",leftarrow]&\Omega_{g,f,I}^d\arrow[d,"T_{D,I}"]\arrow[r,"T^d"]&\Omega_{g,f,I}\arrow[d,"T_{D,I}"]\\
    \Omega_{g,f,D}^n\arrow[r,"(T^n)^{-1}",leftarrow]&\Omega_{g,f,D}^d\arrow[r,"T^d"]&\Omega_{g,f,D}
  \end{tikzcd}\]
  of functors. By the proof of \cite[2.4.35]{CD12}, the upper horizontal arrows are isomorphisms. The vertical arrows are isomorphisms by (\ref{4.2.2}(i)), so the lower horizontal arrows are also isomorphisms. In particular, the natural transformation
  \[\Omega_{g,f,D}^n\stackrel{(T^n)^{-1}}\longleftarrow \Omega_{g,f,D}^d\]
  has the inverse $T^n$.
\end{none}
\begin{none}\label{4.2.12}
  Under the notations and hypotheses of (\ref{4.2.2}), assume that we have a commutative diagram
  \[\begin{tikzcd}
    F_0\arrow[r,"c'"]\arrow[d,"w_0"]&F\arrow[rdd,"r_2'"]\arrow[d,"w"]\\
    E_0\arrow[r,"c"]\arrow[d,"v_0"]&E\arrow[d,"v"]\arrow[rd,"r_2"']\\
    D_0\arrow[r,"b"']&D\arrow[r,"q_2"']&X
  \end{tikzcd}\]
  of $\mathscr{S}$-schemes and that $w:F\rightarrow E$ and $v:E\rightarrow D$ simultaneously satisfy one of the conditions (i)--(iii) of (loc.\ cit). Then the composition $vw:F\rightarrow D$ also satisfies it, and the diagram
  \[\begin{tikzcd}
    \Omega_{g,f,F}\arrow[rr,"T_{E,F}"]\arrow[rd,"T_{D,F}"']&&\Omega_{g,f,E}\arrow[ld,"T_{D,E}"]\\
    &\Omega_{g,f,D}
  \end{tikzcd}\]
  of functors commutes.

  Assume further that $b$, $c$, and $d$ are strict regular embeddings. Then we similarly have the commutative diagrams
  \[\begin{tikzcd}
    \Omega_{g,f,F}^d\arrow[rr,"T_{E,F}"]\arrow[rd,"T_{D,F}"']&&\Omega_{g,f,E}^d\arrow[ld,"T_{D,E}"]\\
    &\Omega_{g,f,D}^d
  \end{tikzcd}
  \quad\begin{tikzcd}
    \Omega_{g,f,F}^n\arrow[rr,"T_{E,F}"]\arrow[rd,"T_{D,F}"']&&\Omega_{g,f,E}^n\arrow[ld,"T_{D,E}"]\\
    &\Omega_{g,f,D}^n
  \end{tikzcd}\]
  of functors.
\end{none}
\subsection{Composition transformations}
\begin{none}\label{4.3.1}
  Let $h:X\rightarrow Y$ and $g:Y\rightarrow S$ be morphisms of $\mathscr{S}$-schemes, and we put $f=gh$. Consider a commutative diagram
  \[\begin{tikzcd}
    &X\arrow[rd,"b"]\arrow[ld,"b'"']\\
    D'\arrow[rd,"u'"]\arrow[rr,"\rho",near start]\arrow[rddd,"q_2'"']&&D\arrow[dd,"\rho'",near start]\arrow[rd,"u"]\arrow[rrrddd,bend left,"q_2"]\\
    &X\times_{Y}X\arrow[dd,"p_2'"]\arrow[uu,leftarrow,"a'"',near start,crossing over]\arrow[rr,"\varphi",crossing over,near end]&&X\times_S
    X\arrow[rrdd,"p_2"]\arrow[lluu,"a"',leftarrow, bend right, crossing over]\\
    &&D''\arrow[rd,"u''"']\arrow[rrrd,"q_2''"]\\
    &X\arrow[rr,"a''"']\arrow[ru,"b''"]&&Y\times_S X\arrow[uu,"\varphi'"',leftarrow,crossing over]\arrow[rr,"p_2''"']&&X
  \end{tikzcd}\]
  of $\mathscr{S}$-schemes. Assume that the exchange transformation
  \begin{equation}\label{4.3.1.1}
    q_2'^*b''^!\stackrel{Ex}\longrightarrow \rho^!\rho'^*
  \end{equation}
  is defined. For example, when the diagram
  \[\begin{tikzcd}
    D'\arrow[d,"q_2'"]\arrow[r,"\rho"]&D\arrow[d,"\rho'"]\\
    X\arrow[r,"b''"]&D''
  \end{tikzcd}\]
  is Cartesian and the exchange transformation
  \[b''^*\rho_*'\stackrel{Ex}\longrightarrow q_{2*}'\rho^*\]
  is an isomorphism, (\ref{4.3.1.1}) is defined. Then the {\it composition transformation}
  \[\Omega_{h,D'}\Omega_{g,f,D''}\stackrel{C}\longrightarrow \Omega_{f,D}\]
  is given by
  \[b'^!q_2'^*b''^!q_2''^*\stackrel{Ex}\longrightarrow b'^!\rho^!\rho'^*q_2''^*\stackrel{\sim}\longrightarrow b^!q_2^!.\]
  Note that it is an isomorphism when the first arrow is an isomorphism. For example, if (Supp) is satisfied and $\rho'$ is strict smooth separated, then the first arrow is an isomorphism by (\ref{2.5.9}).
\end{none}
\begin{none}\label{4.3.2}
  Under the notations and hypotheses of (\ref{4.3.1}), consider a commutative diagram
  \[\begin{tikzcd}
    &X\arrow[rd,"c"]\arrow[ld,"c'"']\\
    E'\arrow[rd,"v'"]\arrow[rr,"\psi",near start]\arrow[rddd,"r_2'"']&&E\arrow[dd,"\psi'",near start]\arrow[rd,"v"]\arrow[rrrddd,bend left,"r_2"]\\
    &D'\arrow[dd,"q_2'"]\arrow[uu,leftarrow,"b'"',near start,crossing over]\arrow[rr,"\rho",crossing over,near end]&&D\arrow[rrdd,"q_2"]\arrow[lluu,"b"',leftarrow, bend right, crossing
    over]\\
    &&E''\arrow[rd,"v''"']\arrow[rrrd,"r_2''"]\\
    &X\arrow[rr,"b''"']\arrow[ru,"c''"]&&D''\arrow[uu,"\rho'"',leftarrow,crossing over]\arrow[rr,"q_2''"']&&X
  \end{tikzcd}\]
  of $\mathscr{S}$-schemes. Assume that the exchange transformation
  \[r_2'^*c''^!\stackrel{Ex}\longrightarrow \psi^!\psi'^*\]
  is also defined. Then in the below two cases, we will show that the diagram
  \begin{equation}\label{4.3.2.1}\begin{tikzcd}
    \Omega_{h,E'}\Omega_{g,f,E''}\arrow[r,"C"]\arrow[d,"T_{D',E'}T_{D'',E''}"]&\Omega_{f,E}\arrow[d,"T_{D,E}"]\\
    \Omega_{h,D'}\Omega_{g,f,D''}\arrow[r,"C"]&\Omega_{f,D}
  \end{tikzcd}\end{equation}
  of functors commutes where the horizontal arrows are define in (loc.\ cit).
  \begin{enumerate}[(i)]
    \item Assume that the exchange transformations
    \[{\rm id}^*b^!\stackrel{Ex}\longrightarrow c^!v^*,\quad {\rm id}^*b'^!\stackrel{Ex}\longrightarrow c'^!v'^*,\quad {\rm id}^*b''^!\stackrel{Ex}\longrightarrow c''^!v''^*\]
    are defined and isomorphisms. Assume {\it further} that the exchange transformation
    \[v'^*\rho^!\stackrel{Ex}\longrightarrow \psi^!v^*\]
    is defined. Then the commutativity of (\ref{4.3.2.1}) is equivalent to the commutativity of the big outside diagram of the diagram
    \[\begin{tikzcd}
      b'^!q_2'^*b''^!q_2''^*\arrow[d,"Ex"]\arrow[r,"Ex"]&b'^!\rho^!\rho'^*q_2''^*\arrow[d,"Ex"]\arrow[r,"\sim"]&b^!q_2^*\arrow[dd,"Ex"]\\
      c'^!v'^*q_2'^*b''^!q_2''^*\arrow[r,"Ex"]\arrow[d,"Ex"]&c'^!v'^*\rho^!\rho'^*q_2''^*\arrow[d,"Ex"]\\
      c'^!v'^*q_2'^*c''^!v''^*q_2''^*\arrow[d,"\sim"]&c'^!\psi^!v^*\rho'^*q_2''^*\arrow[r,"\sim"]\arrow[d,"\sim"]&c^!v^*q_2^*\arrow[d,"\sim"]\\
      c'^!r_2'^*c''^!r_2''^*\arrow[r,"Ex"]&c'^!\psi^!\psi'^*r_2''^*\arrow[r,"\sim"]&c^!r_2^*
    \end{tikzcd}\]
    of functors. It is true since each small diagram commutes.
    \item Assume that the units
    \[{\rm id}\stackrel{ad}\longrightarrow v_*v^*,\quad {\rm id} \stackrel{ad}\longrightarrow v_*'v'^*,\quad {\rm id}\stackrel{ad}\longrightarrow v_*''v''^*\]
    are isomorphisms. Then the commutativity of (\ref{4.3.2.1}) is equivalent to the commutativity of the big outside diagram of the diagram
    \[\begin{tikzpicture}[baseline= (a).base]
    \node[scale=.78] (a) at (0,0)
    {\begin{tikzcd}
      c'^!r_2'^*c''^!r_2''^*\arrow[r,"ad"]\arrow[d,"ad"]&c'^!\psi^!\psi_*r_2'^*c''^!r_2''^*\arrow[r,"Ex^{-1}"]\arrow[d,"ad"]&c'^!\psi^!\psi'^*c_*''c''^!r_2''^*\arrow[r,"ad'"]\arrow[d,"ad"]
      &c'^!\psi^!\psi'^*r_2''^*\arrow[d,"ad"] \arrow[r,"\sim"]&c^!r_2^*\arrow[d,"ad"]\\
      c'^!v'^!v_*'r_2'^*c''^!r_*''^*\arrow[d,"\sim"]&c'^!\psi^!v^!v_*\psi_*r_2'^*c''^!r_2''^*\arrow[r,"Ex^{-1}"]\arrow[d,"\sim",leftarrow]&c'^!\psi^!v^!v_*\psi'^*c_*''c''^!r_2''^*
      \arrow[r,"ad'"] \arrow[d,"Ex",leftarrow]&c'^!\psi^!v^!v_*\psi^*r_2''^* \arrow[r,"\sim"]\arrow[d,"Ex",leftarrow]&c^!v^!v_*r_2^*\arrow[dddd,"\sim"]\\
      b'^!v_*'v'^*q_2'^*c''^!r_2''^*\arrow[r,"ad"]\arrow[d,"ad^{-1}"]&b'^!\rho^!\rho_*v_*'v'^*q_2'^*c''^!r_2''^*\arrow[d,"ad^{-1}"]&c'^!\psi^!v^!\rho'^*v_*''c_*''c''^!r_2''^*
      \arrow[r,"ad'"]\arrow[d,"\sim",leftarrow]&c'^!\psi^!v^!\rho'^*v_*''r_2''^*\arrow[d,"\sim",leftarrow]\\
      b'^!q_2'^*c''^!r_2''^*\arrow[r,"ad"]\arrow[d,"ad"]&b'^!\rho^!\rho_*q_2'^*c''^!r_2''^*\arrow[r,"Ex^{-1}"]\arrow[d,"ad"]&b'^!\rho^!\rho'^*b_*''c''^!r_2''^*\arrow[d,"ad"]
      &b'^!\rho^!\rho'^*v_*''v''^*q_2''^*\arrow[ddd,"ad^{-1}"]\\
      b'^!q_2'^*c''^!v''^!v_*''r_2''^*\arrow[r,"ad"]\arrow[d,"\sim"]&b'^!\rho^!\rho_*q_2'^*c''^!v''^!v_*''r_2''^*\arrow[r,"Ex^{-1}"]\arrow[d,"\sim"]
      &b'^!\rho^!\rho'^*b_*''c''^!v''^!v_*''r_2''^*\arrow[d,"\sim"]\\
      b'^!q_2'^*b''^!v_*''v''^*q_2''^*\arrow[r,"ad"]\arrow[d,"ad^{-1}"]&b'^!\rho^!\rho_*q_2'^*b''^!v_*''v''^*q_2''^*\arrow[r,"Ex^{-1}"]\arrow[d,"ad^{-1}"]
      &b'^!\rho^!\rho'^*b_*''b''^!v_*''b''^*q_2''^*\arrow[d,"ad^{-1}"]\arrow[ruu,"ad'"']&&b^!v_*v^*q_2^*\arrow[d,"ad^{-1}"]\\
      b'^!q_2'^*b''^!q_2''^*\arrow[r,"ad"]&b'^!\rho^!\rho_*q_2'^*b''^!q_2''^*\arrow[r,"Ex^{-1}"]&b'^!\rho^!\rho'^*b_*''b''^!q_2''^*\arrow[r,"ad'"]&b'^!\rho^!\rho'^*q_2''^*\arrow[r,"\sim"]&
      b^!q_2^*
    \end{tikzcd}};
    \end{tikzpicture}\]
    of functors. It is true since each small diagram commutes.
  \end{enumerate}
\end{none}
\subsection{Purity transformations}
\begin{none}\label{4.4.2}
  In this subsection, we will introduce purity transformations and their functorial properties.
\end{none}
\begin{df}\label{4.4.5}
  Let $f:X\rightarrow S$ be an exact log smooth morphism of $\mathscr{S}$-schemes. Assume that ($f$ is proper) or ($f$ is separated and $\mathscr{T}$ satisfies (Supp)). We also assume that we have a commutative diagram
  \[\begin{tikzcd}
    &D\arrow[d,"u"]\arrow[rd,"q_2"]\\
    X\arrow[ru,"b"]\arrow[r,"a"']&X\times_S X\arrow[r,"p_2"']&X
  \end{tikzcd}\]
  of $\mathscr{S}$-schemes where
  \begin{enumerate}
    \item $b$ is a strict regular embedding,
    \item $u$ satisfies one of the conditions (i)--(iii) of (\ref{4.2.2}).
  \end{enumerate}
  Then we denote by
  \[\mathfrak{q}_{f,D}^n:\Omega_{f,D}^n f^!\longrightarrow f^*,\quad \mathfrak{q}_{f,D}^o:\Omega_f^o f^!\longrightarrow f^*\]
  the compositions
  \[f^*\stackrel{\mathfrak{q}_f}\longrightarrow \Omega_f f^!\stackrel{T_D}\longrightarrow \Omega_{f,D}f^!\stackrel{T^d}\longrightarrow \Omega_{f,D}^d f^!\stackrel{T^n}\longrightarrow \Omega_{f,D}^n f^!,\]
  \[f^*\stackrel{\mathfrak{q}_f}\longrightarrow \Omega_f f^!\stackrel{T_D}\longrightarrow \Omega_{f,D}f^!\stackrel{T^d}\longrightarrow \Omega_{f,D}^d f^!\stackrel{T^n}\longrightarrow \Omega_{f,D}^n f^!\stackrel{T^o}\longrightarrow \Omega_{f,D}^o f^!\]
  respectively. Their left adjoints are denoted by
  \[\mathfrak{p}_{f,D}^n:f_\sharp \longrightarrow f_!\Sigma_{f,D}^n,\quad \mathfrak{p}_{f,D}^o:f_\sharp \longrightarrow f_!\Sigma_{f,D}^o\]
  respectively.
\end{df}
\begin{none}\label{4.4.3}
  Let $h:X\rightarrow Y$ and $g:Y\rightarrow S$ be separated $\mathscr{P}$-morphisms of $\mathscr{S}$-schemes, and we put $f=gh$. Then we have the commutative diagram
  \[\begin{tikzcd}
    X\arrow[d,"a'"]\arrow[rd,"a"]\\
    X\times_Y X\arrow[d,"p_2"]\arrow[r,"\varphi"]&X\times_S X\arrow[d,"\varphi'"]\arrow[rd,"p_2"]\\
    X\arrow[r,"a''"']&Y\times_S X\arrow[r,"p_2''"']&X
  \end{tikzcd}\]
  of $\mathscr{S}$-schemes, and the exchange transformation
  \[p_2'^*a''^!\stackrel{Ex}\longrightarrow \varphi^!\varphi'^*\]
  is defined by (eSm-BC). Thus by (\ref{4.3.1}), we have the composition transformation
  \[C:\Omega_h\Omega_{g,f}\rightarrow \Omega_f.\]
  Then the diagram
  \[\begin{tikzcd}
    \Omega_hh^!\Omega_g g^!\arrow[r,"\mathfrak{q}_h\mathfrak{q}_g"]\arrow[d,"Ex",leftarrow]&h^*g^*\arrow[ddd,"\sim"]\\
    \Omega_h\Omega_{g,f}h^!g^!\arrow[d,"\sim"]\\
    \Omega_h\Omega_{g,f}f^!\arrow[d,"C"]\\
    \Omega_ff^!\arrow[r,"\mathfrak{q}_f"]&f^*
  \end{tikzcd}\]
  of functors commutes by the proof of \cite[1.7.3]{Ayo07}.
\end{none}
\begin{none}\label{4.4.4}
  Let $f:X\rightarrow S$ be a separated and vertical $\mathscr{P}$-morphism of $\mathscr{S}$-schemes, and let $i:S\rightarrow X$ be its section. Then we have a commutative diagram
  \[\begin{tikzcd}
    S\arrow[r,"i"]\arrow[d,"i"]&X\arrow[r,"f"]\arrow[d,"i'"]&S\arrow[d,"i"]\\
    X\arrow[r,"a"]&X\times_S X\arrow[d,"p_1"]\arrow[r,"p_2"]&X\arrow[d,"f"]\\
    &X\arrow[r,"f"]&S
  \end{tikzcd}\]
  of $\mathscr{S}$-schemes where
  \begin{enumerate}[(i)]
    \item $a$ denotes the diagonal morphism, and $p_2$ denotes the second projection,
    \item each square is Cartesian.
  \end{enumerate}
  Consider the diagram
  \[\begin{tikzcd}
    \Omega_{f,{\rm id}}i^!f^!\arrow[r,"\sim"]\arrow[d,"Ex"]&\Omega_{f,{\rm id}}\arrow[d,equal]\\
    i^!\Omega_ff^!\arrow[r,"\mathfrak{q}_f"]&i^!f^*
  \end{tikzcd}\]
  of functors. It commutes since the big outside diagram of the diagram
  \[\begin{tikzcd}
  i^!f^*i^!f^!\arrow[r,"\sim"]\arrow[d,"Ex"]&i^!f^*(fi)^!\arrow[d,"Ex"]\arrow[r,"\sim"]&i^!f^*\arrow[dd,equal]\\
  i^!i'^!p_2^*f^!\arrow[d,"\sim"]&i^!(p_1a)^!f^*\arrow[ru,"\sim"]\arrow[d,"\sim"]\\
  i^!a^!p_2^*f^!\arrow[r,"Ex"]&i^!a^!p_1^!f^*\arrow[r,"\sim"]&i^!f^*
  \end{tikzcd}\]
  of $\mathscr{S}$-schemes commutes.
\end{none}
\begin{none}\label{4.4.6}
  Consider a Cartesian diagram
  \[\begin{tikzcd}
    X'\arrow[d,"f'"]\arrow[r,"g'"]&X\arrow[d,"f"]\\
    S'\arrow[r,"g"]&S
  \end{tikzcd}\]
  of $\mathscr{S}$-schemes where $f$ is an exact log smooth morphism. Assume that ($f$ is proper) or ($f$ is separated and (Supp) is satisfied). Then we have the commutative diagram
  \[\begin{tikzcd}
    X'\arrow[r,"a'"]\arrow[d,"g'"]&X'\times_{S'} X'\arrow[d,"g''"]\arrow[r,"p_2'"]&X'\arrow[d,"g'"]\\
    X\arrow[r,"a"]&X\times_S X\arrow[r,"p_2"]&X
  \end{tikzcd}\]
  of $\mathscr{S}$-schemes where each square is Cartesian and $p_2$ denotes the second projection. We also denote by $p_1$ (resp.\ $p_1'$) the first projection $X\times_S X\rightarrow X$ (resp.\ $X'\times_{S'}X'\rightarrow X'$). In this setting, we will show that the diagram
  \[\begin{tikzcd}
    f_\sharp'g'^*\arrow[dd,"Ex"]\arrow[r,"\mathfrak{p}_{f'}"]&f_!'\Sigma_{f'} g'^!\arrow[d,"Ex"]\\
    &f_!'g'^*\Sigma_f\arrow[d,leftarrow,"Ex"]\\
    g^*f_\sharp\arrow[r,"\mathfrak{p}_f"]&g^*f_!\Sigma_f
  \end{tikzcd}\]
  of functors commutes. It is the big outside diagram of the diagram
  \[\begin{tikzcd}
    f_\sharp'g'^*\arrow[r,"\sim"]\arrow[ddd,"Ex"]\arrow[rd,"\sim"]&f_\sharp'p_{1!}'a_*'g'^*\arrow[rr,"Ex"]\arrow[rd,"Ex",leftarrow]&&f_!'p_{2\sharp}'a_*'g'^*\arrow[d,"Ex",leftarrow]\\
    &f_\sharp'g'^*p_{1!}a_*\arrow[r,"Ex"]\arrow[dd,"Ex"]&f_\sharp'p_{1!}'g''^*a_*\arrow[r,"Ex"]&f_!'p_{2\sharp}'g''^*a_*\arrow[d,"Ex"]\\
    &&&f_!'g'^*p_{2\sharp}a_*\arrow[d,"Ex",leftarrow]\\
    g^*f_\sharp\arrow[r,"\sim"]&g^*f_\sharp p_{1!}a_*\arrow[rr,"Ex"]&&g^*f_!p_{2\sharp}a_*
  \end{tikzcd}\]
  of functors. Thus the assertion follows from the fact that each small diagram commutes.
\end{none}
\section{Support property}
\begin{none}
  Throughout this section, fix a $eSm$-premotivic triangulated category satisfying (Adj), (Htp--1), (Htp--2), (Htp--3), (Loc), (s\'et-Sep), and (Stab) that is generated by $eSm$ and $\tau$. In \S\ref{section5.6}, we assume also the axiom (ii) of (\ref{2.9.1}) and (Htp--4).
\end{none}
\subsection{Elementary properties of the support property}
\begin{none}\label{5.1.0}
  We will define the universal and semi-universal support property for not necessarily proper morphisms, and we will show that our definition coincides with the usual definition for proper morphisms in (\ref{5.1.3}). Then we will study elementary properties of the universal support property. Recall from (\ref{2.5.8}) that any proper strict morphism of $\mathscr{S}$-schemes satisfies the support property.
\end{none}
\begin{prop}\label{5.1.1}
  Let $g:Y\rightarrow X$ and $f:X\rightarrow S$ be proper morphisms of $\mathscr{S}$-schemes. If $f$ and $g$ satisfy the support property, then $fg$ also satisfies the support property.
\end{prop}
\begin{proof}
  Consider a commutative diagram
  \[\begin{tikzcd}
    W\arrow[r,"g'"]\arrow[d,"j''"]&V\arrow[d,"j'"]\arrow[r,"f'"]&U\arrow[d,"j"]\\
    Y\arrow[r,"g"]&X\arrow[r,"f"]&S
  \end{tikzcd}\]
  of $\mathscr{S}$-schemes where $j$ is an open immersion and each square is Cartesian. Then the conclusion follows from the commutativity of the diagram
  \[\begin{tikzcd}
    j_\sharp f_*'g_*'\arrow[d,"\sim"]\arrow[r,"Ex"]&f_*j_\sharp'g_*\arrow[r,"Ex"]&f_*g_*j_\sharp''\arrow[d]\\
    j_\sharp(f'g')_*\arrow[rr,"Ex"]&&(fg)_*j_\sharp''
  \end{tikzcd}\]
  of functors.
\end{proof}
\begin{df}\label{5.1.2}
  Let $f:X\rightarrow S$ be a morphism of $\mathscr{S}$-schemes. We say that $f$ satisfies the {\it universal} (resp.\ semi-universal) support property if any pullback of the {\it proper} morphism $X\rightarrow \underline{X}\times_{\underline{S}}S$ (resp.\ any pullback of the {\it proper} morphism $X\rightarrow \underline{X}\times_{\underline{S}}S$ via a strict morphism $Y\rightarrow \underline{X}\times_{\underline{S}}S$) satisfies the support property.
\end{df}
\begin{prop}\label{5.1.3}
  Let $f:X\rightarrow S$ be a {\it proper} morphism of $\mathscr{S}$-schemes. Then $f$ satisfies the {\it universal} (resp.\ semi-universal) support property if and only if any pullback of $f$ (resp.\ any pullback of $f$ via a strict morphism) satisfies the support property.
\end{prop}
\begin{proof}
  If $f$ satisfies the universal (resp.\ semi-universal) support property, let $f':X'\rightarrow S'$ be a pullback of $f$ via a morphism (resp.\ strict morphism) $S'\rightarrow S$. Consider the commutative diagram
  \[\begin{tikzcd}
    X'\arrow[r,"u'"]\arrow[d,"g'"]&\underline{X}\times_{\underline{S}}S'\arrow[d,"g''"]\arrow[r,"v'"]&S'\arrow[d,"g"]\\
    X\arrow[r,"u"]&\underline{X}\times_{\underline{S}}S\arrow[r,"v"]&S
  \end{tikzcd}\]
  of $\mathscr{S}$-schemes. By assumption, $u'$ satisfies the support property. Since $v'$ is strict proper, it satisfies the support property by (\ref{5.1.0}). Thus $f=v'u'$ satisfies the support property by (\ref{5.1.1}).

  Conversely, if the support property is satisfied for any pullback of $f$ (resp.\ for any pullback of $f$ via a strict morphism), we put $T=\underline{X}\times_{\underline{S}}S$, and let $p':X\times_T T'\rightarrow T'$ be a pullback of $X\rightarrow T$ via a morphism (resp.\ strict morphism) $T'\rightarrow T$. The morphism $p'$ has the factorization
  \[X\times_T T'\stackrel{r}\rightarrow X\times_S T'\stackrel{q}\rightarrow T'\]
  where $r$ denotes the morphism induced by $T\rightarrow S$, and $q$ denotes the projection. Then the morphism $r$ is a closed immersion since it is a pullback of the diagonal morphism $T\rightarrow T\times_S T$, so $r$ satisfies the support property, and the morphism $q$
  satisfies the support property since it is a pullback of $f$ via the morphism (resp.\ strict morphism) $T'\rightarrow T$. Thus by (\ref{5.1.1}), the morphism $p'=qr$ satisfies the support property.
\end{proof}
\begin{prop}\label{5.1.7}
  Let $f:X\rightarrow S$ be a morphism of $\mathscr{S}$-schemes. Then the question that $f$ satisfies the universal (resp.\ semi-universal) support property is strict \'etale local on $X$.
\end{prop}
\begin{proof}
  Replacing $f$ by $X\rightarrow \underline{X}\times_{\underline{S}}S$, we may assume that $\underline{f}$ is an isomorphism. Then the question is strict \'etale local on $S$ by (s\'et-Sep), which implies that the question is strict \'etale local on $X$.
\end{proof}
\begin{prop}\label{5.1.4}
  Let $g:Y\rightarrow X$ and $f:X\rightarrow S$ be morphisms of $\mathscr{S}$-schemes.
  \begin{enumerate}[(1)]
    \item If $f$ is strict, then $f$ satisfies the universal support property.
    \item If $f$ and $g$ satisfy the universal (resp.\ semi-universal) support property, then $fg$ also satisfies the universal (resp.\ semi-universal) support property.
    \item Assume that $g$ is proper and that for any pullback $h$ of $g$ via a strict morphism, the unit
    \[{\rm id}\stackrel{ad}\longrightarrow h_*h^*\]
    is an isomorphism. If $fg$ satisfy the semi-universal support property, then $f$ satisfies the semi-universal support property.
  \end{enumerate}
\end{prop}
\begin{proof}
  (1) It is true since the morphism $X\rightarrow \underline{X}\times_{\underline{S}}S$ is an isomorphism when $f$ is strict.\\[4pt]
  (2) The induced morphism $p:Y\rightarrow \underline{Y}\times_{\underline{S}}S$ has the factorization
  \[Y\stackrel{r}\longrightarrow \underline{Y}\times_{\underline{X}}X\stackrel{q}\longrightarrow \underline{Y}\times_{\underline{S}}S\]
  where $r$ denotes the morphism induced by $Y\rightarrow \underline{Y}$ and $Y\rightarrow X$, and $q$ denotes the morphism induced by $X\rightarrow S$.
  Any pullback of $r$ (resp.\ any pullback of $r$ via a strict morphism) satisfies the support property by assumption, and any pullback of $q$ (resp.\ any pullback of $q$ via a strict morphism) satisfies the support property by assumption since $q$ is a pullback of the morphism
  $X\rightarrow\underline{X}\times_{\underline{S}}S$. Thus by (\ref{5.1.1}), any pullback of $p$ (resp.\ any pullback of $p$ via a strict morphism) satisfies the support property, i.e., $fg$ satisfies the universal support property.\\[4pt]
  (3) Replacing $f$ by $X\rightarrow \underline{X}\times_{\underline{S}}S$, we may assume that $\underline{f}$ is an isomorphism. The question is also preserved by any base change via a strict morphism to $S$, so we only need to prove that $f$ satisfies the support property. Consider the commutative diagram
  \[\begin{tikzcd}
    W\arrow[d,"j''"]\arrow[r,"g'"]&V\arrow[d,"j'"]\arrow[r,"f'"]&U\arrow[d,"j"]\\
    Y\arrow[r,"g"]&X\arrow[r,"f"]&S
  \end{tikzcd}\]
  of $\mathscr{S}$-schemes where $j$ is an open immersion and each square is Cartesian. We have the commutative diagram
  \[\begin{tikzcd}
    j_\sharp f_*'\arrow[d,"ad"]\arrow[rrrr,"Ex"]&&&&f_*j_\sharp'\arrow[d,"ad"]\\
    j_\sharp f_*'g_*'g'^*\arrow[r,"\sim"]&j_\sharp (f'g')_*g'^*\arrow[r,"Ex"]&(fg)_*j_\sharp''g'^*\arrow[r,"Ex"]&(fg)_*g^*j_\sharp\arrow[r,"\sim"]&f_*g_*g^*j_\sharp'
  \end{tikzcd}\]
  of functors, and we want to show that the upper horizontal arrow is an isomorphism. The right vertical arrow is isomorphisms by assumption, and the third bottom horizontal arrow is an isomorphism by ($eSm$-BC). Moreover, the morphism
  $f'g':W\rightarrow U$ satisfies the support property by assumption, so the second bottom horizontal arrow is an isomorphism. Thus to show that the upper horizontal arrow is an isomorphism, it suffices to show the left vertical arrow is an isomorphism. To show this, we will show that the unit
  \[{\rm id}\stackrel{ad}\longrightarrow g_*'g'^*\]
  is an isomorphism.

  Consider the commutative diagram
  \[\begin{tikzcd}
    {\rm id}\arrow[rrr,"ad"]\arrow[d,"ad"]&&&g_*'g'^*\arrow[d,"ad"]\\
    j'^*j_\sharp'\arrow[r,"ad"]&j'^*g_*g^*j_\sharp'\arrow[r,"Ex"]&g_*'j''^*g^*j_\sharp'\arrow[r,"Ex"]&g_*j''^*j_\sharp''g'^*
  \end{tikzcd}\]
  of functors. The vertical arrows are isomorphisms by (\ref{2.2.1}), and the lower left horizontal arrow is an isomorphism by assumption. Moreover, the lower middle and right horizontal arrows are isomorphisms by ($eSm$-BC). Thus the upper horizontal arrow is an isomorphism.
\end{proof}
\begin{none}\label{5.1.5}
  Let $g:S'\rightarrow S$ be a morphism of $\mathscr{S}$-schemes. We will sometimes assume that
  \begin{enumerate}[(i)]
    \item for any pullback $g'$ of $g$, $g'^*$ is conservative,
    \item for any commutative diagram
    \[\begin{tikzcd}
      Y'\arrow[r,"g''"]\arrow[d,"h'"]&Y\arrow[d,"h"]\\
      X'\arrow[r,"g'"]\arrow[d,"f'"]&X\arrow[d,"f"]\\
      S'\arrow[r,"g"]&S
    \end{tikzcd}\]
    of $\mathscr{S}$-schemes such that each square is Cartesian, the exchange transformation
    \[g'^*h_*\stackrel{Ex}\longrightarrow h_*'g''^*\]
    is an isomorphism.
  \end{enumerate}
\end{none}
\begin{prop}\label{5.1.6}
  Consider a Cartesian diagram
  \[\begin{tikzcd}
    X'\arrow[r,"g'"]\arrow[d,"f'"]&X\arrow[d,"f"]\\
    S'\arrow[r,"g"]&S
  \end{tikzcd}\]
  of $\mathscr{S}$-schemes. Assume that $g$ satisfies the conditions of (\ref{5.1.5}). If $f'$ satisfies the universal support property, then $f$ satisfies the universal support property.
\end{prop}
\begin{proof}
  Replacing $f$ by $X\rightarrow \underline{X}\times_{\underline{S}}S$, we may assume that $\underline{f}$ is an isomorphism. Then $f'$ is proper, so it satisfies the support property by (\ref{5.1.3}). Since the question is preserved by any base change of $S$, it suffices to show that $f$ satisfies the support property.

  Consider a commutative diagram
  \[\begin{tikzcd}
    V'\arrow[dd,"p'"]\arrow[rr,"q'"]\arrow[rd,"u'"]&&V\arrow[dd,"p",near start]\arrow[rd,"u"]\\
    &X'\arrow[rr,"g'",near start,crossing over]&&X\arrow[dd,"f"]\\
    U'\arrow[rr,"q",near start]\arrow[rd,"j'"]&&U\arrow[rd,"j"']\\
    &S'\arrow[rr,"g"']\arrow[uu,leftarrow,crossing over,"f'"',near end]&&S
  \end{tikzcd}\]
  where $j$ is an open immersion and each small square is Cartesian. We want to show that the natural transformation
  \[j_\sharp p_*\stackrel{Ex}\longrightarrow f_*u_\sharp\]
  is an isomorphism. Consider the commutative diagram
  \[\begin{tikzcd}
    j_\sharp'q^*p_*\arrow[d,"Ex"]\arrow[r,"Ex"]&j_\sharp'p_*'q'^*\arrow[r,"Ex"]&f_*'u_\sharp'q'^*\arrow[d,"Ex"]\\
    g^*j_\sharp p_*\arrow[r,"Ex"]&g^*f_*u_\sharp\arrow[r,"Ex"]&f_*'g'^*u_\sharp
  \end{tikzcd}\]
  of functors. The vertical arrows are isomorphisms by ($eSm$-BC), and the upper left and lower right horizontal arrows are isomorphisms by the assumption that $g$
  satisfies the conditions of (\ref{5.1.5}). Moreover, the upper right horizontal arrow is an isomorphism since $f'$ satisfies the support property. Thus the lower left horizontal arrow
  is an isomorphism. Then the conservativity of $g^*$ implies the support property for $f$.
\end{proof}
\subsection{Conservativity}
\begin{lemma}\label{5.3.2}
  Let $F:\mathcal{C}\rightarrow \mathcal{C}'$ and $G:\mathcal{C}'\rightarrow \mathcal{C}''$ be functors of categories. Assume that for any objects $X$ and $Y$ of $\mathcal{C}$, the
  function
  \[\tau_{XY}:{\rm Hom}_{\mathcal{C}'}(FX,FY)\rightarrow {\rm Hom}_{\mathcal{C''}}(GFX,GFY)\]
  defined by
  \[f\mapsto Gf\]
  is bijective. If $F$ is conservative, then $GF$ is also conservative.
\end{lemma}
\begin{proof}
  Let $X$ and $Y$ be objects of $\mathcal{C}$, and let $\alpha:X\rightarrow Y$ be a morphism in $\mathcal{C}$ such that $GF\alpha$ is an isomorphism. We put $\beta=GF\alpha$. Choose
  the inverse of $\phi:GFY\rightarrow GFX$ of $\alpha$. Then
  \[{\rm id}=\tau_{XX}^{-1}({\rm id})=\tau_{XX}^{-1}(\phi\circ \beta)=\tau_{YX}^{-1}(\phi)\circ \tau_{XY}^{-1}(\beta),\]
  so $F\alpha=\tau_{XY}^{-1}(\beta)$ has a left inverse. Similarly, $F\alpha$ has a right inverse. Thus $F\alpha$ is an isomorphism. Then the conservativity of $F$ implies that
  $\alpha$ is an isomorphism.
\end{proof}
\begin{none}\label{5.3.3}
  Let $f:X\rightarrow S$ be a Kummer log smooth morphism of $\mathscr{S}$-schemes with an fs chart $\theta:P\rightarrow Q$ of Kummer log smooth type. We will construct a homomorphism $\eta:P\rightarrow P'$ of Kummer log smooth over $S$ type with the following properties.
  \begin{enumerate}[(i)]
    \item Let $g:S'\rightarrow S$ denotes the projection $S\times_{\mathbb{A}_P}\mathbb{A}_{P'}\rightarrow S$. For any pullback $u$ of $g$, $u^*$ is conservative.
    \item In the Cartesian diagram
    \[\begin{tikzcd}
      X'\arrow[r,"g'"]\arrow[d,"f'"]&X\arrow[d,"f"]\\
      S'\arrow[r,"g"]&S
    \end{tikzcd}\]
    of $\mathscr{S}$-schemes, $f'$ is strict smooth.
  \end{enumerate}
  This will be used in the proof of (\ref{5.4.1}).

  Consider the homomorphisms
  \[\lambda:P\rightarrow P\oplus Q,\quad \eta:P\rightarrow P^{\rm gp}\oplus Q\]
  defined by $p\mapsto (p,\theta(p))$. Using these homomorphism, we construct the fiber products
  \[S''=S\times_{\mathbb{A}_P}\mathbb{A}_{P\oplus Q},\quad S'=S\times_{\mathbb{A}_P}\mathbb{A}_{P^{\rm gp}\oplus Q}.\]
  Consider the commutative diagram
  \[\begin{tikzcd}
    S\arrow[r,"s"]\arrow[rd,"{\rm id}"']&S''\arrow[d,"h"]\arrow[r,leftarrow,"j"]&S'\arrow[ld,"g"]\\
    &S
  \end{tikzcd}\]
  of $\mathscr{S}$-schemes where $s$ denotes the morphism constructed by the homomorphism $P\oplus Q\rightarrow P$ defined by $(p,q)\mapsto p$, $h$ denotes the projection, and $j$ denotes the open immersion induced by the inclusion $P\oplus Q\rightarrow P^{\rm gp}\oplus Q$.

  From $s^*h^*\cong {\rm id}$, we see that $h^*$ is conservative. We will show that $g^*$ is also conservative. By (Htp--3), the composition
  \[h_\sharp h^*\stackrel{\sim}\longrightarrow g_\sharp j_\sharp j^*g^*\stackrel{ad'}\longrightarrow g_\sharp g^*\]
  is an isomorphism, so for $F,G\in \mathscr{T}(S)$, the homomorphism
  \[{\rm Hom}_{\mathscr{T}(S'')}(h^*F,h^*G)\rightarrow {\rm Hom}_{\mathscr{T}(S')}(g^*F,g^*G)\tag{*}\]
  is an isomorphism. Then (\ref{5.3.2}) implies that $g^*$ is conservative. The same proof can be applied to show that for any pullback $u$ of $g$, $u^*$ is conservative.

  The remaining is to show that $f'$ is strict. The homomorphism $\theta:P\rightarrow Q$ factors through $P'=P^{\rm gp}\oplus Q$ via $p\mapsto (p,\theta(p))$ and $(p,q)\mapsto q$, so
  the morphism $f:X\rightarrow S$ factors through $S'$. Consider the commutative diagram
  \[\begin{tikzcd}
    X'\arrow[d]\arrow[r,"g'"]&X\arrow[d]\\
    X''\arrow[d,"p_1"]\arrow[r]&S'\arrow[d,"g"]\\
    S'\arrow[r,"g"]&S
  \end{tikzcd}\]
  of $\mathscr{S}$-schemes where each square is Cartesian. The morphism $X\rightarrow S'$ is strict, so to show that $f'$ is strict, it suffices to show that $p_1$ is strict. This
  follows from (\ref{D18}) below.
\end{none}
\begin{lemma}\label{D18}
  Let $\theta:P\rightarrow Q$ be a Kummer homomorphism of fs monoids. Then the summation homomorphism $\eta:Q\oplus_P Q\rightarrow Q$ is strict.
\end{lemma}
\begin{proof}
  The homomorphism $\overline{\eta}:\overline{Q\oplus_P Q}\oplus \overline{Q}$ is surjective, so the remaining is to show that $\overline{d}$ is injective. Choose $n\in \mathbb{N}^+$ such that $nq\subset \theta(P)$. For any $q\in Q$, $n(q,-q)=(nq,0)+(0,-nq)=0$ because $nq\in \theta(P)$. Thus $(q,-q)\in (Q\oplus_P Q)^*$ since $Q\oplus_P Q$ is saturated. Let $Q'$ denote the submonoid of $Q\oplus_P Q$ generated by elements of the form $(q,-q)$ for $q\in Q^{\rm gp}$. Then $Q'\subset (Q\oplus_P Q)^*$, and $Q/Q'\cong Q$. The injectivity follows from this.
\end{proof}
\begin{none}\label{5.3.4}
  Let $f:X\rightarrow S$ be a Kummer log smooth morphism of $\mathscr{S}$-schemes. By (\ref{0.1.4}), we can choose a strict \'etale cover $\{u_i:S_i\rightarrow S\}_{i\in I}$ such that for each $i$, there is a commutative diagram
  \[\begin{tikzcd}
    X_i\arrow[r,"v_i"]\arrow[d,"f_i"]&X\arrow[d,"f"]\\
    S_i\arrow[r,"u_i"]&S
  \end{tikzcd}\]
  of $\mathscr{S}$-schemes such that $f_i$ has an fs chart of log smooth type and $\{v_i\}_{i\in I}$ is a strict \'etale cover. Then by (\ref{5.3.3}), there is a Kummer log smooth morphism satisfying the conditions (i) and (ii) of (loc.\ cit). Let $g:S'\rightarrow S$ denote the union of $g_iu_i:S_i'\rightarrow S$. Then $g$ satisfies the condition (i) of (loc.\ cit) by (k\'et-Sep), and $g$ satisfies the condition (ii) of (loc.\ cit) by construction.
\end{none}
\subsection{Support property for Kummer log smooth morphisms}
\begin{prop}\label{5.4.1}
  Let $f$ be a Kummer log smooth morphisms of $\mathscr{S}$-schemes. Then $f$ satisfies the universal support property.
\end{prop}
\begin{proof}
  By (\ref{5.1.7}) and (\ref{0.1.4}), we may assume that $f$ has an fs chart $\theta:P\rightarrow Q$ of Kummer log smooth type. As in (\ref{5.3.3}), choose a Kummer log smooth morphism $g:S'\rightarrow S$ satisfying the condition (i) of (\ref{5.1.5}) such that the pullback of $f$ via $g:S'\rightarrow S$ is strict. Since $g$ is an exact log smooth morphism, the condition (ii) of (loc.\ cit) is satisfied by ($eSm$-BC). Now we can apply (\ref{5.1.6}), so replacing $f$ by the projection $X\times_S S'\rightarrow S'$, we may assume that $f$ is strict. Then the conclusion follows from (\ref{5.1.4}(1)).
\end{proof}
\begin{prop}\label{5.4.2}
  Let $f:X\rightarrow S$ and $g:Y\rightarrow X$ be morphisms of $\mathscr{S}$-schemes such that $g$ is proper. If $f$ and $fg$ satisfy the semi-universal support property, then $g$ satisfies the semi-support property.
\end{prop}
\begin{proof}
  Consider a commutative diagram
  \[\begin{tikzcd}
    T'\arrow[rr,"\beta"]\arrow[d,"\alpha'"]&&T\arrow[d,"\alpha"]\\
    Y\arrow[rr,"g"]\arrow[rd,"fg"']&&X\arrow[ld,"f"]\\
    &S
  \end{tikzcd}\]
  of $\mathscr{S}$-schemes where the small square is Cartesian and $\alpha$ is strict. By (\ref{5.1.3}), it suffices to show that $\beta$ satisfies the support property for any $\alpha$. The morphisms $\alpha$ and $\alpha'$ satisfy the semi-universal support property by (\ref{5.1.4}(1)) since they are strict, so the morphisms $f\alpha$ and $fg\alpha'$satisfies the semi-universal support property by (\ref{5.1.4}(2)). Hence replacing $(Y,X,S)$ by $(T',T,S)$, we reduce to showing that $g$ satisfies the support property.

  Consider a Cartesian diagram
  \[\begin{tikzcd}
    W\arrow[r,"w"]\arrow[d,"g'"]&Y\arrow[d,"g"]\\
    V\arrow[r,"v"]&X
  \end{tikzcd}\]
  of $\mathscr{S}$-schemes where $v$ is an open immersion. By (\ref{0.5.3}), it suffices to show that for any Kummer log smooth morphism $p:X'\rightarrow X$ of $\mathscr{S}$-schemes and any object $K$ of $\mathscr{T}(W)$, the homomorphism
  \[{\rm Hom}_{\mathscr{T}(X)}(M_X(X'),v_\sharp g_*'K)\rightarrow {\rm Hom}_{\mathscr{T}(X)}(M_X(X'),g_*w_\sharp K)\]
  is an isomorphism. It is equivalent to the assertion that
  \[{\rm Hom}_{\mathscr{T}(X')}(1_{X'},p^*v_\sharp g_*'K)\rightarrow {\rm Hom}_{\mathscr{T}(X')}(1_{X'},p^*g_*w_\sharp K)\]
  is an isomorphism since $M_X(X')=p_\sharp 1_{X'}$. By (\ref{5.4.1}), $p$ satisfies the universal support property, so $fp$ and $gp$ satisfy the semi-universal support property by (\ref{5.1.4}(2)). Since $p^*$ commutes with $v_\sharp$, $g_*'$, $g_*$, and $w_\sharp$, replacing $Y\rightarrow X\rightarrow S$ by $Y\times_X X'\rightarrow X'\rightarrow S$, we reduce to showing that
  \[{\rm Hom}_{\mathscr{T}(X)}(1_{X},v_\sharp g_*'K)\rightarrow {\rm Hom}_{\mathscr{T}(X)}(1_{X},g_*w_\sharp K)\]
  is an isomorphism. It is equivalent to showing that
  \[{\rm Hom}_{\mathscr{T}(S)}(1_{S},f_*v_\sharp g_*'K)\rightarrow {\rm Hom}_{\mathscr{T}(S)}(1_{S},f_*g_*w_\sharp K)\]
  is an isomorphism. Hence it suffices to show that the natural transformation
  \[v_\sharp g_*'\stackrel{Ex}\longrightarrow g_*w_\sharp\]
  is an isomorphism.

  The induced morphism $Y\rightarrow \underline{X}\times_{\underline{S}}S$ has the factorization
  \[Y\rightarrow \underline{Y}\times_{\underline{S}}S\rightarrow \underline{X}\times_{\underline{S}}S.\]
  The first arrow satisfies the semi-universal support property by assumption, and the second arrow satisfies the semi-universal support property by (\ref{5.1.4}(1)) since it is strict. Thus the composition also satisfies the semi-universal support property by (\ref{5.1.4}(2)). Since the induced morphism $X\rightarrow \underline{X}\times_{\underline{S}}S$ also satisfies the semi-universal support property by assumption, replacing $(Y,X,S)$ by $(Y,X,\underline{X}\times_{\underline{S}}S)$, we may assume that $\underline{f}$ is an isomorphism.

  Then there is a unique commutative diagram
  \[\begin{tikzcd}
    W\arrow[r,"w"]\arrow[d,"g'"]&Y\arrow[d,"g"]\\
    V\arrow[r,"v"]\arrow[d,"f'"]&X\arrow[d,"f"]\\
    U\arrow[r,"u"]&S
  \end{tikzcd}\]
  of $\mathscr{S}$-schemes such that the lower square is also Cartesian. Since $v$ is an open immersion, $u$ is automatically an open immersion. In the commutative diagram
  \[\begin{tikzcd}
    u_\sharp f_*'g_*'\arrow[d,"\sim"]\arrow[r,"Ex"]&f_*v_\sharp g_*'\arrow[r,"Ex"]&f_*g_*w_\sharp\arrow[d,"\sim"]\\
    u_\sharp (f'g')_*\arrow[rr,"Ex"]&&(fg)_*w_\sharp
  \end{tikzcd}\]
  of functors, the upper left arrow and lower arrows are isomorphisms by assumption. Thus the upper right arrow is an isomorphism. This completes the proof.
\end{proof}
\begin{prop}\label{5.4.3}
  Let $f:X\rightarrow S$ be a proper morphism of $\mathscr{S}$-schemes satisfying the semi-universal support. Let $g:S'\rightarrow S$ be a Kummer log smooth morphism of $\mathscr{S}$-schemes, and consider the Cartesian diagram
  \[\begin{tikzcd}
    X'\arrow[d,"f'"]\arrow[r,"g'"]&X\arrow[d,"f"]\\
    S'\arrow[r,"g"]&S
  \end{tikzcd}\]
  of $\mathscr{S}$-schemes. Then the exchange transformation
  \[g_\sharp f_*'\stackrel{Ex}\longrightarrow f_*g_\sharp'\]
  is an isomorphism.
\end{prop}
\begin{proof}
  Note first that by (\ref{5.4.1}), $g$ and $g'$ satisfies the universal support property, so $fg'$ satisfies the semi-universal support property by (\ref{5.1.4}(2)). Then $f'$ satisfies the support property by (\ref{5.4.2}). By (\ref{5.3.4}), there is a Kummer log smooth morphism $h:T\rightarrow S$ of $\mathscr{S}$-schemes such that $h^*$ is conservative and that the pullback of $g:S'\rightarrow S$ via $h$ is strict smooth. Note also that $h$ satisfies the universal support property by (\ref{5.4.1}). Thus replacing $(X',X,S',S)$ by $(X'\times_S T,X\times_S T,S'\times_S T,T)$, we may assume that $g$ is strict smooth.

  The question is Zariski local on $S'$ since $f'$ satisfies the support property, so we may assume that $g$ is a strict smooth morphism of relative dimension $d$. Choose a compactification
  \[S'\stackrel{j}\longrightarrow S''\stackrel{p}\longrightarrow S\]
  where $j$ is an open immersion and $p$ is a strict proper morphism of $\mathscr{S}$-schemes. Consider the commutative diagram
  \[\begin{tikzcd}
    X'\arrow[r,"j'"]\arrow[d,"f'"]&X''\arrow[d,"f''"]\arrow[r,"p'"]&X\arrow[d,"f"]\\
    S'\arrow[r,"j"]&S''\arrow[r,"p"]&S
  \end{tikzcd}\]
  of $\mathscr{S}$-schemes where each square is Cartesian. Since $g$ is strict smooth, as in \cite[2.4.50]{CD12}, we have the purity isomorphisms
  \[g_\sharp \stackrel{\sim}\longrightarrow p_*j_\sharp (d)[2d],\]
  \[g_\sharp'\stackrel{\sim}\longrightarrow p_*'j_\sharp'(d)[2d].\]
  with the commutative diagram
  \[\begin{tikzcd}
    g_\sharp p_*'\arrow[d,"Ex"]\arrow[rr,"\sim"]&&p_*j_\sharp f_*'(d)[2d]\arrow[d,"Ex"]\\
    f_*g_\sharp'\arrow[r,"\sim"]&f_*p_*'j_\sharp'(d)[2d]\arrow[r,"\sim"]&p_*f_*''j_\sharp'(d)[2d]
  \end{tikzcd}\]
  of functors. The morphisms $p$ and $p'$ satisfies the semi-support property by (\ref{5.1.4}(1)) since they are strict, so the morphism $fp'$ satisfies the semi-support property by (\ref{5.1.4}(2)). Then $f''$ satisfies the semi-support property by (\ref{5.4.2}), so the right vertical arrow is an isomorphism. Thus the left vertical arrow is an isomorphism.
\end{proof}
\begin{prop}\label{5.4.4}
  Let $f:X\rightarrow S$ be a proper morphism of $\mathscr{S}$-schemes satisfying the semi-universal support property. Then $f$ satisfies the projection formula.
\end{prop}
\begin{proof}
  We want to show that for any objects $K$ of $\mathscr{T}_X$ and $L$ of $\mathscr{T}_S$, the morphism
  \[f_*K\otimes_S L\stackrel{Ex}\longrightarrow f_*(K\otimes_X f^*L)\]
  is an isomorphism. By (\ref{0.5.3}), it suffices to show that for any Kummer log smooth morphism $g:S'\rightarrow S$ of $\mathscr{S}$-schemes, the morphism
  \[f_*K\otimes_S g_\sharp 1_{S'}\stackrel{Ex}\longrightarrow f_*(K\otimes f^*g_\sharp 1_{S'})\]
  is an isomorphism. Consider the Cartesian diagram
  \[\begin{tikzcd}
    X'\arrow[r,"g'"]\arrow[d,"f'"]&X\arrow[d,"f"]\\
    S'\arrow[r,"g"]&S
 \end{tikzcd}\]
  of $\mathscr{S}$-schemes. In the commutative diagram
  \[\begin{tikzcd}
    f_*K\otimes_S g_\sharp 1_{S'}\arrow[rr,"Ex"]\arrow[dd,"Ex"]&&f_*(K\otimes f^*g_\sharp 1_{S'})\arrow[d,"Ex"]\\
    &&f_*(K\otimes g'_\sharp f'^*1_{S'})\arrow[d,"Ex"]\\
    g_\sharp (g^*f_*K\otimes 1_{S'})\arrow[d,"\sim"]&&f_*g'_\sharp (g'^*K\otimes f'^*1_{S'})\arrow[d,"\sim"]\\
    g_\sharp g^* f_*K\arrow[r,"Ex"]&g_\sharp f'_*g'^*K\arrow[r,"Ex"]&f_*g'_\sharp g'^*K
  \end{tikzcd}\]
  of $\mathscr{S}$-schemes, the upper left vertical and the middle right vertical arrows are isomorphisms by ($eSm$-PF), and the lower left horizontal and the upper right vertical arrows are isomorphisms by ($eSm$-BC). Moreover, the lower right horizontal arrow is an isomorphism by (\ref{5.4.3}), so the upper horizontal arrow is an isomorphism.
\end{proof}
\begin{prop}\label{5.4.5}
  Let $f:X\rightarrow S$ be a proper morphism of $\mathscr{S}$-schemes satisfying the semi-universal support property. Then the property $({\rm BC}_{f,g})$ holds for any strict morphism $g:S'\rightarrow S$ of $\mathscr{S}$-schemes.
\end{prop}
\begin{proof}
  By (Zar-Sep), we may assume that $g$ is quasi-projective. Then $g$ has a factorization
  \[S'\stackrel{i}\longrightarrow T\stackrel{p}\longrightarrow S\]
  where $i$ is a strict closed immersion and $p$ is strict smooth. By ($eSm$-BC), we only need to deal with the case when $g$ is a strict closed immersion.

  Let $h:S''\rightarrow S$ denote the complement of $g$. Then we have the commutative diagram
  \[\begin{tikzcd}
    X'\arrow[r,"g'"]\arrow[d,"f'"]&X\arrow[d,"f"]\arrow[r,leftarrow,"h'"]&X''\arrow[d,"f''"]\\
    S'\arrow[r,"g"]&S\arrow[r,leftarrow,"h"]&S''
  \end{tikzcd}\]
  of $\mathscr{S}$-schemes where each square is Cartesian. By (Loc), we have the commutative diagram
  \[\begin{tikzcd}
    g^*f_*h_\sharp'h'^*\arrow[d,"Ex"]\arrow[r,"ad'"]&g^*f_*\arrow[d,"Ex"]\arrow[r,"ad"]&g^*f_*g_*'g'^*\arrow[d,"Ex"]\arrow[r,"\partial_{g'}"]&g^*f_*h_\sharp'h'^*[1]\arrow[d,"Ex"]\\
    f_*'g'^*h_\sharp'h'^*\arrow[r,"ad'"]&f_*'g'^*\arrow[r,"ad"]&f_*'g'^*g_*'g'^*\arrow[r,"\partial_{g'}"]&f_*'g'^*h_\sharp'h'^*[1]
  \end{tikzcd}\]
  of functors where the two rows are distinguished triangles. To show that the second vertical arrow is an isomorphism, it suffices to show that the first and third vertical arrows are isomorphisms.

  We have an isomorphism
  \[f_*h_\sharp'\cong h_\sharp f_*'\] since $f$ satisfies the support property, so we have
  \[g^*f_*h_\sharp'h'^*\cong g^*h_\sharp f_*'h'^*=0,\quad f_*g'^*h_\sharp' h'^* =0\]
  since $g^*h_\sharp=g'^*h_\sharp'=0$ by ($eSm$-BC). Thus the first vertical arrow is an isomorphism. The assertion that the third arrow is an isomorphism follows from (Loc), which completes the proof.
\end{proof}
\begin{prop}\label{5.4.6}
  Let $g:Y\rightarrow X$ and $f:X\rightarrow S$ be morphisms of $\mathscr{S}$-schemes. Assume that $g$ is proper and that the unit
  \[{\rm id}\stackrel{ad}\longrightarrow g_*g^*\]
  is an isomorphism. If $g$ and $fg$ satisfy the semi-universal support propoerty, then $f$ satisfies the semi-universal support property.
\end{prop}
\begin{proof}
  By (\ref{5.1.4}(3)), it suffices to show that for any Cartesian diagram
  \[\begin{tikzcd}
    Y'\arrow[r,"g'"]\arrow[d,"h'"]&X'\arrow[d,"h"]\\
    Y\arrow[r,"g"]&X
  \end{tikzcd}\]
  of $\mathscr{S}$-schemes such that $h$ is a strict, the unit
  \[{\rm id}\stackrel{ad}\longrightarrow g_*'g'^*\]
  is an isomorphism.

  By (\ref{5.4.4}), for any object $K$ of $\mathscr{T}(X')$, the composition
  \[g_*'g'^*1_{X'}\otimes_X K\stackrel{Ex}\longrightarrow g_*'(g'^*1_S\otimes_X g'^*K)\stackrel{\sim}\longrightarrow g'_*g'^*K\]
  is an isomorphism, so we only need to show that the morphism
  \[1_{X'}\stackrel{ad}\longrightarrow g_*'g'^*1_{X'}\]
  in $\mathscr{T}(X')$ is an isomorphism. It has the factorization
  \[1_{X'}\stackrel{\sim}\longrightarrow h^*1_X\stackrel{ad}\longrightarrow h^*g_*g^*1_X\stackrel{Ex}\longrightarrow g_*'h'^*g^*1_X\stackrel{\sim}\longrightarrow g_*'g'^*1_{X'}\]
  in $\mathscr{T}(X')$. The second arrow is an isomorphism by assumption, and the third arrow is an isomorphism by (\ref{5.4.5}). Thus the morphism
  \[1_{X'}\stackrel{ad}\longrightarrow g_*'g'^*1_{X'}\]
  in $\mathscr{T}(X')$ is an isomorphism.
\end{proof}
\subsection{Poincar\'e duality for a compactification of \texorpdfstring{$\mathbb{A}_{\mathbb{N}^2}\rightarrow \mathbb{A}_\mathbb{N}$}{AN2AN}}
\begin{none}\label{5.5.1}
  We fix an $\mathscr{S}$-scheme $S$ over $\mathbb{A}_{\mathbb{N}}$, and we put
  \[U=\mathbb{A}_{\mathbb{N}\oplus \mathbb{N}}\times_{\mathbb{A}_\theta,\mathbb{A}_{\mathbb{N}}}S\]
  where $\theta:\mathbb{N}\rightarrow \mathbb{N}\oplus\mathbb{N}$ denotes the diagonal morphism. We want to compactify the projection
  \[h:U\rightarrow S.\]
  Then we will prove the Poinca\'e duality for the compactification.
\end{none}
\begin{none}\label{5.5.2}
  Under the notations and hypotheses of (\ref{5.5.1}), consider the lattice $L=\mathbb{Z}x_1\oplus \mathbb{Z}x_2$, and consider the dual coordinates
  \[e_1=x_1^\vee,\quad e_2=x_2^\vee.\]
  We denote by $\underline{T}$ the toric variety associated to the fan generated by
  \begin{enumerate}[(a)]
    \item $e_1,e_2\geq 0$,
    \item $e_1+e_2\geq 0$, $e_1\leq 0$,
    \item $e_1+e_2\geq 0$, $e_2\leq 0$.
  \end{enumerate}
  We give the log structure on (a), (b), and (c) by
  \[\mathbb{N}x_1\oplus\mathbb{N}x_2\rightarrow \mathbb{Z}[x_1,x_2],\quad \mathbb{N}(x_1x_2)\rightarrow \mathbb{Z}[x_1x_2,x_1^{-1}],\quad \mathbb{N}(x_1x_2)\rightarrow
  \mathbb{Z}[x_1x_2,x_2^{-1}]\]
  respectively. Then we denote by $T$ the resulting $\mathscr{S}$-scheme. Because the support of this fan is $\{(e_1,e_2):e_1+e_2\geq 0\}$, the morphism $T\rightarrow
  \mathbb{A}_{\mathbb{N}}$ induced by the diagonal homomorphism $\mathbb{N}\rightarrow \mathbb{N}x_1\oplus \mathbb{N}x_2$ is proper, so we have the compactification
  \[\begin{tikzcd}
    \mathbb{A}_{\mathbb{N}^2}\arrow[r]\arrow[rd]&T\arrow[d]\\
    &\mathbb{A}_\mathbb{N}
  \end{tikzcd}\]
  of the morphism $\mathbb{A}_{\mathbb{N}^2}\rightarrow \mathbb{A}_\mathbb{N}$. Thus if we put $X=S\times_{\mathbb{A}_{\mathbb{N}}}T$, then we have the compactification
  \[\begin{tikzcd}
    U\arrow[r,"j"]\arrow[rd,"h"']&X\arrow[d,"f"]\\
    &S
  \end{tikzcd}\]
  of $h$. Here, the meaning of compactification is that $j$ is an open immersion and $f$ is proper.
\end{none}
\begin{none}\label{5.5.3}
  Under the notations and hypotheses of (\ref{5.5.2}), consider the lattice
  \[(\mathbb{Z}x_1\oplus \mathbb{Z}x_2)\oplus_\mathbb{Z} (\mathbb{Z}y_1\oplus \mathbb{Z}y_2)\]
  with $x_1+x_2=y_1+y_2$, and consider the dual coordinates
  \[f_1=y_1^\vee,\quad f_2=y_2^\vee.\]
  We denote by $\underline{T'}$ the toric variety associated to the fan generated by
  \begin{enumerate}[(a)]
    \item $e_1,e_2,f_1,f_2\geq 0$, $e_1+e_2=f_1+f_2$,
    \item $e_1+e_2=f_1+f_2\geq 0$, $e_1,f_1\leq 0$,
    \item $e_1+e_2=f_1+f_2\geq 0$, $e_2,f_2\leq 0$.
  \end{enumerate}
  Then we have an open immersion $\underline{T'}\rightarrow \underline{T}\times_{\mathbb{A}^1}\underline{T}$. Thus if we denote by $T'$ the $\mathscr{S}$-scheme whose underlying scheme
  is $\underline{T'}$ and with the log structure induced by the open immersion, then we have the open immersion
  \[T'\rightarrow T\times_S T.\]
  We put $D=(X\times_S X)\times_{T\times_{\mathbb{A}_\mathbb{N}}T} T'$.
\end{none}
\begin{none}\label{5.5.4}
  Under the notations and hypotheses of (\ref{5.5.3}), we denote by $\underline{T''}$ the toric variety associated to the fan generated by
  \begin{enumerate}
    \item[(a)] $e_1,e_2,f_1,f_2\geq 0$, $e_1+e_2=f_1+f_2$, $e_1\geq f_1$,
    \item[(a')] $e_1,e_2,f_1,f_2\geq 0$, $e_1+e_2=f_1+f_2$, $e_1\leq f_1$,
    \item[(b)] $e_1+e_2=f_1+f_2\geq 0$, $e_1,f_1\leq 0$,
    \item[(c)] $e_1+e_2=f_1+f_2\geq 0$, $e_2,f_2\leq 0$.
  \end{enumerate}
  We give the log structure on (a), (a'), (b), and (c) by
  \[\begin{split}
    &\mathbb{N}y_1\oplus \mathbb{N}x_2\oplus \mathbb{N}(x_1y_1^{-1})\rightarrow \mathbb{Z}[y_1,x_2,x_1y_1^{-1}],\\
    &\mathbb{N}x_1\oplus \mathbb{N}y_2\oplus \mathbb{N}(y_1x_1^{-1})\rightarrow \mathbb{Z}[x_1,y_2,y_1x_1^{-1}],\\
    &\mathbb{N}(x_1x_2)\rightarrow \mathbb{Z}[x_1x_2,x_1^{-1},y_1^{-1}],\\
    &\mathbb{N}(x_1x_2)\rightarrow \mathbb{Z}[x_1x_2,x_2^{-1},y_2^{-1}]
  \end{split}\]
  respectively. Then we denote by $T''$ the resulting $\mathscr{S}$-scheme. The supports of this fan and the fan in (\ref{5.5.3}) are equal, so the morphism $T''\rightarrow T'$ induced
  by the fans is proper. Thus if we put $E=Y\times_{T'} T''$, we have the proper morphism
  \[v:E\rightarrow D\]
  of $\mathscr{S}$-schemes.
\end{none}
\begin{none}\label{5.5.5}
  From (\ref{5.5.1}) to (\ref{5.5.4}), we obtain the commutative diagram
  \[\begin{tikzcd}
    &E\arrow[d,"v"]\arrow[rdd,"r_2"]\\
    &D\arrow[d,"u"]\arrow[rd,"q_2"']\\
    X\arrow[ruu,"c"]\arrow[ru,"b"']\arrow[r,"c"']&X\times_S X\arrow[r,"p_2"']&X
  \end{tikzcd}\]
  of $\mathscr{S}$-schemes where
  \begin{enumerate}[(i)]
    \item $p_2$ denotes the second projection,
    \item $a$ denotes the diagonal morphism, and $b$ and $c$ denotes the morphisms induced by $a$.
  \end{enumerate}
  Note that $u$ is an open immersion and that $v$ and $p_2$ are proper. The morphism $E\rightarrow D$ satisfies the condition of (\ref{4.2.2}(2)), and the morphism $D\rightarrow
  X\times_S X$ satisfies the condition of (\ref{4.2.2}(1)). Consider the natural transformation
  \[\mathfrak{q}_f^o:\Omega_{f,E}^of^!\longrightarrow f^*\]
  in (\ref{4.4.5}). Note that we have $\Omega_{f,E}^o=(-1)[-2]$. We also consider the pullback of the above commutative diagram via $i:Z\rightarrow X$ where $i$ denotes the complement of $j:U\rightarrow X$:
  \[\begin{tikzcd}
    &E'\arrow[d,"v'"]\arrow[rdd,"r_2'"]\\
    &D'\arrow[d,"u'"]\arrow[rd,"q_2'"']\\
    Z\arrow[ruu,"c'"]\arrow[ru,"b'"']\arrow[r,"c'"']&X\arrow[r,"p_2'"']&Z
  \end{tikzcd}\]
  Note that $u'$ and $v'$ are isomorphisms.
\end{none}
\begin{prop}\label{5.5.6}
  Under the notations and hypotheses of (\ref{5.5.5}), the natural transformation
  \[f_*f^!(-1)[-2]\stackrel{\mathfrak{q}_f^o}\longrightarrow f_*f^*\]
  is an isomorphism.
\end{prop}
\noindent {\it Proof}.
  We put $\tau=(1)[2]$, and let
  \[\mathfrak{p}_f^o:f_\sharp \longrightarrow f_*\tau\]
  denote the left adjoint of $\mathfrak{q}_f^o$. We have the commutative diagram
  \[\begin{tikzcd}
    Z_1\arrow[d,"i_1''"]\arrow[dd,bend right,"i_1'"']\arrow[dr,"i_1"]&&Z_2\arrow[d,"i_2''"']\arrow[dd,bend left,"i_2'"]\arrow[dl,"i_2"']\\
    U_1'\arrow[d,"j_1''"]\arrow[r,"j_1'"]&X&U_2'\arrow[d,"j_2''"']\arrow[l,"j_2'"']\\
    U_1\arrow[ru,"j_1"']&&U_2\arrow[lu,"j_2"]
  \end{tikzcd}\]
  of $\mathscr{S}$-schemes where
  \begin{enumerate}[(i)]
    \item $j_1'$ and $j_2'$ denote the open immersions induced by the convex sets (b) and (c) of (\ref{5.5.2}) respectively,
    \item $Z_1=U_1'\times_X (X-U)$ and $Z_2=U_2'\times_X (X-U)$,
    \item $U_1$ and $U_2$ denotes the complements of $Z_1$ and $Z_2$ respectively,
    \item $i_1$, $i_1'$, $i_1''$, $i_2$, $i_2'$, and $i_2''$ are the closed immersions,
    \item $j_1$, $j_1''$, $j_2$, $j_2''$ are the open immersions.
  \end{enumerate}
  The {\it key property} of our compactification is that $U_1$ and $U_1'$ (resp.\ $U_2$ and $U_2'$) are {\it log-homotopic equivalent} over $S$. The meaning is that the morphisms
  \[M_S(U_1')\rightarrow M_S(U_1),\quad M_S(U_2')\rightarrow M_S(U_2)\]
  in $\mathscr{T}(S)$ are isomorphisms. More generally, we will show that the natural transformations
  \begin{equation}\label{5.5.5.1}
    f_\sharp j_{1\sharp}j_{1\sharp}''j_1''^*j_1^*f^*\stackrel{ad}\longrightarrow f_\sharp j_{1\sharp}j_1^*f^*,\quad f_\sharp
    j_{2\sharp}j_{2\sharp}''j_2''^*j_2^*f^*\stackrel{ad}\longrightarrow f_\sharp j_{2\sharp}j_2^*f^*
  \end{equation}
  are isomorphisms. To show that the first one is an isomorphism, consider the Mayer-Vietoris triangle
  \[\begin{tikzcd}
     f_\sharp j_{1\sharp}'''j_1'''^*f^*\rightarrow f_\sharp j_{1\sharp}'j_1'^*f^*\oplus f_\sharp j_\sharp j^*f^*\rightarrow f_\sharp j_{1\sharp}j_1^*f^*\rightarrow f_\sharp j_{1\sharp}'''j_1'''^*f^*[1]
  \end{tikzcd}\]
  in $\mathscr{T}(S)$ where $j_1''':U_1'\times_X U\rightarrow X$ denotes the open immersion. It is a distinguished triangle by (\ref{2.2.3}), so it suffices to show that the morphism
  \[f_\sharp j_{1\sharp}'''j_1'''^*f^*\rightarrow f_\sharp j_\sharp j^*f^*\]
  is an isomorphism since $j_1'=j_1j_1''$. It is true by (Htp--3) because
  \[U_1'\times_X U\cong U\times_{\mathbb{A}_{\mathbb{N}^2}}\mathbb{A}_{(\mathbb{N}^2)_F}\]
  where $F$ is the face of $\mathbb{N}^2$ generated by $(1,0)$. Thus the first natural transformation of (\ref{5.5.5.1}) is an isomorphism. We can show that the other one is also an isomorphism similarly.

  Then guided by a method of \cite[1.7.9]{Ayo07}, consider the commutative diagram
  \begin{equation}\label{5.5.6.1}\begin{tikzcd}
    f_\sharp i_{1*}i_1^!f^*\arrow[r,"ad'"]\arrow[d,"\mathfrak{p}_f^o"]&f_\sharp f^*\arrow[r,"ad"]\arrow[d,"\mathfrak{p}_f^o"]&f_\sharp i_{2*}i_2^*f^*\arrow[d,"\mathfrak{p}_f^o"]
    \arrow[r,"\partial"]&f_\sharp i_{1*}i_1^!f^*[1]\arrow[d,"\mathfrak{p}_f^o"] \\
    f_*i_{1*}i_1^!f^*\tau\arrow[r,"ad'"]&f_*f^*\tau\arrow[r,"ad"]&f_*i_{2*}i_2^*f^*\arrow[r,"\partial"]&f_*i_{1*}i_1^!f^*\tau[1]
  \end{tikzcd}\end{equation}
  of functors. Assume that we have proven that
  \begin{enumerate}[(1)]
    \item the rows are distinguished triangles,
    \item the first and third vertical arrows are isomorphisms.
  \end{enumerate}
  Then the second arrow is also an isomorphism, so we are done. Hence in the remaining, we will prove (1) and (2).
  \begin{enumerate}[(1)]
    \item To show that the second row is a distinguished triangle, by (Loc), it suffices to show that the composition
    \[f_*i_{2*}i_2^*f^*\stackrel{\sim}\longrightarrow f_*j_{2*}i_{2*}'i_2'^*j_2^*f^*\stackrel{ad'}\longrightarrow f_*j_{2*}j_2^*f^*\]
    is an isomorphism. We have shown that the natural transformation
    \[f_*j_{2*}j_2^*f^*\stackrel{ad'}\longrightarrow f_*j_{2*}j_{2*}''j_2''^*j_2^*f^*\]
    is an isomorphism, and we have $fi_2={\rm id}$. Hence, it suffices to show that the unit
    \[g_{2\sharp}g_2^*\stackrel{ad'}\longrightarrow {\rm id}\]
    is an isomorphism where $g_2=fj_2'$. It is true by (Htp--1) since the morphism $U_2'\rightarrow S$ is the projection $\mathbb{A}_S^1\rightarrow S$.

    For the first row, first note that by (sSupp), we have an isomorphism
    \[j_{1\sharp}i_{1*}'i_1'^!j_1^*\stackrel{\sim}\longrightarrow i_{1*}i_1^!.\]
    Hence to show that the first row is distinguished, by (Loc), it suffices to show that the natural transformation
    \[f_\sharp j_{1\sharp}i_{1*}'i_1'^!j_1^*f^*\stackrel{ad'}\longrightarrow f_\sharp j_{1\sharp}j_1^*f^*\]
    is an isomorphism. By (sSupp), we have an isomorphism
    \[f_\sharp j_{1\sharp}'i_{1*}''i_1''^!j_1'^*f^*\stackrel{\sim}\longrightarrow f_\sharp j_{1\sharp}i_{1*}'i_1'^!j_1^*f^*,\]
    and we have shown that the natural transformation
    \[f_\sharp j_{1\sharp}j_{1\sharp}''j_1''^*j_1^*f^*\stackrel{ad'}\longrightarrow f_\sharp j_{1\sharp}j_1^*f^*\]
    is an isomorphism. Hence it suffices to show that the natural transformation
    \[g_{1\sharp}i_{1*}''i_1''^!g_1^*\stackrel{ad'}\longrightarrow g_{1\sharp}g_1^*\]
    is an isomorphism where $g_1=fj_1'$. By (Htp--1), the counit
    \[g_{1\sharp}g_1^*\stackrel{ad'}\longrightarrow {\rm id}\]
    is an isomorphism since the morphism $U_1'\rightarrow S$ is the projection $\mathbb{A}_S^1\rightarrow S$, and by (\ref{2.5.10}), the composition
    \[g_{1\sharp}i_{1*}''i_1''^!g_1^*\stackrel{ad'}\longrightarrow g_{1\sharp}g_1^*\stackrel{ad'}\longrightarrow {\rm id}\]
    is an isomorphism. Thus the first row is distinguished.
    \item Consider the diagram
    \[\begin{tikzpicture}[baseline= (a).base]
    \node[scale=.75] (a) at (0,0)
    {\begin{tikzcd}
    \Omega_{f,{\rm id},E'}^oi^!f^!\arrow[r,"\sim"]&\Omega_{f,{\rm id},E'}^ni^!f^!\arrow[r,"\sim"]\arrow[d,"Ex"]&\Omega_{f,{\rm id},E'}^di^!f^!\arrow[r,"\sim"]\arrow[d,"Ex"]&\Omega_{f,{\rm id},E'}i^!f^!\arrow[r,"\sim"]\arrow[d,"Ex"]&\Omega_{f,{\rm
    id},D'}i^!f^!\arrow[r,"\sim"]\arrow[d,"Ex"]&\Omega_{f,{\rm id}}i^!f^!\arrow[r,"\sim"]\arrow[d,"Ex"]&\Omega_{f,{\rm id}}\arrow[d,equal]\\
    i^!\Omega_{f,E}^of^!\arrow[r,"\sim"]&i^!\Omega_{f,E}^nf^!\arrow[r]&i^!\Omega_{f,E}^df^!\arrow[r]&i^!\Omega_{f,E}f^!\arrow[r]&i^!\Omega_{f,D}f^!\arrow[r] &i^!\Omega_ff^!\arrow[r,"\mathfrak{q}_f"]&i^!f^*
    \end{tikzcd}};
    \end{tikzpicture}\]
    of functors. It commutes by (\ref{4.2.6}), and the natural transformation $\Omega_{f,{\rm id},E'}^ni^!f^!\stackrel{Ex}\longrightarrow i^!\Omega_{f,E}^nf^!$ is an isomorphism by (\ref{4.2.7}). Thus the composition of arrows in the second row
    \[i^!\Omega_{f,E}^of^!\stackrel{\mathfrak{q}_f^o}\longrightarrow i^!f^*\]
    is also an isomorphism. Then the first vertical arrow of (\ref{5.5.6.1}) is also an isomorphism. The third vertical arrow of (\ref{5.5.6.1}) is also an isomorphism similarly.\qed
  \end{enumerate}
\begin{thm}\label{5.5.7}
  Under the notations and hypotheses of (\ref{5.5.5}), the natural transformation
  \[\mathfrak{q}_f^o:f^!(-1)[-2]\longrightarrow f^*\]
  is an isomorphism.
\end{thm}
\begin{proof}
  We put $\tau=(1)[2]$, and let
  \[\mathfrak{p}_f^o:f_\sharp \longrightarrow f_*\tau\]
  denote the left adjoint of $\mathfrak{q}_f^o$. Guided by a method of \cite[2.4.42]{CD12}, we will construct a right inverse $\phi_1$ and a left inverse $\phi_2$ to the morphism
  $\mathfrak{p}_f^o$. Note first that the natural transformation
  \[f_\sharp f^*\stackrel{\mathfrak{p}_f^o}\longrightarrow f_*f^*\tau\]
  is an isomorphism by (\ref{5.5.6}). The left inverse $\phi_2$ is constructed by
  \[\phi_2:f_*\tau\stackrel{ad}\longrightarrow f_*\tau f^*f_\sharp \stackrel{(\mathfrak{p}_{f,E}^o)^{-1}}\longrightarrow f_\sharp f^*f_\sharp \stackrel{ad'}\longrightarrow f_\sharp.\]
  To show $\phi_2\circ\mathfrak{p}_f^o={\rm id}$, it suffices to check that the outside diagram of the diagram
  \[\begin{tikzcd}
    f_\sharp \arrow[d,"\mathfrak{p}_{f,E}^o"']\arrow[r,"ad"]&f_\sharp f^*f_\sharp\arrow[d,"\mathfrak{p}_{f,E}^o"]\arrow[rd,equal]\arrow[rrd,"ad'"]\\
    f_*\tau\arrow[r,"ad"']&f_*\tau f^*f_\sharp \arrow[r,"(\mathfrak{p}_{f,E}^o)^{-1}"']&f_\sharp f^*f_\sharp \arrow[r,"ad'"']&f_\sharp
  \end{tikzcd}\]
  of functors commutes since the composition
  \[f_\sharp \stackrel{ad}\longrightarrow f_\sharp f^*f_\sharp\stackrel{ad'}\longrightarrow f_\sharp\]
  is the identity. It is true since each small diagram commutes.

  The right inverse $\phi_1$ is constructed by
  \[\phi_1:f_*\tau\stackrel{ad}\longrightarrow f_*f^*f_*\tau\stackrel{Ex^{-1}}\longrightarrow f_*f^*\tau f_*\stackrel{\sim}\longrightarrow f_* \tau f^*
  f_*\stackrel{(\mathfrak{p}_{f,E}^o)^{-1}}\longrightarrow f_\sharp f^*f_*\stackrel{ad'}\longrightarrow f_\sharp.\]
  To show $\mathfrak{p}_f^o\circ \phi_1={\rm id}$, it suffices to check that the composition of the outer cycle starting from upper $f_*\tau$ in the below diagram of functors is the
  identity:
  \[\begin{tikzcd}
    f_\sharp \arrow[r,"\mathfrak{p}_{f,E}^o"]&f_*\tau \arrow[rr,"ad"]\arrow[rd,equal]&&f_*f^*f_*\tau\arrow[dd,"Ex^{-1}"]\arrow[dl,"ad'"]\\
    &&f_*\tau\\
    f_\sharp f^*f_*\arrow[uu,"ad'"']&f_*\tau f^*f_*\arrow[ru,"ad'"']\arrow[l,"(\mathfrak{p}_{f,E}^o)^{-1}"']\arrow[uu,"ad'"']&&f_*f^*\tau f_*\arrow[ll,"\sim"']
  \end{tikzcd}\]
  It is true since each small diagram commutes.

  Then from the existences of left and right inverses, we conclude that
  \[\mathfrak{p}_f^o:f_\sharp \longrightarrow f_*\tau\]
  is an isomorphism.
\end{proof}
\begin{cor}\label{5.5.8}
  Under the notations and hypotheses of (\ref{5.5.5}), the universal support property holds for $f:X\rightarrow S$.
\end{cor}
\begin{proof}
  We put $E''=E\times_S V$. Consider the Cartesian diagram
  \[\begin{tikzcd}
    V'\arrow[d,"\mu'"]\arrow[r,"f''"]&V\arrow[d,"\mu"]\\
    X\arrow[r,"f"]&S
  \end{tikzcd}\]
  of $\mathscr{S}$-schemes where $\mu$ is an open immersion. Then by the above theorem, the support property for $f$ follows from the commutativity of the diagram
  \[\begin{tikzcd}
    \mu_\sharp f''_\sharp\arrow[r,"\sim"]\arrow[dd,"\mathfrak{p}_{f'',E''}^o"]&f_\sharp\mu'_\sharp\arrow[d,"\mathfrak{p}_{f,E}^o"]\\
    &f_*\tau\mu'_\sharp\arrow[d,"Ex^{-1}"]\\
    \mu_\sharp f''_*\tau\arrow[r,"Ex"]&f_*\tau\mu'_\sharp
  \end{tikzcd}\]
  of functors.

  Then because we can choose $S$ arbitrary, $f$ satisfies the universal support property.
\end{proof}
\begin{cor}\label{5.5.9}
  Under the notations and hypotheses of (\ref{5.5.5}) the universal support property holds for $h:U\rightarrow S$.
\end{cor}
\begin{proof}
  The conclusion follows from (\ref{5.5.8}) and (\ref{5.1.4}(1),(2)).
\end{proof}
\subsection{Support property for the projection \texorpdfstring{$\mathbb{A}_{\mathbb{N}}\times {\rm pt}_\mathbb{N}\rightarrow {\rm pt}_\mathbb{N}$}{ANptNptN}}
\begin{none}\label{5.6.1}
  Let $x$ and $y$ denote the first and second coordinates of $\mathbb{N}\oplus \mathbb{N}$ respectively, and let $S$ be an $\mathscr{S}$-scheme. Consider the morphisms
  \[S\times\mathbb{A}_{(\mathbb{N}\oplus\mathbb{N},(x))}\stackrel{h}\longrightarrow S\times\mathbb{A}_{(\mathbb{N}\oplus\mathbb{N},(xy))}\stackrel{g}\longrightarrow
  S\times\mathbb{A}_{(\mathbb{N}\oplus\mathbb{N},(x))}\stackrel{f}\longrightarrow S\times{\rm pt}_\mathbb{N}\]
  of $\mathscr{S}$-schemes where
  \begin{enumerate}[(i)]
    \item $h$ denotes the obvious closed immersion,
    \item $g$ denotes the morphism induced by
    \[\mathbb{N}\oplus \mathbb{N}\mapsto \mathbb{N}\oplus \mathbb{N},\quad (a,b)\mapsto (a,a+b),\]
    \item $f$ denotes the morphism induced by
    \[\mathbb{N}\mapsto \mathbb{N}\oplus \mathbb{N},\quad a\mapsto (a,0).\]
  \end{enumerate}
  To simplify the notations, we put
  \[X=S\times\mathbb{A}_{(\mathbb{N}\oplus\mathbb{N},(x))},\quad Y=S\times\mathbb{A}_{(\mathbb{N}\oplus\mathbb{N},(xy))},\quad T=S\times{\rm pt}_\mathbb{N}.\]
  Then we have the sequence
  \[X\stackrel{h}\longrightarrow Y \stackrel{g}\longrightarrow X\stackrel{f}\longrightarrow T\]
  of morphisms of $\mathscr{S}$-schemes.
\end{none}
\begin{prop}\label{5.6.2}
  Under the notations and hypotheses of (\ref{5.6.1}), the morphism $gh$ satisfies the semi-universal support property.
\end{prop}
\begin{proof}
  Consider the commutative diagram
  \[\begin{tikzcd}
    X\arrow[rr,"gh"]\arrow[rd,"q"']&&X\arrow[ld,"p"]\\
    &T
  \end{tikzcd}\]
  of $\mathscr{S}$-schemes where $p$ and $q$ denote the morphisms induced by the homomorphisms
  \[\mathbb{N}\rightarrow\mathbb{N}\oplus\mathbb{N},\quad a\mapsto (a,a),\]
  \[\mathbb{N}\rightarrow\mathbb{N}\oplus\mathbb{N},\quad a\mapsto (a,2a),\]
  respectively. By (\ref{5.4.2}), it suffices to show that $p$ and $q$ satisfy the semi-universal support property.

  The morphism $fg$ satisfies the universal support property by (\ref{5.5.9}), and the morphism $h$ satisfies the universal support property by (\ref{5.1.4}(1)) since it is strict. Thus the morphism $p=fgh$ satisfies the universal support property by (\ref{5.1.4}(2)). Hence the remaining is to show that $q$ satisfies the semi-universal support property.

  The morphism $q$ has the factorization
  \[S\times\mathbb{A}_{(\mathbb{N}\oplus \mathbb{N},(x))}\stackrel{i}\longrightarrow S\times\mathbb{A}_{(\mathbb{N}\oplus \mathbb{N}\oplus \mathbb{Z},(x))}\stackrel{u}\longrightarrow S\times\mathbb{A}_{(\mathbb{N}\oplus \mathbb{N},(x))}\stackrel{p}\longrightarrow S\times{\rm pt}_{\mathbb{N}}\]
  where
  \begin{enumerate}[(i)]
    \item $i$ denotes the morphism induced by the homomorphism
    \[\mathbb{N}\oplus\mathbb{N}\oplus \mathbb{Z}\rightarrow \mathbb{N}\oplus \mathbb{N},\quad (a,b,c)\mapsto (a,b),\]
    \item $u$ denotes the morphism induced by the homomorphism
    \[\mathbb{N}\oplus \mathbb{N}\rightarrow \mathbb{N}\oplus\mathbb{N}\oplus \mathbb{Z},\quad (a,b)\mapsto (a,2b,b).\]
  \end{enumerate}
  We already showed that $p$ satisfies the universal support property. The morphism $i$ satisfies the universal support property by (\ref{5.1.4}(1)) since it is strict, and the morphism $u$ satisfies the support property by (\ref{5.4.1}) since it is Kummer log smooth. Thus by (\ref{5.1.4}(2)), the morphism $q=pui$ satisfies the universal support property.
\end{proof}
\begin{cor}\label{5.6.3}
  Under the notations and hypotheses of (\ref{5.6.1}), the morphism $gh$ satisfies the projection formula.
\end{cor}
\begin{proof}
  It follows from (\ref{5.6.2}) and (\ref{5.4.4}).
\end{proof}
\begin{prop}\label{5.6.4}
  Under the notations and hypotheses of (\ref{5.6.1}), the unit
  \[{\rm id}\stackrel{ad}\longrightarrow (gh)_*(gh)^*\]
  is an isomorphism.
\end{prop}
\begin{proof}
  To simplify the notation, we put $v=gh$. By (\ref{5.6.3}), for any object $K$ of $\mathscr{T}(X)$, the composition
  \[v_*v^*1_X\otimes_X K\stackrel{Ex}\longrightarrow v_*(v^*1_S\otimes_X v^*K)\stackrel{\sim}\longrightarrow v_*v^*K\]
  is an isomorphism, so we only need to show that the morphism
  \[1_X\stackrel{ad}\longrightarrow v_*v^*1_X\]
  in $\mathscr{T}(X)$ is an isomorphism.

  We denote by $j:U\rightarrow X$ the verticalization of $X$ via $f$. Then we have the Cartesian diagram
  \[\begin{tikzcd}
    U\arrow[d,"j"]\arrow[r,"{\rm id}"]&U\arrow[d,"j"]\\
    X\arrow[r,"v"]&X
  \end{tikzcd}\]
  of $\mathscr{S}$-schemes, and by (Htp--2), the morphism
  \begin{equation}\label{5.6.4.1}
  1_X\stackrel{ad}\longrightarrow j_*j^*1_X
  \end{equation}
  in $\mathscr{T}(X)$ is an isomorphism. In the commutative diagram
  \[\begin{tikzcd}
    1_X\arrow[r,"ad"]\arrow[d,"ad"]&v_*v^*1_X\arrow[d,"ad"]\\
    j_*j^*1_X\arrow[r,"ad"]&v_*j_*j^*v^*1_X
  \end{tikzcd}\]
  of functors, the vertical arrows are isomorphisms since the morphism (\ref{5.6.4.1}) is an isomorphism and $v^*1_X\cong 1_X$. The lower horizontal arrow is an isomorphism since $vj=j$. Thus the upper horizontal arrow is an isomorphism, which completes the proof.
\end{proof}
\begin{prop}\label{5.6.6}
  Under the notations and hypotheses of (\ref{5.6.1}), the morphism $f$ satisfies the semi-support property.
\end{prop}
\begin{proof}
  By (\ref{5.5.9}), the morphism $fg$ satisfies the semi-universal support property. The morphism $h$ satisfies the semi-universal support property by (\ref{5.1.4}(1)) since it is strict, so by (\ref{5.1.4}(2)), the morphism $fgh$ satisfies the semi-universal support property. Then by (\ref{5.6.2}) and (\ref{5.6.4}), the morphism $gh:X\rightarrow X$ and $f:X\rightarrow S$ satisfy the condition of (\ref{5.4.6}), so (loc.\ cit) implies that $f$ satisfies the semi-universal support property.
\end{proof}
\subsection{Proof of the support property}\label{section5.6}
\begin{none}\label{5.8.1}
  Throughout this subsection, we assume (Htp--4) and the axiom (ii) of (\ref{2.9.1}) for $\mathscr{T}$.
\end{none}
\begin{prop}\label{5.8.2}
  Let $S$ be an $\mathscr{S}$-scheme with an fs chart $\mathbb{N}$. Then the quotient morphism $f:S\rightarrow \underline{S}$ satisfies the support property.
\end{prop}
\begin{proof}
  We have the factorization
  \[S\stackrel{i}\rightarrow \underline{S}\times \mathbb{A}_M\stackrel{g}\rightarrow \underline{S}\times \mathbb{P}^1\stackrel{p}\rightarrow \underline{S}\]
  where
  \begin{enumerate}[(i)]
    \item $i$ denotes the strict closed immersion induced by the chart $S\rightarrow \mathbb{A}_\mathbb{N}$,
    \item $p$ denotes the projection,
    \item $M$ denotes the fs monoscheme that is the gluing of ${\rm spec}\,\mathbb{N}$ and ${\rm spec}\mathbb{N}^{-1}$ along ${\rm spec}\,\mathbb{Z}$,
    \item $g$ denotes the morphism removing the log structure.
  \end{enumerate}
  Then by (sSupp), $i$ and $p$ satisfies the support property. Hence by (\ref{5.1.1}), the remaining is to show the support property for $g$. This question is Zariski local on $\underline{S}\times\mathbb{P}^1$, so we reduce to showing the support property for the morphism
  \[\underline{S}\times\mathbb{A}_\mathbb{N}\rightarrow \underline{S}\times\mathbb{A}^1\]
  removing the log structure. This is the axiom (ii) of (\ref{2.9.1}).
\end{proof}
\begin{prop}\label{5.8.3}
  Let $S$ be an $\mathscr{S}$-scheme with the trivial log structure. Then the semi-universal support property is satisfied for the projection $p:S\times\mathbb{A}_\mathbb{N}\rightarrow S$.
\end{prop}
\begin{proof}
  Let $q$ denote the morphism $S\times \mathbb{A}_\mathbb{N}\rightarrow S\times \mathbb{A}^1$ removing the log structure. By definition, we need to show that $q$ satisfies the semi-universal support property. Any pullback of $q$ via a strict morphism is the quotient morphism $X\rightarrow \overline{X}$ for some $\mathscr{S}$-scheme $X$. Thus the conclusion follows from (\ref{5.8.2}).
\end{proof}
\begin{prop}\label{5.8.4}
  Let $f:X\rightarrow S$ be a morphism of $\mathscr{S}$-schemes, and assume that $S$ has the trivial log structure. Then $f$ satisfies the semi-universal support property.
\end{prop}
\begin{proof}
  By (\ref{5.1.4}(3)) and (Htp--3), the question is dividing local on $X$. Hence by \cite[11.1.9]{CLS11}, we may assume that $X$ has an fs chart $\mathbb{N}^r$. Then $p$ has a factorization
  \[X\stackrel{i}\longrightarrow \underline{X}\times\mathbb{A}_{\mathbb{N}^r}\stackrel{\mathbb{A}_{\theta_r}}\longrightarrow \underline{X}\times\mathbb{A}_{\mathbb{N}^{r-1}}\stackrel{\mathbb{A}_{\theta_{r-1}}}\longrightarrow \cdots \rightarrow \underline{X}\times\mathbb{A}_\mathbb{N}\stackrel{q}\longrightarrow \underline{X}\stackrel{\underline{f}}\longrightarrow S\]
  where
  \begin{enumerate}[(i)]
    \item $i$ denotes the morphism induced by the chart $S\rightarrow \mathbb{A}_{\mathbb{N}^r}$,
    \item $\theta_s:\mathbb{N}^{s-1}\rightarrow \mathbb{N}^{s}$ denotes the homomorphism
    \[(a_1,\ldots,a_{s-2},a_{s-1})\mapsto (a_1,\ldots,a_{s-2},a_{s-1},a_{s-1}),\]
    \item $q$ denotes the projection.
  \end{enumerate}
  The morphism $i$ and $\underline{f}$ satisfy the semi-universal support property by (\ref{5.1.4}(1)) since they are strict, and the morphisms $\mathbb{A}_{\theta_s}$ for $s=2,\ldots,r$ satisfy the universal support property by (\ref{5.5.9}). Thus the conclusion follows from the (\ref{5.8.3}) and (\ref{5.1.4}(2)).
\end{proof}
\begin{thm}\label{5.8.5}
  The support property holds for $\mathscr{T}$.
\end{thm}
\begin{proof}
  Let $f:X\rightarrow S$ be a proper morphism of $\mathscr{S}$-schemes. Consider the commutative diagram
  \[\begin{tikzcd}
    X\arrow[r,"f"]\arrow[d,"q"]&S\arrow[d,"p"]\\
    \underline{X}\arrow[r,"\underline{f}"]&\underline{S}
  \end{tikzcd}\]
  of $\mathscr{S}$-schemes where $p$ and $q$ denote the morphisms removing the log structures. Then $p$ and $q$ satisfy the semi-universal support property by (\ref{5.8.4}), and $\underline{f}$ satisfies the semi-universal support property by (\ref{5.1.4}(1)) since it is strict. Thus the composition $\underline{f}q$ satisfies the semi-support property by (\ref{5.1.4}(2)), and then $f$ satisfies the semi-universal support property by (\ref{5.4.2}).
\end{proof}
\section{Base change and homotopy properties}
\subsection{Homotopy property 5}
\begin{none}\label{6.1.1}
  Let $f:X\rightarrow S$ be a morphism of $\mathscr{S}$-schemes. In this subsection, we often consider the following conditions:
  \begin{enumerate}[(i)]
    \item the morphism $\underline{f}:\underline{X}\rightarrow \underline{S}$ of underlying schemes is an isomorphism,
    \item the induced homomorphism $\overline{\mathcal{M}}_{X,\overline{x}}^{\rm gp}\rightarrow \overline{\mathcal{M}}_{S,\overline{s}}^{\rm gp}$ is an isomorphism.
  \end{enumerate}
\end{none}
\begin{prop}\label{6.1.2}
  Consider the coCartesian diagram
  \[\begin{tikzcd}
    P\arrow[d,"\theta"]\arrow[r,"\eta"]&P'\arrow[d,"\theta'"]\\
    Q\arrow[r,"\eta'"]&Q'
  \end{tikzcd}\]
  of sharp fs monoids such that
  \begin{enumerate}[(i)]
    \item $\theta'^{\rm gp}$ is an isomorphism,
    \item if $F$ is a face of $P'$ such that $F\cap \eta(P)=\langle 0\rangle$, then $\theta'(F)$ is a face of $Q'$,
    \item if $G$ is a face of $Q'$ such that $G\cap \eta'(Q)=\langle 0\rangle$, then $G=\theta'(F)$ for some face $F$ of $P'$.
  \end{enumerate}
  Then the induced morphism
  \[f:\mathbb{A}_{(Q',(\eta'(Q^+)))}\rightarrow \mathbb{A}_{(P',(\lambda(P^+)))}\]
  satisfies the conditions (i) and (ii) of (\ref{6.1.1}).
\end{prop}
\begin{proof}
  By \cite[I.3.3.4]{Ogu17}, $\mathbb{A}_{(P',(\lambda(P^+)))}$ has the stratification
  \[\bigcup_F (\mathbb{A}_{F^*}\times {\rm pt}_{P'/F})\]
  for face $F$ of $P'$ such that $F\cap \eta(P)=\langle 0\rangle$. Similarly, $\mathbb{A}_{(Q',(\lambda'(Q^+)))}$ has the stratification
  \[\bigcup_G (\mathbb{A}_{G^*}\times {\rm pt}_{Q'/G})\]
  for face $Q$ of $Q'$ such that $G\cap \eta(Q)=\langle 0\rangle$.

  Thus by assumption, $f$ is a union of the morphisms
  \[f_F:\mathbb{A}_{{\theta'(F)}^*}\times{\rm pt}_{Q'/\theta'(F)}\rightarrow \mathbb{A}_{F^*}\times {\rm pt}_{P/F}.\]
  This satisfies the condition (i) and (ii) of (\ref{6.1.1}) because $\theta'^{\rm gp}$ is an isomorphism. Then the conclusion follows from \cite[IV.18.12.6]{EGA}.
\end{proof}
\begin{prop}\label{6.1.3}
  Let $f:X\rightarrow S$ be a morphism of $\mathscr{S}$-schemes satisfying the conditions (i) and (ii) of (\ref{6.1.1}). Then the unit
  \[{\rm id}\stackrel{ad}\longrightarrow f_*f^*\]
  is an isomorphism.
  \end{prop}
\begin{proof}
  (I) {\it Locality of the question.} The question is strict \'etale local on $S$, so we may assume that $S$ has an fs chart. Since $\underline{f}$ is an isomorphism, the question is also strict \'etale local on $X$, so we may assume that $X$ has an fs chart.

  Let $i:Z\rightarrow S$ be a closed immersion, and let $j:U\rightarrow S$ denote its complement. By (Loc) and (\ref{2.6.6}), we reduce to the question for $X\times_S Z\rightarrow Z$ and $X\times_S U\rightarrow U$. Hence by the proof of \cite[3.5(ii)]{Ols03}, we reduce to the case when $S$ has a constant log structure. Since $\underline{f}$ is an isomorphism, we can do the same method for $X$, and by \cite[3.5(ii)]{Ols03}, we reduce to the case when $X$ has a constant log structure. Hence we reduce to the case when $f$ is the morphism
  \[\underline{S}\times{\rm pt}_Q\rightarrow \underline{S}\times {\rm pt}_P\]
  induced by a homomorphism $\theta:P\rightarrow Q$ of sharp fs monoids. By assumption, $\theta^{\rm gp}$ is an isomorphism.\\[4pt]
  (II) {\it Induction.} We will use an induction on $n={\rm dim}\,P$. If $n=1$, then we are done since $P=Q$, so we may assume $n>1$.\\[4pt]
  (III) {\it Reduction to the case when $Q=(P+\langle a\rangle )^{\rm sat}$.} Choose generators $a_1,\ldots,a_m$ of $Q$. Then consider the homomorphisms
  \[P\longrightarrow (P+\langle a_1\rangle)^{\rm sat}\longrightarrow \cdots \longrightarrow (P+\langle a_1,\ldots,a_m\rangle)^{\rm sat}\]
  of sharp fs monoids. If we show the question for each morphism \[(P+\langle a_1,\ldots,a_i\rangle)^{\rm sat} \longrightarrow (P+\langle a_1,\ldots,a_{i+1}\rangle)^{\rm sat},\] then we are done, so we reduce to the case when $Q=(P+\langle a\rangle )^{\rm sat}$ for some $a\in P^{\rm gp}$.\\[4pt]
  (IV) {\it Construction of fans.} We put
  \[C=P_\mathbb{Q},\quad D=Q_\mathbb{Q},\]
  and consider the dual cones $C^\vee$ and $D^\vee$. Choose a point $v$ in the interior of $D^\vee$. We triangulate $D^\vee$, and then we triangulate $C^\vee$ such that the triangulations are compatible.

  Let $\{\Delta_i\}_{i\in I}$ denote the set of $(n-1)$-simplexes of the triangulation $C^\vee$ contained in the boundary of $C^\vee$. We put
  \[C_i^\vee=\Delta_i+\langle v\rangle,\quad D_i^\vee=C_i^\vee\cap D^\vee,\]
  and we denote by $C_i$ and $D_i$ the dual cones of $C_i$ and $D_i$ respectively. Now we put
  \[P_i=C_i\cup P^{\rm gp},\quad Q_i=D_i\cap P^{\rm gp},\quad H=(\langle v\rangle)^{\bot},\quad r_i=\Delta_i^\bot.\]
  Then $r_i$ is a ray of $C$ since $\Delta_i$ is an $(n-1)$-simplex, and $H$ is an $(n-1)$-hyperplane such that $H\cap C=H\cap D=\langle 0\rangle$ since $v$ is in the interior of $D^\vee$. For each $i\in I$, we have the following two cases: $C_i\neq D_i$ or $C_i=D_i$.

  If $C_i\neq D_i$, then
  \[C_i=\langle b_1,\ldots,b_{n-1},r_i\rangle,\quad D_i=\langle b_1,\ldots,b_{n-1},a\rangle\]
  for some $b_1,\ldots,b_{n-1}\in H$. Since $H\cap C=H\cap D=\langle 0\rangle$, if $F$ (resp.\ $G$) is a face of $C_i$ (resp.\ $D_i$), then
  \[F\cap C=\langle 0\rangle \;\;\Leftrightarrow \;\;F\subset \langle b_1,\ldots,b_{n-1}\rangle,\]
  \[G\cap D=\langle 0\rangle \;\;\Leftrightarrow \;\;G\subset \langle b_1,\ldots,b_{n-1}\rangle.\]
  Thus the coCartesian diagram
  \[\begin{tikzcd}
    P\arrow[r]\arrow[d]&P_i\arrow[d]\\
    Q\arrow[r]&Q_i
  \end{tikzcd}\]
  satisfies the condition of (\ref{6.1.2}), so by (loc.\ cit), the induced morphism
  \begin{equation}\label{6.1.3.1}
  S\times \mathbb{A}_{Q_i}\times_{\mathbb{A}_Q}{\rm pt}_Q\rightarrow S\times \mathbb{A}_{P_i}\times_{\mathbb{A}_P}{\rm pt}_P
  \end{equation}
  satisfies the conditions (i) and (ii) of (\ref{6.1.1}).

  If $C_i=D_i$, then $C_i=D_i=\langle b_1,\ldots,b_{n-1},r_i\rangle$ for some $b_1,\ldots,b_{n-1}\in H$, and we can similarly show that (\ref{6.1.3.1}) satisfies the conditions (i) and (ii) of (\ref{6.1.1}).

  Let $M$ (resp.\ $N$) denote the fs monoscheme that is a gluing of ${\rm spec}\,P_i$ (resp.\ ${\rm spec}\,Q_i$) for $i\in I$. Then consider the induced commutative diagram
  \[\begin{tikzcd}
    \underline{S}\times\mathbb{A}_N\times_{\mathbb{A}_Q}{\rm pt}_Q\arrow[r,"f'"]\arrow[d,"g'"]&\underline{S}\times\mathbb{A}_M\times_{\mathbb{A}_P}{\rm pt}_P\arrow[d,"g"]\\
    \underline{S}\times {\rm pt}_Q\arrow[r,"f"]&\underline{S}\times {\rm pt}_P
  \end{tikzcd}\]
  of $\mathscr{S}$-schemes. We have shown that $f'$ satisfies the conditions (i) and (ii) of (loc.\ cit).

  Consider the commutative diagram
  \[\begin{tikzcd}
    {\rm id}\arrow[rr,"ad"]\arrow[d,"ad"]&&f_*f^*\arrow[d,"ad"]\\
    g_*g^*\arrow[r,"ad"]&g_*f_*'f'^*g^*\arrow[r,"\sim"]&f_*g_*'g'^*f^*
  \end{tikzcd}\]
  of functors. The vertical arrows are isomorphisms by (Htp--4) since $g$ and $g'$ are dividing covers. Thus the question for $f$ reduces to the question for $f'$.

  Then using Mayer-Vietoris triangle, by induction on ${\rm dim}\,P$, we reduce to the questions for $(P,Q)=(P_i,Q_i)$ for $i\in I$. In particular, we may assume
  \[(P_i)_\mathbb{Q}=\langle b_1,\ldots,b_{r-1},b_r\rangle,\quad (Q_i)_\mathbb{Q}=\langle b_1,\ldots,b_{r-1},a\rangle.\]\\[4pt]
  (V) {\it Final step of the proof.} We put
  \[F=\langle b_1\rangle\cap P,\quad G=\langle b_1\rangle \cap  Q,\quad P'=\langle b_2,\ldots,b_r\rangle \cap P,\quad Q'=\langle b_2,\ldots,b_{r-1},a\rangle \cap Q.\]
  Consider the commutative diagram
  \[\begin{tikzcd}
    \underline{S}\times{\rm pt}_Q\arrow[r,"i'"]\arrow[d,"f"]&\underline{S}\times\mathbb{A}_{(Q,Q-G)}\arrow[d,"f'"]\arrow[rr,"j'",leftarrow]\arrow[rdd,"q",bend left,near start]&&\underline{S}\times(\mathbb{A}_{(Q,Q-G)}-{\rm pt}_Q)\arrow[d,"f''"]\\
    \underline{S}\times{\rm pt}_P\arrow[r,"i"]&\underline{S}\times\mathbb{A}_{(P,P-F)}\arrow[rr,"j",leftarrow,crossing over]\arrow[d,"p"]&&\underline{S}\times(\mathbb{A}_{(P,P-F)}-{\rm pt}_P)\\
    &\underline{S}\times{\rm pt}_{P'}&\underline{S}\times{\rm pt}_{Q'}
  \end{tikzcd}\]
  where
  \begin{enumerate}[(i)]
    \item $p$ denotes the morphism induced by the inclusion $P'\rightarrow P$,
    \item $j$ denotes the complement of $i$,
    \item each square is Cartesian.
  \end{enumerate}
  Then $j$ (resp.\ $j'$) is the verticalization of $\underline{S}\times \mathbb{A}_{(P,P-F)}$ (resp.\ $\underline{S}\times\mathbb{A}_{(Q,Q-G)}$) via $p$ (resp.\ $f'p$). Thus by (Htp--2), the natural transformations
  \[p^*\stackrel{ad}\longrightarrow j_*j^*p^*,\quad q^*\stackrel{ad}\longrightarrow j_*'j'^*q^*\]
  are isomorphisms. From the commutative diagram
  \[\begin{tikzcd}
    p^*\arrow[rr,"ad"]\arrow[d,"ad","\sim"']&&f_*'f'^*p^*\arrow[d,"\sim"]\\
    j_*j^*p^*\arrow[r,"ad","\sim"']&j_*f_*''f''^*j^*p^*\arrow[r,"\sim"]&p_*'j_*'j'^*f'^*p^*
  \end{tikzcd}\]
  of functors, we see that the upper horizontal arrow is an isomorphism.

  In the commutative diagram
  \[\begin{tikzcd}
    i^*p^*\arrow[r,"ad","\sim"']\arrow[rrd,"ad"']&i^*f_*'f'^*p^*\arrow[r,"Ex"]&f_*i'^*f'^*p^*\arrow[d,"\sim"]\\
    &&f_*f^*i^*p^*
  \end{tikzcd}\]
  of functors, the upper right horizontal arrow is an isomorphism by (\ref{2.6.6}). Thus the diagonal arrow is an isomorphism. In particular, the morpism
  \[1_S\stackrel{ad}\longrightarrow f_*f^*1_S\]
  in $\mathscr{T}_S$ is an isomorphism.

  For any object $K$ of $\mathscr{T}_S$, we have the commutative diagram
  \[\begin{tikzcd}
    K\arrow[r,"\sim"]\arrow[rrd,"ad"']&1_S\otimes K\arrow[r,"ad","\sim"']&f_*f^*1_S\otimes K\arrow[d,"Ex"]\\
    &&f_*f^*K
  \end{tikzcd}\]
  in $\mathscr{T}_S$. By (PF), the right vertical arrow is an isomorphism. Thus the diagonal arrow is also an isomorphism, which completes the proof.
\end{proof}
\begin{none}\label{6.1.4}
  So far, we have proven the half of (Htp--5). In the remaining, we will first prove a few lemmas. Then we will prove (Htp--5).
\end{none}
\begin{lemma}\label{6.1.5}
  Let $\theta:P\rightarrow Q$ be a homomorphism of fs sharp monoids such that $\theta^{\rm gp}:P^{\rm gp}\rightarrow Q^{\rm gp}$ is an isomorphism, and let $\eta':Q\rightarrow Q'$ be a homomorphism of fs monoids such that $\overline{\eta}':Q\rightarrow \overline{Q'}$ is Kummer. Then there is a coCartesian diagram
  \[\begin{tikzcd}
    P\arrow[d,"\eta"]\arrow[r,"\theta"]&Q\arrow[d,"\eta'"]\\
    P'\arrow[r]&Q'
  \end{tikzcd}\]
  of fs monoids.
\end{lemma}
\begin{proof}
  Let $P'$ denote the submonoid of $Q'$ consisting of elements $p'\in Q'$ such that $np'\in P+Q'^*$ for some $n\in \mathbb{N}^+$. Since $\overline{\eta}'$ is Kummer, by the construction of pushout in the category of fs monoids, it suffices to verify $P'^{\rm gp}=Q'^{\rm gp}$ to show that the above diagram is coCartesian.

  Let $q'\in Q'$ be an element. Since $\overline{\eta}'$ is Kummer, we can choose $m\in \mathbb{N}^+$ such that $mq'=q+q''$ for some $q\in Q$ and $q''\in Q'^*$. We put $r={\rm dim}\,P$, and choose $r$ linearly independent elements $p_1,\ldots,p_r\in P$ over $\mathbb{Q}$. Then let $(a_1,\ldots,a_r)$ denote the coordinate of $q\in Q$ according to the basis $\{p_1,\ldots,p_r\}$.

  Choose $b_1,\ldots,b_r\in \mathbb{N}^+$ such that $a_1+mb_1,\ldots,a_r+mb_r>0$, and we put $p=(b_1,\ldots,b_r)\in P$. Then
  \[q+mp=(a_1+mb_1,\ldots,a_r+mb_r)\in P,\]
  so $m(q'+p)=q+mp+q''\in P+Q'^*$. Thus $q'+p\in P'$, so $q'\in P'^{\rm gp}$. This proves $P'^{\rm gp}=Q'^{\rm gp}$.
\end{proof}
\begin{lemma}\label{6.1.6}
  Let $P$ be a sharp fs monoid, and let $\eta:P\rightarrow P'$ be a homomorhpism of fs monoids such that $\overline{\eta}:P\rightarrow \overline{P'}$ is Kummer. We denote by $I$ the ideal of $P'$ generated by $\eta(P^+)$. Then the induced morphism
  \[\mathbb{A}_{(P',P'^+)}\rightarrow \mathbb{A}_{(P',I)}\]
  is a bijective strict closed immersion.
\end{lemma}
\begin{proof}
  For any element $p'\in P'^+$, for some $n\in \mathbb{N}^+$, we have $np'\in I$ since $\overline{\eta}$ is Kummer. Let $\mathfrak{m}$ (resp.\ $\mathfrak{n}$) denote the ideal of $\mathbb{Z}[P']$ induced by $P'^+$ (resp.\ $I$). Then by the above argument, for some $m\in \mathbb{N}^+$, we have $\mathfrak{n}^m\subset \mathfrak{m}$. This implies the assertion.
\end{proof}
\begin{none}\label{6.1.7}
  Let $\theta:P\rightarrow Q$ be a homomorphism of sharp fs monoids such that $\theta^{\rm gp}$ is an isomorphism, and consider a coCartesian diagram
  \[\begin{tikzcd}
    P\arrow[d,"\eta"]\arrow[r,"\theta"]&Q\arrow[d,"\eta'"]\\
    P'\arrow[r,"\theta'"]&Q'
  \end{tikzcd}\]
  of fs monoids where $\overline{\eta}:P\rightarrow \overline{P'}$ is Kummer. Note that $\eta'^{\rm gp}$ is also an isomorphism by the construction of the pushout in the category of fs monoids. Then we have the induced commutative diagram
  \[\begin{tikzcd}
    \mathbb{A}_{(Q',Q'^+)}\arrow[r]\arrow[d]&\mathbb{A}_{(P',P'^+)}\arrow[d]\\
    \mathbb{A}_{(Q',J)}\arrow[r]&\mathbb{A}_{(P',I)}
  \end{tikzcd}\]
  of schemes where $I$ (resp.\ $J$) denote the ideal of $P'$ (resp.\ $Q'$) generated by $\eta(P)$ (resp.\ $\eta'(Q)$). By (\ref{6.1.6}), the vertical arrows are bijective strict closed immersions, and the upper horizontal arrow is an isomorphism since they are isomorphic to $\mathbb{A}_{Q'^{\rm gp}}=\mathbb{A}_{P'^{\rm gp}}$. Thus the category of strict \'etale morphisms to $\mathbb{A}_{(Q,J)}$ is equivalent to that of $\mathbb{A}_{(P,I)}$ by \cite[IV.18.1.2]{EGA}.
\end{none}
\begin{prop}\label{6.1.8}
  Let $\theta:P\rightarrow Q$ be a homomorphism of sharp fs monoids such that $\theta^{\rm gp}$ is an isomorphism, and let $S$ be an $\mathscr{S}$-scheme with the trivial log structure. Consider the induced morphism $f:S\times{\rm pt}_Q\rightarrow S\times {\rm pt}_P$ of $\mathscr{S}$-schemes. Then the functor $f^*$ is essentially surjective.
\end{prop}
\begin{proof}
  By (\ref{0.5.3}), it suffices to show that for any Kummer log smooth morphism $g':Y'\rightarrow S\times {\rm pt}_Q$ with an fs chart $\eta':Q\rightarrow Q'$ of Kummer log smooth type, there is a Cartesian diagram
  \[\begin{tikzcd}
    Y'\arrow[d,"g'"]\arrow[r]&Y\arrow[d,"g"]\\
    S\times {\rm pt}_Q\arrow[r,"f"]&S\times {\rm pt}_P
  \end{tikzcd}\]
  of $\mathscr{S}$-schemes such that $g$ is Kummer log smooth. Note that $\overline{\eta'}:Q\rightarrow \overline{Q'}$ is Kummer.

  By definition (\ref{0.1.1}), we can choose a factorization
  \[Y'\rightarrow S\times {\rm pt}_Q\times_{\mathbb{A}_Q}\mathbb{A}_{Q'}\rightarrow S\]
  of $g'$ where the first arrow is strict \'etale and the second arrow is the projection. Then by (\ref{6.1.5}), there is a coCartesian diagram
  \[\begin{tikzcd}
    P\arrow[d,"\eta"]\arrow[r,"\theta"]&Q\arrow[d,"\eta'"]\\
    P'\arrow[r,"\theta'"]&Q'
  \end{tikzcd}\]
  of fs monoids such that $\overline{\eta}:P\rightarrow \overline{P'}$ is Kummer. Now we have the commutative diagram
  \[\begin{tikzcd}
    Y'\arrow[d]\\
    S\times {\rm pt}_Q\times_{\mathbb{A}_Q}\mathbb{A}_{Q'}\arrow[d]\arrow[r]&S\times{\rm pt}_P\times_{\mathbb{A}_P}\mathbb{A}_{P'}\arrow[d]\\
    S\times {\rm pt}_Q\arrow[r]&S\times {\rm pt}_P
  \end{tikzcd}\]
  of $\mathscr{S}$-schemes where the square is Cartesian. Since
  \[{\rm pt}_Q\times_{\mathbb{A}_Q}\mathbb{A}_{Q'}=\mathbb{A}_{(Q,J)},\quad {\rm pt}_P\times_{\mathbb{A}_P}\mathbb{A}_{P'}=\mathbb{A}_{(P,I)}\]
  where $I$ (resp.\ $J$) denotes the ideal of $P'$ (resp.\ $Q'$) generated by $\eta(P)$ (resp.\ $\eta'(Q)$), by (\ref{6.1.7}), there is a commutative diagram
  \[\begin{tikzcd}
    Y'\arrow[d]\arrow[r]&Y\arrow[d]\\
    S\times {\rm pt}_Q\times_{\mathbb{A}_Q}\mathbb{A}_{Q'}\arrow[d]\arrow[r]&S\times{\rm pt}_P\times_{\mathbb{A}_P}\mathbb{A}_{P'}\arrow[d]\\
    S\times {\rm pt}_Q\arrow[r]&S\times {\rm pt}_P
  \end{tikzcd}\]
  of $\mathscr{S}$-schemes where each square is Cartesian and the arrow $Y\rightarrow S\times {\rm pt}_P\times_{\mathbb{A}_P}\mathbb{A}_{P'}$ is strict \'etale. Thus we have constructed the diagram we want.
\end{proof}
\begin{thm}\label{6.1.9}
  The log motivic triangulated category $\mathscr{T}$ satisfies (Htp--5).
\end{thm}
\begin{proof}
  Let $f:X\rightarrow S$ be a morphism of $\mathscr{S}$-schemes satisfying the conditions (i) and (ii) of (\ref{6.1.1}). By (\ref{6.1.3}), it suffices to show that the counit
  \[f^*f_*\stackrel{ad'}\longrightarrow {\rm id}\]
  is an isomorphism.

  Let $\{g_i:S_i\rightarrow S\}_{i\in I}$ be a strict \'etale cover. Consider the Cartesian diagram
  \[\begin{tikzcd}
    X_i\arrow[d,"f_i"]\arrow[r,"g_i'"]&X\arrow[d,"f"]\\
    S_i\arrow[r,"g_i"]&S
  \end{tikzcd}\]
  of $\mathscr{S}$-schemes. Then we have the commutative diagram
  \[\begin{tikzcd}
    g_i'^*f^*f_*\arrow[rrd,"ad'"']\arrow[r,"\sim"]& f_i^*g_i^*f_*\arrow[r,"Ex"]&f_i^*f_{i*}g_i'^*\arrow[d,"ad'"]\\
    &&g_i'^*
  \end{tikzcd}\]
  of functors, and the upper right horizontal arrow is an isomorphism by ($eSm$-BC). Since the family of functors $\{g_i'^*\}_{i\in I}$ is conservative by (k\'et-Sep), we reduce to showing that for any $i\in I$, the counit $f_i^*f_{i*}\stackrel{ad'}\longrightarrow {\rm id}$ is an isomorphism. Using this technique, we reduce to the case when $f$ has an fs chart $\theta:P\rightarrow Q$.

  Then let $i:S'\rightarrow S$ be a strict closed immersion of $\mathscr{S}$-schemes, and let $j:S''\rightarrow S$ denote its complement. Consider the commutative diagram
  \[\begin{tikzcd}
    X'\arrow[d,"f'"]\arrow[r,"i'"]&X\arrow[d,"f"]\arrow[r,"j'",leftarrow]&X''\arrow[d,"f''"]\\
    S'\arrow[r,"i"]&S\arrow[r,"j",leftarrow]&S''
  \end{tikzcd}\]
  of $\mathscr{S}$-schemes where each square is Cartesian. Then we have the commutative diagrams
  \[\begin{tikzcd}
    i'^*f^*f_*\arrow[rrd,"ad'"']\arrow[r,"\sim"]& f'^*i^*f_*\arrow[r,"Ex"]&f'^*f_*'i'^*\arrow[d,"ad'"]\\
    &&i'^*
  \end{tikzcd}
  \quad
  \begin{tikzcd}
    j'^*f^*f_*\arrow[rrd,"ad'"']\arrow[r,"\sim"]&f''^*j^*f_*\arrow[r,"Ex"]&f''^*f_*''j'^*\arrow[d,"ad'"]\\
    &&j^*
  \end{tikzcd}\]
  of functors. As in the above paragraph, by (\ref{2.6.6}) and (Loc), we reduce to showing that the counits
  \[f'^*f_*'\stackrel{ad}\longrightarrow {\rm id},\quad f''^*f_*''\stackrel{ad}\longrightarrow {\rm id}\]
  are isomorphisms. Using this technique, by the proof of \cite[3.5(ii)]{Ols03}, we reduce to the case when $X\rightarrow S$ is the morphism $\underline{S}\times{\rm pt}_Q\rightarrow \underline{S}\times {\rm pt}_P$ induced by the homomorphism $P\rightarrow Q$. In this case, the conclusion follows from (\ref{6.1.8}) since $f^*$ is fully faithful by (\ref{6.1.3}).
\end{proof}
\subsection{Homotopy property 6}
\begin{thm}\label{6.2.1}
  The log motivic triangulated category $\mathscr{T}$ satisfies (Htp--6).
\end{thm}
\begin{proof}
  Let $S$ be an $\mathscr{S}$-scheme, and we put $X=S\times\mathbb{A}_\mathbb{N}$ and $Y=S\times {\rm pt}_\mathbb{N}$. Consider the commutative diagram
  \[\begin{tikzcd}
    S\times{\rm pt}_\mathbb{N}\arrow[r,"i'"]\arrow[d,"g'"']&S\times\mathbb{A}_\mathbb{N}\arrow[d,"f'"]\arrow[r,"j'",leftarrow]&S\times\mathbb{G}_m\arrow[d,"{\rm id}"]\\
    S\arrow[r,"i"]\arrow[rd,"{\rm id}"']&S\times\mathbb{A}^1\arrow[d,"f"]\arrow[r,"j",leftarrow]&S\times\mathbb{G}_m\\
    &S
  \end{tikzcd}\]
  of $\mathscr{S}$-schemes where
  \begin{enumerate}[(i)]
    \item each small square is Cartesian,
    \item $f$ denotes the projection, and $f'$ denotes the morphism removing the log structure.
    \item $i$ denotes the $0$-section, and $j$ denotes its complement.
  \end{enumerate}
  We want to show that the natural transformation
  \[f_*f_*'f'^*f^*\stackrel{ad}\longrightarrow f_*f_*'i_*'i'^*f'^*f^*\stackrel{\sim}\longrightarrow g_*'g'^*\]
  is an isomorphism. By (Loc), it is equivalent to showing
  \[f_*f_*'j_\sharp'j'^*f'^*f^*=0.\]
  Then by (Supp), it is equivalent to showing
  \[f_*j_\sharp j^*f^*=0.\]
  Thus by (Loc), it is equivalent to showing that the natural transformation
  \[f_*f^*\stackrel{ad}\longrightarrow f_*i_*i^*f^*\]
  is an isomorphism, which follows from (Htp--1).
\end{proof}
\begin{none}\label{6.2.2}
  Here, we give an application of (Htp--6). Let $S$ be an $\mathscr{S}$-scheme. Consider the commutative diagram
  \[\begin{tikzcd}
    S\times {\rm pt}_\mathbb{N}\arrow[r,"i_0"]\arrow[rd,"g"']&S\times \mathbb{A}_\mathbb{N}\arrow[d,"f"]\\
    &S
  \end{tikzcd}\]
  of $\mathscr{S}$-schemes where $f$ denotes the projection and $i_0$ denotes the $0$-section. Let $i_1:S\rightarrow S\times \mathbb{A}_\mathbb{N}$ denote the $1$-section. By (Htp--6), the natural transformation
  \[f_*f^*\stackrel{ad}\longrightarrow f_*i_*i^*f^*\]
  is an isomorphism, and $f^*$ is conservative since $fi_1={\rm id}$. Thus by (\ref{5.3.2}), $(fi)^*=g^*$ is conservative.
\end{none}
\subsection{Homotopy property 7}
\begin{thm}\label{6.3.1}
  Let $S$ be an $\mathscr{S}$-scheme with an fs chart $P$, let $\theta:P\rightarrow Q$ be a vertical homomorphism of exact log smooth over $S$ type, and let $G$ be a $\theta$-critical face of $Q$. We denote by
  \[f:S\times_{\mathbb{A}_P}(\mathbb{A}_Q-\mathbb{A}_{Q_G})\rightarrow S\]
  the projection. Then $f_!f^*=0$. In other words, $\mathscr{T}$ satisfies (Htp--7).
\end{thm}
\begin{proof}
  (I) {\it Locality of the question.} Note first that the question is strict \'etale local on $S$ by ($eSm$-BC).

  Let $i:Z\rightarrow S$ be a strict closed immersion of $\mathscr{S}$-schemes, and let $j:U\rightarrow S$ denote its complement. Consider the commutative diagram
  \[\begin{tikzcd}
    Z\times_{\mathbb{A}_P}(\mathbb{A}_Q-\mathbb{A}_{Q_G})\arrow[d,"g"]\arrow[r,"i'"]&S\times_{\mathbb{A}_P}(\mathbb{A}_Q-\mathbb{A}_{Q_G})\arrow[d,"f"]\arrow[r,"j'",leftarrow]& U\times_{\mathbb{A}_P}(\mathbb{A}_Q-\mathbb{A}_{Q_G})\arrow[d,"h"]\\
    Z\arrow[r,"i"]&S\arrow[r,"j",leftarrow]&U
  \end{tikzcd}\]
  of $\mathscr{S}$-schemes where each square is Cartesian. Then by (BC--3), we reduce to showing $g_!g^*=0$ and $h_!h^*=0$. By the proof of \cite[3.5(ii)]{Ols03}, we reduce to the case when $S$ has a constant log structure.\\[4pt]
  (II) {\it Reduction of $G$.} Let $G_1$ be a maximal $\theta$-critical face of $Q$ containing $G$. Consider the induced commutative diagram
  \[\begin{tikzcd}
    S\times_{\mathbb{A}_P}(\mathbb{A}_Q-\mathbb{A}_{Q_G})\arrow[rd,"f"']\arrow[r]&S\times_{\mathbb{A}_P}(\mathbb{A}_Q-\mathbb{A}_{Q_{G_1}})\arrow[d,"f'"]\arrow[r,leftarrow]& S\times_{\mathbb{A}_P}(\mathbb{A}_{Q_G}-\mathbb{A}_{Q_{G_1}})\arrow[ld,"f''"]\\
    &S
  \end{tikzcd}\]
  of $\mathscr{S}$-schemes. By (Loc), to show $f_!f^*=0$, it suffices to show $f_!'f'^*=0$ and $f_!''f''^*=0$. Hence replacing $f$ by $f'$ or $f''$, we reduce to the case when $G$ is a maximal $\theta$-critical face of $Q$.\\[4pt]
  (III) {\it Reduction of $P$.} We denote by $P'$ the submonoid of $Q$ consisting of elements $q\in Q$ such that $nq\in P+Q^*$ for some $n\in \mathbb{N}^+$. The induced homomorphism $\theta':P'\rightarrow Q$ is again a vertical homomorphism of exact log smooth over $S\times_{\mathbb{A}_P}\mathbb{A}_{P'}$ type. Replacing $S$ by $S\times_{\mathbb{A}_P}\mathbb{A}_{P'}$, we reduce to the case when the cokernel of $\theta^{\rm gp}$ is torsion free and $\theta$ is logarithmic.

  Then by (\ref{0.2.2}), since the question is strict \'etale local on $S$, we may assume that $P$ is sharp and that the fs chart $S\rightarrow \mathbb{A}_P$ is neat at some point $s\in S$. Then $P$ and $Q$ are sharp, and with (I), we may further assume that the fs chart $S\rightarrow \mathbb{A}_P$ factors through ${\rm pt}_P$.\\[4pt]
  (IV) {\it Homotopy limit.} Let $G_1=G,\ldots,G_r$ denote the maximal $\theta$-critical faces of $Q$. The condition that $\theta$ is vertical implies $r\geq 2$. For any nonempty subset $I=\{i_1,\ldots,i_s\}\subset \{2,\ldots,r\}$, we put
  \[G_I=G_{i_1}\cap \cdots\cap G_{i_s},\]
  and we denote by
  \[f_I:S\times_{\mathbb{A}_P}\mathbb{A}_{(Q,Q-G_I)}\rightarrow S\]
  the projection. For any face $H$ of $Q$, $\mathbb{A}_H\subset \mathbb{A}_Q-\mathbb{A}_{Q_H}$ if and only if $H\neq Q,G$, which is equivalent to $H\subset G_2\cup \cdots \cup G_r$. Thus the family
  \[\{S\times_{\mathbb{A}_P}\mathbb{A}_{(Q,Q-G_{2})},\ldots,S\times_{\mathbb{A}_P}\mathbb{A}_{(Q,Q-G_{r})}\}\]
  forms a closed cover of $S\times_{\mathbb{A}_P}(\mathbb{A}_Q-\mathbb{A}_{Q_G})$, so for any object $K$ of $\mathscr{T}(S)$, $f_*f^*K$ is the homotopy limit of the \v{C}ech-type sequence
  \[\bigoplus_{|I|=1,|I|\subset \{2,\ldots,r\}} f_{I*}f_I^*K\longrightarrow \cdots \longrightarrow \bigoplus_{|I|=r-1,|I|\subset \{2,\ldots,r\}} f_{I*}f_I^*K\]
  in $\mathscr{T}(S)$. Hence we reduce to showing $f_{I*}f_I^*K=0$ for any nonempty subset $I\subset \{2,\ldots,r\}$. This is proved in (\ref{6.3.1}) below.
\end{proof}
\begin{lemma}\label{6.3.2}
  Let $S$ be an $\mathscr{S}$-scheme with a constant log structure $S\rightarrow {\rm pt}_P$ where $P$ is a sharp fs monoid, let $\theta:P\rightarrow Q$ be a homomorphism of exact log smooth over $S$ type, and let $G$ be a $\theta$-critical face of $Q$ such that $G\neq Q^*$. We denote by
  \[f:S\times_{\mathbb{A}_P}\mathbb{A}_{(Q,Q-G)}\rightarrow S\]
  the projection. Then $f_!f^*=0$.
\end{lemma}
\begin{proof}
  We will use induction on ${\rm dim}\,Q$. We have ${\rm dim}\,Q\geq 1$ since $G\neq Q^*$.\\[4pt]
  (I) {\it Locality of the question.} Note that by ($eSm$-BC), the question is strict \'etale local on $S$.\\[4pt]
  (II) {\it Reduction of $G$.} Let $G'$ be a $1$-dimensional face of $G$, and choose generators $b_1,\ldots,b_r$ of the ideal $Q-G'$ in $Q$. For any nonempty subset $I=\{i_1,\ldots,i_s\}\subset \{1,\ldots,r\}$, we denote by $Q_I$ the localization $Q_{b_{i_1},\ldots,b_{i_s}}$, and we denote by $G_I$ the face of $Q_I$ generated by $G$. The family
  \[\{S\times_{\mathbb{A}_P}\mathbb{A}_{(Q_{b_1},Q_{b_1}-G_{b_1})},\ldots,S\times_{\mathbb{A}_P}\mathbb{A}_{(Q_{b_s},Q_{b_s}-G_{b_s})}\}\]
  forms an open cover of $S\times_{\mathbb{A}_P}(\mathbb{A}_{(Q,Q-G)}-\mathbb{A}_{(Q,Q-G')})$. Thus if we denote by
  \[f_I:S\times_{\mathbb{A}_P}\mathbb{A}_{(Q_I,Q_I-G_I)}\rightarrow S,\]
  \[f':S\times_{\mathbb{A}_P}\mathbb{A}_{(Q,Q-G')}\rightarrow S,\quad f'':S\times_{\mathbb{A}_P}(\mathbb{A}_{(Q,Q-G)}-\mathbb{A}_{(Q,Q-G')})\rightarrow S\]
  the projections, then for any object $K$ of $\mathscr{T}(S)$, $f_!''f''^*K$ is a homotopy colimit of
  \[\bigoplus_{|I|=r} f_{I*}f_I^*K\longrightarrow \cdots \longrightarrow \bigoplus_{|I|=1} f_{I*}f_I^*K\]
  in $\mathscr{T}(S)$. Since ${\rm dim}\,Q_I<{\rm dim}\,Q$, for any nonempty subset $I\subset \{1,\ldots,r\}$, $f_{I!}f_I^*=0$ by induction on ${\rm dim}\,Q$. Thus $f_!''f''^!=0$. Then by (Loc), $f_!f^*=0$ is equivalent to $f_!'f'^*=0$. Hence replacing $G$ by $G'$, we reduce to the case when ${\rm dim}\,G=1$.\\[4pt]
  (III) {\it Reduction of $\theta$.} Let $a_1$ be a generator of $G$, let $H$ be a maximal $\theta$-critical face of $Q$ containing $G$, and choose $a_2,\ldots,a_d\in H$ where $d={\rm dim}\,G$ such that $\overline{a_1},\ldots,\overline{a_d}$ in $\overline{Q}$ are independent over $\mathbb{Q}$.

  We denote by $P'$ (resp.\ $Q'$) the submonoid of $Q$ consisting of elements $q\in Q$ such that $nq\in \langle a_2,\ldots,a_n\rangle +P$ (resp.\ $nq\in \langle a_1,\ldots,a_n\rangle +P$) for some $n\in \mathbb{N}^+$. Then we denote by $G'$ the face of $Q'$ generated by $a_1$, and we denote by $\theta':P'\rightarrow Q'$ the induced homomorphism. Consider the induced morphisms
  \[S\times_{\mathbb{A}_P}\mathbb{A}_{(Q,Q-G)}\stackrel{w}\rightarrow S\times_{\mathbb{A}_P}\mathbb{A}_{(Q',Q'-G')}\stackrel{v}\rightarrow S\times_{\mathbb{A}_P}\mathbb{A}_{(P',P')}\stackrel{u}\rightarrow S\]
  of $\mathscr{S}$-schemes. The induced homomorphism $Q'^{\rm gp}\rightarrow Q^{\rm gp}$ is an isomorphism by \cite[I.4.7.9]{Ogu17}. Thus by (\ref{6.1.3}), the unit ${\rm id}\stackrel{ad}\longrightarrow w_*w^*$ is an isomorphism. Hence to show $f_!f^*=0$, it suffices to show $v_!v^*=0$.

  The cokernel of $\theta'^{\rm gp}$ is torsion free, and the diagram
  \[\begin{tikzcd}
    P'\arrow[d,"\theta'"]\arrow[r]&\overline{P'}\arrow[d,"\overline{\theta'}"]\\
    Q'\arrow[r]&\overline{Q'}
  \end{tikzcd}\]
  is coCartesian where the horizontal arrows are the quotient homomorphisms. Thus we can apply (\ref{0.2.2}), so strict \'etale locally on $S$, we have a Cartesian diagram
  \[\begin{tikzcd}
    S\times_{\mathbb{A}_P}\mathbb{A}_{(Q',Q'-G')}\arrow[d]\arrow[r]&\mathbb{A}_{(\overline{Q'},\overline{Q'}-\overline{G'})}\arrow[d]\\
    S\times_{\mathbb{A}_P}\mathbb{A}_{(P',P')}\arrow[r]&\mathbb{A}_{\overline{P'}}
  \end{tikzcd}\]
  of $\mathscr{S}$-schemes. The homomorphism $\overline{\theta'}:\overline{P'}\rightarrow \overline{Q'}$ is again a homomorphism of exact log smooth over $S\times_{\mathbb{A}_P}\mathbb{A}_{(P',P')}$ type. Replacing $(S\times_{\mathbb{A}_P}\mathbb{A}_{(Q,Q-G)}\rightarrow S,\;\theta:P\rightarrow Q)$ by
  \[(S\times_{\mathbb{A}_P}\mathbb{A}_{(Q',Q'-G')}\rightarrow S\times_{\mathbb{A}_P}\mathbb{A}_{(P',P')},\;\overline{\theta'}:\overline{P'}\rightarrow \overline{Q'}),\]
  we may assume that
  \begin{enumerate}[(i)]
    \item $P$ and $Q$ are sharp,
    \item $Q_\mathbb{Q}=P_\mathbb{Q}\oplus G_\mathbb{Q}$,
    \item ${\rm dim}\,G=1$.
  \end{enumerate}
  (IV) {\it Further reduction of $\theta$.} Since $P_\mathbb{Q}$ and $G_\mathbb{Q}$ generate $Q_\mathbb{Q}$, as in (\ref{5.3.3}), we can choose a homomorphism $P\rightarrow P_1$ of Kummer log smooth over $S$ type such that
  \begin{enumerate}
    \item $P_1$ and $G$ generate $P_1\oplus_P Q$,
    \item the functor $g^*$ is conservative where $g:S\times_{\mathbb{A}_P}\mathbb{A}_{P_1}\rightarrow S$ denotes the projection.
  \end{enumerate}
  We put $Q_1=P_1\oplus_P Q$, and we denote by $G_1$ the face of $Q_1$ generated by $G$. Consider the Cartesian diagram
  \[\begin{tikzcd}
    S\times_{\mathbb{A}_P}\mathbb{A}_{(Q_1,Q_1-G_1)}\arrow[d,"g'"]\arrow[r,"f'"]&S\times_{\mathbb{A}_P}\mathbb{A}_{P_1}\arrow[d,"g"]\\
    S\times_{\mathbb{A}_P}\mathbb{A}_{(Q,Q-G)}\arrow[r,"f"]&S
  \end{tikzcd}\]
  of $\mathscr{S}$-schemes. Since $g^*$ is conservative, to show $f_!f^*=0$, it suffices to show $f_!'f'^*=0$ by ($eSm$-BC).

  By (\ref{0.2.2}), strict \'etale locally on $S\times_{\mathbb{A}_P}\mathbb{A}_{P'}$, there is a Cartesian diagram
  \[\begin{tikzcd}
    S\times_{\mathbb{A}_P}\mathbb{A}_{(Q_1,Q_1-G_1)}\arrow[r]\arrow[d,"f'"]&\mathbb{A}_{(\overline{Q_1},\overline{Q_1}-\overline{G_1})}\arrow[d]\\
    S\times_{\mathbb{A}_P}\mathbb{A}_{P_1}\arrow[r]&\mathbb{A}_{\overline{P_1}}
  \end{tikzcd}\]
  of $\mathscr{S}$-schemes. Replacing $(S\times_{\mathbb{A}_P}\mathbb{A}_{(Q,Q-G)}\rightarrow S,\;\theta:P\rightarrow Q)$ by
  \[(S\times_{\mathbb{A}_P}\mathbb{A}_{(Q_1,Q_1-G_1)}\rightarrow S\times_{\mathbb{A}_P}\mathbb{A}_{P_1},\;\overline{P_1}\rightarrow \overline{Q_1}),\]
  we may assume that
    \begin{enumerate}[(i)]
    \item $P$ and $Q$ are sharp,
    \item $Q=P\oplus G$,
    \item ${\rm dim}\,G=1$.
  \end{enumerate}
  (V) {\it Final step of the proof.} Then $S\times_{\mathbb{A}_P}\mathbb{A}_{(Q,Q-G)}=S\times_{\mathbb{A}_P}\mathbb{A}_Q$. Let $a_1$ denote the generator of $G$. We denote by $T$ the gluing of
  \[{\rm Spec}(Q\rightarrow \mathbb{Z}[Q]),\quad {\rm Spec}(P\rightarrow \mathbb{Z}[P,a_1^{-1}])\]
  along ${\rm Spec}(P\rightarrow \mathbb{Z}[P,a_1^\pm])$. Consider the commutative diagram
  \[\begin{tikzcd}
    &&S\arrow[ld,"i_1"']\arrow[ldd,"{\rm id}",near start]\\
    S\times_{\mathbb{A}_P}\mathbb{A}_Q\arrow[rd,"f"']\arrow[r,"j_1"]&S\times_{\mathbb{A}_P}T\arrow[d,"h"]\arrow[rr,"j_2",leftarrow,crossing over]&& S\times_{\mathbb{A}_P}(\mathbb{A}_P\times\mathbb{A}^1)\arrow[lld,"p"]\\
    &S
  \end{tikzcd}\]
  of $\mathscr{S}$-schemes where
  \begin{enumerate}[(i)]
    \item $j_1$ and $j_2$ denote the induced open immersions,
    \item $h$ and $p$ denote the projections,
    \item $i_1$ is a complement of $j_1$.
  \end{enumerate}
  Then $h$ is exact log smooth, and $j_2$ is the verticalization of $S\times_{\mathbb{A}_P} T$ via $h$. Thus by (Htp--2), the natural transformation
  \[h^*\stackrel{ad}\longrightarrow j_{2*}j_2^*h^*\]
  is an isomorphism. By (Htp--1), the composition
  \[{\rm id}\longrightarrow p_*p^*\stackrel{\sim}\longrightarrow h_*j_{2*}j_2^*h^*\]
  is an isomorphism, so the unit ${\rm id}\stackrel{ad}\longrightarrow h_*h^*$ is an isomorphism. Then by (Loc), for any object $K$ of $\mathscr{T}(S)$, we have the distinguished triangle
  \[h_*j_{1\sharp}j_1^*h^*K\stackrel{ad'}\longrightarrow h_*h^*K\stackrel{ad}\longrightarrow h_*i_{1*}i_1^*h^*K\longrightarrow h_*j_{1\sharp}j_1^*h^*K[1]\]
  in $\mathscr{T}(S)$. Since $i_1h={\rm id}$, the second arrow is the inverse of the unit ${\rm id}\stackrel{ad}\longrightarrow h_*h^*$, which is an isomorphism. Thus $f_!f^*\cong h_*j_{1\sharp}j_1^*h^*=0$.
\end{proof}
\subsection{Base change property 2}
\begin{prop}\label{6.4.1}
  Let $T$ be an $\mathscr{S}$-scheme, and consider the commutative diagram
  \[\begin{tikzcd}
    X'\arrow[d,"f'"]\arrow[r,"g'"]&X\arrow[d,"f"]\arrow[r]&T\times \mathbb{A}_{\mathbb{N}\oplus\mathbb{N}}\arrow[d,"p"]\\
    S'\arrow[r,"g"]&S\arrow[r]&T\times\mathbb{A}_\mathbb{N}
  \end{tikzcd}\]
  of $\mathscr{S}$-schemes where
  \begin{enumerate}[(i)]
    \item each square is Cartesian,
    \item $p$ denotes the morphism induced by the diagonal homomorphism $\mathbb{N}\rightarrow \mathbb{N}\oplus \mathbb{N}$ of fs monoids.
  \end{enumerate}
  Then $({\rm BC}_{f,g})$ is satisfied.
\end{prop}
\begin{proof}
  By (\ref{5.5.7}), the purity transformations
  \[f_\sharp \stackrel{\mathfrak{p}_f^o}\longrightarrow f_!(1)[2],\quad f_\sharp'\stackrel{\mathfrak{p}_{f'}^o}\longrightarrow f_!'(1)[2]\]
  are isomorphisms. Thus to show that the exchange transformation $g^*f_!\stackrel{Ex}\longrightarrow f_!'g'^*$ is an isomorphism, it suffices to show that the exchange transformation
  \[f_\sharp' g'^*\stackrel{Ex}\longrightarrow g^*f_\sharp\]
  is an isomorphism. This follows from ($eSm$-BC).
\end{proof}
\begin{prop}\label{6.4.2}
  Let $f:X\rightarrow S$ be a separated morphism of $\mathscr{S}$-schemes, and let $g:S\times {\rm pt}_\mathbb{N}\rightarrow S$ denote the projection. Then $({\rm BC}_{f,g})$ is satisfied.
\end{prop}
\begin{proof}
  Consider the commutative diagram
  \[\begin{tikzcd}
    X''\arrow[r,"i'"]\arrow[d,"f'"]&X'\arrow[d,"f''"]\arrow[r,"p'"]&X\arrow[d,"f"]\\
    S\times {\rm pt}_\mathbb{N}\arrow[r,"i"]&S\times \mathbb{A}_\mathbb{N}\arrow[r,"p"]&S
  \end{tikzcd}\]
  of $\mathscr{S}$-schemes where
  \begin{enumerate}[(i)]
    \item each square is Cartesian,
    \item $i$ denotes the $0$-section, and $p$ denotes the projection.
  \end{enumerate}
  Then $({\rm BC}_{f,p})$ is satisfied by ($eSm$-BC), and  $({\rm BC}_{f'',i})$ is satisfied by (BC--3). These two imply $({\rm BC}_{f,g})$.
\end{proof}
\begin{prop}\label{6.4.3}
  Let $T$ be an $\mathscr{S}$-scheme, and let $\theta:P\rightarrow Q$ be a locally exact homorphism of sharp fs monoids. We put $X=T\times {\rm pt}_Q$ and $S=T\times {\rm pt}_P$, and consider the morphism $f:X\rightarrow S$ induced by $\theta$. Then $({\rm BC}_{f,g})$ is satisfied for any morphism $g:S'\rightarrow S$ of $\mathscr{S}$-schemes.
\end{prop}
\begin{proof}
  (I) {\it Reduction method 1.} Assume that we have a factorization
  \[P\stackrel{\theta'}\longrightarrow Q'\stackrel{\theta''}\longrightarrow Q\]
  of $\theta$ where $\theta'$ is locally exact, $\theta''^{\rm gp}$ is an isomorphism, and $\theta'$ is a sharp fs monoid. Consider the morphism
  \[f':T\times {\rm pt}_{Q'}\rightarrow T\times {\rm pt}_P\]
  induced by $\theta'$. Then by (Htp--5), (${\rm BC}_{f,g}$) is equivalent to (${\rm BC}_{f',g}$).\\[4pt]
  (II) {\it Reduction method 2.} Let $u:S_0\rightarrow S$ be a morphism of $\mathscr{S}$-schemes, and consider a commutative diagram
  \[\begin{tikzcd}
    X_0'\arrow[dd,"f_0'"]\arrow[rr,"g_0'"]\arrow[rd,"v'"]&&X_0\arrow[rd,"v"]\arrow[dd,"f_0",near start]\\
    &X'\arrow[rr,"g'",crossing over,near start]&&X\arrow[dd,"f"]\\
    S_0'\arrow[rr,"g_0",near start]\arrow[rd,"u'"]&&S_0\arrow[rd,"u"]\\
    &S\arrow[uu,"f'",leftarrow,crossing over,near end]\arrow[rr,"g"]&&S
  \end{tikzcd}\]
  of $\mathscr{S}$-schemes where each small square is Cartesian. Then we have the commutative diagram
  \[\begin{tikzcd}
    u'^*g^*f_*\arrow[d,"\sim"]\arrow[r,"Ex"]&u'^*f_*'g'^*\arrow[r,"Ex"]&f_{0*}'v'^*g'^*\arrow[d,"\sim"]\\
    g_0^*u^*f_*\arrow[r,"Ex"]&g_0^*f_{0*}v^*\arrow[r,"Ex"]&f_{0*}'g_0'^*v^*
  \end{tikzcd}\]
  of functors. Assume that $u'^*$ is conservative. If $({\rm BC}_{f,u})$, $({\rm BC}_{f',u'})$, and $({\rm BC}_{f_0,g_0})$ are satisfied, then the lower left horizontal, upper right horizontal, and lower right horizontal arrows are isomorphisms. Thus the upper left horizontal arrow is also an isomorphism. Then $({\rm BC}_{f,g})$ is satisfied since $u'^*$ is conservative.

  We will apply this technique to the following two cases.
  \begin{enumerate}[(a)]
    \item When $u$ is an exact log smooth morphism such that $u'^*$ is conservative, then $({\rm BC}_{f,u})$ and $({\rm BC}_{f',u'})$ are satisfied by ($eSm$-BC). Thus $({\rm BC}_{f_0,g_0})$ implies $({\rm BC}_{f,g})$.
    \item When $u$ is the projection $S\times {\rm pt}_\mathbb{N}\rightarrow S$, $u'^*$ is conservative by (\ref{6.2.2}). We also have $({\rm BC}_{f,u})$ and $({\rm BC}_{f',u'})$ by (\ref{6.4.2}). Thus  $({\rm BC}_{f_0,g_0})$ implies  $({\rm BC}_{f,g})$.
  \end{enumerate}
  (III) {\it Final step of the proof.} Let $G$ be a maximal $\theta$-critical face of $Q$, and we denote by $Q'$ the submonoid of $Q$ consisting of elements $q\in Q'$ susch that $nq\in P+G$ for some $n\in \mathbb{N}^+$. Then by \cite[I.4.7.9]{Ogu17}, $Q'^{\rm gp}=Q^{\rm gp}$. Thus by (I), we reduce to the case when $Q=Q'$. In this case, we have
  \[Q_\mathbb{Q}=(P\oplus G)_\mathbb{Q}.\]

  Then choose $n\in \mathbb{N}^+$ such that $nq\in P+G$ for any $q\in Q$, and consider the homomorphism
  \[P\rightarrow P^{\rm gp}\oplus P,\quad a\mapsto (a,na).\]
  We put $P'=P^{\rm gp}\oplus P$, and consider the projection $u:S\times_{\mathbb{A}_P}\mathbb{A}_{P'}\rightarrow S$. Then $u^*$ is conservative as in (\ref{5.3.3}). Thus by the case (a) in (II), we can replace $P\rightarrow Q$ by $P'\rightarrow P'\oplus_P Q$. Thus we reduce to the case when
  \[Q=P\oplus G.\] Since $X=S\times {\rm pt}_G$, we reduce to the case when $P=0$.

  By \cite[11.1.9]{CLS11}, there is a homomorphism $\lambda:\mathbb{N}^r\rightarrow Q$ of fs monoids such that $\lambda^{\rm gp}$ is an isomorphism. Thus by (I), we reduce to the case when $Q=\mathbb{N}^r$. Then we have the factorization
  \[S\times {\rm pt}_{\mathbb{N}^r}\rightarrow \cdots \rightarrow S\times {\rm pt}_\mathbb{N}\rightarrow S\]
  of $f$, so we reduce to the case when $Q=\mathbb{N}$. By the case (b) of (II), we reduce to the case when $\theta$ is the first inclusion $\mathbb{N}\oplus \mathbb{N}\oplus \mathbb{N}$. Composing with the homomorphism
  \[\mathbb{N}\oplus\mathbb{N}\rightarrow \mathbb{N}\oplus\mathbb{N},\quad (a,b)\mapsto (a,a+b)\]
  of fs monoids, by (I), we reduce to the case when $\theta$ is the diagonal homomorphism $\mathbb{N}\rightarrow \mathbb{N}\oplus\mathbb{N}$.

  Then $f$ has the factorization
  \[S\times {\rm pt}_{\mathbb{N}\oplus\mathbb{N}}\stackrel{i}\longrightarrow S\times \mathbb{A}_{\mathbb{N}\oplus\mathbb{N}}\times_{\mathbb{A}_\theta,\mathbb{A}_\mathbb{N}}{\rm pt}_\mathbb{N}\stackrel{p}\longrightarrow S\]
  where $i$ denotes the induced strict closed immersion and $p$ denotes the projection. By (\ref{6.4.1}), $({\rm BC}_{p,g})$ is satisfied. If $g''$ denotes the pullback of $g:S'\rightarrow S$ via $p$, $({\rm BC}_{i,g''})$ is satisfied by (BC--3). These two implies $({\rm BC}_{f,g})$.
\end{proof}
\begin{thm}\label{6.4.4}
  The log motivic triangulated category $\mathscr{T}$ satisfies {\rm (BC--2)}.
\end{thm}
\begin{proof}
  Consider a Cartesian diagram
  \[\begin{tikzcd}
    X'\arrow[d,"f'"]\arrow[r,"g'"]&X\arrow[d,"f"]\\
    S'\arrow[r,"g"]&S
  \end{tikzcd}\]
  of $\mathscr{S}$-schemes where $f$ is a separated exact log smooth morphism. We want to show $({\rm BC}_{f,g})$.\\[4pt]
  (I) {\it Reduction method 1.} Let $\{u_i:S_i\rightarrow S\}_{i\in I}$ be a family of strict morphisms. For $i\in I$, consider the commutative diagram
  \[\begin{tikzcd}
    X_i'\arrow[dd,"f_i'"]\arrow[rr,"g_i'"]\arrow[rd,"v'"]&&X_i\arrow[rd,"v"]\arrow[dd,"f_i",near start]\\
    &X'\arrow[rr,"g'",crossing over,near start]&&X\arrow[dd,"f"]\\
    S_i'\arrow[rr,"g_i",near start]\arrow[rd,"u'"]&&S_i\arrow[rd,"u"]\\
    &S\arrow[uu,"f'",leftarrow,crossing over,near end]\arrow[rr,"g"]&&S
  \end{tikzcd}\]
  of $\mathscr{S}$-schemes where each square is Cartesian. Assume that the family of functors $\{u_i^*\}_{i\in I}$ is conservative. Then we have the commutative diagram
  \[\begin{tikzcd}
    u_i'^*g^*f_!\arrow[d,"\sim"]\arrow[r,"Ex"]&u_i'^*f_!'g'^*\arrow[r,"Ex"]&f_{i!}'v_i'^*g'^*\arrow[d,"\sim"]\\
    g_i^*u_i^*f_!\arrow[r,"Ex"]&g_i^*f_{i!}v_i^*\arrow[r,"Ex"]&f_{i*}'g_i'^*v_i^*
  \end{tikzcd}\]
  of functors. By (BC--3), the lower left horizontal and upper right horizontal arrows are isomorphisms. If $({\rm BC}_{f_i,g_i})$ is satisfied for any $i$, then the lower right horizontal arrow is an isomorphism for any $i$, so the upper left horizontal arrow is an isomorphism for any $i$. This implies $({\rm BC}_{f,g})$ since $\{u_i^*\}_{i\in I}$ is conservative.

  We will apply this in the following two situations.
  \begin{enumerate}[(a)]
    \item When $\{u_i\}_{i\in I}$ is a strict \'etale cover, the family of functors $\{u_i^*\}_{i\in I}$ is conservative by (k\'et-sep).
    \item When $u_0$ is a strict closed immersion and $u_1$ is its complement, the pair of functors $(u_0^*,u_1^*)$ is conservative by (Loc).
  \end{enumerate}
  (II) {\it Reduction method 2.} Let $\{v_i:X_i\rightarrow X\}_{i\in I}$ be a family of separated strict morphisms such that the family of functors $\{v_{i}^!\}_{i\in I}$ is conservative. Consider the commutative diagram
  \[\begin{tikzcd}
    X_i'\arrow[d,"v_i'"]\arrow[r,"g_i'"]&X_i\arrow[d,"v_i"]\\
    X'\arrow[r,"g'"]\arrow[d,"f'"]&X\arrow[d,"f"]\\
    S'\arrow[r,"g"]&S
  \end{tikzcd}\]
  of $\mathscr{S}$-schemes where each square is Cartesian. Then we have the commutative diagram
  \[\begin{tikzcd}
    g^*f_!v_{i!}\arrow[d,"\sim"]\arrow[r,"Ex"]&f_!'g'^*v_{i!}\arrow[r,"Ex"]&f_!'v_{i!}'g_i'^*\arrow[d,"\sim"]\\
    g^*(fv_i)_!\arrow[rr,"Ex"]&&(f'v_i')_!g_i'^*
  \end{tikzcd}\]
  of functors. The upper right horizontal arrow is an isomorphism by (BC--1). If $({\rm BC}_{fv_i,g})$ is satisfied, then the lower horizontal arrow is an isomorphism, so the upper left horizontal arrow is an isomorphism. This implies that the natural transformation
  \[v_i^!g_*'f^!\stackrel{Ex}\longrightarrow v_i^!f^!g_*\]
  is an isomorphism. This implies $({\rm BC}_{f,g})$ since $\{v_{i}^!\}_{i\in I}$ is conservative.

  We will apply this technique in the following two situations.
  \begin{enumerate}[(a)]
    \item When $\{v_i\}_{i\in I}$ is a strict \'etale cover such that each $v_i$ is separated, by (k\'et-sep) and (\ref{2.5.7}), the family of functors $\{v_i^!\}_{i\in I}$ is conservative.
    \item When $v_0:S_0\rightarrow S$ is a strict closed immersion and $v_1$ is its complement, by (Loc), the pair of functors $(v_0^!,v_1^!)$ is conservative.
  \end{enumerate}
  (III) {\it Final step of the proof.} By the case (a) of (I) and the case (a) of (II), we reduce to the case when $f$ has an fs chart $\theta:P\rightarrow Q$ of exact log smooth type. Then by the case (b) of (I) and the proof of \cite[3.5(ii)]{Ols03}, we may assume that $S$ has a constant log structure.

  Consider the commutative diagram
  \[\begin{tikzcd}
    Y\arrow[d,"h"]\arrow[r,"g''"]&{\rm pt}_{Q}\arrow[d]\\
    X\arrow[d,"f"]\arrow[r,"g'"]&\mathbb{A}_Q\arrow[d]\\
    S\arrow[r,"g"]&\mathbb{A}_P
  \end{tikzcd}\]
  of $\mathscr{S}$-schemes where each square is Cartesian. By the case (b) of (II) and the proof of \cite[3.5(ii)]{Ols03}, we reduce to showing $({\rm BC}_{fh,g})$.

  Now, consider the the commutative diagram
  \[\begin{tikzcd}
    Y'\arrow[d,"q'"]\arrow[r,"g''"]&Y\arrow[d,"q"]\\
    T'\arrow[d,"p'"]\arrow[r,"g'''"]&\underline{Y}\times_{\underline{S}}S\arrow[d,"p"]\\
    S'\arrow[r,"g"]&S
  \end{tikzcd}\]
  of $\mathscr{S}$-schemes where
  \begin{enumerate}[(i)]
    \item each square is Cartesian,
    \item $q$ denotes the induced morphism and $p$ denotes the projection.
  \end{enumerate}
  Then since $p$ is strict, $({\rm BC}_{p,g})$ is satisfied by (BC--1). By (\ref{6.4.3}), $({\rm BC}_{q,g'''})$ is also satisfied. These two implies $({\rm BC}_{f,g})$.
\end{proof}
\section{Log motivic derivators}
\subsection{Axioms of premotivic triangulated prederivators}
\begin{none}\label{10.1.1}
  Through this subsection, fix a class of morphisms $\mathscr{P}$ of $\mathscr{S}$ containing all isomorphisms and stable by compositions and pullbacks.
\end{none}
\begin{df}\label{10.1.2}
  We will introduce several notations and terminology.
  \begin{enumerate}[(1)]
    \item A small category is called a {\it diagram.} We denote by ${\rm dia}$ the 2-category of small categories.
    \item An {\it $\mathscr{S}$-diagram} is a functor
    \[\mathscr{X}:I\rightarrow \mathscr{S}\]
    where $I$ is a small category. The 2-category of $\mathscr{S}$-diagrams is denoted by $\mathscr{S}^{\rm dia}$. We often write $\mathscr{X}=(\mathscr{X},I)$ for $\mathscr{X}$. The category $I$ is called the {\it index category} of $\mathscr{X}$, and an object $\lambda$ of $I$ is called an {\it index} of $\mathscr{X}$.
    \item For any $\mathscr{S}$-scheme $S$ and any diagram $I$, we denote by $(S,I)$ the functor $I\rightarrow \mathscr{S}$ mapping to $S$ constantly.
    \item Let $f:I\rightarrow J$ be a functor of small categories, and let $\mu$ be an object of $J$. We denote by $I_\mu$ the full subcategory of $I$ such that $\lambda$ is an object of $I_\mu$ if and only if $u(\lambda)$ is isomorphic to $\mu$ in $J$.

        We denote by $I/\mu$ the category where
        \begin{enumerate}[(i)]
          \item object is a pair $(\lambda\in {\rm ob}(I),\;a:f(\lambda)\rightarrow \mu)$,
          \item morphism $(\lambda,a)\rightarrow (\lambda',a')$ is the data of commutative diagrams:
          \[\begin{tikzcd}
            \lambda\arrow[r,"b"]&\lambda'
          \end{tikzcd}
          \quad
          \begin{tikzcd}
            f(\lambda)\arrow[rd,"a"']\arrow[rr,"b"]&&f(\lambda')\arrow[ld,"a'"]\\
            &\mu
          \end{tikzcd}\]
        \end{enumerate}
    \item Let $f:(\mathscr{X},I)\rightarrow (\mathscr{Y},J)$ be a 1-morphism of $\mathscr{S}$-diagrams. Abusing the notation, we denote by $f$ the induced functor $I\rightarrow J$. For an object $\mu$ of $J$, we denote by $\mathscr{X}_\mu$ and $\mathscr{X}/\mu$ the $\mathscr{S}$-diagrams
        \[I_\mu\longrightarrow I\stackrel{\mathscr{X}}\longrightarrow \mathscr{S},\quad I/\mu\longrightarrow I\stackrel{\mathscr{X}}\longrightarrow \mathscr{S}\]
        respectively where the first arrows are the induced functors. Then we denote be
        \[\mu:\mathscr{X}_\mu\rightarrow \mathscr{X},\quad \overline{\mu}:\mathscr{X}/\mu\rightarrow \mathscr{X}\]
        the induced functors.
    \item Let $f:\mathscr{X}\rightarrow \mathscr{Y}$ be a $1$-morphism of $\mathscr{S}$-diagrams, and let $\mu$ be an index of $\mathscr{Y}$. Consider the induced 2-diagrams
    \[\begin{tikzcd}
      \mathscr{X}_\mu\arrow[d]\arrow[r,"\mu"]\arrow[rd,phantom,"{\rotatebox[origin=c]{45}{$\Longleftrightarrow$}}"]&\mathscr{X}\arrow[d,"f"]\\
      \mathscr{Y}_\mu\arrow[r,"\mu"']&\mathscr{Y}
    \end{tikzcd}
    \begin{tikzcd}
      \mathscr{X}/\mu\arrow[d]\arrow[r,"\overline{\mu}"]\arrow[rd,phantom,"{\rotatebox[origin=c]{45}{$\Longleftarrow$}}"]&\mathscr{X}\arrow[d,"f"]\\
      \mathscr{Y}_\mu\arrow[r,"\mu"']&\mathscr{Y}
    \end{tikzcd}
    \begin{tikzcd}
      \mathscr{X}/\mu\arrow[d]\arrow[r,"\overline{\mu}"]\arrow[rd,phantom,"{\rotatebox[origin=c]{45}{$\Longleftrightarrow$}}"]&\mathscr{X}\arrow[d,"f"]\\
      \mathscr{Y}/\mu\arrow[r,"\overline{\mu}"']&\mathscr{Y}
    \end{tikzcd}\]
    of $\mathscr{S}$-diagrams. Here, the arrows $\Leftrightarrow$ and $\Rightarrow$ express the induced $2$-morphisms. Then we denote by
    \[f_\lambda:\mathscr{X}_\mu\rightarrow \mathscr{Y}_\mu,\quad f_{\overline{\mu}\mu}:\mathscr{X}/\mu\rightarrow \mathscr{Y}_\mu,\quad f_{\overline{\mu}}:\mathscr{X}/\mu\rightarrow \mathscr{Y}/\mu\]
    the 1-morphisms in the above $2$-diagrams.

    Let $\lambda$ be an index of $\mathscr{X}$ such that $f(\lambda)$ is isomorphic to $\mu$. Then we denote by
    \[f_{\lambda\mu}:\mathscr{X}_\lambda\rightarrow \mathscr{Y}_\mu\]
    the induced $1$-morphism.
    \item Let $f:\mathscr{X}\rightarrow \mathscr{Y}$ be a $1$-morphism of $\mathscr{S}$-diagrams. For a property ${\bf P}$ of morphisms in $\mathscr{S}$, we say that $f$ is a $\bf P$ {\it morphism} if for any index $\lambda$ of $\mathscr{X}$, the morphism $f_{\lambda\mu}:\mathscr{X}_\lambda\rightarrow \mathscr{Y}_\mu$ where $\mu=f(\lambda)$ is a $\bf P$ morphism in $\mathscr{S}$.
    \item We denote by ${\rm dia}$ the 2-category of small categories.
    \item We denote by $Tri^{\otimes}$ the $2$-category of triangulated symmetric monoidal categories.
    \item We denote by ${\bf e}$ the trivial category.
  \end{enumerate}
\end{df}
\begin{df}\label{10.1.3}
  A $\mathscr{P}${\it -premotivic triangulated prederivator} $\mathscr{T}$ over $\mathscr{S}$ is a 1-contravariant and $2$-contravariant $2$-functor
  \[\mathscr{T}:\mathscr{S}^{\rm dia}\longrightarrow {\rm Tri}^{\otimes}\]
  with the following properties.
  \begin{enumerate}
    \item[(PD--1)] For any $1$-morphism $f:\mathscr{X}\rightarrow \mathscr{Y}$ of $\mathscr{S}$-diagrmas, we denote by $f^*:\mathscr{T}(\mathscr{Y})\rightarrow \mathscr{T}(\mathscr{X})$ the image of $f$ under $\mathscr{T}:\mathscr{S}^{\rm dia}\longrightarrow {\rm Tri}^{\otimes}$. Then the functor $f^*$ admits a right adjoint denoted by $f^*$.\\
        For any $2$-morphism $t:f\rightarrow g$ of $1$-morphisms $f,g:\mathscr{X}\rightarrow \mathscr{Y}$ of $\mathscr{S}$-diagrams, we denote by $t^*:g^*\rightarrow f^*$ the image of $t$ under $\mathscr{T}$.
    \item[(PD--2)] For any $\mathscr{P}$-morphism $f:\mathscr{X}\rightarrow \mathscr{Y}$, the functor $f^*$ admits a left adjoint denoted by $f_\sharp$.
    \item[(PD--3)] For any $\mathscr{S}$-diagram $\mathscr{X}=(\mathscr{X},I)$, if $I$ is a discrete category, then the induced functor
    \[\mathscr{T}(\mathscr{X})\longrightarrow \prod_{\lambda\in {\rm ob}(I)}\mathscr{T}(\mathscr{X}_\lambda)\]
    is an equivalence of categories.
    \item[(PD--4)] For any $\mathscr{S}$-diagram $\mathscr{X}=(\mathscr{X},I)$, the family of functors $\lambda^*$ for $\lambda\in {\rm ob}(I)$ is conservative.
    \item[(PD--5)]  For any object $S$ of $\mathscr{S}$, the fibered category
    \[\mathscr{T}(-,{\bf e})\]
    is a $\mathscr{P}$-premotivic triangulated category.
    \item[(PD--6)] For any morphism $f:\mathscr{X}\rightarrow \mathscr{Y}$ of $\mathscr{S}$-diagrams and any index $\mu$ of $\mathscr{Y}$, in the $2$-diagram
    \[\begin{tikzcd}
      \mathscr{X}/\mu\arrow[d,"f_{\overline{\mu}\mu}"']\arrow[r,"\overline{\mu}"]\arrow[rd,phantom,"\;\;\,\;_t"',"{\rotatebox[origin=c]{45}{$\Longleftarrow$}}"]&\mathscr{X}\arrow[d,"f"]\\
      \mathscr{Y}_\mu\arrow[r,"\mu"']&\mathscr{Y}
    \end{tikzcd}\]
    of $\mathscr{S}$-diagrmas, the exchange transformation
    \[\mu^*f_*\stackrel{ad}\longrightarrow f_{\overline{\mu}\mu*}f_{\overline{\mu}\mu}^*\mu^*f_*\stackrel{t^*}\longrightarrow
     f_{\overline{\mu}\mu*}\overline{\mu}^*f^*f_*\stackrel{ad'}\longrightarrow    f_{\overline{\mu}\mu*}\overline{\mu}^*\]
    is an isomorphism.
  \end{enumerate}
\end{df}
\begin{rmk}
  Our axioms are selected from \cite[2.4.16]{Ayo07} and the axioms of {\it algebraic derivators} in \cite[2.4.12]{Ayo07}.
\end{rmk}
\begin{df}\label{10.1.4}
  Let $\mathscr{T}$ be a $\mathscr{P}$-premotivic triangulated prederivator.
  \begin{enumerate}[(1)]
    \item A {\it cartesian section} of $\mathscr{T}$ is the data of an object $A_\mathscr{X}$ of $\mathscr{T}(\mathscr{X})$ for each $\mathscr{S}$-diagram $\mathscr{X}$ and of isomorphisms
    \[f^*(A_{\mathscr{Y}})\stackrel{\sim}\longrightarrow A_{\mathscr{X}}\]
     for each morphism $f:\mathscr{X}\rightarrow \mathscr{Y}$ of $\mathscr{S}$-diagrams, subject to following coherence conditions as in (\ref{1.1.5}). The tensor product of two cartesian sections is defined termwise.
    \item A set of {\it twists} for $\mathscr{T}$ is a set of Cartesian sections of $\mathscr{T}$ stable by tensor product. For short, we say also that $\mathscr{T}$ is $\tau$-twisted.
  \end{enumerate}
\end{df}
\begin{prop}\label{10.1.5}
  Let $\mathscr{T}$ be a $\mathscr{P}$-premotivic triangulated prederivator, and let $\mathscr{X}$ be an $\mathscr{S}$-diagram. Then $\lambda_\sharp K$ is compact for any index $\lambda$ of $\mathscr{X}$ and any compact object $K$ of $\mathscr{T}(\mathscr{X}_\lambda)$.
\end{prop}
\begin{proof}
  We have the isomorphism
  \[{\rm Hom}_{\mathscr{T}(\mathscr{Z})}(\lambda_\sharp K,-)\cong {\rm Hom}_{\mathscr{T}(\mathscr{X}_\lambda)}(K,\lambda^*(-)).\]
  Using this, ther conclusion follows fom the fact that $\lambda^*$ preserves small sums.
\end{proof}
\begin{df}\label{10.1.6}
  Let $\mathscr{T}$ be a $\tau$-twisted $\mathscr{P}$-premotivic triangulated prederivator. For any object $S$ of $\mathscr{S}$, we denote by $\mathcal{F}_{\mathscr{P}/S}$ the family of motives of the form
  \[M_S(X)\{i\}\]
  for $\mathscr{P}$-morphism $X\rightarrow S$ and twist $i\in \tau$. Then for any $\mathscr{S}$-diagram $\mathscr{X}$, we denote by $\mathcal{F}_{\mathscr{P}/\mathscr{X}}$ the family of motives of the form
  \[\lambda_\sharp K\]
  for index $\lambda$ of $\mathscr{X}$ and object $K$ of $\mathcal{F}_{\mathscr{P}/\mathscr{X}_\lambda}$.
\end{df}
\begin{prop}\label{10.1.7}
  Let $\mathscr{T}$ be a $\tau$-twisted $\mathscr{P}$-premotivic triangulated prederivator. Assume that $\mathscr{T}(-,{\bf e})$ is compactly generated by $\mathscr{P}$ and $\tau$. Then $\mathcal{F}_{\mathscr{P}/\mathscr{X}}$ generates $\mathscr{T}(\mathscr{X})$.
\end{prop}
\begin{proof}
  Let $K\rightarrow K'$ be a morphism in $\mathscr{T}(\mathscr{X})$ such that the homomorphism
  \[{\rm Hom}_{\mathscr{T}(\mathscr{X})}(\lambda_\sharp L,K)\rightarrow {\rm Hom}_{\mathscr{T}(\mathscr{X})}(\lambda_\sharp L,K')\]
  is an isomorphism for any index $\lambda$ of $\mathscr{X}$ and any element $L$ of $\mathcal{F}_{\mathscr{P}/\mathscr{X}_\lambda}$. We want to show that the morphism $K\rightarrow K'$ in $\mathscr{T}(\mathscr{X})$ is an isomorphism.

  The homomorphism
  \[{\rm Hom}_{\mathscr{T}(\mathscr{X}_\lambda)}(L,\lambda^*K)\rightarrow {\rm Hom}_{\mathscr{T}(\mathscr{X}_\lambda)}(L,\lambda^*K')\]
  is an isomorphism, and $\mathcal{F}_{\mathscr{P}/\mathscr{X}_\lambda}$ generates $\mathscr{T}(\mathscr{X}_\lambda)$ by assumption. Thus the morphism $\lambda^*K\rightarrow \lambda^*K'$ in $\mathscr{T}(\mathscr{X}_\lambda)$ is an isomorphism. Then (PD--4) implies that the morphism $K\rightarrow K'$ in $\mathscr{T}(\mathscr{X})$ is an isomorphism, which completes the proof.
\end{proof}
\begin{df}\label{10.1.8}
  We say that a $\mathscr{P}$-premotivic triangulated prederivator over $\mathscr{S}$ is a {\it log motivic derivator} over $\mathscr{S}$ if it satisfies the following conditions.
  \begin{enumerate}
    \item[(i)] The restriction of $\mathscr{T}$ on $\mathscr{S}$ is a log motivic triangulated category over $\mathscr{S}$.
    \item[(ii)] For any $\mathscr{S}$-scheme $S$, the functor $\mathscr{T}(S,-):{\rm dia}\rightarrow {\rm Tri}^{\otimes}$ is a triangulated derivator in the sense of \cite[2.1.34]{Ayo07}.
  \end{enumerate}
\end{df}
\begin{rmk}\label{10.1.9}
  Note that the axiom (ii) is never used in this paper. In particular, (\ref{9.7.2}) is proved without (ii). Nevertheless, we include (ii) in the axioms since any $\mathscr{T}$ came from stable model categories satisfies (ii). Note also that (ii) is one of the axioms of algebraic derivators in \cite[2.4.12]{Ayo07}.
\end{rmk}
\begin{df}[(CD12, 3.2.7)]
  Let $t$ be a topology on $\mathscr{S}$. We say that a $\mathscr{P}$-premotivic triangulated prederivator $\mathscr{T}$ satisfies $t$-{\it descent} if for any $t$-hypercover $f:(\mathscr{X},I)\rightarrow S$, the functor $f^*$ is fully faithful.
\end{df}
\subsection{Consequences of axioms}
\begin{none}\label{10.2.0}
  Throughout this subsection, fix a class of morphisms $\mathscr{P}$ of $\mathscr{S}$ containing all isomorphisms and stable by compositions and pullbacks.
\end{none}
\begin{df}\label{10.2.1}
  Let $f:(\mathscr{X},I)\rightarrow (\mathscr{Y},J)$ be a $1$-morphism of $\mathscr{S}$-diagrams.
  \begin{enumerate}[(1)]
    \item We say that $f$ is {\it reduced} if the functor $f:I\rightarrow J$ is an equivalence.
    \item We say that $f$ is {\it Cartesian} if $f$ is reduced and for any morphism $\mu\rightarrow \mu'$ in $J$, the diagram
    \[\begin{tikzcd}
      \mathscr{X}_\mu\arrow[d]\arrow[r]&\mathscr{X}_{\mu'}\arrow[d]\\
      \mathscr{Y}_\mu\arrow[r]&\mathscr{Y}_{\mu'}
    \end{tikzcd}\]
    in $\mathscr{S}$ is Cartesian.
  \end{enumerate}
\end{df}
\begin{none}\label{10.2.2}
  Let $f:(\mathscr{X},I)\rightarrow (\mathscr{Y},J)$ and $g:(\mathscr{Y}',J')\rightarrow (\mathscr{Y},J)$ be 1-morphisms of $\mathscr{S}$-diagrams. Consider the category $J'\times_J I$. We have the functors
  \[u_1:J'\times_J I\stackrel{p_1}\rightarrow J'\stackrel{\mathscr{Y}'}\rightarrow \mathscr{S},\]
  \[u_2:J'\times_J I\stackrel{p_2}\rightarrow J'\stackrel{\mathscr{X}}\rightarrow \mathscr{S},\]
  \[u:J'\times_J I\stackrel{p}\rightarrow J'\stackrel{\mathscr{Y}}\rightarrow \mathscr{S}\]
  where $p_1$, $p_2$, and $p$ denote the projections. Then we denote by
  \[(\mathscr{Y}'\times_{\mathscr{Y}}\mathscr{X},J'\times_J I)\]
  the functor $J'\times_J I\rightarrow \mathscr{S}$ obtained by taking fiber products $u_1(\lambda)\times_{u(\lambda)}u_2(\lambda)$ for $\lambda\in {\rm ob}(J'\times_J I)$. Note that by \cite[2.4.10]{Ayo07}, the commutative diagram
  \[\begin{tikzcd}
    \mathscr{Y}'\times_\mathscr{Y}\mathscr{X}\arrow[r,"g'"]\arrow[d,"f'"]&\mathscr{X}\arrow[d,"f"]\\
    \mathscr{Y}'\arrow[r,"g"]&\mathscr{Y}
  \end{tikzcd}\]
  of $\mathscr{S}$-diagrams is Cartesian where $g'$ and $f'$ denote the first and second projections respectively.
\end{none}
\begin{prop}\label{10.2.3}
  Let $f:\mathscr{X}\rightarrow \mathscr{Y}$ be a Cartesian $\mathscr{P}$-morphism of $\mathscr{S}$-diagrams, and let $\mu$ be an index of $\mathscr{Y}$. Then in the Cartesian diagram
  \[\begin{tikzcd}
    \mathscr{X}_\mu\arrow[d,"f_\mu"]\arrow[r,"\mu"]&\mathscr{X}\arrow[d,"f"]\\
    \mathscr{Y}_\mu\arrow[r,"\mu"]&\mathscr{Y}
  \end{tikzcd}\]
  of $\mathscr{S}$-diagrams, the exchange transformation
  \[f^*\mu_*\stackrel{Ex}\longrightarrow \mu_*f_\mu^*\]
  is an isomorphism.
\end{prop}
\begin{proof}
  Let $\lambda$ be an index of $\mathscr{X}$ (so an index of $\mathscr{Y}$ since $f$ is Cartesian). By (PD--4), it suffices to show that the natural transformation
  \[\lambda^*f^*\mu_*\stackrel{Ex}\longrightarrow \lambda^*\mu_*f_\mu^*\]
  is an isomorphism.

  Consider the 2-diagrams
  \[\begin{tikzcd}
    \mathscr{X}_\mu/\lambda\arrow[dd,"(f_\mu)_{\overline{\lambda}}"'] \arrow[rrrd,phantom,"{\rotatebox[origin=c]{30}{$\Longleftarrow$}}"]\arrow[rd,"\overline{\lambda}"] \arrow[rr,"\mu_{\overline{\lambda}\lambda}"] &&\mathscr{X}_\lambda\arrow[rd,"\lambda"]\\
    &\mathscr{X}_\mu\arrow[dd,"f_\mu"]\arrow[rr,"\mu"]\arrow[ld,phantom,"{\rotatebox[origin=c]{45}{$\Longleftrightarrow$}}"] &&\arrow[lldd,phantom,"{\rotatebox[origin=c]{45}{$\Longleftrightarrow$}}"]\mathscr{X}\arrow[dd,"f"]\\
    \mathscr{Y}_\mu/\lambda\arrow[rd,"\overline{\lambda}"']\\
    &\mathscr{Y}_\mu\arrow[rr,"\mu"']&&\mathscr{Y}
  \end{tikzcd}\quad \begin{tikzcd}
    \mathscr{X}_\mu/\lambda\arrow[dd,"(f_\mu)_{\overline{\lambda}}"'] \arrow[rrdd,phantom,"{\rotatebox[origin=c]{45}{$\Longleftrightarrow$}}"] \arrow[rr,"\mu_{\overline{\lambda}\lambda}"] &&\mathscr{X}_\lambda\arrow[rd,"\lambda"]\arrow[dd,"f_\lambda"]\\
    &&&\arrow[ld,phantom,"{\rotatebox[origin=c]{45}{$\Longleftrightarrow$}}"]\mathscr{X}\arrow[dd,"f"]\\
    \mathscr{Y}_\mu/\lambda\arrow[rd,"\overline{\lambda}"']\arrow[rr,"\mu_{\overline{\lambda}\lambda}"]&&\mathscr{Y}_\lambda\arrow[rd,"\lambda"] \arrow[ld,phantom,"{\rotatebox[origin=c]{30}{$\Longleftarrow$}}"]\\
    &\mathscr{Y}_\mu\arrow[rr,"\mu"']&&\mathscr{Y}
  \end{tikzcd}\]
  of $\mathscr{S}$-diagrams. Then we have the commutative diagram
  \[\begin{tikzcd}
    \lambda^*f^*\mu_*\arrow[d,"Ex"]\arrow[r,"\sim"]&f_\lambda^*\lambda^*\mu_*\arrow[r,"Ex"]&f_\lambda^*\mu_{\overline{\lambda}\lambda^*}\overline{\lambda}^*\arrow[d,"Ex"]\\
    \lambda^*\mu_*f_\mu^*\arrow[r,"Ex"]&\mu_{\overline{\lambda}\lambda*}\overline{\lambda}^*f_\mu^*\arrow[r,"\sim"]&\mu_{\overline{\lambda}\lambda*}(f_\mu)_{\overline{\lambda}}^* \overline{\lambda}^*
  \end{tikzcd}\]
  of functors. By (PD--6), the lower left horizontal and upper right horizontal arrows are isomorphisms. Thus it suffices to show that the right vertical arrow is an isomorphism. We have the identification
  \[\mathscr{X}_\mu/\lambda=\mathscr{X}_\mu\times {\rm Hom}_J(\mu,\lambda),\quad \mathscr{Y}_\mu/\lambda=\mathscr{Y}_\mu\times {\rm Hom}_J(\mu,\lambda)\]
  where $J$ denotes the index category of $\mathscr{Y}$. Thus by (PD--3), it suffices to show that for any morphism $\mu\rightarrow \lambda$ in $J$, in the induced Cartesian diagram
  \[\begin{tikzcd}
    \mathscr{X}_\mu\arrow[d,"f_\mu"]\arrow[r,"{\rm id}_{\mu\lambda}"]&\mathscr{X}_\lambda\arrow[d,"f_\lambda"]\\
    \mathscr{Y}_\mu\arrow[r,"{\rm id}_{\mu\lambda}"]&\mathscr{Y}_\lambda
  \end{tikzcd}\]
  in $\mathscr{S}$, the exchange transformation
  \[f_\lambda^*{\rm id}_{\mu\lambda*}\stackrel{Ex}\longrightarrow {\rm id}_{\mu\lambda*}f_\mu^*\]
  is an isomorphism. This follows from (PD--5) and the assumption that $f$ is a Cartesian $\mathscr{P}$-morphism.
\end{proof}
\begin{prop}\label{10.2.4}
  Consider a Cartesian diagram
  \[\begin{tikzcd}
    \mathscr{X}'\arrow[d,"f'"]\arrow[r,"g'"]&\mathscr{X}\arrow[d,"f"]\\
    \mathscr{Y}'\arrow[r,"g"]&\mathscr{Y}
  \end{tikzcd}\]
  of $\mathscr{S}$-diagrams where $f$ is a Cartesian $\mathscr{P}$-morphism. Then the exchange transformation
  \[f_\sharp'g'^*\stackrel{Ex}\longrightarrow g^*f_\sharp\]
  is an isomorphism.
\end{prop}
\begin{proof}
  Note that $f'$ is also a Cartesian $\mathscr{P}$-morphism. Let $\mu'$ be an index of $\mathscr{Y}'$. By (PD--4), it suffices to show that the natural transformation
  \[\mu'^*f_\sharp'g'^*\stackrel{Ex}\longrightarrow \mu'^*g^*f_\sharp\]
  is an isomorphism. We put $\mu=g(\mu')$.

  Consider the commutative diagrams
  \[\begin{tikzcd}
    \mathscr{X}_{\mu'}'\arrow[d,"f_{\mu'}"]\arrow[r,"\mu'"]&\mathscr{X}'\arrow[d,"f'"]\arrow[r,"g'"]&\mathscr{X}\arrow[d,"f"]\\
    \mathscr{Y}_{\mu'}'\arrow[r,"\mu'"]&\mathscr{Y}'\arrow[r,"g"]&\mathscr{Y}
  \end{tikzcd}\quad \begin{tikzcd}
    \mathscr{X}_{\mu'}'\arrow[d,"f_{\mu'}"]\arrow[r,"g_{\mu'\mu}'"]&\mathscr{X}_\mu\arrow[d,"f_\mu"]\arrow[r,"g'"]&\mathscr{X}\arrow[d,"f"]\\
    \mathscr{Y}_{\mu'}'\arrow[r,"g_{\mu'\mu}"]&\mathscr{Y}_\mu\arrow[r,"g"]&\mathscr{Y}
  \end{tikzcd}\]
  of $\mathscr{S}$-diagrams. Then we have the commutative diagram
  \[\begin{tikzcd}
    f_{\mu'\sharp}'\mu'^*g'^*\arrow[d,"\sim"]\arrow[r,"Ex"]&\mu'^*f_\sharp'g'^*\arrow[r,"Ex"]&\mu'^*g^*f_\sharp\arrow[d,"\sim"]\\
    f_{\mu'\sharp}'g_{\mu'\mu}^*\mu^*\arrow[r,"Ex"]&g_{\mu'\mu}^*f_{\mu\sharp}\mu^*\arrow[r,"Ex"]&g_{\mu'\mu}^*\mu^*f_\sharp
  \end{tikzcd}\]
  of functors. The upper left horizontal and lower right horizontal arrows are isomorphisms by (\ref{10.2.3}), and the lower left horizontal arrow is an isomorphism by (PD--5) since the commutative diagram
  \[\begin{tikzcd}
    \mathscr{X}_{\mu'}'\arrow[d,"f_{\mu'}'"]\arrow[r,"g_{\mu'\mu}'"]&\mathscr{X}_\mu\arrow[d,"f_\mu"]\\
    \mathscr{Y}_{\mu'}'\arrow[r,"g_{\mu'\mu}"]&\mathscr{Y}_\mu
  \end{tikzcd}\]
  is Cartesian by assumption. Thus the upper right horizontal arrow is also an isomorphism.
\end{proof}
\begin{prop}\label{10.2.5}
  Let $\mathscr{X}$ be an $\mathscr{S}$-diagram. Assume that the index category of $\mathscr{X}$ has a terminal object $\lambda$. Consider the $1$-morphisms
  \[\mathscr{X}_\lambda\stackrel{\lambda}\longrightarrow \mathscr{X}\stackrel{f}\longrightarrow \mathscr{X}_\lambda\]
  where $f$ denotes the morphism induced by the functor $I\rightarrow {\bf e}$ to the terminal object $\lambda$. Then the natural transformation
  \[\lambda_\sharp \lambda^*f^*\stackrel{ad}\longrightarrow f^*\]
  is an isomorphism.
\end{prop}
\begin{proof}
  Let $\lambda'$ be an index of $\mathscr{X}$. By (PD--4), it suffices to show that the natural transformation
  \[\lambda'^*\lambda_\sharp\lambda^*f^*\stackrel{ad'}\longrightarrow \lambda'^*f^*\]
  is an isomorphism. We will show that its right adjoint
  \[f_*\lambda_*'\stackrel{ad}\longrightarrow f_*\lambda_*\lambda^*\lambda_*'\]
  is an isomorphism.

  Consider the diagram
  \[\begin{tikzcd}
    \mathscr{X}_{\lambda'}\arrow[d,"{\rm id}_{\lambda'\lambda}"]\arrow[r,"{\rm id}"]&\mathscr{X}_{\lambda'}\arrow[d,"\lambda'"]\arrow[rd,"{\rm id}_{\lambda'\lambda}"]\\
    \mathscr{X}_\lambda\arrow[r,"\lambda"]&\mathscr{X}\arrow[r,"f"]&\mathscr{X}_\lambda
  \end{tikzcd}\]
  of $\mathscr{S}$-diagrams. Then we have the commutative diagram
  \[\begin{tikzcd}
    f_*\lambda_*'\arrow[d,"ad"]\arrow[rd,"\sim"]\\
    f_*\lambda_*\lambda^*\lambda_*'\arrow[r,"Ex"]&f_*\lambda_*{\rm id}_{\lambda'\lambda*}{\rm id}^*
  \end{tikzcd}\]
  of functors, so it suffices to show that the horizontal arrow is an isomorphism. This follows from (PD--6) since $\mathscr{X}_{\lambda'}=\mathscr{X}_{\lambda'}/\lambda$.
\end{proof}
\begin{prop}\label{10.2.6}
  Let $f:\mathscr{X}\rightarrow \mathscr{Y}$ be a reduced morphism of $\mathscr{S}$-diagrams, and let $\mu$ be an index of $\mathscr{Y}$. Consider the Cartesian diagram
  \[\begin{tikzcd}
    \mathscr{X}_\mu\arrow[d,"f_\mu"]\arrow[r,"\mu"]&\mathscr{X}\arrow[d,"f"]\\
    \mathscr{Y}_\mu\arrow[r,"\mu"]&\mathscr{Y}
  \end{tikzcd}\]
  of $\mathscr{S}$-diagrams. Then the exchange transformation
  \[\mu^*f_*\stackrel{Ex}\longrightarrow f_{\mu*}\mu^*\]
  is an isomorphism.
\end{prop}
\begin{proof}
  Consider the 2-diagram
  \[\begin{tikzcd}
    \mathscr{X}_\mu\arrow[rdd,"f_\mu"',bend right]\arrow[rrd,"\mu",bend left]\arrow[rd,"{\rm id}_{\mu\overline{\mu}}"]
    \arrow[rdd,phantom,"\Longleftrightarrow"]\arrow[rrd,phantom,"{\rotatebox[origin=c]{90}{$\Longleftrightarrow$}}"]    \\
    &\mathscr{X}/\mu\arrow[r,"\overline{\mu}"]\arrow[rd,phantom,"{\rotatebox[origin=c]{45}{$\Longleftarrow$}}"]\arrow[d,"f_{\overline{\mu}\mu}"]&\mathscr{X}\arrow[d,"f"]\\
    &\mathscr{Y}_\mu\arrow[r,"\mu"]&\mathscr{Y}
  \end{tikzcd}\]
  of $\mathscr{S}$-diagrams. Then the exchange transformation
  \[\mu^*f_*\stackrel{Ex}\longrightarrow f_{\mu*}\mu^*\]
  has the decomposition
  \[\mu^*f_*\stackrel{Ex}\longrightarrow f_{\mu\mu*}\overline{\mu}^*\stackrel{ad}\longrightarrow f_{\overline{\mu}\mu*}{\rm id}_{\mu\overline{\mu}*} {\rm id}_{\mu\overline{\mu}}^*\overline{\mu}^*\stackrel{\sim}\longrightarrow f_{\mu*}\mu^*.\]
  By (PD--6), the first arrow is an isomorphism. Thus it suffices to show that the second arrow is an isomorphism.

  Consider the 1-morphisms
  \[\mathscr{X}_\mu\stackrel{{\rm id}_{\mu\overline{\mu}}}\rightarrow \mathscr{X}/\mu\stackrel{{\rm id}_{\overline{\mu}\mu}}\rightarrow \mathscr{X}_\mu\stackrel{f_\mu}\rightarrow \mathscr{Y}_\mu\]
  of $\mathscr{S}$-diagrams. Then it suffices to show that the natural transformation
  \[{\rm id}_{\overline{\mu}\mu*}{\rm id}_{\mu\overline{\mu}*}{\rm id}_{\mu\overline{\mu}}^*\stackrel{ad}\longrightarrow {\rm id}_{\overline{\mu}\mu*}\]
  is an isomorphism, which follows from (\ref{10.2.5}).
\end{proof}
\begin{prop}\label{10.2.7}
  Consider a Cartesian diagram
  \[\begin{tikzcd}
    \mathscr{X}'\arrow[d,"f'"]\arrow[r,"g'"]&\mathscr{X}\arrow[d,"f"]\\
    \mathscr{Y}'\arrow[r,"g"]&\mathscr{Y}
  \end{tikzcd}\]
  of $\mathscr{S}$-diagrams where
  \begin{enumerate}[(i)]
    \item $f$ is reduced,
    \item for any index $\mu'$ of $\mathscr{Y}'$, in the Cartesian diagram
    \[\begin{tikzcd}
      \mathscr{X}_{\mu'}'\arrow[d,"f_{\mu'}'"]\arrow[r,"g_{\mu'\mu}'"]&\mathscr{X}_\mu\arrow[d,"f_\mu"]\\
      \mathscr{Y}_{\mu'}'\arrow[r,"g_{\mu'\mu}"]&\mathscr{Y}_\mu
    \end{tikzcd}\]
    in $\mathscr{S}$ where $\mu=g(\mu')$, the exchange transformation
    \[g_{\mu'\mu}^*f_{\mu*}\stackrel{Ex}\longrightarrow f_{\mu'*}'g_{\mu'\mu}'^*\]
    is an isomorphism.
  \end{enumerate}
  Then the exchange transformation
  \[g^*f_*\stackrel{Ex}\longrightarrow f_*'g'^*\]
  is an isomorphism.
\end{prop}
\begin{proof}
  Note that $f'$ is also reduced. Let $\mu'$ be an index of $\mathscr{Y}'$. By (PD--4), it suffices to show that the natural transformation
  \[\mu'^*g^*f_*\stackrel{Ex}\longrightarrow \mu'^*f_*'g'^*\]
  is an isomorphism. We put $\mu=g(\mu')$.

  Consider the commutative diagrams
  \[\begin{tikzcd}
    \mathscr{X}_{\mu'}'\arrow[d,"f_{\mu'}"]\arrow[r,"\mu'"]&\mathscr{X}'\arrow[d,"f'"]\arrow[r,"g'"]&\mathscr{X}\arrow[d,"f"]\\
    \mathscr{Y}_{\mu'}'\arrow[r,"\mu'"]&\mathscr{Y}'\arrow[r,"g"]&\mathscr{Y}
  \end{tikzcd}\quad \begin{tikzcd}
    \mathscr{X}_{\mu'}'\arrow[d,"f_{\mu'}"]\arrow[r,"g_{\mu'\mu}'"]&\mathscr{X}_\mu\arrow[d,"f_\mu"]\arrow[r,"g'"]&\mathscr{X}\arrow[d,"f"]\\
    \mathscr{Y}_{\mu'}'\arrow[r,"g_{\mu'\mu}"]&\mathscr{Y}_\mu\arrow[r,"g"]&\mathscr{Y}
  \end{tikzcd}\]
  of $\mathscr{S}$-diagrams. Then we have the commutative diagram
  \begin{equation}\label{10.2.7.1}\begin{tikzcd}
   \mu'^*g^*f_*\arrow[d,"\sim"]\arrow[r,"Ex"]&\mu'^*f_*'g'^*\arrow[r,"Ex"]&f_{\mu'*}\mu'^*g'^*\arrow[d,"\sim"]\\
    g_{\mu'\mu}^*\mu^*f_*\arrow[r,"Ex"]&g_{\mu'\mu}^*f_{\mu*}\mu^*\arrow[r,"Ex"]&g_{\mu'\mu}^*f_{\mu'*}'g'^*
  \end{tikzcd}\end{equation}
  of functors. The upper right horizontal and lower left horizontal arrows are isomorphisms by (\ref{10.2.6}), and the lower right horizontal arrow is an isomorphism by (PD--5) since the commutative diagram
  \[\begin{tikzcd}
    \mathscr{X}_{\mu'}'\arrow[d,"f_{\mu'}'"]\arrow[r,"g_{\mu'\mu}'"]&\mathscr{X}_\mu\arrow[d,"f_\mu"]\\
    \mathscr{Y}_{\mu'}'\arrow[r,"g_{\mu'\mu}"]&\mathscr{Y}_\mu
  \end{tikzcd}\]
  is Cartesian by assumption. Thus the upper left horizontal arrow of (\ref{10.2.7.1}) is also an isomorphism.
\end{proof}
\begin{none}\label{10.2.10}
  Under the notations and hypotheses of (\ref{10.2.7}), we will give two examples satisfying the conditions of (loc.\ cit).
  \begin{enumerate}[(1)]
    \item When $f$ is reduced and $g$ is a $\mathscr{P}$-morphism, the conditions are satisfied by ($\mathscr{P}$-BC).
    \item Assume that $\mathscr{T}(-,{\bf e})$ satisfies (Loc). Then the conditions are satisfied when $f$ is a reduced strict closed immersion by (\ref{2.6.7}).
  \end{enumerate}
\end{none}
\begin{none}\label{10.2.8}
  Let $i:\mathscr{Z}\rightarrow \mathscr{X}$ be a Cartesian strict closed immersion of $\mathscr{S}$-diagrmas. Then for any morphism $\lambda\rightarrow \lambda'$ in the index category of $\mathscr{X}$, we have the commutative diagram
  \[\begin{tikzcd}
    \mathscr{Z}_\lambda\arrow[r,"i_\lambda"]\arrow[d,"{\rm id}_{\lambda\lambda'}"]&\mathscr{X}_\lambda\arrow[d,"{\rm id}_{\lambda\lambda'}"]\arrow[r,leftarrow,"j_{\lambda}"]& \mathscr{U}_\lambda\arrow[d]\\
    \mathscr{Z}_{\lambda'}\arrow[r,"i_{\lambda'}"]&\mathscr{X}_{\lambda'}\arrow[r,"j_{\lambda'}",leftarrow]&\mathscr{U}_{\lambda'}
  \end{tikzcd}\]
  in $\mathscr{S}$ where each square is Cartesian and $j_\lambda$ (resp.\ $j_{\lambda'}$) denotes the complement of $i_\lambda$ (resp.\ $i_{\lambda'}$). From this, we obtain the Cartesian open immersion $j:\mathscr{U}\rightarrow \mathscr{X}$. It is called the {\it complement} of $i$.
\end{none}
\begin{none}\label{10.2.9}
  We have assumed or proven the axioms DerAlg 0, DerAlg 1, DerAlg 2d, DerAlg 2g, DerAlg 3d, and DerAlg 3g in \cite[4.2.12]{Ayo07}. With the additional assumption that $\mathscr{T}(-,{\bf e})$ satisfies (Loc), the following results are proved in \cite[\S 2.4.3]{Ayo07}.
  \begin{enumerate}[(1)]
    \item Let $i:\mathscr{Z}\rightarrow \mathscr{X}$ be a Cartesian strict closed immersion, and let $j:\mathscr{U}\rightarrow \mathscr{X}$ denote its complement. Then the pair of functors $(i^*,j^*)$ is conservative.
    \item Let $i:\mathscr{Z}\rightarrow \mathscr{X}$ be a strict closed immersion. Then the counit
    \[i^*i_*\stackrel{ad'}\longrightarrow {\rm id}\]
    is an isomorphism.
    \item Consider a Cartesian diagram
    \[\begin{tikzcd}
      \mathscr{X}'\arrow[d,"f'"]\arrow[r,"g'"]&\mathscr{X}\arrow[d,"f"]\\
      \mathscr{Y}'\arrow[r,"g"]&\mathscr{Y}
    \end{tikzcd}\]
    of $\mathscr{S}$-diagrams where $f$ is a $\mathscr{P}$-morphism and $g$ is a Cartesian strict closed immersion. Then the exchange transformation
    \[f_\sharp g_*'\stackrel{Ex}\longrightarrow g_*f_\sharp'\]
    is an isomorphism.
  \end{enumerate}
\end{none}
\begin{none}
  The notion of $\mathscr{P}$-premotivic triangulated prederivators can be used to descent theory of $\mathscr{P}$-premotivic triangulated categories. Let $t$ be a Grothendieck topology on $\mathscr{S}$. Recall from \cite[3.2.5]{CD12} that $\mathscr{T}$ satisfies $t$-descent if the unit
  \[{\rm id}\stackrel{ad}\longrightarrow f_*f^*\]
  is an isomorphism for any $t$-hypercover (see \cite[3.2.1]{CD12} for the definition of $t$-hypercover) $f:\mathscr{X}\rightarrow \mathscr{Y}$ of $\mathscr{S}$-diagrams.
\end{none}
\section{Poincar\'e duality}
\begin{none}
  Throughout this section, fix a log motivic derivator $\mathscr{T}$ over $\mathscr{S}$ satisfying the strict \'etale descent. We will often use notations for various natural transformations in \S \ref{Sec4}.
\end{none}
\subsection{Compactified exactifications}
\begin{none}\label{9.1.0}
  {\it Compactification via toric geometry.} Let $\theta:P\rightarrow Q$ be a homomorphism of fs monoids such that $\theta^{\rm gp}$ is an isomorphism. Choose a fan $\Sigma$ of the dual lattice $(\overline{P}^{\rm gp})^\vee$ whose support is $(\overline{P})^\vee$ and containing $(Q/\theta(P^*))^\vee$ as a cone. This fan induces a factorization
  \[{\rm spec}(Q/\theta({P^*}))\rightarrow M\rightarrow {\rm spec}\,\overline{P}\]
  of the morphism ${\rm spec}\,(Q/\theta(P^*))\rightarrow {\rm spec}\,\overline{P}$ for some fs monoscheme $M$. Consider the open immersions ${\rm spec}\,P_i\rightarrow M$ of fs monoschemes induced by the fan, and, we denote by $P_i'$ the preimage of $P_i$ via the homomorphism $P^{\rm gp}\rightarrow \overline{P}^{\rm gp}$. Then the family of $P_i'$ forms an fs monoscheme $M'$ with the factorization
  \[{\rm spec}\,Q\rightarrow M'\rightarrow {\rm spec}\,P\]
  of the morphism ${\rm spec}\,Q\rightarrow {\rm spec}\,P$. Here, the first arrow is an open immersion, and the second arrow is a proper log \'etale monomorphism.

  We will sometimes use this construction later.
\end{none}
\begin{df}\label{9.1.1}
  Let $f:X\rightarrow S$ be an exact log smooth separated morphism of $\mathscr{S}$-schemes, let $a:X\rightarrow X\times_S X$ denote the diagonal morphism, and let $p_1,p_2:X\times_S X\rightrightarrows X$ denote the first and second projections respectively. A {\it compactified exactification} of the diagram $X\rightarrow X\times_S X\rightrightarrows X$ is a commutative diagram
  \[\begin{tikzcd}
    &D\arrow[d,"u"]\\
    X\arrow[ru,"b"]\arrow[r,"a"]&X\times_S X\arrow[r,"p_1",shift left]\arrow[r,"p_2"',shift right]&X
  \end{tikzcd}\]
  of $\mathscr{S}$-schemes such that
  \begin{enumerate}[(i)]
    \item there is an open immersion $v:I\rightarrow D$ of fs log schemes such that the compositions $p_1uv$ and $p_2uv$ are strict,
    \item $b$ is a strict closed immersion and factors through $I$,
    \item $u$ is a proper and log \'etale monomorphism of fs log schemes.
  \end{enumerate}
  We often say that $u:D\rightarrow X$ is a compactified exactification of $a$ if no confusion seems likely to arise. We also call $I$ an {\it interior} of $E$. Then $p_1uv$ and $p_2uv$ are strict log smooth, and the morphism $X\rightarrow I$ of $\mathscr{S}$-schemes induced by $b$ is a strict regular embedding. Note also that the natural transformation
  \[{\rm id}\stackrel{ad}\longrightarrow u_*u^*\]
  is an isomorphism by (Htp--4) and that the natural transformation
  \[\Omega_{f,I}\stackrel{T_{D,I}}\longrightarrow \Omega_{f,D}\]
  given in (\ref{4.2.2}) is an isomorphism by construction.
\end{df}
\begin{none}\label{9.1.2}
  Under the notations and hypotheses of (\ref{9.1.1}), let $\mathcal{CE}_a$ denote the category whose objects consist of compactified exactifications of $a$ and morphisms consist of commutative diagrams
  \[\begin{tikzcd}
    &E\arrow[rdd,"r_2"]\arrow[d,"v"]\\
    &D\arrow[rd,"q_2"']\arrow[d,"u"]\\
    X\arrow[r,"a"']\arrow[ru,"b"']\arrow[ruu,"c"]&X\times_S X\arrow[r,"p_2"']&X
\end{tikzcd}\]
  of $\mathscr{S}$-schemes. Note that $v$ is a proper log \'etale monomorphism. For such a morphism in $\mathcal{CE}_a$, we associate the natural transformation
  \[T_{D,E}:\Omega_{f,E}\longrightarrow \Omega_{f,D}\]
  given in (\ref{4.2.2}).  We will show that it is an isomorphism. Let $I$ be an interior of $D$, and let $J$ be an interior of $E$ contained in $I\times_D E$. Consider the induced commutative diagram
  \[\begin{tikzcd}
    &J\arrow[rdd,"r_2'"]\arrow[d,"v'"]\\
    &I\arrow[rd,"q_2'"']\arrow[d,"u'"]\\
    X\arrow[r,"a"']\arrow[ru,"b'"']\arrow[ruu,"c'"]&X\times_S X\arrow[r,"p_2"']&X
  \end{tikzcd}\]
  of $\mathscr{S}$-schemes. Then $v'$ is a strict \'etale monomorphism, so it is an open immersion by \cite[IV.17.9.1]{EGA}. Consider the diagram
  \[\begin{tikzcd}
    \Omega_{f,J}\arrow[d,"T_{I,J}"]\arrow[r,"T_{E,J}"]&\Omega_{f,E}\arrow[d,"T_{D,E}"]\\
    \Omega_{f,I}\arrow[r,"T_{D,I}"]&\Omega_{f,D}
  \end{tikzcd}\]
  of functors. It commutes by (\ref{4.2.12}), and the horizontal arrows are isomorphisms by (\ref{9.1.1}). The left vertical arrow is also an isomorphism since $v'$ is an open immersion, so $T_{D,E}$ is an isomorphism.
\end{none}
\begin{df}\label{9.1.5}
  Let $\theta:P\rightarrow Q$ be a homomorphism of fs monoids. Then the submonoid of $P^{\rm gp}$ consisting of elements $p\in P^{\rm gp}$ such that $n\theta^{\rm gp}(p)\in Q$ for some $n\in \mathbb{N}^+$ is called the {\it fs exactification} of $\theta$. It is the fs version of \cite[I.4.2.17]{Ogu17}.
\end{df}
\begin{none}\label{9.1.3}
  Let $f:X\rightarrow S$ be an exact log smooth separated morphism of $\mathscr{S}$-schemes with an fs chart $\theta:P\rightarrow Q$ of exact log smooth type, and let $Q_1$ denote the fs exactification of the summation homomorphism of $Q^{\rm gp}\oplus_{P^{\rm gp}}Q^{\rm gp}$. Applying (\ref{9.1.0}), we obtain the morphisms
  \[{\rm spec}\,Q_1\rightarrow M\rightarrow {\rm spec}(Q\oplus_P Q)\]
  of fs monoschemes. If we put
  \[I=(X\times_S X)\times_{\mathbb{A}_{Q\oplus_P Q}}\mathbb{A}_{Q_1},\quad D=(X\times_S X)\times_{\mathbb{A}_{Q\oplus_P Q}}\mathbb{A}_M,\]
  then the projection $u:D\rightarrow X\times_S X$ is a compactified exactification of the diagonal morphism $a:X\rightarrow X\times_S X$ with an interior $I$. In particular, $a$ has a compactified exactification.
\end{none}
\begin{prop}\label{9.1.4}
  Let $f:X\rightarrow S$ be an exact log smooth separated morphism of $\mathscr{S}$-schemes, and let $a:X\rightarrow X\times_S X$ denote the diagonal morphism. For any compactified exactifications $u:D\rightarrow X\times_SX$ and $u':D\rightarrow X\times_S X$, the morphism $D\times_{X\times_S X}D'\rightarrow X\times_S X$ is a compactified exactification.
\end{prop}
\begin{proof}
  Consider the induced commutative diagram
  \[\begin{tikzcd}
    &&I\arrow[d,"w"']\arrow[rrd,"r_2"',near end]&I'\arrow[ld,"w'"',crossing over,near end]\arrow[rd,"r_2'"]\\
    X\arrow[rru,"c"']\arrow[rrru,bend left,"c'"]\arrow[rr,"a"]&&X\times_S X\arrow[rr,"p_2"']&&X
  \end{tikzcd}\]
  of $\mathscr{S}$-schemes where $I$ (resp.\ $I'$) is an interior of $D$ (resp.\ $D'$). To show the claim, it suffices to construct an open immersion $I''\rightarrow I\times_{X\times_S X}I'$ such that $a$ factors through $I''$ and that the morphisms $I''\rightrightarrows X$ induced by $p_1$ and $p_2$ are strict.

  We have the induced morphisms
  \[X\stackrel{\alpha}\rightarrow I\times_{X\times_S X}I'\stackrel{\beta}\rightarrow I\times_{r_2,X,r_2'}I'\]
  of $\mathscr{S}$-schemes. Let $x\in X$ be a point. Consider the associated homomorphisms
  \[ \overline{\mathcal{M}}_{I\times_{r_2,X,r_2'}I',\overline{\beta\alpha(x)}}\stackrel{\lambda}\rightarrow \overline{\mathcal{M}}_{I\times_{X\times_S X}I',\overline{\alpha(x)}} \stackrel{\eta}\rightarrow \overline{\mathcal{M}}_{X,\overline{x}}\]
  of fs monoids. Then $\eta\lambda$ is an isomorphism since $r_2$ and $r_2'$ are strict. In particular, $\lambda$ is injective. Since $\beta$ is a pullback of the diagonal morphism $X\times_S X\rightarrow (X\times_S X)\times_{p_2,X,p_2} (X\times_S X)$ that is a closed immersion, $\lambda$ is a pushout of a $\mathbb{Q}$-surjective homomorphism. Thus $\lambda$ is $\mathbb{Q}$-surjective, so $\lambda$ is Kummer. Then by (\ref{0.3.5}), $\eta$ is an isomorphism, i.e., $\alpha$ is strict. Thus the conclusion follows from (\ref{0.3.3}).
\end{proof}
\begin{cor}\label{9.1.6}
  Let $f:X\rightarrow S$ be an exact log smooth separated morphism of $\mathscr{S}$-schemes, and let $a:X\rightarrow X\times_S X$ denote the diagonal morphism. Then the category $\mathcal{CE}_a$ is connected.
\end{cor}
\begin{proof}
  It is a direct consequence of (\ref{9.1.4}).
\end{proof}
\subsection{Functoriality of purity transformations}
\begin{none}\label{9.2.1}
  Let $h:X\rightarrow Y$ and $g:Y\rightarrow S$ be exact log smooth separated morphisms of $\mathscr{S}$-schemes. We put $f=gh$. Consider the commutative diagram
  \[\begin{tikzcd}
    X\arrow[d,"a'"]\arrow[rd,"a"]\\
    X\times_Y X\arrow[d,"p_2'"]\arrow[r,"\varphi"]&X\times_S X\arrow[d,"\varphi'"]\arrow[rd,"p_2"]\\
    X\arrow[d,"h"]\arrow[r,"a''"]&Y\times_S X\arrow[r,"p_2''"]\arrow[d,"\varphi''"]&X\arrow[d,"h"]\\
    Y\arrow[r,"a'''"]&Y\times_S Y\arrow[r,"p_2'''"]&Y
  \end{tikzcd}\]
  of $\mathscr{S}$-schemes where
  \begin{enumerate}[(i)]
    \item $a$, $a'$, and $a'''$ denote the diagonal morphisms,
    \item $p_2$, $p_2'$, and $p_2'''$ denote the second projections,
    \item each small square is Cartesian.
  \end{enumerate}
\end{none}
\begin{none}\label{9.2.2}
  Under the notations and hypotheses of (\ref{9.2.1}), assume that we have a commutative diagram
  \begin{equation}\label{9.2.2.1}\begin{tikzcd}
    &X\arrow[rd,"b"]\arrow[ld,"b'"']\\
    D'\arrow[rd,"u'"]\arrow[rr,"\rho",near start]\arrow[rddd,"q_2'"']&&D\arrow[dd,"\rho'",near start]\arrow[rd,"u"]\arrow[rrrddd,bend left,"q_2"]\\
    &X\times_{Y}X\arrow[dd,"p_2'"]\arrow[uu,leftarrow,"a'"',near start,crossing over]\arrow[rr,"\varphi",crossing over,near end]&&X\times_S
    X\arrow[rrdd,"p_2"]\arrow[lluu,"a"',leftarrow, bend right, crossing over]\\
    &&D''\arrow[rd,"u''"']\arrow[dd,"\rho'"',near start]\arrow[rrrd,"q_2''"]\\
    &X\arrow[rr,"a''"',crossing over,near start]\arrow[dd,"h"]\arrow[ru,"b''"]&&Y\times_S X\arrow[uu,"\varphi'"',leftarrow,crossing over]\arrow[rr,"p_2''"']&&X\arrow[dd,"h"]\\
    &&D'''\arrow[rd,"u'''"']\arrow[rrrd,"q_2'''"]\\
    &Y\arrow[ru,"b'''"]\arrow[rr,"a'''"']&&Y\times_S Y\arrow[uu,"\varphi''"',leftarrow,crossing over]\arrow[rr,"p_2'''"']&&Y
  \end{tikzcd}\end{equation}
  of $\mathscr{S}$-schemes where each small square is Cartesian and $u$ (resp.\ $u'$, resp.\ $u'''$) is a compactified exactification of $a$ (resp.\ $a'$, resp.\ $a'''$).

  We will use these notations and hypotheses later.
\end{none}
\begin{none}\label{9.2.3}
  Under the notations and hypotheses of (\ref{9.2.2}), let $\alpha:S_0\rightarrow S$ be a morphism of $\mathscr{S}$-schemes, and consider the commutative diagrams
  \[\begin{tikzcd}
    X_0\arrow[d,"\beta"]\arrow[r,"h_0"]&Y_0\arrow[d]\arrow[r,"g_0"]&S_0\arrow[d,"\alpha"]\\
    X\arrow[r,"h"]&Y\arrow[r,"g"]&S
  \end{tikzcd}\]
  \[\begin{tikzcd}
    X_0\arrow[dd,"b_0'"]\arrow[rd,"\beta"]\arrow[rrdd,"b_0",bend left]\\
    &X\\
    D_0'\arrow[rr,"\rho_0",near end]\arrow[rd,"\gamma'"]\arrow[dd,"q_{02}'"]&&D_0\arrow[dd,"\rho_0'",near start]\arrow[rd,"\gamma"']\arrow[rrdd,"q_{02}",bend left]\\
    &D'\arrow[rr,"\rho",near start,crossing over]\arrow[uu,"b'",crossing over,near end,leftarrow]&&D\arrow[lluu,leftarrow,"b"',bend right,crossing over]\arrow[rrdd,"q_2",bend left,crossing over]\\
    X_0\arrow[rd,"\beta"]\arrow[rr,"b_0''",near end]&&D_0''\arrow[rd,"\gamma''"]\arrow[rr,"q_{02}''",near end]&&X_0\arrow[rd,"\beta"]\\
    &X\arrow[rr,"b''"]\arrow[uu,"q_2'",crossing over,near end,leftarrow]&&D''\arrow[rr,"q_2''"]\arrow[uu,"\rho'",near end,leftarrow,crossing over]&&X
  \end{tikzcd}\]
  of $\mathscr{S}$-schemes where each small square is Cartesian. We put $f_0=g_0h_0$. Then the diagram
  \[\begin{tikzcd}
    \beta^*\Omega_{h,D'}\Omega_{g,f,D''}\arrow[r,"C"]\arrow[d,"Ex"]&\beta^*\Omega_{f,D}\arrow[dd,"Ex"]\\
    \Omega_{h_0,D_0'}\beta^*\Omega_{g,f,D''}\arrow[d,"Ex"]\\
    \Omega_{h_0,D_0'}\Omega_{g_0,f_0,D_0''}\beta^*\arrow[r,"C"]&\Omega_{f_0,D_0}\beta^*
  \end{tikzcd}\]
  of functors commutes since it is the big outside diagram of the commutative diagram
  \[\begin{tikzcd}
    \beta^*b'^!q_2'^*b''^!q_2''^*\arrow[d,"Ex"]\arrow[r,"Ex"]&\beta^*b'^!\rho^!\rho'^*q_2''^*\arrow[d,"Ex"]\arrow[r,"\sim"]&\beta^*b^!q_2^*\arrow[dd,"Ex"]\\
    b_0'^!\gamma'^*q_2'^*b''^!q_2''^*\arrow[d,"\sim"]\arrow[r,"Ex"]&b_0'^!\gamma'^*\rho^!\rho'^*q_2''^*\arrow[d,"Ex"]\\
    b_0'^!q_{02}'^*\beta^*b''^!q_2''^*\arrow[d,"Ex"]&b_0'^!\rho_0^!\gamma^*\rho'^*q_2''^*\arrow[dd,"\sim"]\arrow[r,"\sim"]&b_0^!\gamma^*q_2^*\arrow[dd,"\sim"]\\
    b_0'^!q_{02}'^*b_0''^!\gamma''^*q_2''^*\arrow[d,"\sim"]\\
    b_0'^!q_{02}'^*b_0''^!q_{02}''^*\alpha^*\arrow[r,"Ex"]&b_0'^!\rho_0^!\rho_0'^*q_2''^*\beta^*\arrow[r,"\sim"]&b_0^!q_{02}^*\beta^*
  \end{tikzcd}\]
  of functors.

  We will use these notations and hypotheses later.
\end{none}
\begin{none}\label{9.2.4}
  Under the notations and hypotheses of (\ref{9.2.2}), we denote by $\mathcal{I}$ (resp.\ $\mathcal{I}'$, resp.\ $\mathcal{I}''$) the ideal sheaf of $X$ on $D$ (resp.\ $D'$, resp.\ $D'''$). Then by \cite[IV.3.2.2]{Ogu17}, the morphisms
  \[\mathcal{I}/\mathcal{I}^2\rightarrow b^*\Omega_{D/X}^1,\quad \mathcal{I}'/\mathcal{I}'^2\rightarrow b'^*\Omega_{D'/X}^1,\quad \mathcal{I}''/\mathcal{I}''^2\rightarrow b''^*\Omega_{D''/X}^1\]
  of quasi-coherent sheaves on $\underline{X}$ are isomorphisms, and by \cite[IV.3.2.4, IV.1.2.15]{Ogu17}, the morphisms
  \[\Omega_{D/X}^1\rightarrow u^*\Omega_{X\times_S X/X}^1\rightarrow u^*p_1^*\Omega_{X/S}^1,\]
  \[\Omega_{D'/X}^1\rightarrow u'^*\Omega_{X\times_Y X/X}^1\rightarrow u'^*p_1'^*\Omega_{X/Y}^1,\]
  \[\Omega_{D''/X}^1\rightarrow u''^*\Omega_{Y\times_S X/X}^1\rightarrow u''^*p_1''^*\Omega_{Y/S}^1\]
  of quasi-coherent sheaves on $\underline{X}$ are isomorphisms where
  \[p_1:X\times_S X\rightarrow X,\quad p_1':X\times_Y X\rightarrow X,\quad p_1'':Y\times_S X\rightarrow Y\]
  denote the first projections. Then from the exact sequence
  \[0\longrightarrow h^*\Omega_{Y/S}^1\longrightarrow \Omega_{X/S}^1\longrightarrow \Omega_{X/Y}^1\longrightarrow 0\]
  of quasi-coherent sheaves on $\underline{X}$ given in \cite[IV.3.2.3]{Ogu17}, we have the exact sequence
  \[0\longrightarrow \mathcal{I}''/\mathcal{I}''^2\longrightarrow \mathcal{I}/\mathcal{I}^2\longrightarrow \mathcal{I}'/\mathcal{I}'^2\longrightarrow 0\]
  of quasi-coherent sheaves on $\underline{X}$. This shows that the induced diagram
  \begin{equation}\label{9.2.4.1}\begin{tikzcd}
    N_X{D'}\arrow[d,"t_2'"]\arrow[r,"\chi"]&N_X{D}\arrow[d,"\chi'"]\\
    X\arrow[r,"e_2"]&N_XD''
  \end{tikzcd}\end{equation}
  of $\mathscr{S}$-schemes is Cartesian. Thus the induced diagram
  \[\begin{tikzcd}
    D_X{D'}\arrow[d]\arrow[r]&D_X{D}\arrow[d]\\
    X\times\mathbb{A}^1\arrow[r]&D_XD''
  \end{tikzcd}\]
  of $\mathscr{S}$-schemes is also Cartesian. Then as in (\ref{4.3.1}), we have the natural transformations
  \[\Omega_{h,D'}^d\Omega_{g,f,D''}^d\stackrel{C}\longrightarrow \Omega_{f,D}^d,\]
  \begin{equation}\label{9.2.4.2}
    \Omega_{h,D'}^n\Omega_{g,f,D''}^n\stackrel{C}\longrightarrow \Omega_{f,D}^n
  \end{equation}
  These are called again {\it composition transformations.} Note that the left adjoint versions are
  \[\Sigma_{f,D}^d\stackrel{C}\longrightarrow \Sigma_{h,D'}^d\Sigma_{g,f,D''}^d,\]
  \[\Sigma_{f,D}^n\stackrel{C}\longrightarrow \Sigma_{h,D'}^n\Sigma_{g,f,D''}^n.\]
  In the Cartesian diagram (\ref{9.2.4.1}), the morphisms $e_2$, $t_2'$, $\chi$, and $\chi'$ are strict, $\chi'$ are strict smooth, and $e_2$ is a strict closed immersion. Thus by (\ref{2.5.9}) and (\ref{4.3.1}), the natural transformation (\ref{9.2.4.2}) is an isomorphism.

  Applying (\ref{9.2.3}) to the cases when the diagram
  \[\begin{tikzcd}
    X_0\arrow[d,"\beta"]\arrow[r,"b_0''"]&D_0\arrow[d,"\gamma"]\arrow[r,"q_{02}"]&X_0\arrow[d,"\beta"]\\
    X\arrow[r,"b''"]&D\arrow[r,"q_2"]&X
  \end{tikzcd}\]
  is equal to one of the diagrams in (\ref{4.1.1.1}) and similar things are true for $D'$ and $D''$, we have the commutative diagram
  \[\begin{tikzcd}
    \Omega_{h,D'}^n\Omega_{g,f,D''}^n\arrow[r,"C"]\arrow[d,"T^nT^n"]&\Omega_{f,D}^n\arrow[d,"T^n"]\\
    \Omega_{h,D'}^d\Omega_{g,f,D''}^d\arrow[r,"C"]\arrow[d,"T^dT^d"]&\Omega_{f,D}^d\arrow[d,"T^d"]\\
    \Omega_{h,D'}\Omega_{g,f,D''}\arrow[r,"C"]&\Omega_{f,D}
  \end{tikzcd}\]
  of functors.
\end{none}
\begin{none}\label{9.2.5}
  Under the notations and hypotheses of (\ref{9.2.2}), consider the diagram
  \begin{equation}\label{9.2.5.1}\begin{tikzcd}
    f_\sharp\arrow[rr,"\mathfrak{p}_{f,D}^n"]\arrow[ddd,"\sim"]&&f_!\Sigma_{f,D}^n\arrow[d,"C"]\\
    &&f_!\Sigma_{g,f,D''}^n\Sigma_{h,D'}^n\arrow[d,"\sim"]\\
    &&g_!h_!\Sigma_{g,f,D''}^n\Sigma_{h,D'}^n\arrow[d,"Ex",leftarrow]\\
    g_\sharp h_\sharp\arrow[rr,"\mathfrak{p}_{g,D'''}^n\mathfrak{p}_{h,D'}^n"]&&g_!\Sigma_{g,D'''}^nh_!\Sigma_{h,D'}^n
  \end{tikzcd}\end{equation}
  of functors. We will show that it commutes. Its right adjoint is the big outside diagram of the diagram
  \[\begin{tikzpicture}[baseline= (a).base]
    \node[scale=.96] (a) at (0,0)
    {\begin{tikzcd}
    \Omega_{h,D'}^nh^!\Omega_{g,D'''}^ng^!\arrow[r,"T^nT^n"]\arrow[d,"Ex",leftarrow]&\Omega_{h,D'}^dh^!\Omega_{g,D'''}^dg^!\arrow[r,"T^dT^d"]\arrow[d,"Ex",leftarrow]& \Omega_{h,D'}h^!\Omega_{g,D'''}g^!\arrow[r,"T_{D'}T_{D'''}"]\arrow[d,"Ex",leftarrow]&\Omega_h h^!\Omega_g g^!\arrow[d,"Ex",leftarrow]\arrow[r,"\mathfrak{q}_h\mathfrak{q}_g"]& h^*g^*\arrow[ddd,"\sim"]\\
    \Omega_{h,D'}^n\Omega_{g,f,D''}^nh^!g^!\arrow[r,"T^nT^n"]\arrow[d,"\sim"]&\Omega_{h,D'}^d\Omega_{g,f,D''}^dh^!g^!\arrow[r,"T^dT^d"]\arrow[d,"\sim"] &\Omega_{h,D'}\Omega_{g,f,D''}h^!g^!\arrow[r,"T_{D'}T_{D''}"]\arrow[d,"\sim"]& \Omega_h\Omega_{g,f}h^!g^!\arrow[d,"\sim"]\\
    \Omega_{h,D'}^n\Omega_{g,f,D''}^nf^!\arrow[r,"T^nT^n"]\arrow[d,"C"]&\Omega_{h,D'}^d\Omega_{g,f,D''}^df^!\arrow[r,"T^dT^d"]\arrow[d,"C"]&\Omega_{h,D'}\Omega_{g,f,D''}f^! \arrow[d,"C"]\arrow[r,"T_{D'}T_{D''}"]& \Omega_h\Omega_{g,f}f^!\arrow[d,"C"]\\
    \Omega_{f,D}^nf^!\arrow[r,"T^n"]&\Omega_{f,D}^df^!\arrow[r,"T^d"]&\Omega_{f,D}f^!\arrow[r,"T_D"]&\Omega_ff^!\arrow[r,"\mathfrak{q}_f"]&f^*
    \end{tikzcd}};
  \end{tikzpicture}\]
  of functors. It commutes by (\ref{4.2.6}), (\ref{4.3.2}), (\ref{4.4.3}), and (\ref{9.2.4}). Thus (\ref{9.2.5.1}) also commutes.

  Note also that the right vertical top arrow of (\ref{9.2.5.1}) is an isomorphism by (\ref{9.2.4}) and that the right vertical bottom arrow of (\ref{9.2.5.1}) is an isomorphism by (\ref{4.2.10}).
\end{none}
\begin{none}\label{9.2.6}
  Let $h:X\rightarrow Y$ and $g:Y\rightarrow S$ be exact log smooth separated morphisms of $\mathscr{S}$-schemes. Assume that $f$ (resp.\ $g$) has an fs chart $\theta:Q\rightarrow R$ (resp.\ $\eta:R\rightarrow P$) of exact log smooth type. In this setting, we will construct the diagram (\ref{9.2.2.1}) and verify the hypotheses of (\ref{9.2.2}).

  We denote by $T$ and $T'''$ the fs exactification of the summation homomorphisms
  \[Q\oplus_P Q\rightarrow Q,\quad R\oplus_P R\rightarrow R\]
  respectively. Then we put
  \[T'=T\oplus_{Q\oplus_P Q}(Q\oplus_R Q).\]
  By (\ref{9.1.0}), we have the factorization
  \[{\rm spec}\,T'''\rightarrow M'''\rightarrow {\rm spec}\,R\oplus_P R\]
  such that the first arrow is an open immersion of fs monoschemes and the second arrow is a proper log \'etale monomorphism of fs monoschemes. We put
  \[M''=M'''\times_{{\rm spec}(R\oplus_P R)}{\rm spec}(R\oplus_P Q).\]
  Consider the induced morphism
  \[{\rm spec}\,T\rightarrow M''\times_{{\rm spec}(R\oplus_P Q)}{\rm spec}(Q\oplus_P Q).\]
  By the method of (\ref{9.1.0}), it has a factorization
  \[{\rm spec}\,T\rightarrow M\rightarrow M''\times_{{\rm spec}(R\oplus_P Q)}{\rm spec}(Q\oplus_P Q)\]
  where the first arrow is an open immersion of fs monoschemes and the second arrow is a proper log \'etale monomorphism of fs monoschemes. We put
  \[M'=M\times_{{\rm spec}(Q\oplus_P Q)}{\rm spec}(Q\oplus_R Q),\]
  and we put
  \[I=(X\times_S X)\times_{\mathbb{A}_{Q\oplus_P Q}}\mathbb{A}_T,\quad I'=(X\times_Y X)\times_{\mathbb{A}_{Q\oplus_R Q}}\mathbb{A}_{T'},\quad I'''=(Y\times_S Y)\times_{\mathbb{A}_{R\oplus_P R}}\mathbb{A}_{T'''},\]
  \[D=(X\times_S X)\times_{\mathbb{A}_{Q\oplus_P Q}}\mathbb{A}_M,\quad D'=(X\times_Y X)\times_{\mathbb{A}_{Q\oplus_R Q}}\mathbb{A}_{M'},\quad D'''=(Y\times_S Y)\times_{\mathbb{A}_{R\oplus_P R}}\mathbb{A}_{M'''}.\]
  Then we have the commutative diagram (\ref{9.2.2.1}). By construction, $D$ (resp.\ $D'$, resp.\ $D'''$) are compactified exactifications of the diagonal morphism $a:X\rightarrow X\times_S X$ (resp.\ $a':X\rightarrow X\times_Y X$, resp.\ $a''':Y\rightarrow Y\times_S Y$) with an interior $I$ (resp.\ $I'$, resp.\ $I'''$).
\end{none}
\begin{prop}\label{9.2.7}
  Let $f:X\rightarrow S$ be an exact log smooth separated morphism of $\mathscr{S}$-schemes, let $D$ be a compactified exactification of the diagonal morphism $a:X\rightarrow X\times_S X$, and let $g:S'\rightarrow S$ be a morphism of $\mathscr{S}$-schemes. We put
  \[X'=X\times_S S',\quad D'=D\times_{X\times_S X}(X'\times_{S'}X').\]
  Then the diagram
  \begin{equation}\label{9.2.7.1}\begin{tikzcd}
    f_\sharp'g'^*\arrow[dd,"Ex"]\arrow[r,"\mathfrak{p}_{f'}^n"]&f_!'\Sigma_{f',D'}^n g'^!\arrow[d,"Ex"]\\
    &f_!'g'^*\Sigma_{f,D}^n\arrow[d,leftarrow,"Ex"]\\
    g^*f_\sharp\arrow[r,"\mathfrak{p}_f^n"]&g^*f_!\Sigma_{f,D}^n
  \end{tikzcd}\end{equation}
  of functors commutes.
\end{prop}
\begin{proof}
  The right adjoint of (\ref{9.2.7.1}) is the big outside diagram of the diagram
  \[\begin{tikzcd}
    \Omega_f^nf^!g_*\arrow[d,"Ex",leftarrow]\arrow[r,"T^n"]&\Omega_f^d f^!g_*\arrow[d,"Ex",leftarrow]\arrow[r,"T^d"]&\Omega_f f^!g_*\arrow[d,"Ex",leftarrow]\arrow[r,"\mathfrak{q}_f"]& f^*g_*\arrow[dd,"Ex"]\\
    \Omega_f^n g_*'f'^!\arrow[d,"Ex"]\arrow[r,"T^n"]&\Omega_f^dg_*'f'^!\arrow[d,"Ex"]\arrow[r,"T^d"]&\Omega_fg_*'f'^!\arrow[d,"Ex"]\\
    g_*'\Omega_{f'}^nf'^!\arrow[r,"T^n"]&g_*'\Omega_{f'}^df'^!\arrow[r,"T^d"]&g_*'\Omega_{f'}f'^!\arrow[r,"\mathfrak{q}_{f'}"]&g_*'f'^*
  \end{tikzcd}\]
  of functors. By (\ref{4.2.6}) and (\ref{4.4.6}), each small diagram commutes. The conclusion follows from this.
\end{proof}
\subsection{Poincar\'e duality for Kummer log smooth separated morphisms}
\begin{prop}\label{9.4.1}
  Let $f:X\rightarrow S$ be a strict smooth separated morphism of $\mathscr{S}$-schemes. Then the natural transformation
  \[\mathfrak{p}_f^n:f_\sharp\longrightarrow f_!\Sigma_f^n\]
  is an isomorphism.
\end{prop}
\begin{proof}
  It follows from (\ref{2.5.7}) and (\ref{4.2.1}).
\end{proof}
\begin{none}\label{9.4.3}
  Let $f:X\rightarrow S$ be a Kummer log smooth separated morphism of $\mathscr{S}$-schemes. Then the diagonal morphism $a:X\rightarrow X\times_S X$ is a strict regular embedding by (\ref{0.3.4}). In particular, we can use the notation $\Sigma_f^n$.
\end{none}
\begin{prop}\label{9.4.2}
  Let $f:X\rightarrow S$ be a Kummer log smooth separated morphism of $\mathscr{S}$-schemes. Then the natural transformation
  \[\mathfrak{p}_f^n:f_\sharp \longrightarrow f_!\Sigma_f^n\]
  is an isomorphism.
\end{prop}
\begin{proof}
  By (\ref{5.3.3}), there is a Cartesian diagram
  \[\begin{tikzcd}
    X'\arrow[d,"g'"]\arrow[r,"f'"]&S'\arrow[d,"g"]\\
    X\arrow[r,"f"]&S
  \end{tikzcd}\]
  of $\mathscr{S}$-schemes such that
  \begin{enumerate}[(i)]
    \item $g$ is Kummer log smooth,
    \item $g^*$ is conservative,
    \item $f'$ is strict.
  \end{enumerate}
  Hence we reduce to showing that the natural transformation
  \[g^*f_\sharp \stackrel{\mathfrak{p}_f^n}\longrightarrow g^*f_!\Sigma_f^n\]
  is an isomorphism.

  Consider the commutative diagram (\ref{9.2.7.1}). The left vertical arrow and right lower vertical arrow are isomorphisms since $f$ and $g$ are exact log smooth. The right upper vertical arrow is an isomorphism by (\ref{4.2.8}). Thus we reduce to showing that the upper horizontal arrow is an isomorphism. It follows from (\ref{9.4.1}).
\end{proof}
\subsection{Poincar\'e duality for \texorpdfstring{$\mathbb{A}_Q\rightarrow \mathbb{A}_P$}{AQAP}}
\begin{df}\label{9.3.4}
  In this subsection, we will consider the following conditions:
  \begin{enumerate}
    \item[$({\rm PD}_{f,D})$] Let $f:X\rightarrow S$ be an exact log smooth separated morphism of $\mathscr{S}$-schemes, and let $D$ be a compactified exactification of the diagonal morphism $a$. We denote by $({\rm PD}_{f,D})$ the condition that the natural transformation
        \[\mathfrak{p}_f^n:f_\sharp\longrightarrow f_!\Sigma_{f,D}^n\]
        is an isomorphism.
    \item[$({\rm PD}_f)$] Let $f:X\rightarrow S$ be an exact log smooth separated morphism of $\mathscr{S}$-schemes, and let $a:X\rightarrow X\times_S X$ denote the diagonal morphism. We denote by $({\rm PD}_f)$ the conditions that
        \begin{enumerate}[(i)]
          \item there is a compactified exactification of $a$,
          \item for any compactified exactification $D$ of $a$, $({\rm PD}_{f,D})$ is satisfied.
        \end{enumerate}
    \item[$({\rm PD}^{m})$] We denote by $({\rm PD}^m)$ the condition that $({\rm PD}_f)$ is satisfied for any {\it vertical} exact log smooth separated morphism $f:X\rightarrow S$ with an fs chart $\theta:P\rightarrow Q$ such that $\theta$ is a vertical homomorphism of exact log smooth type and
        \[\max_{x\in X}{\rm rk}\,\overline{\mathcal{M}}_{X,\overline{x}}^{\rm gp}+\max_{s\in S}{\rm rk}\,\overline{\mathcal{M}}_{S,\overline{s}}^{\rm gp}\leq m.\]
  \end{enumerate}
  Note that by (\ref{2.8.2}), we get equivalent conditions if we use $\Sigma_{f,D}^o$ instead of $\Sigma_{f,D}^n$
\end{df}
\begin{prop}\label{9.3.1}
  Let $f:X\rightarrow S$ be a vertical exact log smooth separated morphism of $\mathscr{S}$-schemes, and let $E\rightarrow D$ be a morphism in $\mathcal{CE}_a$ where $a:X\rightarrow X\times_S X$ denotes the diagonal morphism. Then $({\rm PD}_{f,D})$ is equivalent to $({\rm PD}_{f,E})$.
\end{prop}
\begin{proof}
  The diagram
  \[\begin{tikzcd}
    \Omega_{g,f,E}^n\arrow[d,"T_{D,E}"]\arrow[r,"T^n"]&\Omega_{g,f,E}^d\arrow[d,"T_{D,E}"]\arrow[r,"T^d"]&\Omega_{g,f,E}\arrow[d,"T_{D,E}"]\arrow[r,"T_E"]&\Omega_f f^!\arrow[d,equal] \arrow[r,"\mathfrak{q}_f"]&f^*\arrow[d,equal]\\
    \Omega_{g,f,D}^n\arrow[r,"T^n"]&\Omega_{g,f,D}^d\arrow[r,"T^d"]&\Omega_{g,f,D}\arrow[r,"T_D"]&\Omega_f f^!\arrow[r,"\mathfrak{q}_f"]&f^*
  \end{tikzcd}\]
  of functors commutes by (\ref{4.2.11}) and (\ref{4.2.12}). The left vertical arrow is an isomorphism because the normal cones $N_XD$ and $N_XE$ are isomorphic to the vector bundle associated to the sheaf $\Omega_{X/S}^1$. Then the conclusion follows from the fact that the composition of row arrows are $\mathfrak{q}_{f,E}^n$ and $\mathfrak{q}_{f,D}^n$ respectively.
\end{proof}
\begin{cor}\label{9.3.2}
  Let $f:X\rightarrow S$ be a vertical exact log smooth separated morphism of $\mathscr{S}$-schemes such that there is a compactified exactification $D$ of the diagonal morphism $a:X\rightarrow X\times_S X$. Then $({\rm PD}_f)$ is equivalent to $({\rm PD}_{f,D})$.
\end{cor}
\begin{proof}
  Since $\mathcal{CE}_a$ is connected by (\ref{9.1.4}), the conclusion follows from (\ref{9.3.1}).
\end{proof}
\begin{prop}\label{9.3.3}
  Under the notations and hypotheses of (\ref{9.2.2}), if $({\rm PD}_{g,D'''})$ and $({\rm PD}_{h,D'})$ are satisfied, then $({\rm PD}_{f,D})$ is satisfied.
\end{prop}
\begin{proof}
  By (\ref{9.2.6}), we can use (\ref{9.2.5}). Then by (loc.\ cit), in the commutative diagram (\ref{9.2.5.1}), the upper horizontal arrow is an isomorphism if and only if the lower horizontal arrow is an isomorphism. The conclusion follows from this.
\end{proof}
\begin{prop}\label{9.3.10}
  Let $f:X\rightarrow S$ be an exact log smooth separated morphism of $\mathscr{S}$-schemes, and let $D$ be a compactified exactification of the diagonal morphism $a:X\rightarrow X\times_S X$. Then $({\rm PD}_f)$ is strict \'etale local on $X$.
\end{prop}
\begin{proof}
  Let $\{u_i:X_i\rightarrow X\}_{i\in I}$ be a strict \'etale cover of $X$. We put
  \[f_i=fu_i,\quad D_i=D\times_{X\times_S X}(X_i\times_S X_i),\quad D_i''=D\times_{X\times_S X}(X_i\times_S X_i)\]
  Then $D_i$ is a compactified exactification of the diagonal morphism $a_i:X_i\rightarrow X_i\times_S X_i$. Hence by (\ref{9.3.2}), it suffices to show that $({\rm PD}_{f,D})$ is satisfied if and only if $({\rm PD}_{f_i,D_i})$ is satisfied for all $i$. Since $\mathscr{T}$ satisfies the strict \'etale descent, by \cite[3.2.8]{CD12}, $({\rm PD}_{f,D})$ is equivalent to the condition that the natural transformation
  \[u_i^*f^!\Omega_{f,D}^n\stackrel{\mathfrak{q}_{f,D}^n}\longrightarrow u_i^*f^*\]
  is an isomorphism for any $i\in I$. This is equivalent to the condition that the natural transformation
  \[f_\sharp u_{i\sharp}\stackrel{\mathfrak{p}_{f,D}^n}\longrightarrow f_!\Sigma_{f,D}^nu_{i\sharp}\]
  is an isomorphism for any $i\in I$.

  By (\ref{9.2.6}), we can use (\ref{9.2.5}) for $u_i:X_i\rightarrow X$ and $X\rightarrow S$. Then by (loc.\ cit), in the commutative diagram
  \[\begin{tikzcd}
    f_{i\sharp}\arrow[rr,"\mathfrak{p}_{f_i,D_i}^n"]\arrow[ddd,"\sim"]&&f_{i!}\Sigma_{f_i,D_i}^n\arrow[d,"C"]\\
    &&f_{i!}\Sigma_{u_i,f_i,D_i''}^n\Sigma_{u_i}^n\arrow[d,"\sim"]\\
    &&f_!u_{i!}\Sigma_{u_i,f_i,D_i''}^n\Sigma_{u_i}^n\arrow[d,"Ex",leftarrow]\\
    f_\sharp u_{i\sharp}\arrow[r,"\mathfrak{p}_{f,D}^n"]&f_!\Sigma_{f,D}^n u_{i\sharp}\arrow[r,"\mathfrak{p}_{u_i}^n"]&f_!\Sigma_{f,D}^nu_!\Sigma_{u_i}^n
  \end{tikzcd}\]
  of functors, the right vertical top arrow and the right vertical bottom arrow are isomorphisms. The lower horizontal right arrow is also an isomorphism by (\ref{9.4.2}). Thus the upper horizontal arrow is an isomorphism if and only if the lower horizontal left arrow is an isomorphism, which is what we want to prove.
\end{proof}
\begin{none}\label{9.3.5}
  Let $S$ be an $\mathscr{S}$-scheme with an fs chart $P$ that is exact at some point $s\in S$, and let $\theta:P\rightarrow Q$ be a locally exact, injective, logarithmic, and vertical homomorphism of fs monoids such that the cokernel of $\theta^{\rm gp}$ is torsion free. We put
  \[X=S\times_{\mathbb{A}_P}\mathbb{A}_Q,\quad m={\rm dim}\,P+{\rm dim}\,Q,\]
  and assume $m>0$. By (\ref{9.1.0}), there is a compactified exactification $D$ of the diagonal morphism $a:X\rightarrow X\times_S X$.
\end{none}
\begin{prop}\label{9.3.6}
  Under the notations and hypotheses of (\ref{9.3.5}), the natural transformation
  \[f_\sharp f^*\stackrel{\mathfrak{p}_{f,D}^n}\longrightarrow f_!\Sigma_{f,D}^nf^*\]
  is an isomorphism.
\end{prop}
\begin{proof}
  Let $G$ be a maximal $\theta$-critical face of $Q$, and we put
  \[U=S\times_{\mathbb{A}_P}\mathbb{A}_{Q_F},\quad D'=D\times_{X\times_S X}(U\times_S U).\]
  We denote by $j:U\rightarrow X$ the induced open immersion. Then the diagram
  \[\begin{tikzcd}
    (fj)_\sharp j^*f^*\arrow[rr,"\sim"]\arrow[d,"\mathfrak{p}_{fj,D'}^n"]&&f_\sharp j_\sharp j^*f^*\arrow[d,"\mathfrak{p}_{f,D}^n"]\arrow[r,"ad'"]&f_\sharp f^*\arrow[d,"\mathfrak{p}_{f,D}^n"]\\
    (fj)_!\Sigma_{fj,D'}^nf^*f^*\arrow[r,"\sim"]&f_!j_\sharp \Sigma_{fj,D'}^nj^*f^*\arrow[r,"Ex"]&f_!\Sigma_{f,D}^nj_\sharp j^*f^*\arrow[r,"ad'"]&f_!\Sigma_{f,D}^nf^*
  \end{tikzcd}\]
  of functors commutes by (\ref{9.2.5}) and (\ref{9.2.6}). By (Htp--3) and (Htp--7), the upper and lower right side horizontal arrows are isomorphisms, and the lower middle horizontal arrow is an isomorphism by (\ref{4.2.9}). The composition $fj:U\rightarrow S$ is Kummer log smooth and separated, so the left vertical arrow is an isomorphism by (\ref{9.4.2}). Thus the right vertical arrow is also an isomorphism.
\end{proof}
\begin{prop}\label{9.3.7}
  Under the notations and hypotheses of (\ref{9.3.5}), $({\rm PD}^{m-1})$ implies $({\rm PD}_f)$.
\end{prop}
\begin{proof}
  Assume $({\rm PD}^{m-1})$. By (\ref{9.3.2}), it suffices to show $({\rm PD}_{f,D})$. We put
  \[d={\rm rk}\,Q^{\rm gp}-{\rm rk}\,P^{\rm gp},\quad \tau=(d)[2d].\] Then it suffices to show that the natural transformation
  \[\mathfrak{p}_{f,D}^o:f_\sharp\longrightarrow f_!\tau\]
  is an isomorphism. Guided by a method of \cite[2.4.42]{CD12}, we will construct its left inverse $\phi_2$ as follows:
  \[\phi_2:f_!\tau\stackrel{ad}\longrightarrow f_!\tau f^*f_\sharp \stackrel{(\mathfrak{p}_{f,D}^o)^{-1}}\longrightarrow f_\sharp f^*f_\sharp \longrightarrow f_\sharp\]
  Here, the second arrow is defined and an isomorphism by (\ref{9.3.6}). We have $\phi_2\circ \mathfrak{p}_{f,D}^o={\rm id}$ as in the proof of (\ref{5.5.7}).

  To construct a right inverse of $\mathfrak{p}_{f,D}^o$, consider the Cartesian diagram
  \[\begin{tikzcd}
    Z\arrow[d,"i"]\arrow[r]&\mathbb{A}_Q-\mathbb{A}_{(Q,Q^+)}\arrow[d]\\
    X\arrow[r]&\mathbb{A}_Q
  \end{tikzcd}\]
  of $\mathscr{S}$-schemes, and we put $g=fi$. Note that the morphism $\underline{g}:\underline{Z}\rightarrow \underline{S}$ of underlying schemes is an isomorphism by assumption on $\theta$. Consider the commutative diagram
  \[\begin{tikzcd}
    f_\sharp j_\sharp j^*\arrow[r,"ad'"]\arrow[d,"\mathfrak{p}_{f,D}^o"]&f_\sharp\arrow[r,"ad"]\arrow[d,"\mathfrak{p}_{f,D}^o"]&f_\sharp i_*i^*\arrow[r,"\partial_i"]\arrow[d,"\mathfrak{p}_{f,D}^o"] &f_\sharp j_\sharp j^*[1]\arrow[d,"\mathfrak{p}_{f,D}^o"]\\
    f_!\tau j_\sharp j^*\arrow[r,"ad'"]&f_!\tau\arrow[r,"ad"]&f_!\tau i_*i^*\arrow[r,"\partial_i"]&f_!\tau j_\sharp j^*[1]
  \end{tikzcd}\]
  of functors where $j$ denotes the complement of $i$. The two rows are distinguished triangles by (Loc). The first vertical arrow is an isomorphism by $({\rm PD}^{m-1})$, and $\phi_2$ induces the left inverse to the third vertical arrow. If we show that the third vertical arrow is an isomorphism, then the second vertical arrow is also an isomorphism. Hence it suffices to construct a right inverse of the natural transformation
  \[f_\sharp i_*\stackrel{\mathfrak{p}_{f,D}^o}\longrightarrow f_!\tau i_*.\]

  Consider also the commutative diagram
  \[\begin{tikzcd}
    f_\sharp j_\sharp j^*f^*\arrow[r,"ad'"]\arrow[d,"\mathfrak{p}_{f,D}^o"]&f_\sharp f^*\arrow[r,"ad"]\arrow[d,"\mathfrak{p}_{f,D}^o"]&f_\sharp i_*i^*f^*\arrow[r,"\partial_i"]\arrow[d,"\mathfrak{p}_{f,D}^o"] &f_\sharp j_\sharp j^*f^*[1]\arrow[d,"\mathfrak{p}_{f,D}^o"]\\
    f_!\tau j_\sharp j^*f^*\arrow[r,"ad'"]&f_!\tau f^*\arrow[r,"ad"]&f_!\tau i_*i^*f^*\arrow[r,"\partial_i"]&f_!\tau j_\sharp j^*f^*[1]
  \end{tikzcd}\]
  of functors. The two rows are distinguished triangles by (Loc), and the first vertical arrow is an isomorphism by $({\rm PD}^{m-1})$. Since the second vertical arrow is an isomorphism by (\ref{9.3.6}), the third vertical arrow is also an isomorphism.

  Then a right inverse of $f_\sharp i_*\stackrel{\mathfrak{p}_{f,D}^o}\longrightarrow f_!\tau i_*$ is constructed by
  \[\begin{split}
    \phi_1':f_!\tau i_*&\stackrel{Ex}\longrightarrow f_!i_*\tau\stackrel{\sim}\longrightarrow g_*\tau\stackrel{ad}\longrightarrow g_*g^*g_*\tau\stackrel{Ex^{-1}}\longrightarrow g_*g^*\tau g_*\\
    &\stackrel{\sim}\longrightarrow f_!i_*\tau g^*g_*\stackrel{Ex^{-1}}\longrightarrow f_!\tau i_*g^*g_*\stackrel{(\mathfrak{p}_{f,D}^o)^{-1}}\longrightarrow f_\sharp i_*g^*g_*\stackrel{ad'}\longrightarrow f_\sharp i_*.
  \end{split}\]
  Here, the fourth and sixth natural transformations are defined and isomorphisms by (Stab), and the seventh natural transformation is an isomorphism by the above paragraph. To show that $\phi_1'$ is a right inverse of $f_\sharp i_*\stackrel{\mathfrak{p}_{f,D}^o}\longrightarrow f_!\tau i_*$, it suffices to check that the composition of the outer cycle of the diagram
  \[\begin{tikzcd}
    f_\sharp i_*\arrow[dd,leftarrow,"ad'"]\arrow[r,"\mathfrak{p}_{f,D}^o"]&f_!\tau i_*\arrow[dd,"ad'",leftarrow]\arrow[r,"\sim"]\arrow[rd,equal]&g_*\tau\arrow[rd,equal]\arrow[rr,"ad"]&& g_*g^*g_*\tau\arrow[dd,"Ex^{-1}"]\\
    &&f_!\tau i_*\arrow[r,"\sim"]&g_*\tau\arrow[ru,leftarrow,"ad'"]\\
    f_\sharp i_*g^*g_*\arrow[r,leftarrow,"(\mathfrak{p}_{f,D}^o)^{-1}"']&f_!\tau i_*g^*g_*\arrow[r,leftarrow,"Ex^{-1}"']\arrow[ru,"ad'"]&f_!i_*\tau g^*g_*\arrow[rr,leftarrow,"\sim"']&&g_*g^*\tau g_*
  \end{tikzcd}\]
  of functors is the identity. It is true since each small diagram commutes.
\end{proof}
\begin{prop}\label{9.3.8}
  Let $f:X\rightarrow S$ be a vertical exact log smooth separated morphism of $\mathscr{S}$-schemes with an fs chart $\theta:P\rightarrow Q$ where $\theta$ is a vertical homomorphism of exact log smooth type. Then $({\rm PD}^{m-1})$ implies $({\rm PD}_f)$.
\end{prop}
\begin{proof}
  By (\ref{9.3.10}) and \cite[II.2.3.2]{Ogu17}, we may assume that $f$ has a factorization
  \[X\stackrel{u}\rightarrow S_0\stackrel{v}\rightarrow S\]
  such that
  \begin{enumerate}[(i)]
    \item $v$ strict \'etale,
    \item the fs chart $S_0\rightarrow \mathbb{A}_P$ is exact at some point of $s_0\in S_0$,
    \item $s_0$ is in the image of $u$.
  \end{enumerate}
  By (\ref{9.2.6}), we can use (\ref{9.3.3}) for the morphisms $X\rightarrow S_0$ and $S_0\rightarrow S$, and by (\ref{9.4.2}), $({\rm PD}_{v})$ is satisfied. Hence replacing $S$ by $S_0$, we may assume that the fs chart $S\rightarrow \mathbb{A}_P$ is exact at some point of $s\in S$ and that $s$ is in the image of $f$.

  By assumption, the induced morphism
  \[X\rightarrow S\times_{\mathbb{A}_P}\mathbb{A}_Q\]
  is strict \'etale and separated. We denote by $P'$ the submonoid of $Q$ consisting of elements $q\in Q$ such that $nq\in P+Q^*$ for some $n\in \mathbb{N}^+$. Then the induced homomorphism $\theta':P'\rightarrow Q$ is locally exact, injective, logarithmic, and vertical, and the cokernel of $\theta^{\rm gp}$ is torsion free. In particular, the induced morphism
  \[S\times_{\mathbb{A}_P}\mathbb{A}_Q\rightarrow S\times_{\mathbb{A}_P}\mathbb{A}_{P'}\]
  is exact log smooth, so it is an open morphism by \cite[5.7]{Nak09}.

  We denote by $S'$ the image of $X$ via the composition
  \[X\rightarrow S\times_{\mathbb{A}_P}\mathbb{A}_Q\rightarrow S\times_{\mathbb{A}_P}\mathbb{A}_{P'}.\]
  we consider $S'$ as an open subscheme of $S\times_{\mathbb{A}_P}\mathbb{A}_{P'}$. Then the induced morphism $g:S'\rightarrow S$ has the fs chart $\theta':P\rightarrow Q'$ of Kummer log smooth type. Consider the factorization
  \[X\stackrel{g_3}\rightarrow S'\times_{\mathbb{A}_{P'}}\mathbb{A}_Q\stackrel{g_2}\rightarrow S'\stackrel{g_1}\rightarrow S\]
  of $f:X\rightarrow S$. Then $({\rm PD}_{g_1})$ and $({\rm PD}_{g_3})$ are satisfied by (\ref{9.4.2}) since $g_1$ and $g_3$ are Kummer log smooth and separated. The set $g_1^{-1}(a)$ is nonempty since $s$ is in the image of $f$, and the chart $S'\rightarrow \mathbb{A}_{P'}$ is exact at a point in $g_1^{-1}(s)$. Thus $({\rm PD}_{g_2})$ is satisfied by (\ref{9.3.7}), so $({\rm PD}_f)$ is satisfied by (\ref{9.2.6}) and (\ref{9.3.3}).
\end{proof}
\begin{thm}\label{9.3.9}
  Let $f:X\rightarrow S$ be a vertical exact log smooth separated morphism of $\mathscr{S}$-schemes with an fs chart $\theta:P\rightarrow Q$ where $\theta$ is a vertical homomorphism of exact log smooth type. Then $({\rm PD}_f)$ is satisfied.
\end{thm}
\begin{proof}
  It suffices to show $({\rm PD}^m)$ for any $m$. By (\ref{9.4.2}), $({\rm PD}^0)$ is satisfied. Then the conclusion follows from (\ref{9.3.8}) and induction on $m$.
\end{proof}
\begin{cor}\label{9.3.11}
  Under the notations and hypotheses of (\ref{9.3.9}), let $D$ be a compactified exactification of the diagonal morphism $a:X\rightarrow X\times_S X$. Then the composition
  \[\Omega_{f,D}f^!\stackrel{T_D}\longrightarrow \Omega_ff^!\stackrel{\mathfrak{q}_f}\longrightarrow f^*\]
  is an isomorphism.
\end{cor}
\begin{proof}
  By (\ref{9.3.9}), the composition
  \[\Omega_{f,D}^nf^!\stackrel{T^n}\longrightarrow \Omega_{f,D}^df^!\stackrel{T^d}\longrightarrow \Omega_{f,D}f^!\stackrel{T_D}\longrightarrow \Omega_ff^!\stackrel{\mathfrak{q}_f}\longrightarrow f^*\]
  is an isomorphism. By (\ref{4.2.1}), the first and second arrows are isomorphisms, so the composition of the third and fourth arrows are also an isomorphism.
\end{proof}
\subsection{Purity}
\begin{prop}\label{9.5.2}
  Let $f:X\rightarrow S$ be a vertical exact log smooth separated morphism of $\mathscr{S}$-schemes with an fs chart $\theta:P\rightarrow Q$ where $\theta$ is a vertical homomorphism of exact log smooth type. Consider a Cartesian diagram
  \[\begin{tikzcd}
    X'\arrow[d,"f'"]\arrow[r,"g'"]&X\arrow[d,"f"]\\
    S'\arrow[r,"g"]&S
  \end{tikzcd}\]
  of $\mathscr{S}$-schemes. Then the exchange transformation
  \[Ex:g^*f_!\longrightarrow f_!'g'^*\]
  is an isomorphism.
\end{prop}
\begin{proof}
  By (\ref{9.1.3}), there is a compactified exactification $D$ of the diagonal morphism $a:X\rightarrow X\times_S X$. Then by (\ref{9.2.7}), the diagram
  \[\begin{tikzcd}
    f_\sharp'g'^*\arrow[dd,"Ex"]\arrow[r,"\mathfrak{p}_{f'}^n"]&f_!'\Sigma_{f',D'}^n g'^!\arrow[d,"Ex"]\\
    &f_!'g'^*\Sigma_{f,D}^n\arrow[d,leftarrow,"Ex"]\\
    g^*f_\sharp\arrow[r,"\mathfrak{p}_f^n"]&g^*f_!\Sigma_{f,D}^n
  \end{tikzcd}\]
  of functors commutes. The left vertical arrow is an isomorphism by ($eSm$-BC), and the right upper vertical arrow is an isomorphism by (\ref{4.2.10}). The horizontal arrows are also isomorphisms by (\ref{9.3.9}). Thus the right lower vertical arrow is an isomorphism. Then the conclusion follows from the fact that the functor
  \[\Sigma_{f,D}^n\cong \Sigma_{f,D}^o\]
  is an equivalence of categories.
\end{proof}
\begin{none}\label{9.5.3}
  Under the notations and hypotheses of (\ref{4.2.5}), note that by (\ref{9.5.2}), the condition $(CE^!)$ is satisfied when $\eta$ in (loc.\ cit) is a vertical exact log smooth separated morphism of $\mathscr{S}$-schemes with an fs chart $\theta:P\rightarrow Q$ where $\theta$ is a vertical homomorphism of exact log smooth type.
\end{none}
\begin{prop}\label{9.5.1}
  Let $f:X\rightarrow S$ be a vertical exact log smooth separated morphism of $\mathscr{S}$-schemes with an fs chart $\theta:P\rightarrow Q$ where $\theta$ is a vertical homomorphism of exact log smooth type, and let $D$ be a compactified exactification of the diagonal morphism $a:X\rightarrow X\times_S X$. Then the transition transformation
  \[T_D:\Omega_{f,D}\longrightarrow \Omega_f\]
  is an isomorphism.
\end{prop}
\begin{proof}
  Let $v:I\rightarrow D$ be an interior of $D$. Consider the commutative diagram
  \[\begin{tikzcd}
    &D'\arrow[rd,"u'"']\arrow[rrrd,"q_2'"]\arrow[dd,"\rho",near start]\\
    X'\arrow[ru,"b'"]\arrow[dd,"p_2"]\arrow[rr,"a'"',near start,crossing over]&&Y\times_S X'\arrow[rr,"p_2'"']&&X'\arrow[dd,"p_2"]\\
    &D\arrow[rd,"u"']\arrow[rrrd,"q_2"]\\
    X\arrow[ru,"b"]\arrow[rr,"a"']&&Y\times_S X\arrow[uu,"\eta'",leftarrow,crossing over]\arrow[rr,"p_2"']&&X
  \end{tikzcd}\]
  of $\mathscr{S}$-schemes where $p_2$ denotes the second projection and each square is Cartesian. We denote by $v':I'\rightarrow D'$ the pullback of $v:I\rightarrow D$. Then $I'$ is also an interior of $D'$. By (\ref{9.5.3}) and (\ref{4.2.5}), we have the exchange transformations
  \[\Omega_{p_2}p_2^!\stackrel{Ex}\longrightarrow p_2^!\Omega_f,\quad \Omega_{p_2,D'}p_2^!\stackrel{Ex}\longrightarrow p_2^!\Omega_{f,D},\quad \Omega_{p_2,I'}p_2^!\stackrel{Ex}\longrightarrow p_2^!\Omega_{f,I},\]
  and we have the commutative diagram
  \begin{equation}\label{9.5.1.1}\begin{tikzcd}
    a^!\Omega_{f',I'}p_2^!\arrow[r,"T_{I,D}"]\arrow[d,"Ex"]&a^!\Omega_{p_2,D'}p_2^!\arrow[d,"Ex"]\arrow[r,"T_D'"]&a^!\Omega_{p_2}p_2^!\arrow[d,"Ex"]\arrow[rdd,"\mathfrak{q}_f"]\\
    a^!p_2^!\Omega_{f,I}\arrow[r,"T_{I,D}"]&a^!p_2^!\Omega_{f,D}\arrow[d,"\sim"]\arrow[r,"T_D"]&a^!p_2^!\Omega_f\arrow[d,"\sim"]\\
    &\Omega_{f,D}\arrow[r,"T_D"]&\Omega_f\arrow[r,equal]&a^!p_2^*
  \end{tikzcd}\end{equation}
  of functors. The natural transformations
  \[a^!p_2^!\Omega_{f,I}\stackrel{T_{I,D}}\longrightarrow a^!p_2^!\Omega_{f,D},\quad a^!\Omega_{f',I'}p_2^!\stackrel{T_{I,D}}\longrightarrow a^!\Omega_{f',D'}p_2^!\]
  are isomorphisms by (\ref{4.2.9}), and the natural transformation
  \[a^!\Omega_{f',I'}p_2^!\stackrel{Ex}\longrightarrow a^!p_2^!\Omega_{f,I}\]
  is an isomorphism by (\ref{4.2.13}) since $r_2=q_2v$ is strict by definition of interior. The composition
  \[a^!\Omega_{p_2,D'}^np_2^!\stackrel{T^n}\longrightarrow a^!\Omega_{p_2,D'}^dp_2^!\stackrel{T^d}\longrightarrow  a^!\Omega_{p_2,D'}p_2^!\stackrel{T_{D'}}\longrightarrow a^!\Omega_{p_2}p_2^!\stackrel{q_f}\longrightarrow a^!p_2^*\]
  is also an isomorphism by (\ref{9.3.11}). Applying these to (\ref{9.5.1.1}), we conclude that the natural transformation
  \[T_D:\Omega_{f,D}\longrightarrow \Omega_f\]
  is an isomorphism.
\end{proof}
\begin{prop}\label{9.5.4}
  Let $f:X\rightarrow S$ be a vertical exact log smooth separated morphism of $\mathscr{S}$-schemes. Then $({\rm Pur}_f)$ is satisfied.
\end{prop}
\begin{proof}
  We want to show that the purity transformation
  \[\mathfrak{q}_f:\Omega_ff^!\longrightarrow f^*\]
  is an isomorphism. It is equivalent to showing that the natural transformation
  \[\mathfrak{p}_f:f_\sharp \longrightarrow f_!\Sigma_f\]
  is an isomorphism.\\[4pt]
  (I) {\it Locality on $S$.} Let $\{u_i:S_i\rightarrow S\}_{i\in I}$ be a strict \'etale separated cover of $S$. Consider the Cartesian diagram
  \[\begin{tikzcd}
    X_i\arrow[d,"f_i"]\arrow[r,"u_i'"]&X\arrow[d,"f"]\\
    S_i\arrow[r,"u_i"]&S
  \end{tikzcd}\]
  of $\mathscr{S}$-schemes. Then by (\ref{4.4.6}), the diagram
  \[\begin{tikzcd}
    f_{i\sharp}u_i'^*\arrow[dd,"Ex"]\arrow[r,"\mathfrak{p}_{f_i}"]&f_{i!}\Sigma_{f_i} u_i'^*\arrow[d,"Ex"]\\
    &f_{i!}u_i'^*\Sigma_f\arrow[d,leftarrow,"Ex"]\\
    u_i^*f_\sharp\arrow[r,"\mathfrak{p}_f"]&u_i^*f_!\Sigma_f
  \end{tikzcd}\]
  of functors commutes. The left vertical arrow is an isomorphism by (eSm-BC), and the right lower vertical arrow is an isomorphism since $u_i$ is exact log smooth. The right upper vertical arrow is also an isomorphism by (\ref{2.5.9}). Thus the upper horizontal arrow is an isomorphism if and only if the lower horizontal arrow is an isomorphism.

  Since the family of functors $\{u_i^*\}_{i\in I}$ is conservative by (k\'et-sep), the lower horizontal arrow is an isomorphism if and only if the natural transformation
  \[\mathfrak{p}_f:f_\sharp\longrightarrow f_!\Sigma_f\]
  is an isomorphism. Thus we have proven that the question is strict \'etale separated local on $S$.\\[4pt]
  (II) {\it Locality on $X$.} Let $\{v_i:X_i\rightarrow X\}_{i\in J}$ be a strict \'etale separated cover of $X$. By (\ref{4.4.3}), we have the commutative diagram
  \[\begin{tikzcd}
    \Omega_{v_i}v_i^!\Omega_f f^!\arrow[r,"\mathfrak{q}_{v_i}"]\arrow[d,"Ex",leftarrow]&v_i^*\Omega_f f^!\arrow[r,"\mathfrak{q}_f"]&v_i^*f^*\arrow[ddd,"\sim"]\\
    \Omega_{v_i}\Omega_{f,fv_i}v_i^!f^!\arrow[d,"\sim"]\\
    \Omega_{v_i}\Omega_{f,fv_i}(fv_i)^!\arrow[d,"C"]\\
    \Omega_{fv_i}(fv_i)^!\arrow[rr,"\mathfrak{q}_{fv_i}"]&&(fv_i)^*
  \end{tikzcd}\]
  of functors. The left top vertical arrow is an isomorphism by (\ref{2.5.9}), and the left bottom vertical arrow is an isomorphism by (\ref{4.3.1}). The upper left horizontal arrow is also an isomorphism by (\ref{2.5.7}), so the upper right horizontal arrow is an isomorphism if and only if the lower horizontal arrow is an isomorphism.

  Since the family of functors $\{v_i^*\}_{i\in J}$ is conservative, the lower horizontal arrow is an isomorphism if and only if the natural transformation
  \[\mathfrak{q}_f:\Omega_f f^!\longrightarrow f^*\]
  is an isomorphism. Thus we have proven that the question is strict \'etale separated local on $X$.\\[4pt]
  (III) {\it Final step of the proof.} Since the question is strict \'etale separated local on $X$ and $S$, we may assume that $f:X\rightarrow S$ has an fs chart $\theta:P\rightarrow Q$ of exact log smooth type by (\ref{0.1.4}) (in (loc.\ cit), if we localize $X$ and $S$ further so that $\underline{X}$ and $\underline{S}$ are affine, then the argument is strict \'etale separated local instead of strict \'etale local). Localizing $Q$ further, since $f$ is vertical, we may assume that $\theta$ is vertical.

  By (\ref{9.1.3}), there is a compactified exactification $D$ of the diagonal morphism $a:X\rightarrow X\times_S X$. Then we have the natural transformation
  \[\Omega_{f,D}f^!\stackrel{T_D}\longrightarrow \Omega_f f^!\stackrel{\mathfrak{q}_f}\longrightarrow f^!.\]
  The composition is an isomorphism by (\ref{9.3.11}), and the first arrow is an isomorphism by (\ref{9.5.1}). Thus the second arrow is an isomorphism.
\end{proof}
\begin{thm}\label{9.5.5}
  Let $f:X\rightarrow S$ be an exact log smooth separated morphism of $\mathscr{S}$-schemes. Then $({\rm Pur}_f)$ is satisfied.
\end{thm}
\begin{proof}
  Let $j:U\rightarrow X$ denote the verticalization of $f$. By (\ref{4.4.3}), the diagram
  \[\begin{tikzcd}
    \Omega_jj^!\Omega_f f^!\arrow[r,"\mathfrak{q}_j"]\arrow[d,"Ex",leftarrow]&j^*\Omega_f f^!\arrow[r,"\mathfrak{q}_f"]&j^*f^*\arrow[ddd,"\sim"]\\
    \Omega_j\Omega_{f,fj}j^!f^!\arrow[d,"\sim"]\\
    \Omega_j\Omega_{f,fj}(fj)^!\arrow[d,"C"]\\
    \Omega_{fj}(fj)^!\arrow[rr,"\mathfrak{q}_f"]&&(fj)^*
  \end{tikzcd}\]
  of functors commutes. The left top vertical arrow is an isomorphism by (\ref{2.5.7}), and the left bottom vertical arrow is an isomorphism by (\ref{4.3.1}). The upper left horizontal arrow is an isomorphism by (\ref{2.5.7}), and the lower horizontal arrow is an isomorphism by (\ref{9.5.4}). Thus the upper right horizontal arrow is an isomorphism.

  Then consider the commutative diagram
  \[\begin{tikzcd}
    \Omega_ff^!\arrow[d,"ad"]\arrow[r,"\mathfrak{q}_f"]&f^*\arrow[d,"ad"]\\
    j_*j^*\Omega_f f^!\arrow[r,"\mathfrak{q}_f"]&j_*j^*f^*
  \end{tikzcd}\]
  of functors. We have shown that the lower horizontal arrow is an isomorphism. Since the right vertical arrow is an isomorphism by (Htp--2), the remaining is to show that the left vertical arrow is an isomorphism.

  Consider the commutative diagram
  \[\begin{tikzcd}
    U\arrow[d,"a'"]\arrow[r,"j"]&X\arrow[d,"a"]\\
    U\times_S X\arrow[d,"p_2'"]\arrow[r,"j'"]&X\times_S X\arrow[d,"p_2"]\\
    U\arrow[r,"j"]&X
  \end{tikzcd}\]
  of $\mathscr{S}$-schemes where
  \begin{enumerate}[(i)]
    \item $p_2$ denotes the second projection.
    \item $a$ denotes the diagonal morphism,
    \item each square is Cartesian.
  \end{enumerate}
  Then $j'$ is the verticalization of $p_2$, so by (Htp--2), the natural transformation
  \[p_2^*\stackrel{ad}\longrightarrow j_*'j'^*p_2^*\]
  is an isomorphism. Consider the natural transformations
  \[\Omega_f \stackrel{\sim}\longrightarrow a^!p_2^*\stackrel{ad}\longrightarrow  a^!j_*'j'^*p_2^*\stackrel{Ex}\longleftarrow j_*a'^!j'^*p_2^*\stackrel{Ex}\longleftarrow j_*j^*a^!p_2^*\stackrel{\sim}\longleftarrow j_*j^*\Omega_f.\]
  We have shown that the second arrow is an isomorphism. The third arrow is an isomorphism by (eSm-BC), and the fourth arrow is an isomorphism by (Supp). This completes the proof.
\end{proof}
\subsection{Purity transformations}
\begin{df}\label{9.6.1}
  Let $i:(\mathscr{X},I)\rightarrow (\mathscr{Y},I)$ be a Cartesian strict regular embedding of $\mathscr{S}$-diagrams. For any morphism $\lambda\rightarrow \mu$ in $I$, there are induced morphisms
  \[D_{\mathscr{X}_\lambda}\mathscr{Y}_\lambda\rightarrow D_{\mathscr{X}_\mu}\mathscr{Y}_\mu,\quad N_{\mathscr{X}_\lambda}\mathscr{Y}_\lambda\rightarrow N_{\mathscr{X}_\mu}\mathscr{Y}_\mu\]
  of $\mathscr{S}$-schemes. Using these, we have the following $\mathscr{S}$-schemes.
  \begin{enumerate}[(1)]
    \item $D_\mathscr{X}\mathscr{Y}$ denotes the $\mathscr{S}$-diagram constructed by $D_{\mathscr{X}_\lambda}\mathscr{Y}_\lambda$ for $\lambda\in I$,
    \item $N_\mathscr{X}\mathscr{Y}$ denotes the $\mathscr{S}$-diagram constructed by $N_{\mathscr{X}_\lambda}\mathscr{Y}_\lambda$ for $\lambda\in I$.
  \end{enumerate}
  Note that if the induced morphism $\underline{\mathscr{Y}_\lambda}\rightarrow \underline{\mathscr{Y}_\mu}$ is flat for any morphism $\lambda\rightarrow \mu$ in $I$, then the induced morphisms $\mathscr{X}\rightarrow D_\mathscr{X}\mathscr{Y}$ and $\mathscr{X}\rightarrow N_\mathscr{X}\mathscr{Y}$ are Cartesian strict regular embeddings by \cite[B.7.4]{Ful98}.
\end{df}
\begin{df}\label{9.6.2}
  Let $f:X\rightarrow S$ be an exact log smooth morphism of $\mathscr{S}$-schemes, and let $h:\mathscr{X}\rightarrow X$ be a morphism of $\mathscr{S}$-diagrams. Then we denote by
  \[N_\mathscr{X}(X\times_S \mathscr{X})\]
  the vector bundle of $\mathscr{X}$ associated to the dual free sheaf $(h^*\Omega_{X/S})^\vee$. Note that when the induced morphism $\mathscr{X}\rightarrow X\times_S \mathscr{X}$ is a Cartesian strict regular embedding, this definition is equivalent to the definition in (\ref{9.6.1}).
\end{df}
\begin{none}\label{9.6.3}
  Let $f:X\rightarrow S$ be a separated vertical exact log smooth morphism of $\mathscr{S}$-schemes. We will construct several $\mathscr{S}$-diagrams and their morphisms.\\[4pt]
  (1) {\it Construction of $\mathscr{X}$.} Let $\{h_\lambda:\mathscr{X}_\lambda\rightarrow X\}_{\lambda\in I_0}$ be a strict \'etale cover such that there is a commutative diagram
  \[\begin{tikzcd}
    \mathscr{X}_\lambda\arrow[r,"h_\lambda"]\arrow[d,"f_\lambda"]&X\arrow[d,"f"]\\
    S_\lambda\arrow[r,"l_\lambda"]&S
  \end{tikzcd}\]
    of $\mathscr{S}$-schemes where
    \begin{enumerate}[(i)]
      \item $\mathscr{X}_\lambda$ and $S_\lambda$ are affine,
      \item $f_\lambda$ has an fs chart $\theta_\lambda:P_\lambda\rightarrow Q_\lambda$ of exact log smooth type,
      \item $l_\lambda$ is strict \'etale.
    \end{enumerate}
  Then we denote by $\mathscr{X}=(\mathscr{X},I)$ the \v{C}ech hypercover associated to $\{h_\lambda:\mathscr{X}_\lambda\rightarrow X\}_{\lambda\in I_0}$.\\[4pt]
  (2) {\it Construction of $\mathscr{D}$.} For $\lambda\in I_0$, we denote by $h_\lambda'$ the induced morphism
  \[X\times_S \mathscr{X}_\lambda\rightarrow X\times_S X\]
  of $\mathscr{S}$-schemes, and let $U_\lambda$ denote the open subscheme
  \[X\times_S \mathscr{X}_\lambda-(h_\lambda')^{-1}(a(X)).\]
  Then we denote by $\mathscr{D}_\lambda$ the \v{C}ech hypercover of $X\times_S \mathscr{X}_\lambda$ associated to
  \[\{\mathscr{X}_\lambda\times_{S_\lambda}\mathscr{X}_\lambda\rightarrow X\times_S \mathscr{X}_\lambda,\;U_\lambda\rightarrow X\times_S \mathscr{X}_\lambda\},\]
  and we denote by $\mathscr{D}=(\mathscr{D},J)$ the \v{C}ech hypercover of $X\times_S X$ associated to
  \[\{\mathscr{D}_\lambda\rightarrow X\times_S X\}_{\lambda\in I_0}.\]
  Note that from our construction, we have the morphism
  \[u:\mathscr{D}\rightarrow X\times_S \mathscr{X}\]
  of $\mathscr{S}$-diagrams.\\[4pt]
  (3) {\it Construction of $\mathscr{E}$.} For $\lambda\in I_0$, we put $Y_\lambda=\mathscr{X}_\lambda\times_{S_\lambda\times_S X}\mathscr{X}_\lambda$. Then the composition
  \[Y_\lambda\rightarrow \mathscr{X}_\lambda\rightarrow \mathbb{A}_{Q_\lambda}\]
  where the first arrow is the first projection gives an fs chart of $Y_\lambda$. The induced morphism $Y_\lambda\rightarrow S_\lambda$ also has an fs chart $P_\lambda\rightarrow Q_\lambda$. As in (\ref{9.1.3}), choose a proper birational morphism
  \[M_\lambda\rightarrow {\rm spec}(Q_\lambda\oplus_{P_\lambda}Q_\lambda)\]
  of fs monoschemes, and we put
  \[E_\lambda=(\mathscr{X}_\lambda\times_{S_\lambda}\mathscr{X}_\lambda)\times_{\mathbb{A}_{Q_\lambda\oplus_{P_\lambda}Q_\lambda}}\mathbb{A}_{M_\lambda},\]
  and let $u_\lambda'':E_\lambda\rightarrow \mathscr{X}_\lambda\times_{S_\lambda}\mathscr{X}_\lambda$ denote the projection. Note that the diagonal morphism $\mathscr{X}_\lambda\rightarrow \mathscr{X}_\lambda\times_{S_\lambda}\mathscr{X}_\lambda$ factors through $E_\lambda$ by construction in (loc.\ cit). Let $b_\lambda'':\mathscr{X}\rightarrow E_\lambda$ denote the factorization. We will show that the projection
  \[Y_\lambda\times_{\mathscr{X}_\lambda\times_{S_\lambda}\mathscr{X}_\lambda}E_\lambda\rightarrow Y_\lambda\]
  is an isomorphism. Consider the commutative diagram
  \[\begin{tikzcd}
    &E_\lambda'\arrow[r]\arrow[d]&E_\lambda\arrow[d]\arrow[r]&\mathbb{A}_{M_\lambda}\arrow[d]\\
    Y_\lambda\arrow[r]\arrow[d]&Y_\lambda\times_{S_\lambda}Y_\lambda\arrow[d,"\iota_1"]\arrow[r]&\mathscr{X}_\lambda\times_{S_\lambda}\mathscr{X}_\lambda\arrow[r]&\mathbb{A}_{Q_\lambda \oplus_{P_\lambda}Q_\lambda}\\
    \mathbb{A}_{M_\lambda}\arrow[r]&\mathbb{A}_{Q_\lambda\oplus_{P_\lambda}Q_\lambda}
  \end{tikzcd}\]
  of $\mathscr{S}$-diagrams where
  \begin{enumerate}[(i)]
    \item $\iota_1$ denotes the fs chart induced by the fs charts $P_\lambda\rightarrow Q_\lambda$ of $Y_\lambda\rightarrow S_\lambda$ defined above,
    \item the arrow $Y_\lambda\times_{S_\lambda}Y_\lambda\rightarrow \mathscr{X}_\lambda\times_{S_\lambda}\mathscr{X}_\lambda$ is the morphism induced by the first projection $Y_\lambda\rightarrow \mathscr{X}_\lambda$, the identity $S_\lambda\rightarrow S_\lambda$, and the second projection $Y_\lambda\rightarrow \mathscr{X}_\lambda$,
    \item $E_\lambda'=(Y_\lambda\times_{S_\lambda}Y_\lambda)\times_{\mathscr{X}_\lambda\times_{S_\lambda}\mathscr{X}_\lambda}E_\lambda$.
  \end{enumerate}
  By (\ref{0.2.3}), we have an isomorphism
  \[(Y_\lambda\times_{S_\lambda}Y_\lambda)\times_{\iota_1,\mathbb{A}_{Q_\lambda\oplus_{P_\lambda}Q_\lambda}}\mathbb{A}_M\cong E_\lambda',\]
  and this shows the assertion since the morphism $\mathbb{A}_{M_\lambda}\rightarrow \mathbb{A}_{Q_\lambda\oplus_{P_\lambda}Q_\lambda}$ is a monomorphism of fs log schemes.

  Now, we denote by $\mathscr{E}_\lambda$ the \v{C}ech hypercover of $X\times_S \mathscr{X}_\lambda$ associated to
  \[\{E_\lambda\rightarrow X\times_S \mathscr{X}_\lambda,\;U_\lambda\rightarrow X\times_S \mathscr{X}_\lambda\},\]
  and we denote by $\mathscr{E}=(\mathscr{E},J)$ the \v{C}ech hypercover of $X\times_S X$ associated to
  \[\{\mathscr{E}_\lambda\rightarrow X\times_S X\}_{\lambda\in I_0}.\]
  Note that from our construction, we have the morphism
  \[v:\mathscr{E}\rightarrow \mathscr{D}\]
  of $\mathscr{S}$-diagrams. We put
  \[\mathscr{Y}=\mathscr{X}\times_{X\times_S \mathscr{X}}\mathscr{D}.\]
  Then the assertion in the above paragraph shows that the projection $\mathscr{Y}\times_{\mathscr{D}}\mathscr{E}\rightarrow \mathscr{Y}$ is an isomorphism, so the projection $b:\mathscr{Y}\rightarrow \mathscr{D}$ factors through $c:\mathscr{Y}\rightarrow \mathscr{E}$.\\[4pt]
  (4) {\it Commutative diagrams.} Now, we have the commutative diagram
  \begin{equation}\label{9.6.3.1}\begin{tikzcd}
    &\mathscr{E}\arrow[rdd,"r_2"]\arrow[d,"v"]\\
    \mathscr{Y}\arrow[ru,"c"]\arrow[r,"b"]\arrow[d,"u_0"]&\mathscr{D}\arrow[d,"u"]\arrow[rd,"q_2"]\\
    \mathscr{X}\arrow[r,"a'"']\arrow[d,"h"]&X\times_S \mathscr{X}\arrow[d,"h'"]\arrow[r,"p_2'"']&\mathscr{X}\arrow[d,"h"]\\
    X\arrow[r,"a"']&X\times_S X\arrow[r,"p_2"']&X
  \end{tikzcd}\end{equation}
  of $\mathscr{S}$-diagrams where
  \begin{enumerate}[(i)]
    \item each small square is Cartesian,
    \item $u$, $v$, and $c$ are the morphisms constructed above.
  \end{enumerate}
  As in (\ref{4.1.1}), we also have the commutative diagrams
  \[\begin{tikzcd}
    \mathscr{Y}\arrow[d,"\gamma_1"]\arrow[r,"c"]&\mathscr{E}\arrow[d,"\beta_1"]\arrow[r,"r_2"]&\mathscr{X}\arrow[d,"\alpha_1"]\\
    \mathscr{Y}\times\mathbb{A}^1\arrow[d,"\phi"]\arrow[r,"d"]&D_{\mathscr{Y}}\mathscr{E}\arrow[r,"s_2"]&\mathscr{X}\times\mathbb{A}^1\arrow[d,"\pi"]\\
    \mathscr{Y}&&\mathscr{X}
  \end{tikzcd}\quad
  \begin{tikzcd}
    \mathscr{Y}\arrow[d,"\gamma_0"]\arrow[r,"e"]&\mathscr{E}\arrow[d,"\beta_0"]\arrow[r,"t_2"]&\mathscr{X}\arrow[d,"\alpha_0"]\\
    \mathscr{Y}\times\mathbb{A}^1\arrow[d,"\phi"]\arrow[r,"d"]&D_{\mathscr{Y}}\mathscr{E}\arrow[r,"s_2"]&\mathscr{X}\times\mathbb{A}^1\arrow[d,"\pi"]\\
    \mathscr{Y}&&\mathscr{X}
  \end{tikzcd}\]
  of $\mathscr{S}$-diagrams where
  \begin{enumerate}[(i)]
    \item each square is Cartesian,
    \item $\alpha_0$ denotes the 0-section, and $\alpha_1$ denotes the 1-section,
    \item $d$ and $s_2$ are the morphisms constructed as in (\ref{4.1.1.1}),
    \item $\pi$ and $\phi$ denotes the projections.
  \end{enumerate}
  Then we have the commutative diagram
  \[\begin{tikzcd}
    \mathscr{Y}\arrow[d,"u_0"]\arrow[r,"e"]&N_\mathscr{Y}\mathscr{E}\arrow[d,"u_1"]\arrow[rd,"t_2"]\\
    \mathscr{X}\arrow[d,"h"]\arrow[r,"e'"]&N_\mathscr{X}(X\times_S \mathscr{X})\arrow[d,"h_1"]\arrow[r,"t_2'"]&\mathscr{X}\arrow[d,"h"]\\
    X\arrow[r,"e''"]&N_X(X\times_S X)\arrow[r,"t_2''"]&X
  \end{tikzcd}\]
  of $\mathscr{S}$-diagrams where
  \begin{enumerate}[(i)]
    \item each small square is Cartesian,
    \item $e''$ denotes the $0$-section, and $t_2''$ denotes the projection.
  \end{enumerate}
  For $\lambda\in I$, we also have the corresponding commutative diagrams
  \[\begin{tikzcd}
    &\mathscr{E}_\lambda\arrow[rdd,"r_{2\lambda}"]\arrow[d,"v_\lambda"]\\
    \mathscr{Y}_\lambda\arrow[ru,"c_\lambda"]\arrow[r,"b_\lambda"']\arrow[d,"u_{0\lambda}"]&\mathscr{D}_\lambda\arrow[d,"u_\lambda"]\arrow[rd,"q_{2\lambda}"']\\
    \mathscr{X}_\lambda\arrow[r,"a_\lambda'"']&X\times_S \mathscr{X}_\lambda\arrow[r,"p_{2\lambda}'"']&\mathscr{X}_\lambda\\
  \end{tikzcd}\]
  \[\begin{tikzcd}
    \mathscr{Y}_\lambda\arrow[d,"\gamma_{1\lambda}"]\arrow[r,"c_\lambda"]&\mathscr{E}_\lambda\arrow[d,"\beta_{1\lambda}"]\arrow[r,"r_{2\lambda}"]&\mathscr{X}_\lambda \arrow[d,"\alpha_{1\lambda}"]\\
    \mathscr{Y}_\lambda\times\mathbb{A}^1\arrow[d,"\phi_\lambda"]\arrow[r,"d_\lambda"]&D_{\mathscr{Y}_\lambda}\mathscr{E}_\lambda\arrow[r,"s_{2\lambda}"] &\mathscr{X}_\lambda\times\mathbb{A}^1\arrow[d,"\pi_\lambda"]\\
    \mathscr{Y}_\lambda&&\mathscr{X}_\lambda
  \end{tikzcd}\quad
  \begin{tikzcd}
    \mathscr{Y}_\lambda\arrow[d,"\gamma_{0\lambda}"]\arrow[r,"e_\lambda"]&N_{\mathscr{Y}_\lambda}\mathscr{E}_\lambda\arrow[d,"\beta_{0\lambda}"]\arrow[r,"t_{2\lambda}"]&\mathscr{X}_\lambda \arrow[d,"\alpha_{0\lambda}"]\\
    \mathscr{Y}_\lambda\times\mathbb{A}^1\arrow[d,"\phi_\lambda"]\arrow[r,"d_\lambda"]&D_{\mathscr{Y}_\lambda}\mathscr{E}_\lambda\arrow[r,"s_{2\lambda}"] &\mathscr{X}_\lambda\times\mathbb{A}^1\arrow[d,"\pi_\lambda"]\\
    \mathscr{Y}_\lambda&&\mathscr{X}_\lambda
  \end{tikzcd}\]
  \[\begin{tikzcd}
    \mathscr{Y}_\lambda\arrow[d,"u_{0\lambda}"]\arrow[r,"e_\lambda"]&N_{\mathscr{Y}_\lambda}\mathscr{E}_\lambda\arrow[d,"u_{1\lambda}"]\arrow[rd,"t_{2\lambda}"]\\
    \mathscr{X}_\lambda\arrow[r,"e_\lambda'"]&N_{\mathscr{X}_\lambda}(X\times_S \mathscr{X}_\lambda)\arrow[r,"t_{2\lambda}'"]&\mathscr{X}_\lambda
  \end{tikzcd}\]
  of $\mathscr{S}$-schemes. We also put
  \[g=fh,\quad g_\lambda=fh_\lambda,\]
\end{none}
\begin{none}\label{9.6.9}
  Under the notations and hypotheses of (\ref{9.6.3}), we have an isomorphism $N_{\mathscr{Y}}\mathscr{E}\cong N_{\mathscr{X}}(X\times_S \mathscr{X})\times_\mathscr{X} \mathscr{Y}$ by \cite[IV.1.2.15]{Ogu17}. In particular, the morphism $u_1:N_{\mathscr{Y}}\mathscr{E}\rightarrow N_{\mathscr{X}}(X\times_S \mathscr{X})$ is a strict \'etale hypercover. We have the natural transformation
  \[\Omega_{f,g}^n\stackrel{T^{n'}}\longrightarrow \Omega_{f,g,\mathscr{E}}^n\]
  given by
  \[e'^!t_2'^*\stackrel{ad}\longrightarrow e'^!u_{1*}u_1^*t_2'^*\stackrel{\sim}\longrightarrow e'^!u_{1*}t_2^*\stackrel{Ex^{-1}}\longrightarrow u_{0*}e^!t_2^*.\]
  Here, the first arrow is an isomorphism since $\mathscr{T}$ satisfies strict \'etale descent, and the third arrow is defined and an isomorphism by (\ref{10.2.10}) since $e'$ is a Cartesian strict closed immersion. Thus the composition is an isomorphism.

  We similarly have the natural transformation
  \[\Omega_{f,g_\lambda}^n\stackrel{T^{n'}}\longrightarrow \Omega_{f,g_\lambda,\mathscr{E}_\lambda}^n,\]
  which is also an isomorphism.
\end{none}
\begin{none}\label{9.6.4}
  Under the notations and hypotheses of (\ref{9.6.3}), for $\lambda\in I_0$, we temporary put
  \[A_\lambda= X\times_S \mathscr{X}_\lambda,\quad B_\lambda =\mathscr{X}_\lambda\times_{S_\lambda}\mathscr{X}_\lambda\]
  for simplicity. We had the Cartesian diagram
  \[\begin{tikzcd}
    E_\lambda\arrow[d]\arrow[r]&\mathbb{A}_{M_\lambda}\arrow[d]\\
    B_\lambda\arrow[r]&\mathbb{A}_{Q_\lambda\oplus_{P_\lambda}Q_\lambda}
  \end{tikzcd}\]
  of $\mathscr{S}$-schemes. Consider the commutative diagram
  \[\begin{tikzcd}
    E_\lambda\times_{A_\lambda}E_\lambda\arrow[d]\arrow[r,"\zeta_2''"]&E_\lambda\times_{A_\lambda}B_\lambda\arrow[r]\arrow[d,"\zeta_1'"]&\mathbb{A}_{M_\lambda}\arrow[d]\\
    B_\lambda\times_{A_\lambda}E_\lambda\arrow[d]\arrow[r,"\zeta_2'"]&B_\lambda\times_{A_\lambda}B_\lambda\arrow[d,"\zeta_1"]\arrow[r,"\zeta_2"]&\mathbb{A}_{Q_\lambda\oplus_{P_\lambda} Q_\lambda}\\
    \mathbb{A}_{M_\lambda}\arrow[r]&\mathbb{A}_{Q_\lambda\oplus_{P_\lambda}Q_\lambda}
  \end{tikzcd}\]
  of $\mathscr{S}$-schemes where
  \begin{enumerate}[(i)]
    \item each square is Cartesian,
    \item $\zeta_1$ denotes the composition $B_\lambda\times_{A_\lambda}B_\lambda\rightarrow B_\lambda\rightarrow \mathbb{A}_{Q_\lambda\oplus_{P_\lambda}Q_\lambda}$ where the first arrow is the first projection,
    \item $\zeta_2$ denotes the composition $B_\lambda\times_{A_\lambda}B_\lambda\rightarrow B_\lambda\rightarrow \mathbb{A}_{Q_\lambda\oplus_{P_\lambda}Q_\lambda}$ where the first arrow is the second projection.
  \end{enumerate}
  By (\ref{0.2.3}), we have isomorphisms
  \[B_\lambda\times_{A_\lambda}E_\lambda\cong B_\lambda\times_{\zeta_1,\mathbb{A}_{Q_\lambda\oplus_{P_\lambda}Q_\lambda}}\mathbb{A}_M\cong B_\lambda\times_{\zeta_2,\mathbb{A}_{Q_\lambda\oplus_{P_\lambda}Q_\lambda}}\mathbb{A}_M\cong E_\lambda\times_{A_\lambda}B_\lambda,\]
  so using this, we have the Cartesian diagram
  \[\begin{tikzcd}
    E_\lambda\times_{A_\lambda}E_\lambda\arrow[r,"\zeta_2''"]\arrow[d]&E_\lambda\times_{A_\lambda}B_\lambda\arrow[d,"\zeta_1'"]\\
    E_\lambda\times_{A_\lambda}B_\lambda\arrow[r,"\zeta_1'"]&B_\lambda\times_{A_\lambda}B_\lambda
  \end{tikzcd}\]
  of $\mathscr{S}$-schemes. Since $\zeta_1'$ is a pullback of $\mathbb{A}_M\rightarrow \mathbb{A}_{Q_\lambda\oplus_{P_\lambda}Q_\lambda}$ that is a monomorphism, the morphism $\zeta_1''$ is an isomorphism. From this, we conclude that the induced morphism
  \[E_\lambda\times_{X\times_S \mathscr{X}_\lambda}\mathscr{E}_\lambda\rightarrow E_\lambda\times_{X\times_S \mathscr{X}_\lambda}\mathscr{D}_\lambda\]
  is an isomorphism for $\lambda\in I_0$.

  Now, for $\lambda\in I$ instead of $\lambda\in I_0$, if $D_\lambda=D_{\lambda_1}\times_{X\times_S X}\cdots \times_{X\times_S X}D_{\lambda_r}$ for some $\lambda_1,\ldots,\lambda_r\in I_0$, we put
  \[E_\lambda=E_{\lambda_1}\times_{X\times_S X}\cdots \times_{X\times_S X}E_{\lambda_r}.\]
  From the result in the above paragraph, we see that the induced morphism
  \[E_\lambda\times_{X\times_S \mathscr{X}_\lambda}\mathscr{E}_\lambda\rightarrow E_\lambda\times_{X\times_S \mathscr{X}_\lambda}\mathscr{D}_\lambda\]
  is an isomorphism. In particular, the the first projection
  \[E_\lambda\times_{X\times_S \mathscr{X}_\lambda}\mathscr{E}_\lambda\rightarrow E_\lambda\]
  is a strict \'etale \v{C}ech hypercover.
\end{none}
\begin{none}\label{9.6.5}
  Under the notations and hypotheses of (\ref{9.6.4}), consider the commutative diagram
  \[\begin{tikzcd}
    &&\mathscr{E}_\lambda\arrow[rd,"v_\lambda"]\\
    \mathscr{Y}_\lambda\arrow[r,"c_\lambda''"]\arrow[d,"u_{0\lambda}"]&\mathscr{E}_\lambda'\arrow[ru,"w_\lambda''"]\arrow[d,"v_\lambda''"]\arrow[r,"v_\lambda'"] &\mathscr{D}_\lambda'\arrow[d,"u_\lambda'"] \arrow[r,"w_\lambda'"]&\mathscr{D}_\lambda\arrow[d,"u_\lambda"]\arrow[rd,"q_{2\lambda}"]\\
    \mathscr{X}_\lambda\arrow[r,"b_\lambda''"]&E_\lambda\arrow[r,"u_\lambda''"]&\mathscr{X}_\lambda\times_{S_\lambda}\mathscr{X}_\lambda\arrow[r,"w_\lambda"]&X\times_S \mathscr{X}_\lambda \arrow[r,"p_{2\lambda}'"]&\mathscr{X}_\lambda
  \end{tikzcd}\]
  of $\mathscr{S}$-diagrams where each small square is Cartesian and $w_\lambda$ denotes the induced morphism. Then we have the commutative diagram
  \begin{equation}\label{9.6.5.1}\begin{tikzcd}
    \Omega_{f,g_\lambda,\mathscr{E}_\lambda'}^n\arrow[d,"T_{E_\lambda,\mathscr{E}_\lambda'}"]\arrow[r,leftarrow,"(T^n)^{-1}"]&\Omega_{f,g,\mathscr{E}_\lambda'}^d \arrow[d,"T_{E_\lambda,\mathscr{E}_\lambda'}"]\arrow[r,"T^d"]&\Omega_{f,g,\mathscr{E}_\lambda'}\arrow[d,"T_{E_\lambda,\mathscr{E}_\lambda'}"] \arrow[rr,"T_{\mathscr{D}_\lambda',\mathscr{E}_\lambda'}"]&&\Omega_{f,g_\lambda,\mathscr{D}_\lambda'}\arrow[d,"T_{\mathscr{X}_\lambda \times_{S_\lambda}\mathscr{X}_\lambda, D_\lambda'}"]\arrow[rr,"T_{\mathscr{D}_\lambda,\mathscr{D}_\lambda'}"]&&\Omega_{f,g_\lambda,\mathscr{D}_\lambda}\arrow[d,"T_{\mathscr{D}_\lambda}"]\\
    \Omega_{f,g_\lambda,E_\lambda}^n\arrow[r,leftarrow,"(T^n)^{-1}"']&\Omega_{f,g,E_\lambda}^d\arrow[r,"T^d"']&\Omega_{f,g,E_\lambda} \arrow[rr,"T_{\mathscr{X}_\lambda\times_{S_\lambda} \mathscr{X}_\lambda,E_\lambda}"']&&\Omega_{f,g_\lambda,\mathscr{X}_\lambda\times_{S_\lambda}\mathscr{X}_\lambda} \arrow[rr,"T_{\mathscr{X}_\lambda\times_{S_\lambda}\mathscr{X}_\lambda}"'] && \Omega_{f,g_\lambda}
  \end{tikzcd}\end{equation}
  of functors. Here, the arrows are defined by the $\mathscr{S}$-diagram versions of (\ref{4.2.2}), (\ref{4.2.1}), and (\ref{4.2.11}). Since $u_\lambda$ is an exact log smooth morphism and $a_\lambda'$ is reduced, the exchange transformation
  \[u_{0\lambda*}b_\lambda^!\stackrel{Ex}\longrightarrow a_\lambda'^!u_{\lambda*}\]
  for the commutative diagram
  \[\begin{tikzcd}
    \mathscr{Y}_\lambda\arrow[d,"u_{0\lambda}"]\arrow[r,"b_\lambda"]&\mathscr{D}_\lambda\arrow[d,"u_\lambda"]\\
    \mathscr{X}_\lambda\arrow[r,"a_\lambda'"]&X\times_S \mathscr{X}
  \end{tikzcd}\]
  is an isomorphism by (\ref{10.2.10}). The unit ${\rm id}\stackrel{ad}\longrightarrow u_{\lambda*}u_\lambda^*$ is also an isomorphism since $\mathscr{T}$ satisfies strict \'etale descent. Thus by construction in (\ref{4.2.2}(ii)), the transition transformation $T_{\mathscr{D}_\lambda}$ is an isomorphism. Similarly, the other vertical arrows of (\ref{9.6.5.1}) are isomorphisms.

  By construction in (\ref{9.6.3}) using (\ref{9.1.3}), the conditions of (\ref{4.1.2}) are satisfied, so by (\ref{4.2.14}), the lower horizontal arrows of (\ref{9.6.5.1}) denoted by $(T^n)^{-1}$ and $T^d$ are isomorphisms. The lower horizontal arrow of (\ref{9.6.5.1}) denoted by $T_{\mathscr{X}_\lambda\times_{S_\lambda} \mathscr{X}_\lambda,E_\lambda}$ is an isomorphism by (\ref{9.5.1}), and the lower horizontal arrow of (\ref{9.6.5.1}) denoted by $T_{\mathscr{X}_\lambda\times_{S_\lambda} \mathscr{X}_\lambda}$ is an isomorphism by construction (\ref{4.2.2}(iii)). Thus we have shown that the lower horizontal arrows of (\ref{9.6.5.1}) are all isomorphisms, so the upper horizontal arrows of (\ref{9.6.5.1}) are also isomorphisms.

  Now, consider the commutative diagram
  \[\begin{tikzcd}
    \Omega_{f,g_\lambda,\mathscr{E}_\lambda'}^n\arrow[d,"T_{\mathscr{E}_\lambda,\mathscr{E}_\lambda'}"]\arrow[r,leftarrow,"(T^n)^{-1}"]& \Omega_{f_\lambda,g,\mathscr{E}_\lambda'}^d \arrow[d,"T_{\mathscr{E}_\lambda,\mathscr{E}_\lambda'}"]\arrow[r,"T^d"]& \Omega_{f_\lambda,g,\mathscr{E}_\lambda'}\arrow[d,"T_{\mathscr{E}_\lambda,\mathscr{E}_\lambda'}"] \arrow[r,"T_{\mathscr{D}_\lambda',\mathscr{E}_\lambda'}"]& \Omega_{f,g_\lambda,\mathscr{D}_\lambda'}\arrow[d,"T_{\mathscr{D}_\lambda,\mathscr{D}_\lambda'}"] \arrow[rr,"T_{\mathscr{X}_\lambda\times_{S_\lambda}\mathscr{X}_\lambda,\mathscr{D}_\lambda'}"]&&\Omega_{f,g_\lambda,\mathscr{X}_\lambda\times_{S_\lambda}\mathscr{X}_\lambda} \arrow[d,"T_{\mathscr{X}_\lambda\times_{S_\lambda}\mathscr{X}_\lambda}"]\\
    \Omega_{f,g_\lambda,\mathscr{E}_\lambda}^n\arrow[r,leftarrow,"(T^n)^{-1}"]& \Omega_{f_\lambda,g,\mathscr{E}_\lambda}^d \arrow[r,"T^d"]& \Omega_{f_\lambda,g,\mathscr{E}_\lambda} \arrow[r,"T_{\mathscr{D}_\lambda,\mathscr{E}_\lambda}"]&\Omega_{f,g_\lambda,\mathscr{D}_\lambda}\arrow[rr,"T_{\mathscr{D}_\lambda}"]&&\Omega_{f,g_\lambda}
  \end{tikzcd}\]
  of functors. Here, the arrows are defined by the $\mathscr{S}$-diagram versions of (\ref{4.2.2}), (\ref{4.2.1}), and (\ref{4.2.11}). We have shown that the upper horizontal arrows and the right side vertical arrow are isomorphisms, and the other vertical arrows are also isomorphisms by (\ref{4.2.2}) and (\ref{4.2.11}). The lower horizontal arrows are isomorphisms. In particular, the natural transformation
  \[\Omega_{f,g_\lambda,\mathscr{E}_\lambda}^n\stackrel{(T^n)^{-1}}\longleftarrow \Omega_{f_\lambda,g,\mathscr{E}_\lambda}^d \]
  is an isomorphism. Let $T^n$ denote its inverse.
\end{none}
\begin{none}\label{9.6.6}
  Under the notations and hypotheses of (\ref{9.6.5}), as in (\ref{4.2.5}), we have several exchange transformations (or inverse exchange transformations) as follows.
  \begin{enumerate}[(1)]
    \item We put $\Omega_{f,g,\lambda}=a'^!\lambda_*p_{2\lambda'}^*$. Then we have the natural transformations
    \[\lambda_*\Omega_{f,g_\lambda}\stackrel{Ex}\longrightarrow \Omega_{f,g,\lambda}\stackrel{Ex}\longrightarrow \Omega_{f,g}\lambda_*\]
    given by
    \[\lambda_*a_\lambda'^!p_{2\lambda}'^*\stackrel{Ex}\longrightarrow a'^!\lambda_*p_{2\lambda}'^*\stackrel{Ex^{-1}}\longrightarrow a'^!p_2'^*\lambda_*.\]
    Here, the first arrow is an isomorphism by (\ref{10.2.6}), and the second arrow is defined and an isomorphism by (\ref{10.2.4}) since $p_2'$ is Cartesian exact log smooth.
    \item We put $\Omega_{f,g,\mathscr{D},\lambda}=u_{0*}b^!\lambda_*q_{2\lambda}^*$. Then we have the natural transformations
    \[\lambda_*\Omega_{f,g_\lambda,\mathscr{D}_\lambda}\stackrel{Ex}\longrightarrow \Omega_{f,g,\mathscr{D},\lambda}\stackrel{Ex^{-1}}\longleftarrow\Omega_{f,g,\mathscr{D}}\lambda_*\]
    given by
    \[\lambda_*u_{0\lambda*}b_\lambda^!q_{2\lambda}^*\stackrel{\sim}\longrightarrow u_{0*}\lambda_*b_\lambda^!q_{2\lambda}^*\stackrel{Ex}\longrightarrow u_{0*}b^!\lambda_*q_{2\lambda}^*\stackrel{Ex}\longleftarrow u_{0*}b^!q_2^*\lambda_*.\]
    \item We have the {\it inverse} exchange transformation
    \[\Omega_{f,g,\mathscr{E}}\lambda_*\stackrel{Ex^{-1}}\longrightarrow \lambda_*\Omega_{f,g_\lambda,\mathscr{E}_\lambda}\]
    given by
    \[u_{0*}c^!r_2^*\lambda_*\stackrel{Ex}\longrightarrow u_{0*}c^!\lambda_*r_{2\lambda}^*\stackrel{Ex^{-1}}\longrightarrow u_{0*}\lambda_*c_\lambda^!r_{2\lambda}^*\stackrel{\sim}\longrightarrow \lambda_*u_{0\lambda*}c_\lambda^!r_{2\lambda}^*.\]
    Here, the second arrow is defined and an isomorphism by (\ref{10.2.10}) since $c$ is a Cartesian strict closed immersion.
    \item We have the {\it inverse} exchange transformation
    \[\Omega_{f,g,\mathscr{E}}^d\lambda_*\stackrel{Ex^{-1}}\longrightarrow \lambda_*\Omega_{f,g_\lambda,\mathscr{E}_\lambda}^d\]
    given by
    \[\begin{split}
      u_{0*}\phi_*d^!s_2^*\pi^*\lambda_*&\stackrel{Ex}\longrightarrow u_{0*}\phi_*d^!s_2^*\lambda_*\pi_\lambda^*\stackrel{Ex}\longrightarrow u_{0*}\phi_*d^!\lambda_*s_{2\lambda}^*\pi_\lambda^*\\
      &\stackrel{Ex^{-1}}\longrightarrow u_{0*}\phi_*\lambda_*d_\lambda^!s_{2\lambda}^*\pi_\lambda^*\stackrel{\sim}\longrightarrow \lambda_*u_{0\lambda*}\phi_{\lambda*}d_\lambda^!s_{2\lambda}^*\pi_\lambda^*.
    \end{split}\]
    Here, the third arrow is defined and an isomorphism by (\ref{10.2.10}) since $d$ is a Cartesian strict closed immersion.
    \item We have the {\it inverse} exchange transformation
    \[\Omega_{f,g,\mathscr{E}}^n\lambda_*\stackrel{Ex^{-1}}\longrightarrow \lambda_*\Omega_{f,g_\lambda,\mathscr{E}_\lambda}^n\]
    given by
    \[u_{0*}e^!t_2^*\lambda_*\stackrel{Ex}\longrightarrow u_{0*}e^!\lambda_*t_{2\lambda}^*\stackrel{Ex^{-1}}\longrightarrow u_{0*}\lambda_*e_\lambda^!t_{2\lambda}^* \stackrel{\sim}\longrightarrow \lambda_*u_{0\lambda*}e_\lambda^!t_{2\lambda}^*.\]
    Here, the second arrow is defined and an isomorphism by (\ref{10.2.10}) since $e$ is a Cartesian strict closed immersion.
    \item We have the {\it inverse} exchange transformation
    \[\Omega_{f,g}^n\lambda_*\stackrel{Ex^{-1}}\longrightarrow \lambda_*\Omega_{f,g_\lambda}^n\]
    given by
    \[e'^!t_2'^*\lambda_*\stackrel{Ex}\longrightarrow e'^!\lambda_*t_{2\lambda}'^*\stackrel{Ex^{-1}}\longrightarrow \lambda_*e_\lambda'^!t_{2\lambda}'^*.\]
    Here, the first arrow is an isomorphism by (\ref{10.2.4}) since $t_2'$ is Cartesian exact log smooth, and the second arrow is defined and an isomorphism by (\ref{10.2.10}) since $e'$ is a Cartesian strict closed immersion. Thus the composition is also an isomorphism.
    \item We have the exchange transformation
    \[\lambda^*\Omega_{f,g,\mathscr{E}}^d\stackrel{Ex}\longrightarrow \Omega_{f,g_\lambda,\mathscr{E}_\lambda}^d\lambda^*\]
    given by
    \[\begin{split}
      \lambda^*u_{0*}\phi_*d^!s_2^*\pi^*&\stackrel{Ex}\longrightarrow u_{0\lambda*}\lambda^*\phi_*d^!s_2^*\pi^*\stackrel{Ex}\longrightarrow u_{0\lambda*}\phi_{\lambda*}\lambda^*d^!s_2^*\pi^*\\
      &\stackrel{Ex}\longrightarrow u_{0\lambda*}\phi_{\lambda*}d_\lambda^!\lambda^*s_2^*\pi^*\stackrel{\sim}\longrightarrow u_{0\lambda*}\phi_{\lambda*}d_\lambda^!s_{2\lambda}^*\pi_\lambda^*\lambda^*.
    \end{split}\]
    Here, the first and second arrows are isomorphisms by (\ref{10.2.7}), and the third arrow is defined and an isomorphism by (\ref{10.2.9}) since $d$ is a Cartesian strict closed immersion. Thus the composition is an isomorphism.
    \item We have the exchange transformation
    \[\lambda^*\Omega_{f,g,\mathscr{E}}^n\stackrel{Ex}\longrightarrow \Omega_{f,g_\lambda,\mathscr{E}_\lambda}^n\lambda^*\]
    given by
    \[\lambda^*u_{0*}e^!t_2^*\pi^*\stackrel{Ex}\longrightarrow u_{0\lambda*}\lambda^*e^!t_2^*\pi^*\stackrel{Ex}\longrightarrow u_{0\lambda*}e_\lambda^!\lambda^*t_2^*\pi^* \stackrel{\sim}\longrightarrow w_{\lambda*}e_\lambda^!t_{2\lambda}^*\pi_\lambda^*\lambda^*.\]
    Here, the first arrow is an isomorphism by (\ref{10.2.7}), and the second arrow is defined and an isomorphism by (\ref{10.2.10}) since $e$ is a Cartesian strict closed immersion. Thus the composition is an isomorphism.
  \end{enumerate}
  We also have the natural transformation
  \[\Omega_{f,g,\mathscr{D},\lambda}\stackrel{T_{\mathscr{D}}}\longrightarrow \Omega_{f,g,\lambda}\]
  given by
  \[u_{0*}b^!\lambda_*q_{2\lambda}^*\stackrel{Ex}\longrightarrow a'^!u_*\lambda_*q_{2\lambda}^*\stackrel{\sim}\longrightarrow a'^!\lambda_*u_{\lambda*}u_\lambda^*p_{2\lambda}^* \stackrel{ad^{-1}}\longrightarrow a'^!\lambda_*p_{2\lambda}^*.\]
  Here, the third arrow is defined and an isomorphism since $\mathscr{T}$ satisfies strict \'etale descent.
\end{none}
\begin{none}\label{9.6.7}
  Under the notations and hypotheses of (\ref{9.6.6}), for $\lambda\in I$, we have the commutative diagram
  \[\begin{tikzcd}
    \lambda^*\Omega_{f,g,\mathscr{E}}^n\arrow[d,"Ex"]\arrow[r,"(T^n)^{-1}"]&\lambda^*\Omega_{f,g,\mathscr{E}}^d\arrow[d,"Ex"]\\
    \Omega_{f,g_\lambda,\mathscr{E}_\lambda}^n\lambda^*\arrow[r,"(T^n)^{-1}"]&\Omega_{f,g_\lambda,\mathscr{E}_\lambda}^d\lambda^*
  \end{tikzcd}\]
  of functors. By (loc.\ cit), the vertical arrows are isomorphisms, and by (\ref{9.6.5}), the lower horizontal arrow is an isomorphism. Thus the upper horizontal arrow is also an isomorphism. Then by (PD--4), the natural transformation
  \[\Omega_{f,g,\mathscr{E}}^n\stackrel{(T^n)^{-1}}\longleftarrow \Omega_{f,g,\mathscr{E}}^d\]
  is an isomorphism. Let $T^n$ denote its inverse.

  Now, consider the commutative diagram
    \[\begin{tikzpicture}[baseline= (a).base]
    \node[scale=.98] (a) at (0,0)
    {\begin{tikzcd}
    \lambda_*\Omega_{f,g_\lambda}^n\arrow[dd,"Ex^{-1}",leftarrow]\arrow[r,"T^{n'}"]&\lambda_*\Omega_{f,g_\lambda,\mathscr{E}_\lambda}^{n}\arrow[dd,"Ex^{-1}",leftarrow] \arrow[r,"T^n"]&\lambda_* \Omega_{f,g_\lambda,\mathscr{E}_\lambda}^d\arrow[dd,"Ex^{-1}",leftarrow]\arrow[r,"T^d"]&\lambda_*\Omega_{f,g_\lambda,\mathscr{E}_\lambda}  \arrow[dd,"Ex^{-1}",leftarrow] \arrow[r,"T_{\mathscr{D}_\lambda,\mathscr{E}_\lambda}"] &\lambda_*\Omega_{f,g_\lambda,\mathscr{D}_\lambda}\arrow[d,"Ex"]\arrow[r,"T_{\mathscr{D}_\lambda}"]& \lambda_*\Omega_{f,g_\lambda}\arrow[d,"Ex"]\\
    &&&&\Omega_{f,g,\lambda}\arrow[d,"Ex^{-1}",leftarrow]\arrow[r,"T_{\mathscr{D}}"]&\Omega_{f,g,\lambda}\arrow[d,"Ex"]\\
    \Omega_{f,g}^n\lambda_*\arrow[r,"T^{n'}"]&\Omega_{f,g,\mathscr{E}}^n\lambda_*\arrow[r,"T^n"]&\Omega_{f,g,\mathscr{E}}^d\lambda_*\arrow[r,"T^d"]& \Omega_{f,g,\mathscr{E}}\lambda_*\arrow[r,"T_{\mathscr{D},\mathscr{E}}"]&\Omega_{f,g,\mathscr{D}}\lambda_*\arrow[r,"T_\mathscr{D}"]&\Omega_{f,g}\lambda_*
  \end{tikzcd}};
  \end{tikzpicture}\]
  of functors. Here, the arrows are constructed in (\ref{9.6.6}), (\ref{9.6.9}), and the $\mathscr{S}$-diagram version of (\ref{4.2.2}). The top horizontal arrows are isomorphisms by (\ref{9.6.9}) and (\ref{9.6.6}), and we have shown in (loc.\ cit) that the left side vertical and the right side vertical arrows are isomorphisms. Thus the composition of the five lower horizontal arrows
  \[\Omega_{f,g}^n\lambda_*\longrightarrow \Omega_{f,g}\lambda_*\]
  is an isomorphism. Then its left adjoint
  \[\lambda^*\Sigma_{f,g}\longrightarrow \Sigma_{f,g}^n\lambda^*\]
  is also an isomorphism where
  \[\Sigma_{f,g}=p_{2\sharp}'a_*',\quad \Sigma_{f,g}^n=t_{2\sharp}'e_*'.\]

  We also denote by $T_\mathscr{E}^n$ the composition
  \[\Omega_{f,g}^n\stackrel{T^{n'}}\longrightarrow \Omega_{f,g,\mathscr{E}}^n\stackrel{T^n}\longrightarrow \Omega_{f,g,\mathscr{E}}^d\stackrel{T^d}\longrightarrow \Omega_{f,g,\mathscr{E}} \stackrel{T_{\mathscr{D},\mathscr{E}}}\longrightarrow \Omega_{f,g,\mathscr{D}}\stackrel{T_{\mathscr{D}}}\longrightarrow \Omega_{f,g}.\]
  It is called again a {\it transition transformation.}
  Then its left adjoint
  \[\Sigma_{f,g}\longrightarrow \Sigma_{f,g}^n\]
  is an isomorphism by (PD--4) and the above paragraph. Therefore, we have proven the following theorem.
\end{none}
\begin{thm}\label{9.6.8}
  Under the notations and hypotheses of (\ref{9.6.6}), the transition transformation
  \[\Omega_{f,g}^n\stackrel{T_\mathscr{E}^n}\longrightarrow \Omega_{f,g}\]
  is an isomorphism.
\end{thm}
\begin{none}\label{9.6.11}
  Under the notations and hypotheses of (\ref{9.6.6}), we put
  \[\Omega_{f,X\times_S \mathscr{X}}=h_*\Omega_{f,g}h^*,\quad \Omega_f^n=h_*\Omega_{f,g}^nh^*.\]
  Then the natural transformation
  \[\Omega_f^n=h_*\Omega_{f,g}^nh^*\stackrel{T_\mathscr{E}^n}\longrightarrow h_*\Omega_{f,g}h^*=\Omega_{f,X\times_S \mathscr{X}}\]
  is an isomorphism by (\ref{9.6.8}). We also have the natural transformation
  \[T_{X\times_S \mathscr{X}}:\Omega_{f,X\times_S \mathscr{X}}\longrightarrow \Omega_f\]
  given by
  \[h_*a'^!p_2'^*h^*\stackrel{Ex}\longrightarrow a^!h_*'p_2'^*h^*\stackrel{\sim}\longrightarrow a^!h_*'h'^*p_2^*\stackrel{ad^{-1}}\longrightarrow a^!p_2^*.\]
  Here, the first arrow is an isomorphism by (\ref{10.2.10}), and the third arrow is defined and an isomorphism since $\mathscr{T}$ satisfies strict \'etale descent. Thus the composition is also an isomorphism.

  Now, consider the natural transformations
  \[\Omega_f^n \stackrel{T_\mathscr{E}^n}\longrightarrow \Omega_{f,X\times_S \mathscr{X}}\stackrel{T_{X\times_S \mathscr{X}}}\longrightarrow \Omega_f.\]
  The composition is also denoted by $T_\mathscr{E}^n$. It is an isomorphism by (\ref{9.6.8}) and the above paragraph.

  Then consider the natural transformations
  \[\Omega_f^nf^!\stackrel{T_\mathscr{E}^n}\longrightarrow \Omega_{f}f^!\stackrel{\mathfrak{q}_f} \longrightarrow f^*.\]
  The composition is denoted by $\mathfrak{q}_{f,\mathscr{E}}^n$. By (\ref{9.5.5}) and the above paragraph, we have proven the following theorem.
\end{none}
\begin{thm}\label{9.6.10}
  Under the notations and hypotheses of (\ref{9.6.9}), the natural transformation
   \[\Omega_f^nf^!\stackrel{\mathfrak{q}_{f,\mathscr{E}}^n}\longrightarrow f^*.\]
  is an isomorphism.
\end{thm}
\subsection{Canonical version of purity transformations}
\begin{none}\label{9.7.1}
  Let $f:X\rightarrow S$ be a separated vertical exact log smooth morphism of $\mathscr{S}$-schemes. The category of localized compactified exactifications of the diagonal morphism $a:X\rightarrow X\times_S X$, denoted by $\mathcal{LCE}_a$, is the category whose object is the data of $v_\lambda:\mathscr{E}_i\rightarrow \mathscr{X}_\lambda\times_{S_\lambda}\mathscr{X}_\lambda$ and commutative diagram
  \[\begin{tikzcd}
    \mathscr{X}_\lambda\arrow[r,"h_\lambda"]\arrow[d,"f_\lambda"]&X\arrow[d,"f"]\\
    S_\lambda\arrow[r,"l_\lambda"]&S
  \end{tikzcd}\]
  for $\lambda\in I$ where
  \begin{enumerate}
    \item $I$ is a set, and the diagram commutes,
    \item $f_\lambda$ and $l_\lambda$ are strict \'etale,
    \item $v_i$ is a compactified exactification of the diagonal morphism $\mathscr{X}_\lambda\rightarrow \mathscr{X}_\lambda\times_{S_\lambda}\mathscr{X}_\lambda$.
  \end{enumerate}
  Morphism is the data of
  \[S_\lambda'\rightarrow S_\lambda,\quad \mathscr{X}_\lambda'\rightarrow \mathscr{X}_\lambda,\quad \mathscr{E}_\lambda'\rightarrow \mathscr{E}_\lambda\]
  compatible with the morphisms in (\ref{9.6.3.1}).

  Then $\mathcal{LCE}_a$ is not empty by (\ref{9.6.3}), and as in (\ref{9.1.6}), it is connected since we can take the fiber products of $(\mathscr{X}_\lambda,S_\lambda,\mathscr{E}_\lambda)_{\lambda\in I}$ and $(\mathscr{X}_\lambda',S_\lambda',\mathscr{E}_\lambda')_{\lambda\in I}$. For any object $\omega$ of $\mathcal{LCE}_a$, as in (\ref{9.6.8}), we can associate the natural transformation
  \[\Omega_{f}^n\stackrel{T_\omega^n}\longrightarrow \Omega_{f}.\]
  Then as in (\ref{4.2.12}), we have the compatibility, i.e., this defines the functor
  \[T^n:\mathcal{LCE}_a\rightarrow {\rm Hom}(\Omega_f^n,\Omega_f).\]
  To make various natural transformations $T_\omega^n$ canonical, we take the limit
  \[\varprojlim_{\omega}T_\omega^n.\]
  It is denoted by $T^n:\Omega_f^n\rightarrow \Omega_f$. Now, the definition of the {\it purity transformation} is the composition
  \[\Omega_f^n f^!\stackrel{T^n}\longrightarrow \Omega_f f^!\stackrel{\mathfrak{q}_f}\longrightarrow f^*,\]
  and it is denoted by $\mathfrak{q}_f^n$. By (\ref{9.6.10}), we have the following theorem.
\end{none}
\begin{thm}\label{9.7.2}
  Let $f:X\rightarrow S$ be a separated vertical exact log smooth morphism. Then the purity transformation
   \[\Omega_f^nf^!\stackrel{\mathfrak{q}_f^n}\longrightarrow f^*.\]
  is an isomorphism.
\end{thm}
\titleformat*{\section}{\center \scshape }


\begin{thebibliography}{9}
\small
  \bibitem[Ayo07]{Ayo07} J. Ayoub, {\it Les six op\'erations de Grothendieck et le formalisme des cycles \'evanescents dans le monde motivique}, Ast\'erisque, vol. {\bf 314} and {\bf 315} (2007).
  \bibitem[CD12]{CD12} D.-C. Cisinski and F. D\'eglise, {\it Triangulated categories of mixed motives}, preprint, arXiv:0912.2110v3, 2012.
  \bibitem[CD13]{CD13} D.-C. Cisinski and F. D\'eglise, {\it \'Etale motives}, Compositio Mathematica, \'a para\^itre, 2015.
  \bibitem[CLS11]{CLS11} D. Cox, J. Little, and H. Schenck, {\it Toric varieties}, Graduate Studies in Math. vol. 124, AMS, 2011.
  \bibitem[EGA]{EGA} J.\ Dieudonn'e and A.\ Grothendieck. --- {\it \'El\'ements de g\'eom\'etrie alg\'ebrique.} Inst.\ Hautes \'Etudes Sci.\ Publ.\ Math. \textbf{4, 8, 11, 17, 20, 24, 28, 32}, (1961-1967).
  \bibitem[Ful98]{Ful98} W. Fulton, {\it Intersection theory}, 2nd ed., Ergebnisse der Mathematikund ihrer Grenzgebiete. 3. Folge. A Series of Modern Surveys in Mathematics, Springer-Verlag, Berlin, vol. {\bf 2} (1998).
  \bibitem[Kat00]{Kat00} F. Kato, {\it Log smooth deformations and moduli of smooth curves}, Internat. J. Math. {\bf 11} (2000), no. 2, 215--232.
  \bibitem[Nak97]{Nak97} C. Nakayama, {\it Logarithmic \'etale cohomology}, Math. Ann. {\bf 308}, (1997), 365--404.
  \bibitem[Nak09]{Nak09} C. Nakayama, {\it Quasi-sections in log geometry}, Osaka J. Math. {\bf 46} (2009), 1163–-1173
  \bibitem[NO10]{NO10} C. Nakayama and A. Ogus, {\it Relative rounding in toric and logarithmic geometry}, Geometry \& Topology {\bf 14} (2010), 2189–2241.
  \bibitem[Ols03]{Ols03} M. Olsson, {\it Logarithmic geometry and algebraic stacks}, Ann. Sci. d'ENS 36 (2003), 747--791.
  \bibitem[Ogu17]{Ogu17} A. Ogus, {\it Lectures on Logarithmic Algebraic Geometry}, preprint, Version of April 18, 2017.
    \bibitem[SGA4]{SGA4} M. Artin, A. Grothendieck, and J.-L. Verdier, {\it Th\'eorie des topos et cohomologie \'etale des sch\'emas,}
Lecture Notes in Mathematics, vol. {\bf 269, 270, 305}, Springer-Verlag, 1972–1973, S\'eminaire de G\'eom\'etrie
Alg\'ebrique du Bois–Marie 1963–64 (SGA 4).
\end{thebibliography}
\end{document}